\documentclass[11pt, a4paper]{article}
\usepackage{t1enc}
\usepackage[latin1]{inputenc}
\usepackage[english]{babel}

\usepackage{amsmath,amsthm,amssymb}
\usepackage{amsfonts}
\usepackage{latexsym}
\usepackage[dvips]{graphicx}
\usepackage{graphicx}
\usepackage{float}
\usepackage[natural]{xcolor}
\usepackage{algorithm}
\usepackage{algorithmic}
\usepackage{enumerate,enumitem}
\usepackage{mathrsfs}
\usepackage{multirow}
\usepackage{xcolor}
\usepackage[colorlinks,linkcolor=blue]{hyperref}
\usepackage{lineno}
\usepackage{setspace}
\usepackage{fullpage}
\usepackage[title]{appendix}
\usepackage{todonotes}
\usepackage{amssymb}
\usepackage{psfrag}
\usepackage{listings}
\usepackage{bbm}

\usepackage{fullpage}
\usepackage{hyperref}
\usepackage[margin=2.3cm]{geometry}
\usepackage{enumitem,verbatim}

\newtheorem{theorem}{Theorem}[section]
\newtheorem{claim}[theorem]{Claim}

\newtheorem{lemma}[theorem]{Lemma}

\newtheorem{corollary}[theorem]{Corollary}

\newtheorem{question}[theorem]{Question}

\newtheorem{observation}[theorem]{Observation}
\newenvironment{claimproof}[1][Proof of Claim]
  {\begin{proof}[#1]}
  {\end{proof}}

\theoremstyle{definition}
\newtheorem{problem}{Problem}

\newtheorem{definition}[theorem]{Definition}

\title{\LARGE Transversal Hamilton cycles in digraph collections}

\author{
Yangyang Cheng\thanks{Mathematical Institute, University of Oxford, Oxford, UK. Email: {\tt yangyang.cheng@maths.ox.ac.uk}. Supported by the PhD studentship of ERC Advanced Grant (883810).}
\and
Heng Li \thanks{School of Mathematics, Shandong University, Jinan, China. Email: {\tt heng.li@sdu.edu.cn}.  Supported by National Natural Science Foundation of China (12501487)}
\and
Wanting Sun \thanks{Data Science Institute, Shandong University, Jinan, China. Email: {\tt wtsun@sdu.edu.cn}. Supported by the Natural Science Foundation of China (12501488), the China Postdoctoral Science Foundation (2023M742092) and the Natural Science Foundation of Shandong Province (ZR2024QA023).}
\and
Guanghui Wang \thanks{School of Mathematics, and State Key Laboratory of Cryptography and Digital Economy Security, Shandong University, Jinan, China. Email: {\tt ghwang@sdu.edu.cn}. Supported by  the Natural Science Foundation of China (12231018).}
}
\date{}

\begin{document}
\maketitle

\begin{abstract}
Given a collection $\mathcal{D} =\{D_1,D_2,\ldots,D_m\}$ of digraphs on the common vertex set $V$, an $m$-edge digraph $H$ with vertices in $V$ is \textit{transversal} in $\mathcal{D}$ if there exists a bijection $\varphi :E(H)\rightarrow [m]$ such that 
$e \in E(D_{\varphi(e)})$ for all $e\in E(H)$.  Ghouila-Houri proved that any $n$-vertex digraph with minimum semi-degree at least $\frac{n}{2}$ contains a directed Hamilton cycle. In this paper, we provide a transversal generalization of Ghouila-Houri's theorem, thereby solving a problem proposed by Chakraborti, Kim, Lee and Seo \cite{2023Tournament}.  Our proof utilizes the absorption method for transversals, the regularity method for digraph collections, as well as the transversal blow-up lemma \cite{cheng2023transversals} and the related machinery. As an application, when $n$ is sufficiently large, our result implies the transversal version of Dirac's theorem, which was proved by Joos and Kim \cite{2021jooskim}.
\end{abstract}

\maketitle

%\linenumbers

\section{Introduction}
A transversal $X$ over a collection $\mathcal{F}=\{F_1,\ldots,F_m\}$ of objects represents an object that intersects every $F_i$. Here, the object can be any mathematical object, such as sets, spaces, set systems, matroids, and so forth. Several classical theorems, which guarantee the existence of certain objects, have been restated in the context of transversals, illustrating the feasibility or infeasibility of obtaining such an object as a transversal over a certain collection. To name a few, transversals are extensively studied for Carath\'eodory's theorem \cite{1982Barany}, Helly's theorem \cite{2005Kalai_Helly},  Erd\H{o}s-Ko-Rado theorem \cite{2017AhaEKR}, Rota's basis conjecture \cite{1994HuangLatin,2020PokrovskiyRota}, and others. 

Following this trend, there has been ample research studying transversal generalizations of classical results
in graph theory. These generalizations are not only interesting in their own right but also often provide a strengthening of the original results, reflecting the robustness of graph properties to some extent.  Transversals over a collection of graphs have been implicitly used in the literature under the name
of rainbow coloring and explicitly defined in \cite{2021jooskim} as follows.

\begin{definition}
     For a given collection $\mathcal{G}=\{G_1,\ldots,G_m\}$ of graphs/hypergraphs/digraphs with the same vertex set $V$, an $m$-edge graph/hypergraph/digraph $H$ with vertices in $V$ is  \textit{transversal} in $\mathcal{G}$ if there exists a bijection ${\varphi}:E(H)\to[m]$ such that $e\in E(G_{{\varphi}(e)})$ for all $e\in E(H).$
\end{definition}

By interpreting each $G_i$ as the set of edges colored with color $i$, the function ${\varphi}$ is often referred to as a coloring. We say that the coloring ${\varphi}$ is \textit{rainbow} if it is injective, and $H$ is a \textit{rainbow subgraph} if there exists a rainbow coloring on it. Many classical results in extremal graph theory have been extended to such transversal settings, exhibiting interesting phenomena. The central question in studying transversal generalizations can be stated as follows. 

\begin{problem}[\cite{2021jooskim}]
    Let $H$ be a graph/hypergraph/digraph with $m$ edges, $\mathcal{G}=\{G_1,G_2,\dots,G_m\}$ be a collection of graphs/hypergraphs/digraphs on a common {{vertex}} set of size $n$. Which property $\mathcal{P}$ imposed on $\mathcal{G}$ will guarantee the existence of a transversal copy of $H$?
\end{problem}

Since every object in the family $\mathcal{G}$ can be the same, a necessary condition for a positive answer to the
above question is that every $n$-vertex graph/hypergraph/digraph satisfying the property $\mathcal{P}$ contains a copy of $H$. However, this alone is not always sufficient. For example, Aharoni, DeVos, de la Maza, Montejano, and \v{S}\'{a}mal 
\cite{2019Aharoni} proved that if $\mathcal{G}=\{G_1, G_2,G_3\}$  is a collection of graphs on a common vertex set of size $n$ and $|E(G_i)|>(\frac{26-2\sqrt{7}}{81})n^2$ for all $i\in [3]$, then  $\mathcal{G}$ contains a rainbow triangle. Moreover,  the constant $\frac{26-2\sqrt{7}}{81}$ is optimal, while Mantel's theorem states that any $n$-vertex graph with more than $\lfloor\frac{n^2}{4}\rfloor$ edges must contain a triangle.

In 2020, Aharoni \cite{2019Aharoni} conjectured that Dirac's theorem can be extended to a transversal version. Cheng, Wang and Zhao \cite{2021ChengWangZhao} solved  this conjecture asymptotically, and it was completely confirmed by Joos and Kim \cite{2021jooskim}. 
\begin{theorem}[\cite{2021jooskim}]\label{Ahaconj}
Suppose $\mathcal{G} = \{G_1, \ldots, G_n\}$ is a collection of graphs with the same vertex set of size $n$. If $\delta(G_i)\geq \frac{n}{2}$ for all $i\in [n]$, then $\mathcal{G}$ contains a transversal Hamilton cycle.
   % Given graphs $G_1,\dots, G_n$ on the same vertex set of size $n$, each having minimum degrees at least $\frac{n}{2}$, there exists a rainbow Hamilton cycle: a cycle with edge-set $\{e_1,\dots,e_n\}$ such that $e_i\in E(G_i)$ for $i=1,\ldots,n$.
\end{theorem}

There are numerous established results on transversal problems in graph collections.  For more results on transversals, we refer the reader to the survey \cite{surveytransversal}.

A natural next step is to study {{the rainbow analogues}} in digraph collections. Cheng, Han, Wang and Wang \cite{2023chengspan} studied the transversal version of Hajnal-Szemer\'{e}di theorem in digraph collections, which asymptotically generalized the corresponding results by  Czygrinow-DeBiasio-Kierstead-Molla \cite{Czygrinow} and Treglown \cite{Treglown}. 
It is well-known that every tournament contains a Hamilton path, and every strongly connected tournament contains a Hamilton cycle.  Chakraborti, Kim, Lee and Seo \cite{2023Tournament} proved the existence of a transversal directed Hamilton path and cycle in a family of tournaments. They \cite{2024KimTour} also proved that in a digraph collection consisting of tournaments, there exist transversal Hamilton cycles of all possible orientations possibly except the consistently oriented one. These results generalized the classical theorem of Thomason \cite{1986ThomasonTour}.  Babi{\'n}ski, Grzesik and Prorok \cite{2023BabinskiTria} studied the transversal Mantel's theorem in digraph and oriented graph collections. Gerbner, Grzesik, Palmer and  Prorok \cite{2024GerbnerDstar} determined the minimum number of edges in each color which guarantees the existence {{of}} any rainbow directed star. %For more results on transversals, we refer the reader to the survey \cite{surveytransversal}. 

\subsection{Main result}
In this paper, we study the transversal problem in digraph collections. Let $D$ be a digraph with no loops and at most one edge in each direction between every pair of vertices. Denote the set of vertices of $D$ by $V(D)$. Define $\delta^{+}(D)$ and $\delta^{-}(D)$ as the minimum out-degree and in-degree of $D$, respectively, and let the \textit{minimum semi-degree} be $\delta^0(D) = \min{\delta^+(D), \delta^-(D)}$. For a collection $\mathcal{D} = \{D_1, \dots,  D_m\}$ of digraphs (not necessarily distinct) on a common vertex set $V$, define $\delta^0(\mathcal{D}) =\min\{\delta^0(D_i):i\in [m]\}$.

%A tournament is an oriented graph where every pair of distinct vertices are adjacent.
%Let $\delta^{+}(D)$ and $\delta^{-}(D)$ be the minimum out-degree and minimum in-degree  of $D$, respectively. Let $\delta^0(D)=\min\{\delta^+(D),{\delta^-}(D)\}$ be the \textit{ minimum semi-degree} of $D.$ Let $\mathcal{D} =\{D_1, D_2,\dots,D_m\}$ be a collection of not necessarily distinct digraphs with common vertex set $V$. Define $\delta^0(\mathcal{D}):=\min\{\delta^0(D_i):i\in [m]\}$. 

In 1960, Ghouila-Houri \cite{Houri} generalized  Dirac's theorem to digraphs.
\begin{theorem}[\cite{Houri}]\label{Ghouila-Houri}
    For each integer $n\geq 3$, every $n$-vertex digraph $D$ with $\delta^0(D)\geq \frac{n}{2}$ contains a directed Hamilton cycle.
\end{theorem}
Chakraborti, Kim, Lee, and Seo \cite{2023Tournament} suggested that it would be interesting to consider a transversal version of Ghouila-Houri's theorem. Our main result solves this.
%Chakraborti, Kim, Lee, Seo \cite{2023Tournament} suggested to consider a transversal version of Ghouila-Houri's theorem would be interesting. Our main result solves this problem.
\begin{theorem}\label{mainthm}
 Let $n$ be a sufficiently large integer. Suppose $\mathcal{D} = \{D_1, \ldots, D_n\}$ is a collection of digraphs on a common vertex set of size $n$. If $\delta^0(\mathcal{D})\geq \frac{n}{2}$, then $\mathcal{D}$ contains a transversal directed Hamilton cycle.
\end{theorem}
In an undirected graph, one may replace each undirected edge by two arcs with opposite directions. Thus, for sufficiently large $n$, Theorem \ref{Ahaconj} is a direct consequence of Theorem \ref{mainthm}. Furthermore, by setting $D_1 = D_2 = \cdots = D_n$ in Theorem \ref{mainthm}, we immediately obtain Theorem \ref{Ghouila-Houri} for all sufficiently large $n$.

%by letting $D_1=D_2=\cdots=D_n$ in Theorem \ref{mainthm}, we get Theorem \ref{Ghouila-Houri} immediately when $n$ is large. 

\subsection{Notation and Organization}
{\bf Notation}. Throughout the paper, we use standard graph theory notation and terminology. Let $D$ be a digraph with  vertex set $V(D)$ and edge set $E(D)$, and sometimes we write $e(D):=|E(D)|$. For $x,y\in V(D)$, denote by $xy$ the directed edge from $x$ to $y$ and we always write $xy\in D$. Denote by $P_n$ and $C_n$ the directed path and directed cycle on $n$ vertices, respectively. 

We use $H\subseteq D$ to denote that $H$ is a sub-digraph of $D$. For $S\subseteq V(D)$, we write $D[S]$ for the sub-digraph of $D$ induced by $S$ and $D-S$ for $D[V(D)\setminus S]$. For two not necessarily disjoint vertex subsets 
$A,B\subseteq V(D)$, let $D[A,B]$ be the sub-digraph of $D$ with vertex set $A\cup B$ and edge set $\{xy\in E(D):x\in A, y\in B\}$ and let $e_{D}(A, B):=|E(D[A,B])|$. Denote $D^{\pm}[A,B]=D[A,B]\cup D[B,A]$. When $A=B$, we often abbreviate $D[A,B]$ and $e_{D}(A, B)$ as $D[A]$ and $e_D(A)$  respectively.  

For a digraph $D$ and $v\in V(D)$, we denote the sets of out-neighbors and in-neighbors of $v$ by $N_D^+(v)$ and $N_D^-(v)$, respectively, with the corresponding out- and in-degrees given by $d_D^+(v) := |N_D^+(v)|$ and $d_D^-(v) := |N_D^-(v)|$. 
%we denote the set of out-neighbours and in-neighbours of $v$ by $N_D^+(v)$ and $N_D^-(v)$, respectively, and we write $d_D^+(v):=|N_D^+(v)|$ and $d_D^-(v):=|N_D^-(v)|$ for the out- and indegree of $v$ in $D$, respectively. 
For $S\subseteq V(D)$, we write $d_D^+
(v,S):=|N_D^+(v)\cap S|$ and $d_D^-
(v,S):=|N_D^-(v)\cap S|$. %{\color {blue} We write $\delta^+(D)$ and $\delta^-(D)$, respectively, for the minimum out- and indegree of $D$ and $\delta^0(D):=\min\{\delta^+(D),\delta^-(D)\}$ for the \textit{minimum semi-degree} of $D$.}  
We often omit subscripts when they are clear from the context.

Let $\mathcal{D}=\{D_i:i\in [m]\}$ be a collection of digraphs on a common vertex set $V$ and $H$ be a rainbow digraph inside $\mathcal{D}$.  For a vertex $v\in V$ and each $i\in [m]$, let $N_i^+(v):=N_{D_i}^+(v)$, $N_i^-(v):=N_{D_i}^-(v)$, {$d_i^+(v):=|N_i^+(v)|$ and $d_i^-(v):=|N_i^-(v)|$}. For any two vertex subsets $X,\,Y \subseteq V$, define $\mathcal{D}[X]:=\{D_i[X]:i\in [m]\}$, $\mathcal{D}[X,Y]:=\{D_i[X,Y]:i\in [m]\}$ and $\mathcal{D}^{\pm}[X,Y]:=\{D_i^{\pm}[X,Y]:i\in [m]\}$. If $E(D)=\emptyset$, then we simply write $D= \emptyset$. Similarly, if there is no edge in $D_i$ for all $i \in [m]$, then we write $\mathcal{D}= \emptyset$. Denote by ${\rm col}(H)$ the set of colors appearing in $H$.  Two rainbow graphs $H_1$ and $H_2$ are said to be \textit{disjoint} if $V(H_1)\cap V(H_2)=\emptyset$ and ${\rm col}(H_1)\cap {\rm col}(H_2)=\emptyset$.

The notation above extends to undirected graphs in the obvious ways. For $n\in \mathbb{N}$, we denote the set $\{1,2,\ldots,n\}$ by $[n]$. For any two constants $\alpha,\beta\in(0,1)$, we write $\alpha\ll\beta$ if there exists a
 function $f=f(\beta)$ such that the subsequent arguments hold for all $0<\alpha\leq f.$ 
 \medskip

\noindent{\bf Organization.} We conclude this section with a sketch of the proof of Theorem \ref{mainthm}. In Section~2, we introduce two key techniques that we will   need later, i.e., the regularity lemma in digraph collections and a stability result for transversal perfect matchings in bipartite graph collections. For readability, we postpone {their} proofs to the appendix. The proof of Theorem \ref{mainthm} will be divided into ``stable case'' and ``extremal case''. Section~3 introduces notions of extremality and stability for digraphs and digraph collections; we prove several results
about absorption that will be used in the proof of the stable case. In {Section 4}, we show that there is always a transversal directed Hamilton cycle in {the} stable case. {Sections 5-7} deal with the extremal case. We finish our paper with some concluding remarks in the final section. 

\subsection{Discussion of the strategy}
Joos and Kim \cite{2021jooskim}  proved the  transversal version of Dirac's theorem by introducing the auxiliary digraph technique, along with some rotations. However, this technique is not feasible in digraph collections because we cannot extend an existing directed cycle into a longer cycle by using the same rotation. In \cite{cheng2024stability}, the first author and Staden further developed the absorption method for graph collections, combined with the application of the transversal blow-up lemma, they finally proved a stability version of Joos and Kim's result \cite{2021jooskim}. The absorption technique in \cite{cheng2024stability} originates from \cite{2021ChengWangZhao} and \cite{2023chengspan}, which adapted the absorption technique first introduced in \cite{rodl2008approximate}. This approach has become a powerful tool for solving various transversal graph embedding problems.

 In this paper, we adapt the method developed in \cite{cheng2024stability} and generalize their stability from graph collections to digraph collections. By analyzing the extremal cases, we proved a tight bound for the transversal Ghouila-Houri's theorem. We believe that our method could be used to demonstrate stabilities or tight bounds for other transversal graph embedding {problems}  (see the survey \cite{surveytransversal}).

\subsection{Sketch of the proof of Theorem \ref{mainthm}}
Our proof utilizes the regularity blow-up method in digraph collections, and is divided into two parts: stable case and extremal case. %, which is totally different with that used for proving the transversal version of Dirac's theorem in undirected graph collections (which was proved by Joos and Kim \cite{2021jooskim} by introducing the auxiliary digraph technique). %In undirected graph collections, Joos and Kim \cite{2021jooskim} introduced the auxiliary digraph technique and, together with some rotations, extended the length of the existing rainbow cycle. However, their method is not feasible in digraph collections because we cannot extend an existing directed cycle into a longer cycle by using rotations.
 %For the stable case, we prove a stability result.

\medskip
\noindent{\bf Stable case.} We first characterize the extremality of a single digraph, which forms the basis of our extremal case distinction. Let $0<\frac{1}{n}\ll\epsilon\ll 1$ and $D$ be an $n$-vertex digraph with minimum semi-degree at least $(\frac{1}{2}-\epsilon)n$. Suppose $D$ has no directed Hamilton cycle. Then $D$ contains a sub-digraph that is close to one of the digraphs in Figure \ref{fig0}. By Lemma~\ref{thm-single-graph}, such a sub-digraph induces a characteristic partition of $D$; we say that $D$ is $(\epsilon, \mathrm{EC}k)$-extremal for some $k \in [3]$.  Furthermore, this $k$ is unique. %Moreover, the digraph $D$ can be $(\epsilon,{\rm EC}k)$-extremal for exactly one $k$.

Assume that $\mathcal{D} = \{D_1,\ldots, D_n\}$ is a collection of digraphs with the same vertex set $V$ of size $n$. Suppose that $\delta^0(\mathcal{D})\geq (\frac{1}{2}-\epsilon)n$ and $\mathcal{D}$ contains no transversal directed Hamilton cycles. We suspect that either almost all digraphs in $\mathcal{D}$ are $(\epsilon,{\rm EC}k)$-extremal for some $k\in [2]$ (allowing $k=1$ for some digraphs and 
$k=2$ for others), or almost all digraphs in $\mathcal{D}$ are $(\epsilon,{\rm EC}3)$-extremal; moreover those digraphs have similar vertex partitions. Based on this observation, we define two kinds of stability for a digraph collection with minimum degree at
least $(\frac{1}{2}-\epsilon)n$. We say that
\begin{itemize}
    \item $\mathcal{D}$ is strongly stable if $\mathcal{D}$ contains many digraphs each of which is not $(\epsilon,{\rm EC}k)$-extremal for any $k\in [3]$,% does not contain sub-digraphs close to either EC1, EC2 or EC3,
    \item $\mathcal{D}$ is weakly stable if for almost all colors $i\in [n]$, each $D_i$ is $(\epsilon,{\rm EC}k)$-extremal for some $k\in [3]$, but their vertex partitions are not similar,
    \item $\mathcal{D}$ is stable if it is either strongly stable or weakly stable.
\end{itemize}
In the stable case, we are to show that if $\mathcal{D}$ is stable, then it contains a transversal directed Hamilton cycle. 

\medskip
{\bf Step 1. Build an absorbing cycle when $\mathcal{D}$ is stable. (Section 3)} 
\medskip

We build {a} ``directed absorbing rainbow cycle'' $C$ for $\mathcal{D}$ with the property that $C$ is very small and there is a set $A$ consisting of some color-vertex-vertex triples $(c,u,v)$ such that whenever $A_0\subseteq A$ is sufficiently small compared to $|C|$, $C$ can absorb all of its elements. Additionally, we construct  an additional auxiliary set $\mathcal{Q}$ consisting of color-color-vertex-vertex tuples $(c,c',v,v')$ to absorb those colors and vertices which $C$ cannot. Delete all colors and vertices in $C$ and $\mathcal{Q}$ from $\mathcal{D}$.

\medskip
{\bf Step 2. Use the regularity-blow-up method for transversals in digraph collections to cover almost all vertices and colors with long rainbow directed  paths.}
\medskip

We first prove a regularity lemma for digraph collections (Section 2.1), and then apply it to $\mathcal{D}$, yielding a reduced digraph collection $\mathcal{R}$ that inherits the minimum degree condition and  the stability property of $\mathcal{D}$. By considering the characteristic bipartite graph of each digraph, we get a bipartite graph collection $\mathcal{B_R}$ for the reduced digraph collection $\mathcal{R}$. Notice that $\mathcal{R}$ has a disjoint rainbow directed cycle partition if and only if $\mathcal{B_R}$ has a transversal perfect matching. We establish a stability result for transversal perfect matching in bipartite graph collections (Section 2.2), ensuring that the disjoint rainbow directed cycle partition of $\mathcal{R}$ exists. Using the blow-up method, we obtain a set of almost spanning disjoint rainbow directed paths within each directed cycle of $\mathcal{R}$, which covers almost all vertices outside the absorbing cycle $C$ and the auxiliary set $\mathcal{Q}$ (Section 4). 

\medskip
{\bf Step 3. Connect the obtained directed paths and cover the remaining vertices via the absorbing cycle.}
\medskip

In this step, we apply the absorbing property of $C$ as well as the auxiliary set $\mathcal{Q}$ to connect all the
directed rainbow paths and to cover the remaining vertices, {which ultimately forms} a transversal directed Hamilton cycle.
\medskip

\noindent{\bf Extremal case (Sections 5-7).} Assume that  $\delta^0(\mathcal{D}) \geq \frac{n}{2}$ and $\mathcal{D}$ contains no transversal Hamilton cycles. Let $0<\frac{1}{n}\ll\epsilon\ll \delta \ll \eta\ll 1$. Denote $\mathcal{C}_k:=\{i\in [n]:D_i\ {\text {is ($\epsilon$,EC$k$)-extremal}}\}$ for $k\in [3]$ and $\mathcal{C}_{\rm bad}:=[n]\setminus(\mathcal{C}_1\cup \mathcal{C}_2\cup \mathcal{C}_3)$. Notice that $\mathcal{C}_i\cap \mathcal{C}_j=\emptyset$ for $1\leq i<j\leq 3$. By ``stable case'', we know that almost all digraphs in $\mathcal{D}$ are extremal and $|\mathcal{C}_1| +|\mathcal{C}_2|\leq 2\sqrt{\delta}n$ if $|\mathcal{C}_3|\geq 2\sqrt{\delta}n$. In view of Lemma \ref{thm-single-graph}, we can fix a characteristic partition $({A_i},{B_i},{L_i})$ for each 
graph $D_i$ with $i\in \mathcal{C}_1\cup \mathcal{C}_2$, and a characteristic partition $({C_i^1},C_i^2,C_i^3,C_i^4,L_i)$ for each  graph $D_i$ with $i\in \mathcal{C}_3$. By swapping labels, we know that $|{A_1}\triangle {A_i}|,|{B_1}\triangle {B_i}|<\delta n$ for every $i\in (\mathcal{C}_1\cup \mathcal{C}_2)\setminus \{1\}$ and $|C_1^k\triangle C_i^k|<2\delta n$ for every $i\in \mathcal{C}_3\setminus \{1\}$. We proceed {with} our proof by considering the sizes of $|\mathcal{C}_1|$, $|\mathcal{C}_2|$ and $|\mathcal{C}_3|$. We will take the case $|\mathcal{C}_1|< \eta n$, $|\mathcal{C}_2|\geq  \eta n$ and $|\mathcal{C}_3|\leq 2\sqrt{\delta}n$ as an example to illustrate the general idea of our proof. 

%Firstly, expand $A_1\cup B_1$ into an equitable partition $A\cup B$ of $V$. For a vertex $v\in A$, there are three bad cases (similar for vertices in $B$): $v$ lies in almost all $B_i$ for $i\in \mathcal{C}_1\cup \mathcal{C}_2$; $v$ lies in almost all $L_i$ for $i\in \mathcal{C}_1\cup \mathcal{C}_2$; $v\in A_i$ for many $i\in \mathcal{C}_1\cup \mathcal{C}_2$ and $v\in B_j$ for many $j\in \mathcal{C}_1\cup \mathcal{C}_2$. Move  vertices in the first case from $A$ to $B$, and do the same operation for such vertices in $B$. Denote the set of vertices in the other two bad cases by $V_{\rm bad}$. After this, the number of vertices in $A$ and $B$ may be unbalanced,  but their difference is still small. Without loss of generality, assume that $|A|\leq |B|$. In order to apply the transversal blow-up lemma to find long rainbow directed paths inside $\mathcal{D}$, the following four steps are required.% we should  In the final of our proof, we would like to use the transversal blow-up lemma (see \cite{2023chengBlowup} or Lemma~\ref{lemma4.1} in Section 4), so we need the following four steps.% strive to keep $A$ and $B$ balance during the proof process.

Firstly, expand $A_1\cup B_1$ into an equitable partition $A\cup B$ of $V$. To apply the transversal blow-up lemma (Lemma~\ref{lemma4.1}) for finding long rainbow directed paths in $\mathcal{D}$, we need to ensure most vertices in the partition $A \cup B$ satisfy uniform neighborhood properties (for example, a  vertex $v\in A$ should satisfy $v\in A_i$ for the majority of colors $i\in \mathcal{C}_2$). We identify three vertex types in $A$ (and symmetrically in $B$) that may disrupt the desired uniformity: 

\begin{enumerate}
\addtolength{\itemindent}{1cm}
\item[{\rm Type 1.}] $v \in A$ belongs to $L_i$ for many $i \in \mathcal{C}_2$;
\item[{\rm Type 2.}] $v \in A$ belongs to $A_i$ for many $i \in \mathcal{C}_2$, and $v\in B_j$ for many other $j \in \mathcal{C}_2$;
    \item[{\rm {Type 3.}}] $v \in A$ belongs to nearly all $B_i$ for $i \in \mathcal{C}_2$.    
\end{enumerate}
We handle these cases as follows: For Type 3 vertices in $A$, we move them to $B$ (and perform the analogous operation for Type 3 vertices in $B$). We define $V_{\rm bad}$ as the union of vertices of Type~1 and Type 2. After this reallocation, every vertex in $A\setminus V_{\rm bad}$ (resp. $B\setminus V_{\rm bad}$) lies in $A_i$ (resp. $B_i$) for the majority of colors $i\in \mathcal{C}_2$. However, the partition sizes $|A|$ and $|B|$ may become unbalanced, with their difference bounded by $9\delta n$.  To balance them, we require the following four steps.

\medskip
{\bf Step 1. Balance the number of vertices in $A$ and $B$.}
\medskip

If there exists a set of  disjoint rainbow paths in $\mathcal{D}[B]$ such that after deleting their vertices, the number of vertices in $A$ and $B$ is balanced, then we are done. Otherwise, we will show that $\mathcal{D}$ must contain a transversal directed Hamilton cycle (see Lemma \ref{Y-large}). 

\medskip
{\bf Step 2. Cover vertices in  $V_{\rm bad}$.} 
\medskip

Using colors in $\mathcal{C}_2$, we are to find a sequence of disjoint rainbow directed paths $P_3$ such that their centers are all vertices in $V_{\rm bad}$ and endpoints are unused vertices in $(A\cup B)\setminus V_{\rm bad}$.

\medskip
{\bf Step 3. Deal with colors in  $\mathcal{C}_{\rm bad}\cup \mathcal{C}_3$.}
\medskip

Choose a maximal rainbow matching $M$ using colors in $\mathcal{C}_{\rm bad}$ and avoiding vertices used in Steps 1-2. {It is routine to verify that for each $j \in (\mathcal{C}_{\rm bad}\cup \mathcal{C}_3)\setminus {\rm col}(M)$, the graph $D_{j}$ is EC1-extremal.}

 \medskip
{\bf Step 4. Deal with colors in $(\mathcal{C}_{1}\cup \mathcal{C}_{\rm bad}\cup \mathcal{C}_3)\setminus {\rm col}(M)$.}
\medskip

We greedily select two rainbow directed paths using colors from  $(\mathcal{C}_{1}\cup \mathcal{C}_{\rm bad}\cup \mathcal{C}_3)\setminus {\rm col}(M)$ while avoiding the vertices used in the previous three steps. 

\medskip
Based on the minimum degree condition and the characteristic partition of extremal digraphs, one may use colors in $\mathcal{C}_2$ to connect all rainbow paths obtained in the above four steps into a single short  rainbow directed path, say $P$, such that $V_{\rm bad}\subseteq V(P)$,  $\mathcal{C}_1\cup \mathcal{C}_3\cup \mathcal{C}_{\rm bad}\subseteq {\rm col}(P)$ and $|A\setminus V(P)|$ and $|B\setminus V(P)|$ are nearly equal. By applying  
the transversal blow-up lemma and some structural analysis, we obtain that $\mathcal{D}$ contains a transversal directed Hamilton cycle. 

\section{Preliminaries}

\subsection{Regularity  for digraph collections}
We use the following version of the regularity lemma for digraph collections, which is obtained by
applying the degree version of the weak regularity lemma to the 4-graph of $\mathcal{D}$. The proofs of Lemma \ref{regularity-lemma} and Lemma  \ref{degree-inheritance} are postponed to the appendix. We begin by defining regularity for digraph collections.
\begin{definition}
  Suppose that $\mathcal{D}=\{D_c:c\in \mathcal{C}\}$ is a collection of bipartite digraphs on a common vertex partition $V_1\cup V_2$. We say that
  %\begin{itemize} \item 
    $\mathcal{D}$ is $(\epsilon,d)$-\textit{regular} if whenever $V_i'\subseteq V_i$ with $|V_i'|\geq \epsilon |V_i|$ for $i\in [2]$ and $\mathcal{C}'\subseteq \mathcal{C}$ with $|\mathcal{C}'|\geq \epsilon |\mathcal{C}|$ we have 
        $$
          \left|\frac{\sum_{c\in \mathcal{C}'}e_{D_c}(V_1',V_2')}{|\mathcal{C}'||V_1'||V_2'|}-\frac{\sum_{c\in \mathcal{C}}e_{D_c}(V_1,V_2)}{|\mathcal{C}||V_1||V_2|}\right|<\epsilon
        $$
        and $\sum_{c\in \mathcal{C}}e_{D_c}(V_1,V_2)\geq d|\mathcal{C}||V_1||V_2|$.
    %\item $\mathcal{D}$ is $(\epsilon,d)$-superregular if it is $(\epsilon,d)$-regular and $\sum_{c\in \mathcal{C}}d_{D_c}^+(x)\geq d|\mathcal{C}||V_2|$ for all $x\in V_1$, and $e_{D_c}(V_1,V_2)\geq d|V_1||V_2|$ for all $c\in \mathcal{C}$.
  %\end{itemize}
\end{definition}
\begin{lemma}[Regularity lemma for digraph collections]\label{regularity-lemma}
For every integer $L_0\geq 1$ and every $\epsilon,\delta>0$, there is an integer $n_0:=n_0(\epsilon,\delta,L_0)$  such that for every $d\in [0,1)$ and every digraph collection $\mathcal{D}=\{D_i:i\in \mathcal{C}\}$ on vertex set $V$ of size $n\geq n_0$ with $\delta n\leq |\mathcal{C}|\leq \frac{n}{\delta}$, there exists a partition of $V$ into $V_0,V_1,\ldots,V_L$, of $\mathcal{C}$ into $\mathcal{C}_0,\mathcal{C}_1,\ldots,\mathcal{C}_M$ and a spanning sub-digraph $D_c'$ of $D_c$ for each $c\in \mathcal{C}$ satisfying the following properties:
\begin{enumerate}
  \item[{\rm (i)}] $L_0<L,M<n_0$ and $|V_0|+|\mathcal{C}_0|\leq \epsilon n$,
  \item[{\rm (ii)}] $|V_1|=\cdots=|V_L|=|\mathcal{C}_1|=\cdots=|\mathcal{C}_M|=:m$,
  \item[{\rm (iii)}] $\sum_{c\in \mathcal{C}}d_{D_c'}^+(v)\geq \sum_{c\in \mathcal{C}}d_{D_c}^+(v)-(\frac{24d}{\delta^3}+\epsilon)n^2$ and  $\sum_{c\in \mathcal{C}}d_{D_c'}^-(v)\geq \sum_{c\in \mathcal{C}}d_{D_c}^-(v)-(\frac{24d}{\delta^3}+\epsilon)n^2$ for all $v\in V$, and $e(D_c')\geq e(D_c)-(\frac{24d}{\delta^3}+\epsilon)n^2$ for all $c\in \mathcal{C}$,
  \item[{\rm (iv)}] if  for each $c\in \mathcal{C}$, the digraph $D_c'$ has an edge with both end  vertices in a single cluster $V_i$ for some $i\in [L]$, then $c\in \mathcal{C}_0$,
  \item[{\rm (v)}] for all triples $\{(h,i),j\}\in \binom{[L]}{2}\times [M]$, we have either $D_c'[V_h,V_i]=\emptyset$ for all $c\in \mathcal{C}_j$, or $\mathcal{D}_{hi,j}':=\{D_c'[V_h,V_i]:c\in \mathcal{C}_j\}$ is $(\epsilon,d)$-regular.
\end{enumerate}
\end{lemma}

The sets $V_i$ are called {\textit{vertex clusters}} and the sets $\mathcal{C}_i$ are called {\textit{color clusters}}, while $V_0$ and $\mathcal{C}_0$
are the exceptional vertex and color sets respectively. Now, we define the reduced digraph collection. 

\begin{definition}[Reduced digraph collection]
  Given a digraph collection $\mathcal{D}=\{D_i:i\in \mathcal{C}\}$ on $V$ and
parameters $\epsilon>0$, $d\in [0,1)$ and $L_0\geq 1$, the {\textit{reduced digraph collection}} $\mathcal{R}:=\mathcal{R}(\epsilon,d,L_0)$ of $\mathcal{D}$ is defined
as follows. Apply Lemma \ref{regularity-lemma} to $\mathcal{D}$ with parameters $\epsilon, \delta, d, L_0$ to obtain $\mathcal{D}'$, a partition $V_0,V_1,\ldots,V_L$ of $V$, and a partition $\mathcal{C}_0,\mathcal{C}_1,\ldots,\mathcal{C}_M$ of $\mathcal{C}$,  where $V_0, \mathcal{C}_0$ are the exceptional sets and $V_1,\ldots,V_L$ are the vertex clusters and $\mathcal{C}_1,\ldots,\mathcal{C}_M$ are the color clusters. Then $\mathcal{R}:=\{R_1,\ldots,R_M\}$ is a digraph collection of $M$ digraphs each on the same vertex set $[L]$, where, for $\{(h,i),j\}\in \binom{[L]}{2}\times [M]$, we have ${hi}\in R_j$ if and only if $\mathcal{D}_{hi,j}'$ is $(\epsilon,d)$-regular.
\end{definition}
The next lemma states that clusters inherit a minimum semi-degree bound in the reduced digraph collection.
\begin{lemma}[Degree inheritance]\label{degree-inheritance}
Suppose $L_0\geq 1$ and $0<\frac{1}{n}\ll \epsilon\leq d\ll\delta,\gamma,p\leq 1$. Let $\mathcal{D}=\{D_i:i\in \mathcal{C}\}$ be a digraph collection on vertex set $V$ of size $n$ with $\delta^0(\mathcal{D})\geq (p+\gamma)n$ and $\delta n\leq |\mathcal{C}|\leq\frac{n}{\delta}$. Let $\mathcal{R}:=\mathcal{R}(\epsilon,d,L_0)$ be the reduced digraph collection of $\mathcal{D}$ consisting of $M$ digraphs on a common  vertex set $[L]$. Then
\begin{enumerate}
  \item[{\rm (i)}] for every $i\in [L]$, there are at least $(1-d^{\frac{1}{4}})M$ colors $j\in [M]$ for which $d_{R_j}^+(i)\geq (p+\frac{\gamma}{2})L$ and $d_{R_j}^-(i)\geq (p+\frac{\gamma}{2})L$,
  \item[{\rm (ii)}] for every $j\in [M]$, there are at least $(1-d^{\frac{1}{4}})L$ vertices $i\in [L]$ for which $d_{R_j}^+(i)\geq (p+\frac{\gamma}{2})L$ and $d_{R_j}^-(i)\geq (p+\frac{\gamma}{2})L$.
\end{enumerate}
\end{lemma}
\begin{proof}
  (i) Choose a vertex $v\in V\setminus V_0$. By Lemma \ref{regularity-lemma} (i) and (iii), we have
  $$
  \sum_{c\in \mathcal{C}}d_{D_c'-V_0}^+(v)\geq \sum_{c\in \mathcal{C}}d_{D_c}^+(v)-\left(\frac{24d}{\delta^3}+\epsilon\right)n^2-\epsilon n|\mathcal{C}|\geq \sum_{c\in \mathcal{C}}d_{D_c}^+(v)-\frac{26dn^2}{\delta^3}.
  $$
  Let $\mathcal{D}_v:=\left\{c\in \mathcal{C}\setminus \mathcal{C}_0:d_{D_c'-V_0}^+(v)\geq d_{D_c}^+(v)-\sqrt{d}n\right\}$. % be the collection of colors $c$ in $\mathcal{C}\setminus \mathcal{C}_0$ for which $d_{D_c'-V_0}^+(v)\geq d_{D_c}^+(v)-\sqrt{d}n$. 
  Then
  $$
    \sum_{c\in \mathcal{C}\setminus \mathcal{C}_0}d_{D_c'-V_0}^+(v)\leq \sum_{c\in \mathcal{C}\setminus \mathcal{C}_0}d_{D_c}^+(v)-\sqrt{d}n |\mathcal{C}\setminus (\mathcal{C}_0\cup \mathcal{D}_v)|.
  $$
  Therefore, $|\mathcal{C}\setminus (\mathcal{C}_0\cup \mathcal{D}_v)|\leq \frac{d^{\frac{1}{3}}n}{2}$ since $d\ll \delta$. It follows that $|\mathcal{D}_v|\geq |\mathcal{C}|-\frac{d^{\frac{1}{3}}n}{2}-\epsilon n\geq |\mathcal{C}|-d^{\frac{1}{3}}n$. Recall that $mM\leq |\mathcal{C}|\leq mM+\epsilon n$ and $mL\leq n\leq mL+\epsilon n$. Together with $\delta n\leq |\mathcal{C}|\leq \frac{n}{\delta}$, one has $\delta L\leq 2(1-\epsilon)M$. Hence the number of color clusters $\mathcal{C}_j$ containing at least one color of $\mathcal{D}_v$ is at least 
  $$
    \frac{|\mathcal{D}_v|}{m}\geq M-\frac{d^{\frac{1}{3}}n}{m}\geq M-\frac{d^{\frac{1}{3}}L}{1-\epsilon}\geq M-\frac{2d^{\frac{1}{3}}M}{\delta}\geq \left(1-\frac{d^{\frac{1}{4}}}{2}\right)M.
  $$
  Now let $i\in [L]$ and $v\in V_i$. For each cluster $\mathcal{C}_j$ as above, choose an arbitrary color $c_j\in \mathcal{C}_j\cap \mathcal{D}_v$. Thus, the number of vertex clusters $V_h$ containing some vertices in $N_{D_{c_j}'}^+(v)$ is at least
  $$
    \frac{d_{D_{c_j}}^+(v)-\sqrt{d}n}{m}\geq \frac{(p+\gamma-\sqrt{d})n}{m}\geq \left(p+\frac{\gamma}{2}\right)L.
  $$
  Lemma \ref{regularity-lemma} (v) implies that $i$ is adjacent to each such $V_h$ in $R_j$. Therefore,  for every $i\in [L]$, $d_{R_j}^+(i)\geq (p+\frac{\gamma}{2})L$ for at least $(1-\frac{d^{\frac{1}{4}}}{2})M$ colors $j$. The proofs of other cases are similar and we omit the proof.
\end{proof}

\subsection{Transversal perfect matching in bipartite graph  collections}

 %In this subsection, we mainly  introduce some conditions for the existence of transversal perfect matching  inside $\mathcal{G}$. 
 
{{In the study of digraphs, a natural approach is to consider  their characteristic bipartite graphs. This subsection introduces the corresponding characteristic bipartite graph collection associated with a digraph system, and studies the existence of transversal perfect matchings in this bipartite system.}}

Let $D$ be a digraph with vertex set $V$. We define $B_D$ to be the characteristic  bipartite graph of $D$ with bipartition $V_1\cup V_2$ where $V_1=V_2=V$; for each $u\in V_1$ and $v\in V_2$, $uv\in E(B_D)$ if and
only if $uv\in E(D)$. Let $\mathcal{D}=\{D_1,\ldots,D_n\}$ be a collection of digraphs on a common vertex set $V$. Define the characteristic  bipartite graph collection of $\mathcal{D}$ to be $\mathcal{B_D}=\{B_{D_1},\ldots,B_{D_n}\}$.  A \textit{matching} in a (di)graph $G$ is a collection of vertex-disjoint edges $M\subseteq E(G)$. We say that $M$ is a \textit{perfect matching} if $V(M) = V(G)$. 

It is straightforward  to check that $\mathcal{B_D}$ contains a transversal perfect matching if and only if the vertex set of $\mathcal{D}$ can be covered by a set of disjoint rainbow cycles inside $\mathcal{D}$. Therefore, in the remainder of this section, we are to find a transversal perfect matching in a bipartite graph collection. 

Joos and Kim \cite{2021jooskim} established a minimum degree condition that guarantees the existence of transversal perfect matchings in a graph collection.
Bradshaw \cite{bradshaw2021transversals} proved the analogue result in bipartite graph collections, which can be stated as follows. %generalized Joos and Kim's results to bipartite graph collection $\mathcal{G}$ and provides the minimum degree condition for $\mathcal{G}$ to contain a perfect matching.
\begin{theorem}[\cite{bradshaw2021transversals}]\label{THEOREM:MATCHING-2021}
    Let $\mathcal{G}=\{G_1,\ldots,G_n\}$ be a collection of bipartite graphs on a common vertex bipartition $V_1\cup V_2$ with $|V_1|=|V_2|=n$. If for each $i\in [n]$ we have the following hold:
    \begin{itemize}
        \item [{\rm (i)}] $d_{G_i}(v)> \frac{n}{2}$ for each $v\in V_1$,
        \item [{\rm (ii)}] $d_{G_i}(v)\geq \frac{n}{2}$ for each $v\in V_2$,
    \end{itemize}
then $\mathcal{G}$ contains a transversal perfect matching.
\end{theorem}
In this subsection, we establish a stability result for transversal perfect matching in bipartite graph collections, that is, the minimum degree required to guarantee a transversal
perfect matching can be below that of Theorem \ref{THEOREM:MATCHING-2021}, as long as the bipartite graph collection $\{G_1,\ldots,G_n\}$ is far in edit distance from several extremal cases. %We weakened the minimum degree condition in Bradshaw's results and provided structural conditions for $\mathcal{G}$ to contain a transversal perfect matching when the minimum degree of $\mathcal{G}$ is close to $n/2$ (see Lemma \ref{stable-matching}).
%In order to better describe the structural conditions for the existence of transversal perfect matching in $\mathcal{G}$, we need to make some preparations. Now 
For this, we first define the extremality for a single bipartite graph.

\begin{definition}[nice, extremal]
    Let $\epsilon >0$ and $G$ be a balanced bipartite graph on the vertex set $V_1\cup V_2$ of size $2n$.  We say that
    \begin{itemize}
        \item $G$ is $\epsilon$-{\em nice} if for any two sets $A\subseteq V_1$ and $B\subseteq V_2$ of size at least $(\frac{1}{2}-\epsilon)n$ we have $e_G(A,B)\geq \epsilon n^2$;
        \item $G$ is $\epsilon$-{\em extremal} if it is not $\epsilon^5$-nice.
    \end{itemize}
\end{definition}
The following lemma is the basis for our proof. %characterizing the extremal cases. 
\begin{lemma}\label{LEMMA:MATCHING-CONS}
Suppose that $0<\frac{1}{n}\ll d\ll\epsilon \leq 1$. Let $G$ be a balanced bipartite graph on a vertex set $V=V_1\cup V_2$ of size $2n$ with $d_{G}(x)\geq (\frac{1}{2}-\epsilon^5)n$ for all but at most $dn$ vertices $x\in V$ which is $\epsilon$-extremal. Then there is a characteristic partition $(A_1, B_1, A_2, B_2, C_1, C_2)$ of $G$ such that the following hold:
\begin{itemize}
    \item[{\rm (i)}] $A_i, B_i, C_i\subseteq V_i$, $|A_i|=|B_i|=(\frac{1}{2}-\epsilon)n$ and $|C_i|=2\epsilon n$ for each $i\in [2]$,
    \item[{\rm (ii)}] $d_G(v, X_i)\geq (\frac{1}{2}-2\epsilon)n$ for $X\in \{A, B\}$ and $v\in X_{3-i}$ with  $i\in [2]$,
    \item[{\rm (iii)}] either $e_G(A_1, B_2)\leq \epsilon n^2$ or $e_G(A_2, B_1)\leq \epsilon n^2$.
\end{itemize}
\end{lemma}

\begin{definition}[crossing, cross graph]
  Let $0<\frac{1}{n},\epsilon,\delta<1$ where $n\in \mathbb{N}$ and let $\mathcal{G}=\{G_1,\ldots,G_n\}$ be a balanced bipartite graph collection on a common vertex set $V$ of size $2n$. Given $i,j\in [n]$ such that $G_i$ and  $G_j$ are both $\epsilon$-extremal, assume that the characteristic partition corresponding to $G_{\ell}$ is $(A_1^{\ell},B_1^{\ell},C_1^{\ell},A_2^{\ell},B_2^{\ell},C_2^{\ell})$ for $\ell\in \{i, j\}$. We say that they are $\delta$-\textit{crossing} if $|A^i_1\triangle A^j_1|\geq \delta n$ and $|A^i_1\triangle B^j_1|\geq \delta n$.
  
We define the \textit{cross graph}  $C_{\mathcal{G}}^{\epsilon,\delta}$ to be the graph with vertex set $[n]$ where $i$ is adjacent to $j$ if and only if $G_i$ and $G_j$ are both $\epsilon$-extremal and $\delta$-crossing.
\end{definition}

\begin{definition}[strongly stable, weakly stable\footnote{To avoid introducing too many terms, we have similarly defined the concepts of ``nice, extremal, crossing, cross graph, strongly stable, weakly stable, stable'' in the context of digraph collections. We should note that, unless explicitly stated otherwise, those concepts in bipartite graph collections are only discussed in Section 2.2 and the appendix.}]
  Let $0<\gamma,\alpha,\epsilon,\delta<1$. Suppose that $\mathcal{G}=\{G_1,\ldots, G_n\}$ is a collection of bipartite graphs on a common vertex partition $V_1\cup V_2$ with $|V_1|=|V_2|=n$. We say that 
  \begin{itemize}
    \item $\mathcal{G}$ is $(\gamma,\alpha)$-\textit{strongly stable} if $G_i$ is $\alpha$-nice for at least $\gamma n$ colors $i\in [n]$;
    \item $\mathcal{G}$ is $(\epsilon,\delta)$-\textit{weakly stable} if $e(C_{\mathcal{G}}^{\epsilon,\delta})\geq \delta n^2$;
    \item $\mathcal{G}$ is $(\gamma,\alpha,\epsilon,\delta)$-\textit{stable} if  it is either $(\gamma,\alpha)$-strongly stable or $(\epsilon,\delta)$-weakly stable.% $e(C_{\mathcal{G}}^{\epsilon,\delta})\geq \delta n^2$.
  \end{itemize} 
  
\end{definition}

The following is a stability result for transversal perfect matching in bipartite graph collections. 
\begin{theorem}\label{stable-matching}
Let $0<\frac{1}{n}\ll d\ll \mu\ll \alpha\ll \gamma,\epsilon\ll \delta\ll1$. Suppose that $\mathcal{G}=\{G_1,\ldots,G_n\}$ is a collection of bipartite graphs on a common vertex partition $V=V_1\cup V_2$ with $|V_1|=|V_2|=n$. % and $\$ $\delta(\mathcal{G})\geq \left(\frac{1}{2}-\mu\right)n$. 
If the following hold:
\begin{itemize}
    \item $\mathcal{G}$ is $(\gamma,\alpha,\epsilon,\delta)$-stable, 
    \item for every $i\in [n]$, $d_{G_i}(x)\geq \left(\frac{1}{2}-\mu\right)n$ for all but at most $dn$ vertices $x\in V_1\cup V_2$, 
    \item for every $x\in V_1\cup V_2$, $d_{G_i}(x)\geq \left(\frac{1}{2}-\mu\right)n$ for all but at most $dn$ colors $i\in [n]$, 
\end{itemize}
then $\mathcal{G}$ contains a transversal perfect matching.
\end{theorem}

%In fact, Lemma \ref{stable-matching} can be regarded as a stability result of Theorem \ref{THEOREM:MATCHING-2021}. 
For convenience, we postpone the proofs of Lemma \ref{LEMMA:MATCHING-CONS} and Theorem \ref{stable-matching} to the appendix.
\subsection{Probabilistic tool}
We use the following version of Chernoff's bound.
\begin{lemma}
    Let $X$ be a random variable with binomial or hypergeometric distribution, and let $0<\epsilon<\frac{3}{2}$. Then 
    $$
    \mathbb{P}[|X-\mathbb{E}(X)|\geq \epsilon \mathbb{E}(X)]\leq 2e^{-\frac{\epsilon^2}{3}\mathbb{E}(X)}.
      $$
\end{lemma}

%\subsection{Other lemmas}
%An $n$-vertex graph $H$ is $\mu$-separable if there is $X \subseteq V(H)$ of size at most $\mu n$ such that $H - X$ consists of components of size at most $\mu n$.
% \begin{lemma}[Transversal blow-up lemma \cite{cheng2023transversals}]\label{LEMMA:Transversal-blow-up-lemma}
%Let $0 \le 1/n \ll \epsilon, \mu,\alpha, \ll \nu, d, \delta, 1/\triangle \leq 1$. Let $\mathcal{C}$ be a set of at least $\delta n$ colors and let  $G = \{G_c : c \in \mathcal{C}\}$ be a collection of bipartite graphs with the same vertex partition $V_1$, $V_2$, where $n \leq  |V_1| \leq |V_2| \leq n/\delta$, such that $G$ is $(\epsilon, d)$-superregular. Let $H$ be a $\mu$-separable bipartite graph with parts $A_1$, $A_2$ of sizes $|V_1|$, $|V_2|$ respectively, and $|\mathcal{C}|$ edges and maximum degree $\triangle$. Suppose further that, for $i = 1, 2$, there is a set $U_i \subseteq A_i$ with $|U_i| \leq \alpha|A_i|$ and for each $x \in U_i$, a target set $T_x \subseteq V_i$ with $|T_x| \geq \nu|V_i|$. Then $G$ contains a transversal copy of $H$ such that for $i = 1, 2$, every $x \in U_i$ is embedded inside $T_x$.
%\end{lemma}

\section{Absorbing}

{{In this section, we first define the extremality for a single digraph and distinguish the characteristic partition of such digraphs. We then introduce our absorbing structure and prove that every digraph collection contains a directed  absorbing cycle whenever either a large proportion of its members are non-extremal, or almost all of them are extremal while their characteristic partitions differ significantly.}} 

\begin{definition}[nice, extremal]\label{nice}
Let $\epsilon>0$ and $D$ be a digraph on  vertex set $V$ of size $n$. We say that
\begin{itemize}
    \item $D$ is $\epsilon$-\textit{nice} if for any two sets $A,B\subseteq V$ of size at least $(\frac{1}{2}-\epsilon)n$, we have $e_D(A,B)\geq \epsilon n^2$;
    \item $D$ is $\epsilon$-\textit{extremal} if it is not $\epsilon$-nice.
\end{itemize}
\end{definition}
Cheng, Wang and Yan \cite{wang2024}  characterized the structure of $\epsilon$-extremal digraphs,  which
forms the basis of our extremal case distinction. 

\begin{lemma}[\cite{wang2024}]\label{thm-single-graph}
    Suppose that $0<\frac{1}{n}\ll d\ll\mu\ll\epsilon\leq 1$.  Let $D$ be an $\epsilon$-extremal digraph  on vertex set $V$ of size $n$ with $\min\{d^+(x),d^-(x)\}\geq (\frac{1}{2}-\mu)n$ for all but at most $dn$ vertices $x\in V$. Then exactly one of the following holds: 
\begin{enumerate}
    \item [{\rm (i)}] there is a partition $A\cup B\cup L$ of $V$ such that $|A|=|B|=(\frac{1}{2}-\epsilon)n$,  
$d^+(a,{A}),d^-(a,{A})\geq (\frac{1}{2}-2\epsilon)n$ for all $a\in {A}$, $d^+(b,{B}),d^-(b,B)\geq (\frac{1}{2}-2\epsilon)n$ for all $b\in {B}$ and either $e(A,B)\leq \epsilon n^2$ or $e(B,A)\leq \epsilon n^2$; here we say $D$ is $(\epsilon,{\rm EC}1)$-extremal,
\item [{\rm (ii)}] 
there is a partition $A\cup B\cup L$ of $V$ such that $|A|=|B|=(\frac{1}{2}-\epsilon)n$, 
$d^+(a,B),d^-(a,B)\geq (\frac{1}{2}-2\epsilon)n$ for all $a\in {A}$, $d^+(b,A),d^-(b,A)\geq (\frac{1}{2}-2\epsilon)n$ for all $b\in {B}$ and either $e(A)\leq \epsilon n^2$ or $e(B)\leq \epsilon n^2$; here we say $D$ is $(\epsilon,{\rm EC}2)$-extremal,
\item [{\rm (iii)}] 
there is a constant $\zeta$ with $\epsilon^{1/3}\leq  \zeta\leq\frac{1}{2}-\epsilon-\epsilon^{1/3}$ and a partition $C_1\cup C_2\cup C_3\cup C_4\cup L$ of $V$ such that $|C_1|=|C_3|=\zeta n$, $|C_2|=|C_4|=(\frac{1}{2}-\zeta-\epsilon)n$,  for all $a\in C_i$ we have  $d^+(a,C_i)\geq (\zeta-\epsilon)n$ if  $i\in \{1,3\}$, $d^+(a,C_{6-i})\geq (\frac{1}{2}-\zeta-2\epsilon)n$ if $i\in \{2,4\}$, and $d^+(a,C_{i+1})\geq (|C_{i+1}|-\epsilon)n$ for all $i\in [4]$ 
{(here we identify $C_5$ with $C_1$)}, and 
either $e(C_1,C_3)\leq \epsilon n^2$ or $e(C_3,C_1)\leq \epsilon n^2$, either  $e(C_2)\leq \epsilon n^2$ or $e(C_4)\leq \epsilon n^2$; here we say $D$ is $(\epsilon,{\rm EC}3)$-extremal, (see Figure \ref{fig0} for the structures of EC1, EC2 and EC3).
\end{enumerate}
\end{lemma}
%\begin{figure}[ht!]
%  \centering
%  \psfrag{a}{$A$}\psfrag{b}{$B$}\psfrag{c}{$C_1$}\psfrag{d}{$C_2$}
%  \psfrag{e}{$C_3$}
%  \psfrag{f}{$C_4$}
%  \psfrag{h}{EC1}
%  \psfrag{i}{EC2}
%  \psfrag{j}{EC3}
%  \includegraphics[width=100mm]{extremal.eps}\\
%  \caption{Extremal digraphs.}\label{fig1}
%  \end{figure}

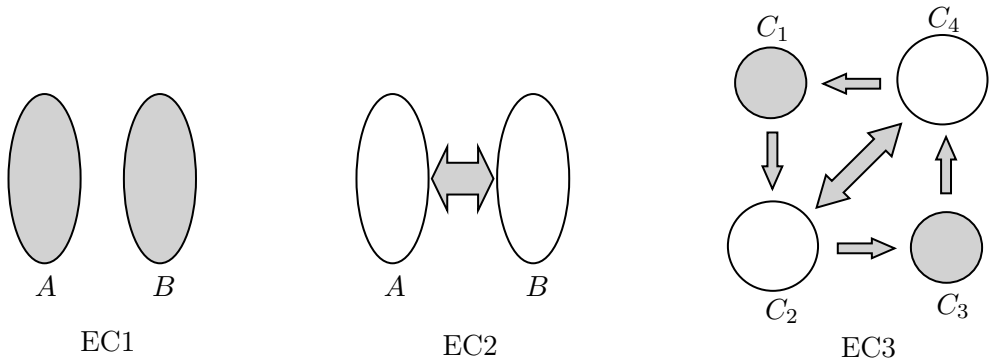
\begin{figure}[ht!]
    \centering

\tikzset{every picture/.style={line width=0.75pt}} %set default line width to 0.75pt        

\begin{tikzpicture}[x=0.75pt,y=0.75pt,yscale=-0.90,xscale=0.90]
%uncomment if require: \path (0,433); %set diagram left start at 0, and has height of 433

%Shape: Ellipse [id:dp3192174066641109] 
\draw  [fill={rgb, 255:red, 155; green, 155; blue, 155 }  ,fill opacity=0.45 ] (77,162.3) .. controls (77,136.18) and (85.95,115) .. (97,115) .. controls (108.05,115) and (117,136.18) .. (117,162.3) .. controls (117,188.42) and (108.05,209.6) .. (97,209.6) .. controls (85.95,209.6) and (77,188.42) .. (77,162.3) -- cycle ;
%Shape: Ellipse [id:dp3028622952003852] 
\draw  [fill={rgb, 255:red, 155; green, 155; blue, 155 }  ,fill opacity=0.45 ] (141,162.3) .. controls (141,136.18) and (149.95,115) .. (161,115) .. controls (172.05,115) and (181,136.18) .. (181,162.3) .. controls (181,188.42) and (172.05,209.6) .. (161,209.6) .. controls (149.95,209.6) and (141,188.42) .. (141,162.3) -- cycle ;
%Shape: Ellipse [id:dp4775050090635111] 
\draw  [fill={rgb, 255:red, 255; green, 255; blue, 255 }  ,fill opacity=1 ] (270,162.3) .. controls (270,136.18) and (278.95,115) .. (290,115) .. controls (301.05,115) and (310,136.18) .. (310,162.3) .. controls (310,188.42) and (301.05,209.6) .. (290,209.6) .. controls (278.95,209.6) and (270,188.42) .. (270,162.3) -- cycle ;
%Shape: Ellipse [id:dp5077151486177411] 
\draw  [fill={rgb, 255:red, 255; green, 255; blue, 255 }  ,fill opacity=1 ] (348,162.3) .. controls (348,136.18) and (356.95,115) .. (368,115) .. controls (379.05,115) and (388,136.18) .. (388,162.3) .. controls (388,188.42) and (379.05,209.6) .. (368,209.6) .. controls (356.95,209.6) and (348,188.42) .. (348,162.3) -- cycle ;
%Left Right Arrow [id:dp7963942770839543] 
\draw  [fill={rgb, 255:red, 155; green, 155; blue, 155 }  ,fill opacity=0.45 ] (311.8,162.3) -- (320.35,144.55) -- (320.35,153.42) -- (337.45,153.42) -- (337.45,144.55) -- (346,162.3) -- (337.45,180.05) -- (337.45,171.17) -- (320.35,171.17) -- (320.35,180.05) -- cycle ;
%Shape: Circle [id:dp4018866230811009] 
\draw  [fill={rgb, 255:red, 155; green, 155; blue, 155 }  ,fill opacity=0.45 ] (480,108.75) .. controls (480,97.84) and (488.84,89) .. (499.75,89) .. controls (510.66,89) and (519.5,97.84) .. (519.5,108.75) .. controls (519.5,119.66) and (510.66,128.5) .. (499.75,128.5) .. controls (488.84,128.5) and (480,119.66) .. (480,108.75) -- cycle ;
%Shape: Circle [id:dp8020973041589226] 
\draw  [fill={rgb, 255:red, 255; green, 255; blue, 255 }  ,fill opacity=1 ] (476,200) .. controls (476,186.19) and (487.19,175) .. (501,175) .. controls (514.81,175) and (526,186.19) .. (526,200) .. controls (526,213.81) and (514.81,225) .. (501,225) .. controls (487.19,225) and (476,213.81) .. (476,200) -- cycle ;
%Shape: Circle [id:dp41217482366774694] 
\draw  [fill={rgb, 255:red, 255; green, 255; blue, 255 }  ,fill opacity=1 ] (570,107) .. controls (570,93.19) and (581.19,82) .. (595,82) .. controls (608.81,82) and (620,93.19) .. (620,107) .. controls (620,120.81) and (608.81,132) .. (595,132) .. controls (581.19,132) and (570,120.81) .. (570,107) -- cycle ;
%Shape: Circle [id:dp6961549202312327] 
\draw  [fill={rgb, 255:red, 155; green, 155; blue, 155 }  ,fill opacity=0.45 ] (577.5,200.75) .. controls (577.5,189.84) and (586.34,181) .. (597.25,181) .. controls (608.16,181) and (617,189.84) .. (617,200.75) .. controls (617,211.66) and (608.16,220.5) .. (597.25,220.5) .. controls (586.34,220.5) and (577.5,211.66) .. (577.5,200.75) -- cycle ;
%Down Arrow [id:dp09054921618976852] 
\draw  [fill={rgb, 255:red, 155; green, 155; blue, 155 }  ,fill opacity=0.45 ] (495.2,155.4) -- (497.95,155.4) -- (497.95,136.5) -- (503.45,136.5) -- (503.45,155.4) -- (506.2,155.4) -- (500.7,168) -- cycle ;
%Down Arrow [id:dp1082430573541513] 
\draw  [fill={rgb, 255:red, 155; green, 155; blue, 155 }  ,fill opacity=0.45 ] (602.5,151.6) -- (599.75,151.6) -- (599.75,170.5) -- (594.25,170.5) -- (594.25,151.6) -- (591.5,151.6) -- (597,139) -- cycle ;
%Down Arrow [id:dp21222095697178878] 
\draw  [fill={rgb, 255:red, 155; green, 155; blue, 155 }  ,fill opacity=0.45 ] (541.55,103.75) -- (541.55,106.5) -- (560.45,106.5) -- (560.45,112) -- (541.55,112) -- (541.55,114.75) -- (528.95,109.25) -- cycle ;
%Down Arrow [id:dp02821814810342671] 
\draw  [fill={rgb, 255:red, 155; green, 155; blue, 155 }  ,fill opacity=0.45 ] (555.85,206.75) -- (555.85,204) -- (536.95,204) -- (536.95,198.5) -- (555.85,198.5) -- (555.85,195.75) -- (568.45,201.25) -- cycle ;
%Up Down Arrow [id:dp06951295556963455] 
\draw  [fill={rgb, 255:red, 155; green, 155; blue, 155 }  ,fill opacity=0.45 ] (553.83,136.87) -- (572.04,131.66) -- (566.83,149.87) -- (563.58,146.62) -- (540.16,170.04) -- (543.41,173.29) -- (525.2,178.5) -- (530.41,160.29) -- (533.66,163.54) -- (557.08,140.12) -- cycle ;

% Text Node
\draw (155,215) node [anchor=north west][inner sep=0.75pt]   [align=left] {$B$};
% Text Node
\draw (316,248) node [anchor=north west][inner sep=0.75pt]   [align=left] {EC2};
% Text Node
\draw (537,249) node [anchor=north west][inner sep=0.75pt]   [align=left] {EC3};
% Text Node
\draw (90,215) node [anchor=north west][inner sep=0.75pt]   [align=left] {$A$};
% Text Node
\draw (115,246) node [anchor=north west][inner sep=0.75pt]   [align=left] {EC1};
% Text Node
\draw (362,215) node [anchor=north west][inner sep=0.75pt]   [align=left] {$B$};
% Text Node
\draw (283.5,215) node [anchor=north west][inner sep=0.75pt]   [align=left] {$A$};
% Text Node
\draw (590,225) node [anchor=north west][inner sep=0.75pt]   [align=left] {$C_3$};
% Text Node
\draw (495,228) node [anchor=north west][inner sep=0.75pt]   [align=left] {$C_2$};
% Text Node
\draw (585,65) node [anchor=north west][inner sep=0.75pt]   [align=left] {$C_4$};
% Text Node
\draw (490,70) node [anchor=north west][inner sep=0.75pt]   [align=left] {$C_1$};

\end{tikzpicture}
    \caption{{\small Extremal digraphs EC1, EC2, and EC3. The gray shaded elliptical indicates that the digraph induced by this vertex set is complete, the gray shaded arrow between two vertex sets indicates that the induced digraph by them is   complete in this direction.}}
    %The shaded parts represent complete digraphs, with the shaded arrows indicating the directions of the edges in these complete directed graphs.}}
    \label{fig0}
\end{figure}

We call the partition obtained  in Lemma \ref{thm-single-graph} a \textit{characteristic partition} of $D$. Assume that  $0<\frac{1}{n}\ll \mu\ll \epsilon\leq 1$ and  $\mathcal{D}=\{D_1,\ldots,D_n\}$ is a collection of digraphs on a common vertex set $V$ of size $n$ with $\delta^0(\mathcal{D})\geq (\frac{1}{2}-\mu)n$. In view of Lemma \ref{thm-single-graph},  whenever $D_i$ is $\epsilon$-extremal, one may fix a characteristic partition $(A_i,B_i,L_i)$ of $D_i$ if it is $(\epsilon,{\rm EC}k)$-extremal for some $k\in [2]$,  a characteristic partition $(C_i^1,C_i^2,C_i^3,C_i^4,L_i)$ of $D_i$ if it is $(\epsilon,{\rm EC}3)$-extremal. For convenience, define $W_i^j:=C_i^j\cup C_{i}^{j+1}$ for $j\in [4]$ {(here we identify $C_{i}^{5}$ with $C_{i}^{1}$)}. We say that a vertex $v\in V$ is \textit{$D_i$-good} if either $D_i$ is not $\epsilon$-extremal or $D_i$ is $\epsilon$-extremal and $v\in V\setminus L_i$. 
For every pair of distinct vertices $x,y\in V$, we define $L({xy}):=\{i\in [n]:{xy}\in E(D_i)\}$ to be the set of colors appearing on ${xy}$. 

\begin{definition}[directed absorbing path and directed absorbing cycle]\label{def-abaosbing}
Given any two not necessarily distinct vertices $u,v\in V$ and a rainbow directed path $P=v_1v_2v_3v_4$ with $u,v\not\in V(P)$, we call $P$ a {\textit{Type-{\rm I} directed $c$-absorbing path}} of $(v,u)$ if $c\in L({v_2v})$ and ${\rm col}({v_2v_3})\in L({uv_3})$ (see Figure~\ref{fig1}), a {\textit{Type-{\rm II} directed $c$-absorbing path}} of $(v,u)$ if $c\in L({vv_3})$ and ${\rm col}({v_2v_3})\in L({v_2u})$ (see Figure \ref{fig2}).  

    Given $\delta,\delta',\gamma,\gamma'>0$ and ${\rm K}\in \{{\rm I,II}\}$, a rainbow directed cycle $C=v_1v_2\ldots v_tv_1$ is a {\textit{Type-{\rm K} directed absorbing  cycle}} with parameters $(\delta,\delta',\gamma,\gamma')$ if $t\leq \gamma n$ and there exists a color set $\mathcal{C}$ of size at least $\delta n$ such that 
    \begin{itemize}
        \item given any color $c\in \mathcal{C}$ and any $D_c$-good vertex $v$, for all but at most $\delta'n$ vertices $u\in V$, there are at least $\gamma' n$ disjoint Type-K directed  $c$-absorbing paths of $(v,u)$ inside $C$,
        \item given any color $c\in \mathcal{C}$, for all but at most $\delta'n$ $D_c$-good vertices $v$, there are at least $\gamma' n$ disjoint Type-K directed $c$-absorbing paths of $(v,v)$ inside $C$.
    \end{itemize}
    \end{definition} 
If there is no ambiguity, we typically refer to the Type-I and Type-II directed $c$-absorbing path (cycle) simply as a directed $c$-absorbing path (cycle). 
%we usually refer to the Type-I directed $c$-absorbing path (cycle) and Type-II directed $c$-absorbing path (cycle) as directed $c$-absorbing path (cycle). 

\begin{figure}[ht!]
  \centering
  \psfrag{a}{$v_1$}\psfrag{b}{$v_2$}\psfrag{c}{$v_3$}\psfrag{d}{$v_4$}
  \psfrag{e}{$v$}
  \psfrag{f}{$u$}
  \psfrag{g}{$c$}
  \includegraphics[width=100mm]{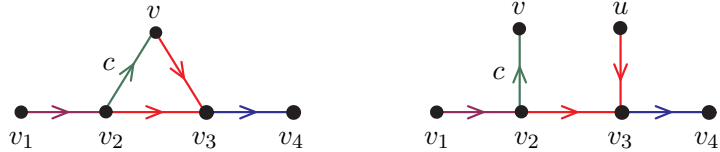}\\
  \caption{A Type-I directed  $c$-absorbing path of $(v,v)$ and a Type-I 
 directed $c$-absorbing path of $(v,u)$ with $v\neq u$.}\label{fig1}
\end{figure}

\begin{figure}[ht!]
  \centering
  \psfrag{a}{$v_1$}\psfrag{b}{$v_2$}\psfrag{c}{$v_3$}\psfrag{d}{$v_4$}
  \psfrag{e}{$v$}
  \psfrag{f}{$u$}
  \psfrag{g}{$c$}
  \includegraphics[width=100mm]{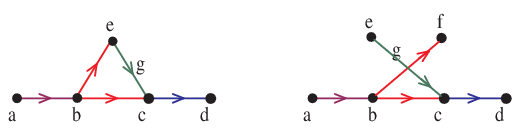}\\
  \caption{A Type-II directed  $c$-absorbing path of $(v,v)$ and a Type-II 
 directed $c$-absorbing path of $(v,u)$ with $v\neq u$.}\label{fig2}
\end{figure}

Denote by $V^k$ the set of all $k$-tuples of distinct elements of $V$. A pair $H=(V,E)$ where
$V$ is a set and $E\subseteq V^k$ is called a \textit{directed $k$-graph}, while if we allow $E$ to be a multiset, it is a
\textit{directed multi-$k$-graph}. We use the same notation for directed $k$-graphs as for graphs. The following 
lemma was obtained by Cheng and Staden \cite{cheng2024stability}, which is the key tool that we will use to build an absorbing structure. 

\begin{lemma}[\cite{cheng2024stability}]\label{LEMMA:directed-k-graph-matching}
Suppose that $k, C, n \in \mathbb{N}$, $0 <\frac{1}{n} \ll \gamma\ll \epsilon \ll \frac{1}{k}, \frac{1}{C} \leq 1$, $m \in [n^C]$ and $t := \gamma n$. Let $\mathbf{H} = \{H_1, \ldots , H_t\}$ be a collection of directed $k$-graphs and let $\mathbf{Z}= \{Z_1, \ldots , Z_m\}$ be a collection of directed multi-$k$-graphs all defined on a common vertex set $V$ of size $n$. Suppose that $|E(H_i)| \geq \epsilon n^k$ for all $i \in [t]$, and for each $j \in [m]$ we have $|E(Z_j ) \cap E(H_i)|\geq \epsilon n^k$ for at least $\epsilon t$ indices $i\in [t]$. Then there is a rainbow matching $M$ in $\mathbf{H}$ of size at least $(1 - \frac{\epsilon^2}{4})t$ and
$|E(Z_j ) \cap E(M)| \geq \frac{\epsilon^2 t}{4}$ for each $j \in [m]$.
\end{lemma}

{We first consider the case where a digraph collection contains many non-extremal digraphs; such a collection is said to be strongly stable.}
\begin{definition}[strongly stable]
    Let $0<\gamma,\alpha<1$. Suppose that $\mathcal{D}=\{D_1,\ldots,D_n\}$ is a collection of digraphs  on a common vertex set $V$ of size $n$. We say that $\mathcal{D}$ is $(\gamma,\alpha)$-\textit{strongly stable} if $D_i$ is $\alpha$-nice for at least $\gamma n$ colors $i\in [n]$.
\end{definition}

By applying Lemma \ref{LEMMA:directed-k-graph-matching}, we establish the following lemma, which yields a directed absorbing cycle when  $\mathcal{D}$ is strongly stable.

\begin{lemma}\label{strongly}
    Let $0<\frac{1}{n}\ll \lambda ,\mu\ll \gamma,\alpha\ll 1$. Assume that $\mathcal{D}=\{D_1,\ldots,D_n\}$ is a collection of digraphs on a common vertex set $V$ of size $n$ with $\delta^0(\mathcal{D})\geq \frac{1-\mu}{2}n$. If $\mathcal{D}$ is $(\gamma,\alpha)$-strongly stable, then there exists a Type-{\rm I} directed absorbing  cycle with parameters $(1,0,\lambda,\lambda^2)$.
\end{lemma}
\begin{proof}
Since $\mathcal{D}$ is $(\gamma,\alpha)$-strongly stable, one may assume that $D_1,\ldots,D_{\gamma n}$ are $\alpha$-nice and $6| \lambda n$. We divide the color set $[\frac{\lambda n}{2}]$ into consecutive sets $\{1,2,3\},\{4,5,6\},\ldots,\{\frac{\lambda n}{2}-2,\frac{\lambda n}{2}-1,\frac{\lambda n}{2}\}$. For each $i\in I:=[\frac{\lambda n}{6}]$, define $\mathcal{F}_i$ to be a directed $4$-graph with vertex $V$ and edge set 
$$
\left\{(v_1,v_2,v_3,v_4)\in V^4:3(i-1)+\ell\in L({v_{\ell}v_{\ell+1}})\ \text{for each}\ \ell\in [3]\right\}.
$$
The minimum semi-degree condition implies that $e(\mathcal{F}_i)\geq n\left(\frac{1}{2}-\mu\right)^3n^3\geq \frac{n^4}{9}$.

Given a color $c\in [n]$ and two not necessarily distinct vertices $v,v'\in V$ (which are not required to be $D_c$-good), let  $Z_i(c,vv')$ be the collection of $(v_1,v_2,v_3,v_4)$ for which $v_1v_2v_3v_4$ is a Type-I directed $c$-absorbing path of $(v,v')$. Since $D_{3i+2}$ is $\alpha$-nice, there are at least $\alpha n^2$ ways to choose $v_2,v_3$ with $v_2\in N_{c}^-(v), v_3\in N_{{3i+2}}^+(v')$ and ${v_2v_3}\in E(D_{3i+2})$. When $v_2,v_3$ are fixed, there are at least $\left(\frac{1}{2}-\mu\right)^2n^2$ ways to choose $v_1$ and $v_4$. Thus $|Z_i(c,vv')|\geq \alpha\left(\frac{1}{2}-\mu\right)^2n^4\geq \frac{\alpha n^4}{5}$.

Applying Lemma \ref{LEMMA:directed-k-graph-matching} with 
\begin{align*}
  &t:=|I|,\ \epsilon:=\frac{\alpha}{5},\ \mathbf{H}:=\{\mathcal{F}_i:i\in I\},\ \text{and}\\
  &\mathbf{Z}:=\left\{Z(c,vv'):=\bigcup_{i\in I}Z_i(c,vv'):c\in [n],v,v'\in V\right\},
\end{align*}
we obtain a rainbow  matching $M$ inside $\mathbf{H}$ of size at least $(1-\frac{\alpha^2}{100})\frac{\lambda n}{6}$ such that $|E(Z(c,vv'))\cap E(M)|\geq \frac{\alpha^2\lambda n}{600}$ for all $c\in [n]$ and $v,v'\in V$. That is, there is a set $I'\subseteq I$ with $|I'|\geq (1-\frac{\alpha^2}{100})\frac{\lambda n}{6}$ such that there is a rainbow directed path $P^i=v_1^iv_2^iv_3^iv_4^i$ with $(v_1^i,v_2^i,v_3^i,v_4^i)\in \mathcal{F}_i$ for each $i\in I'$, and for every $c\in [n]$ and $v,v'\in V$, there are at least $\frac{\alpha^2\lambda n}{600}$ directed $c$-absorbing paths $P^i$ of $(v,v')$.

By relabeling indices, one may assume $I'=:[s]$ with $s=(1-\frac{\alpha^2}{100})\frac{\lambda n}{6}$. We are to connect the directed paths $P^1,\ldots,P^s$ one by one into a rainbow directed cycle $C$. Firstly, we connect $P^1$ and $P^2$ into a single rainbow directed path
$P^1xyP^2$. For this, choose distinct unused colors $c_1,c_2\in [\lambda n+1,n]$ and $c_3\in [\frac{\lambda n}{2}+1,\lambda n]$. Let
$U:=V\setminus V(\bigcup_{i\in [s]}P^i)$. Clearly, $|U|\geq n-\lambda n$. It follows from $\delta^0(\mathcal{D})\geq \frac{1-\mu}{2}n$ that 
\begin{align}\label{eq:degree}
  d_{{c_1}}^+(v_4^1,U),\,d_{{c_2}}^-(v_1^2,U)\geq \left(\frac{1-\mu}{2}-\lambda\right)n\geq \left(\frac{1}{2}-\alpha\right)n.
\end{align}
Notice that $D_{c_3}$ is $\alpha$-nice. Therefore, $|E(D_{c_3}[N_{{c_1}}^+(v_4^1,U),N_{{c_2}}^-(v_1^2,U)])|\geq \alpha n^2$. % there are at least $\alpha n^2$ $c_3$-colored out edges between $N_{{c_1}}^+(v_4^1,U)$ and $N_{{c_2}}^-(v_1^2,U)$. 
Hence, there is an edge ${xy}\in E(D_{c_3})$ with $x\in N_{{c_1}}^+(v_4^1,U)$ and $y\in N_{{c_2}}^-(v_1^2,U)$ such that $P^1xyP^2$ is a rainbow directed path with color set ${\rm col}(P^1\cup P^2)\cup \{c_1, c_2, c_3\}$. 

We apply the above process for each pair $(P^i,P^{i+1})$ for each $i\in [s]$ with $P^{s+1}:=P^1$. In each process, we need two unused colors from $[\lambda n+1,n]$ and one unused color from $[\frac{\lambda n}{2}+1,\lambda n]$. This is feasible since $2s\leq \frac{\lambda n}{3}$. The total number of used vertices at each step is at most $6s\leq {\lambda n}$, so we have the
same degree bounds as \eqref{eq:degree}. It is straightforward to verify that the resulting directed cycle $C$ is a Type-I directed absorbing  cycle with parameters $(1,0,\frac{6s}{n},\frac{\alpha^2\lambda}{600})$ and hence with parameters $(1,0,\lambda,\lambda^2)$ since $\lambda\ll \alpha$.
\end{proof}

The extremality of a digraph collection also depends crucially on the relationships between different digraphs. This leads us to introduce the following definitions.

%Extremality of a digraph collection depends on where digraphs are in relation to one another, hence we make the following definitions.

\begin{definition}[crossing, cross graph, weakly stable] 
Let $0<\frac{1}{n}\ll \epsilon,\delta<1$ with $n\in  \mathbb{N}$. Let $\mathcal{D}=\{D_1,\ldots,D_n\}$ be a collection of digraphs on a common vertex set $V$ of size $n$. Choose $i,j\in [n]$ such that $D_i$ is $(\epsilon,{\rm EC}t)$-extremal and $D_j$ is $(\epsilon,{\rm EC}s)$-extremal with $t,s\in [3]$. We say that $D_i$ and $D_j$ are $\delta$-\textit{crossing} if one of the following holds:
\begin{enumerate}
  \item[{\rm \bf (A1)}]  $t,s\in \{1,2\}$, $|A_i\triangle A_j|\geq \delta n$ and  $|A_i\triangle B_j|\geq \delta n$;
  \item[{\rm \bf (A2)}]  $t=s=3$, either $|W_i^1\triangle W_j^1|\geq \delta n$ and $|W_i^1\triangle W_j^3|\geq \delta n$, or $|W_i^2\triangle W_j^2|\geq \delta n$ and $|W_i^2\triangle W_j^4|\geq \delta n$;
  \item[{\rm \bf (A3)}]  $t\in \{1,2\},s=3$, either $|A_i\triangle W_j^1|\geq \delta n$ and $|A_i\triangle W_j^3|\geq \delta n$, or $|A_i\triangle W_j^2|\geq \delta n$ and $|A_i\triangle W_j^4|\geq \delta n$.
\end{enumerate}

For $k\in [3]$, define the \textit{cross graph} $C_{\mathcal{D},k}^{\epsilon,\delta}$ to be the graph with vertex set $[n]$ where $i$ and $j$ are adjacent if and only if $D_i,D_j$ are satisfying {\bf (Ak)}. 
We say that $\mathcal{D}$ is $(\epsilon,\delta)$-\textit{weakly stable} if $e(C_{\mathcal{D},k}^{\epsilon,\delta})\geq \delta n^2$ for some $k\in [3]$.

\end{definition}
The subsequent observation is an immediate consequence of the above definition. 
\begin{observation}
  Assume that $0<\epsilon\leq \frac{\delta}{8}$. 
  \begin{itemize}
    \item If $D_i$ and $D_j$ are satisfying {\bf (A1)}, then $|X_i\cap Y_j|\geq \frac{\delta n}{4}$, where $X,Y\in \{A,B\}$.
    \item If $D_i$ and $D_j$ are satisfying {\bf (A2)}, then $|W_i^x\cap W_j^y|\geq \frac{\delta n}{4}$, where either $x,y\in \{1,3\}$ or $x,y\in \{2,4\}$.
    \item If $D_i$ and $D_j$ are satisfying {\bf (A3)}, then $|X_i\cap W_j^y|\geq \frac{\delta n}{4}$, where $X\in \{A,B\}$ and either $y\in \{1,3\}$ or $y\in \{2,4\}$.
  \end{itemize}
\end{observation}
%\begin{definition}[Type-II directed absorbing path and absorbing cycle]
%Given any two not necessarily distinct vertices $u,v\in V$, a transversal directed path $P=v_1v_2v_3v_4$ with $u,v\not\in V(P)$ is called a directed $c$-absorbing path of $(v,u)$ if $c\in L({vv_3})$ and ${\rm col}({v_2v_3})\in L({v_2u})$. Similar to Definition \ref{def-abaosbing}, we define Type-II directed absorbing  cycle.
%\end{definition}
    
Our following lemma finds a directed absorbing cycle when $\mathcal{D}$ is weakly stable.
\begin{lemma}\label{weakly}
    Let $0<\frac{1}{n}\ll \lambda,\mu \ll 8\epsilon^{1/2}\leq  \delta\leq \epsilon^{1/3}\ll 1$.  %and $\delta\leq\min\{\zeta,\frac{1}{2}-\zeta-\epsilon\}$ (where $\zeta$ is defined in Lemma~\ref{thm-single-graph}).
    Assume that $\mathcal{D}=\{D_1,\ldots,D_n\}$ is a collection of  digraphs on a common vertex set $V$ of size $n$ with $\delta^0(\mathcal{D})\geq \frac{1-\mu}{2}n$. If $\mathcal{D}$ is $(\epsilon,\delta)$-weakly stable, then there exists a directed  absorbing  cycle with parameters $(\frac{\delta}{3},\sqrt{\epsilon},\lambda,\lambda^2)$.
\end{lemma}
\begin{proof}
Since $\mathcal{D}$ is $(\epsilon,\delta)$-weakly stable, one has $e(C_{\mathcal{D},k}^{\epsilon,\delta})\geq \delta n^2$ for some $k\in [3]$. In particular, if $k=2$, then we consider only $|W_i^1\triangle W_j^1|\geq \delta n$ and $|W_i^1\triangle W_j^3|\geq \delta n$ for all $ij\in E(C_{\mathcal{D},2}^{\epsilon,\delta})$; if $k=3$, then we consider only $|A_i\triangle W_j^1|\geq \delta n$ and $|A_i\triangle W_j^3|\geq \delta n$ for all $ij\in E(C_{\mathcal{D},3}^{\epsilon,\delta})$ (since the other case can be discussed similarly).

For each $i\in [\lfloor\frac{n}{3}\rfloor]$, define $\mathcal{F}_i$ to be a directed $4$-graph with vertex set $V$ and edge set
$$
\left\{(v_1,v_2,v_3,v_4)\in V^4:3(i-1)+\ell\in L(v_{\ell}v_{\ell+1})\ \text{for}\ \ell\in [3]\right\}.
$$
If follows from $\delta^0(\mathcal{D})\geq \frac{1-\mu}{2}n$ that $e(\mathcal{F}_i)\geq n\left(\left(\frac{1}{2}-\mu\right)n\right)^3\geq \frac{n^4}{9}$.  

In order to apply Lemma \ref{LEMMA:directed-k-graph-matching}, we first give the following claim, which requires only the minimum semi-degree condition.
%{\color{blue}Next, we define a Type-II directed $c$-absorbing path: Given any two not necessarily distinct vertices $u,v\in V$, a transversal directed path $P=v_1v_2v_3v_4$ with $u,v\not\in V(P)$ is called a directed $c$-absorbing path of $(v,u)$ if $c\in L({vy})$ and ${\rm col}(v_2v_3)\in L({xu})$. By a similar discussion as above,  we can show the following claim. }
\begin{claim}\label{weakly-absorbing}
Assume that $ij\in E(C_{\mathcal{D},k}^{\epsilon,\delta})$ for some $k\in [3]$. Let $u$ be a $D_j$-good vertex and $v$ be a $D_i$-good vertex with $j+1<3n$.
For ${\rm K}\in \{{\rm I,II}\}$, let $Z_j^{\rm K}(i,vu)$ be the collection of $(v_1,v_2,v_3,v_4)\in \mathcal{F}_{\frac{j+1}{3}}$ where $v_1v_2v_3v_4$ is a Type-{\rm K} directed $i$-absorbing path
of $(v,u)$. Then $|Z_j^{\rm K}(i,vu)|\geq 2^{-7}\delta n^4$.
\end{claim}

\begin{claimproof}[Proof of Claim \ref{weakly-absorbing}]
    Fix any such $i,j,u,v$.
 We proceed by considering the  following cases.

    \medskip
      {\bf Case 1.  $k=1$.}
    \medskip
    
    In this case, we prove $|Z_j^{\rm I}(i,vu)|\geq 2^{-7}\delta n^4$. By Lemma \ref{thm-single-graph} (i) and (ii), we have either $d_{i}^-(v,A_i)\geq (\frac{1}{2}-2\epsilon)n$ or $d_{i}^-(v,B_i)\geq (\frac{1}{2}-2\epsilon)n$. Without loss of generality, we assume that the former case holds since the latter case can be proved similarly. Note that $D_i$ and $D_j$ are $\delta$-crossing. Then $|A_i\cap  A_j|\geq \frac{\delta n}{4}$ and  $|A_i\cap B_j|\geq \frac{\delta n}{4}$. Therefore,  
    $$
    d_{i}^-(v,A_j)\geq |A_i\cap  A_j|-(|A_i|-d_{i}^-(v,A_i))\geq \frac{\delta n}{4}-\epsilon n\geq \frac{\delta n}{5}\ (\text{since}\ \epsilon\leq  \delta^2).
    $$
    Similarly, we have $d_{i}^-(v,B_j)\geq\frac{\delta n}{5}$. 
    
    Since $u$ is $D_j$-good, one has $u\in Y_j$ for some $Y\in \{A,B\}$. Let $Z=Y$ if $D_j$ is $(\epsilon,{\rm EC}1)$-extremal, and let $Z\in \{A,B\}\setminus \{Y\}$ if $D_j$  is $(\epsilon,{\rm EC}2)$-extremal. Notice that $|N_{j}^+(w,Z_j)|\geq (\frac{1}{2}-2\epsilon)n$ for each vertex $w\in Y_j$.     Choose $x\in N_{i}^-(v,Y_j)$. Then 
$$|N_{j}^+(u,Z_j)\cap N_{j}^+(x,Z_j)|\geq |N_{j}^+(u,Z_j)|+|N_{j}^+(x,Z_j)|-|Z_j|\geq \left(\frac{1}{2}-3\epsilon\right)n.$$
Any choice  
of $y\in N_{j}^+(u,Z_j)\cap N_{j}^+(x,Z_j),\ x'\in N_{{j-1}}^-(x)$ and $y'\in N_{{j+1}}^+(y)$ yields a directed $i$-absorbing path $x'xyy'$ of $(v,u)$ in $\mathcal{F}_{\frac{j+1}{3}}$. The
number of such paths is therefore at least $\frac{\delta n}{5}(\frac{1}{2}-3\epsilon)n\left(\frac{1}{2}-\mu\right)^2n^2\geq 2^{-7}\delta n^4$.

   % We divide the proof into the following two subcases.

%{\bf Subcase 1.1.} $u\in A_j$ and $D_j$ is $\epsilon$-close to EC1.

%Notice that $|N_{j}^+(w,Y_j)|\geq (\frac{1}{2}-2\epsilon)n$ for each vertex $w\in Y_j$ with $Y\in \{A,B\}$. Choose $x\in N_{i}^-(v,A_j)$. Then 
%$$|N_{j}^+(u,A_j)\cap N_{j}^+(x,A_j)|\geq |N_{j}^+(u,A_j)|+|N_{j}^+(x,A_j)|-|A_j|\geq \left(\frac{1}{2}-3\epsilon\right)n.$$
%Any choice of $y\in N_{j}^+(u,A_j)\cap N_{j}^+(x,A_j),\ x'\in N_{{j-1}}^-(x)$ and $y'\in N_{{j+1}}^+(y)$ yields a directed $i$-absorbing path $x'xyy'$ of $(v,u)$ in $\mathcal{F}_{\frac{j+1}{3}}$. The number of such paths is therefore at least $\frac{\delta n}{5}(\frac{1}{2}-3\epsilon)n\left(\frac{1}{2}-\mu\right)^2n^2\geq 2^{-7}\delta n^4$.

%{\bf Subcase 1.2.} $u\in A_j$ and $D_j$ is $(\epsilon,{\rm EC}2)$-extremal,. 

%Notice that $|N_{j}^+(w,Y_j)|\geq (\frac{1}{2}-2\epsilon)n$ for each vertex $w\in Z_j$ with $\{Y,Z\}= \{A,B\}$. Choose $x\in N_{i}^-(v,A_j)$. Then 
%$$
%|N_{j}^+(u,B_j)\cap N_{j}^+(x,B_j)|\geq |N_{j}^+(u,B_j)|+|N_{j}^+(x,B_j)|-|B_j|\geq \left(\frac{1}{2}-3\epsilon\right)n.
%$$
%Any choice of $y\in N_{j}^+(u,B_j)\cap N_{j}^+(x,B_j),\ x'\in N_{{j-1}}^-(x)$ and $y'\in N_{{j+1}}^+(y)$ yields a directed $i$-absorbing path $x'xyy'$ of $(v,u)$ in $\mathcal{F}_{\frac{j+1}{3}}$. The number of such paths is therefore at least $\frac{\delta n}{5}(\frac{1}{2}-3\epsilon)n\left(\frac{1}{2}-\mu\right)^2n^2\geq 2^{-7}\delta n^4$.

%The remaining cases when $u\in B_j$ are identical, so we omit its proof.

\medskip
{\bf Case 2. $k=2$.}
\medskip

%Since $|E(C_{\mathcal{D},2}^{\epsilon,\delta})|\geq \delta n^2$, b
%Based on {\bf (A2)}, without loss of generality, we may assume that $|W_i^1\triangle W_j^1|\geq \delta n$ and $|W_i^1\triangle W_j^3|\geq \delta n$. % holds for at least $\frac{\delta}{2}n^2$ pairs. 
In this case, we prove $|Z_j^{\rm II}(i,vu)|\geq 2^{-7}{{\delta}}n^4$. By Lemma \ref{thm-single-graph} (iii), we have either $d_{D_i}^+(v,W_i^1)\geq (\frac{1}{2}-3\epsilon)n$ or $d_{D_i}^+(v,W_i^3)\geq (\frac{1}{2}-3\epsilon)n$. Without loss of generality, we assume that the former case holds since the latter case can be proved similarly. Note that $D_i$ and $D_j$ are $\delta$-crossing. Then $|W_i^1\cap  W_j^1|\geq \frac{\delta n}{4}$ and  $|W_i^1\cap W_j^3|\geq \frac{\delta n}{4}$. Therefore,  
$$
    d_{i}^+(v,W_j^1)\geq |W_i^1\cap  W_j^1|-\left(|W_i^1|-d_{i}^+(v,W_i^1)\right)\geq \frac{\delta n}{4}-2\epsilon n\geq \frac{\delta n}{5}\ (\text{since}\ \epsilon\leq  \delta^2).
$$
Similarly, we have $d_{i}^+(v,W_j^3)\geq\frac{\delta n}{5}$. 

Since $u$ is $D_j$-good, one has $u\in Y_j$ for some $Y\in \{W^1,W^3\}$. Let $Z=W^4$ if $Y=W^1$, and let $Z=W^2$ if $Y=W^3$.  
Note that $|N_{j}^-(w,Z_j)|\geq (\frac{1}{2}-3\epsilon)n$ for each vertex $w\in Y_j$. Choose $x\in N_{i}^+(v,Y_j)$. Then 
$$
|N_{j}^-(u,Z_j)\cap N_{j}^-(x,Z_j)|\geq |N_{j}^-(u,Z_j)|+|N_{j}^-(x,Z_j)|-|Z_j|\geq \left(\frac{1}{2}-5\epsilon\right)n.
$$
Any choice  
of $y\in N_{j}^-(u,Z_j)\cap N_{j}^-(x,Z_j),\ x'\in N_{{j+1}}^+(x)$ and $y'\in N_{{j-1}}^-(y)$ yields a Type-II directed  $i$-absorbing  path $y'yxx'$ of $(v,u)$ in $\mathcal{F}_{\frac{j+1}{3}}$. Therefore, the number of such paths is at least $\frac{\delta n}{5}(\frac{1}{2}-5\epsilon)n\left(\frac{1}{2}-\mu\right)^2n^2\geq 2^{-7}\delta n^4$.

\medskip
{\bf Case 3. $k=3$.}
\medskip

This case can be proved by a similar discussion, so we omit the proof.
\end{claimproof}

For each vertex $v$, let
$$
\mathcal{C}_v:=\{i\in [n]:v\ \text{is}\ D_i\text{-good}\}.%\ \text{and}\ D_i\ \text{is $\epsilon$-close to EC}\},
$$
%where EC=EC1 or EC2 if $k\in\{1,3\}$ and EC=EC2 if $k=2$. 
Since $e(C_{\mathcal{D},k}^{\epsilon,\delta})\geq \delta n^2$, there is a subgraph $H$ of $C_{\mathcal{D},k}^{\epsilon,\delta}$ such that $|V(H)|\geq \delta n$ and $\delta(H)\geq \delta n$. For each $i\in V(H)$, define 
$$
T_i:=\left\{u\in V: |N_H(i)\setminus \mathcal{C}_u|\geq \frac{\delta n}{2}\right\}.
$$
Notice that $|L_i|=2\epsilon n$ when $D_i$ is $\epsilon$-extremal. Hence  
$|T_i|\frac{\delta n}{2}\leq 2\epsilon n^2$, that is, 
$|T_i|\leq \frac{\sqrt{\epsilon}}{2}n$ since $8\epsilon^{1/2}\leq\delta$. 

For each $i\in V(H)$, let $\overline{T_i}:=V\setminus T_i$. Then $|\overline{T_i}|\geq (1-\frac{\sqrt{\epsilon}}{2})n$.  For each $u\in \overline{T_i}$ with $i\in V(H)$, one has $|N_H(i)\cap \mathcal{C}_u|\geq \frac{\delta n}{2}$. 
Now we independently and randomly select vertices from $V(H)$ with probability $\kappa:=\frac{\lambda}{14}$ to obtain a color set $\mathcal{U}$. Using Chernoff's bound, with high probability the following hold:
\begin{enumerate}
    \item[{\rm ($*$)}] $\frac{\kappa |V(H)|}{2}\leq t:=|\mathcal{U}|\leq 2\kappa |V(H)|$,
    \item[{\rm ($**$)}] for every $i\in V(H)$ and $u\in \overline{T_i}$, we have $|N_H(i,\mathcal{U})\cap \mathcal{C}_u|\geq \frac{\delta \kappa n}{4}$.
\end{enumerate}
Fix such a color set $\mathcal{U}$ and one may assume $\mathcal{U}=\{3j-1:j\in [t]\}$. Let $\overline{\mathcal{U}}:=V(H)\cap [3t+1,n]$. It is easy to see that $\mathcal{U}\cap \overline{\mathcal{U}}=\emptyset$, $|\overline{\mathcal{U}}|\geq \frac{\delta n}{2}$ and $\delta(H[\overline{\mathcal{U}}])\geq \frac{\delta n}{2}$. 
%For each color $j\in [t]$, define $I_j$ be the collection of $(i,vu)\in [n]\times V^2$ with $i\in \overline{\mathcal{U}}$, $u\in \overline{T_i}$, $j\in \mathcal{C}_u\cap N_H(i,\mathcal{U})$ and $v$ is $D_i$-good.

%Given $i$ and $u$, the number of choices of $j$ is at least $\frac{\delta \kappa n}{4}$ by ($**$).  
Given $i\in \overline{\mathcal{U}}$ and $u\in \overline{T_i}$, choose $j\in \mathcal{C}_u\cap N_H(i,\mathcal{U})$. Then $\frac{j+1}{3}\in [t]$ and the number of such choices of $j$ is at least $\frac{\delta \kappa n}{4}\geq \frac{\delta t}{8}$ by ($**$). Let $v$ be a $D_i$-good vertex. It follows from Claim \ref{weakly-absorbing} that $|Z_j^{\rm K}(i,vu)|\geq 2^{-7}\delta n^4$ for some ${\rm K}\in \{{\rm I,II}\}$. %there are at least $2^{-7}\delta n^4$ Type-K directed $i$-absorbing paths of $(v,u)$ whose ordered vertex set is in $\mathcal{F}_{\frac{j+1}{3}}$ for some ${\rm K}\in \{{\rm I,II}\}$, that is, 

By ($*$) and the fact that $\lambda\ll \delta$, one has $\frac{t}{n}\leq 2\kappa=\frac{\lambda}{7}\ll 2^{-7}\delta$. Applying Lemma \ref{LEMMA:directed-k-graph-matching} with 
\begin{align*}
   &\gamma:=\frac{\lambda}{7},\ \epsilon:=2^{-7}\delta,\ \mathbf{H}:=\{\mathcal{F}_j:j\in [t]\}\ \  \text{and} \\
  & \mathbf{Z}:=\left\{Z(i,vu):=\bigcup_{j\in \mathcal{C}_u\cap N_H(i,\mathcal{U})}Z_j^{\rm K}(i,vu):i\in \overline{\mathcal{U}},\ u\in \overline{T_i}\ \text{and}\ v\ \text{is}\ D_i\text{-good}\right\},
  %&{\rm or}\  \mathbf{Z}:=\{Z(i,vu)=\bigcup_{j\in \mathcal{C}_u\cap N_H(i,\mathcal{U})}Z_j^{\rm II}(i,vu):i\in \overline{\mathcal{U}}, u\in \overline{T_i}\ \text{and}\ v\ \text{is}\ D_i\text{-good}\}, 
\end{align*}
%The minimum semi-degree condition implies that $e(\mathcal{F}_i)\geq \frac{4}{9}n^4$ for all $j\in [t]$. %Furthermore, for each $j\in $
%Furthermore, by ($*$) we know that $\mathbf{H}$ contains $\frac{\kappa \delta n}{2}\leq t\leq 2\kappa n$ graphs, so $\frac{t}{n}\leq 2\kappa=\frac{\lambda}{7}\ll 2^{-7}\delta^2$. 
%For every $(i,vu)$ with $Z(i,vu)\in \mathbf{Z}$ and every fixed $j\in \mathcal{C}_u\cap N_H(i,\mathcal{U})$, we have $\frac{j+1}{3}\in [t]$ and $|E(Z_j^{\rm I}(i,vu))\cap E(\mathcal{F}_{\frac{j+1}{3}})|\geq 2^{-7}\delta^2n^4$. The number of such $j$ is at least $\frac{\delta \kappa n}{4}\geq \frac{\delta t}{8}$ by ($**$). 
we obtain that there is a rainbow matching $M$ of $\mathbf{H}$ with size at least $(1-2^{-16}\delta^2)t$ such that $|E(Z(i,vu))\cap E(M)|\geq 2^{-16}\delta^2t$ for all $Z(i,vu)\in \mathbf{Z}$. That is, there is a subset $I\subseteq [t]$ with $|I|\geq (1-2^{-16}\delta^2)t$ such that there is a Type-K directed path $P^j=v_1^jv_2^jv_3^jv_4^j\in\mathcal{F}_j$ for each $j\in I$, and for every $i\in \overline{\mathcal{U}}$, $u\in \overline{T_i}$ and a $D_i$-good vertex $v$, there are at least $2^{-16}\delta^2 t$ Type-K directed $i$-absorbing paths $P^j$ of $(v,u)$.

Relabel indices such that $I=:[s]$ with $s\geq (1-2^{-16}\delta^2)t$. We are to connect those disjoint rainbow  directed paths $P^1,\ldots,P^s$ into a directed  absorbing  cycle $C$. Notice that  $\overline{\mathcal{U}}\cap {\rm col}(\bigcup_{i\in I}P^i)=\emptyset$. We first connect $P^1$ and $P^2$ into a single rainbow directed path $P^1v_4^1xzyv_1^2P^2$ or $P^1v_4^1xzv_1^2P^2$. For this, choose distinct colors $c_1,c_2,c_3,c_4\in \overline{\mathcal{U}}$ such that $c_1c_2\in E(H)$. Clearly, there are at least $|\overline{\mathcal{U}}|\geq \frac{\delta n}{2}$ choices for $c_1$, and given this at least $\delta(H[\overline{\mathcal{U}}])\geq \frac{\delta n}{2}$ choices for $c_2$, and at least  $|\overline{\mathcal{U}}|-3>\frac{\delta n}{3}$ choices for each of  $c_3,c_4$. Next, choose an unused  $D_{c_1}$-good vertex $x\in N^+_{{c_3}}(v_4^1)$ and an unused $D_{c_2}$-good vertex $y\in N^-_{{c_4}}(v_1^2)$. The following claim considers the number of choices for $x,y$. % two unused vertices $x,y$ such that $x\in N^+_{{c_3}}(v_4^1)$ is $D_{c_1}$-good and $y\in N^-_{{c_4}}(v_1^2)$ which is $D_{c_2}$-good. 
%there are at least $\left(\frac{1}{2}-\mu\right)n-4e(M)-2\epsilon n$ choices for each of these. Finally, choose $z\in N_{{c_1}}^+(x)\cap N_{{c_2}}^-(y)$ or $N_{{c_1}}^+(x)\cap N_{{c_4}}^-(v_1^2)$ %and the size of $|N_{{c_1}}^+(x)\cap N_{{c_2}}^-(y)|$ is determined 
%by the following claim, each of which has size at least $\frac{\delta n}{5}$. 

\begin{claim}\label{numberz}
  One of the following holds:
  \begin{itemize}
      \item there exists a rainbow directed path $P^1v_4^1xzyv_1^2P^2$ with colors ${\rm col}(P^1\cup P^2)\cup \{c_1,c_2,c_3,c_4\}$ and $z\in N_{{c_1}}^+(x)\cap N_{{c_2}}^-(y)$, each of $x,y,z$ has at least $\frac{\delta n}{10}$ choices; 
      \item there exists a rainbow directed path $P^1v_4^1xyv_1^2P^2$ with colors ${\rm col}(P^1\cup P^2)\cup \{c_1,c_3,c_4\}$ and $y\in N_{{c_1}}^+(x)\cap N_{{c_4}}^-(v_1^2)$, each of $x,y$ has at least $\frac{\delta n}{10}$ choices.
  \end{itemize}%Either $|N_{{c_1}}^+(x)\cap N_{{c_2}}^-(y)|\geq \frac{\delta n}{5}$ or 
\end{claim}
\begin{claimproof}[Proof of Claim \ref{numberz}] 
We proceed by considering the following three cases. 

\medskip
\textbf{Case 1. $k=1$.} 
\medskip

In this case, $D_{c_1}$ and $D_{c_2}$ are either $(\epsilon,{\rm EC}1)$-extremal or $(\epsilon,{\rm EC}2)$-extremal. Then $x\in A_{c_1}\cup B_{c_1}$ and $y\in A_{c_2}\cup B_{c_2}$. Clearly, each of $x,y$ has at least 
$\frac{n}{4}$ choices. Recall that $D_{c_1}$ and $D_{c_2}$ are $\delta$-crossing. Then, $|X_{c_1}\cap Y_{c_2}|\geq \frac{\delta n}{4}$ whenever $X,Y\in \{A,B\}$. By Lemma~\ref{thm-single-graph}, there are $Z,W\in \{A,B\}$ such that $d_{{c_1}}^+(x,Z_{c_1})\geq |Z_{c_1}|-\epsilon n$ and $d_{{c_2}}^-(y,W_{c_2})\geq |W_{c_2}|-\epsilon n$. Choose $z\in N_{{c_1}}^+(x)\cap N_{{c_2}}^-(y)$ and the number of choices for $z$ is  at least 
$$
|N_{{c_1}}^+(x)\cap N_{{c_2}}^-(y)|-7e(M)\geq |N_{{c_1}}^+(x,Z_{c_1})\cap N_{{c_2}}^-(y,W_{c_2})|-14\kappa n\geq |Z_{c_1}\cap W_{c_2}|-2\epsilon n-12\kappa n\geq \frac{\delta n}{5},
$$
as desired.

\medskip
\textbf{Case 2. $k=2$.} 
\medskip

In this case, both $D_{c_1}$ and $D_{c_2}$ are $(\epsilon,{\rm EC}3)$-extremal. Then $|W_{c_1}^i\cap W_{c_2}^j|\geq \frac{\delta n}{4}$ for   $i,j\in \{1,3\}$. Notice that either $|W_{c_1}^1\cap W_{c_2}^2|\geq \frac{n}{5}$ or $|W_{c_1}^3\cap W_{c_2}^2|\geq \frac{n}{5}$ holds. Without loss of generality, assume the latter case hold. Choose $z\in N_{{c_1}}^+(x)\cap N_{{c_2}}^-(y)$. Next, we consider the size of $|N_{{c_1}}^+(x)\cap N_{{c_2}}^-(y)|$. 

If $|W_{c_1}^2\cap N^+_{{c_3}}(v_4^1)|\geq \delta n$ and $|W_{c_2}^3\cap N^-_{{c_4}}(v_1^2)|\geq \delta n$, then choose $x\in W_{c_1}^2\cap N^+_{{c_3}}(v_4^1)$ and $y\in W_{c_2}^3\cap N^-_{{c_4}}(v_1^2)$.  Hence, there are at least $\frac{\delta n}{2}$ choices for each of $x$ and $y$. Furthermore, $|N_{{c_1}}^+(x)\cap N_{{c_2}}^-(y)|\geq |W_{c_1}^3\cap W_{c_2}^2|-4\epsilon n\geq \frac{\delta n}{5}$, as desired.

If $|W_{c_1}^2\cap N^+_{{c_3}}(v_4^1)|< \delta n$, % or $|W_{c_2}^3\cap N^-_{{c_4}}(v_1^2)|<\delta n$. We only consider $|W_{c_1}^2\cap N^+_{{c_3}}(v_4^1)|< \delta n$, that is, 
then $|W_{c_1}^4\cap N^+_{{c_3}}(v_4^1)|\geq (\frac{1}{2}-\mu-\delta-2\epsilon) n$. Choose $x\in W_{c_1}^4\cap N^+_{{c_3}}(v_4^1)$. Then $d_{{c_1}}^+(x,W_{c_1}^1)\geq |W_{c_1}^1|-2\epsilon n$. Notice that either $|W_{c_2}^1\cap N^-_{{c_4}}(v_1^2)|\geq \delta n$ or $|W_{c_2}^3\cap N^-_{{c_4}}(v_1^2)|\geq \delta n$. We may assume the former case holds. Choose $y\in W_{c_2}^1\cap N^-_{{c_4}}(v_1^2)$. Then $d_{{c_2}}^-(y,W_{c_2}^4)\geq |W_{c_2}^4|-2\epsilon n$. 
If $|W_{c_1}^1\cap W_{c_2}^4|\geq \delta n$, then $|N_{{c_1}}^+(x)\cap N_{{c_2}}^-(y)|\geq |W_{c_1}^1\cap W_{c_2}^4|-4\epsilon n\geq \frac{\delta n}{5}$, as desired. 
If $|W_{c_1}^1\cap W_{c_2}^4|< \delta n$ and $|W_{c_2}^3\cap N^-_{{c_4}}(v_1^2)|\geq \frac{\delta n}{9}$, then $|W_{c_1}^1\cap W_{c_2}^2|\geq (\frac{1}{2}-\delta-3\epsilon)n$ and choose $y\in W_{c_2}^3\cap N^-_{{c_4}}(v_1^2)$. Similarly, $|N_{{c_1}}^+(x)\cap N_{{c_2}}^-(y)|\geq |W_{c_1}^1\cap W_{c_2}^2|-4\epsilon n\geq \frac{\delta n}{5}$, as desired. 
If $|W_{c_1}^1\cap W_{c_2}^4|< \delta n$ and $|W_{c_2}^3\cap N^-_{{c_4}}(v_1^2)|<\frac{\delta n}{9}$, then %$|W_{c_1}^1\cap W_{c_2}^2|\geq |W_{c_2}^2|-(\delta+2\epsilon)n$ and
$d^-_{{c_4}}(v_1^2,W_{c_2}^1)\geq |W_{c_2}^1|-(\mu+2\epsilon+\frac{\delta}{9})n$. %It follows that $|W_{c_1}^1\cap W_{c_2}^1|\geq|W_{c_1}^1\cap B_{c_2}^2|\geq |B_{c_2}^2|-(\delta+2\epsilon)n\geq 5\delta n$. 
Therefore, 
$$
|N_{{c_1}}^+(x)\cap N_{{c_4}}^-(v_1^2)|\geq |N_{{c_1}}^+(x,W_{c_1}^1)\cap N_{{c_4}}^-(v_1^2,W_{c_2}^1)|\geq |W_{c_1}^1\cap W_{c_2}^1|-\left(\mu+4\epsilon+\frac{\delta}{9}\right) n\geq \frac{\delta n}{9}.
$$
Now, we connect $P^1$ and $P^2$ by using a rainbow directed  path $P^1v_4^1xyv_1^2P^2$ with colors ${\rm col}(P^1\cup P^2)\cup \{c_1,c_3,c_4\}$, as desired.

\medskip
\textbf{Case 3. $k=3$.}
\medskip

In this case, $D_{c_1}$ is either  $(\epsilon,{\rm EC}1)$-extremal or  $(\epsilon,{\rm EC}2)$-extremal, and $D_{c_2}$ is $(\epsilon,{\rm EC}3)$-extremal. %Hence $x\in A_{c_1}\cup B_{c_1}$ and $y\in W_{c_2}^2\cup W_{c_2}^4$. Note that $D_{c_1}$ and $D_{c_2}$ are $\delta$-crossing. Then $|X_{c_1}\cap Y_{c_2}|\geq \frac{\delta n}{4}$ whenever $X\in \{A,B\},\ Y\in \{W^2,W^4\}$. By Lemma \ref{thm-single-graph}, there are $Z\in \{A,B\}$, and $W\in \{W^2,W^4\}$ such that $d_{D_{c_1}}^+(x,Z_{c_1})\geq |Z_{c_1}|-\epsilon n$ and $d_{D_{c_2}}^-(y,W_{c_2})\geq |W_{c_2}|-\epsilon n$. Thus $|N_{{c_1}}^+(x)\cap N_{{c_2}}^-(y)|\geq |Z_{c_1}\cap W_{c_2}|-2\epsilon n\geq \frac{\delta n}{5}$. This completes the proof of the claim.
Hence $x\in A_{c_1}\cup B_{c_1}$ and $y\in W_{c_2}^1\cup W_{c_2}^3$.  Choose $z\in N_{{c_1}}^+(x)\cap N_{{c_2}}^-(y)$.  Note that $D_{c_1}$ and $D_{c_2}$ are $\delta$-crossing. Then $|X_{c_1}\cap Y_{c_2}|\geq \frac{\delta n}{4}$ whenever $X\in \{A,B\}$ and $Y\in \{W^1,W^3\}$. By Lemma \ref{thm-single-graph}, there are $Z\in \{A,B\}$ and $W\in \{W^2,W^4\}$ such that $d_{{c_1}}^+(x,Z_{c_1})\geq |Z_{c_1}|-\epsilon n$ and $d_{{c_2}}^-(y,W_{c_2})\geq |W_{c_2}|-2\epsilon n$. If $|Z_{c_1}\cap W_{c_2}|\geq \frac{\delta n}{4}$, then $|N_{{c_1}}^+(x)\cap N_{{c_2}}^-(y)|\geq |Z_{c_1}\cap W_{c_2}|-3\epsilon n\geq \frac{\delta n}{5}$, as desired. 

Now, we consider $|Z_{c_1}\cap W_{c_2}|< \frac{\delta n}{4}$. Without loss of generality, assume that $Z=A$ and $W=W^2$. It follows that $|A_{c_1}\cap W_{c_2}^4|\geq |W_{c_2}^4|-(\frac{\delta}{4}+2\epsilon)n$. This implies $|A_{c_1}\cap W_{c_2}^3|\geq |C_{c_2}^4|-(\frac{\delta}{4}+2\epsilon)n\geq \frac{\delta n}{2}$ since $|C^2_4|\geq \delta n$. If $|W_{c_2}^1\cap N_{{c_4}}^-(v_2^1)|\geq \frac{\delta n}{9}$, then choose $y\in W_{c_2}^1\cap N_{{c_4}}^-(v_2^1)$. Hence $|N_{{c_1}}^+(x)\cap N_{{c_2}}^-(y)|\geq |A_{c_1}\cap W_{c_2}^4|-3\epsilon n\geq \frac{\delta n}{5}$. If $|W_{c_2}^1\cap N_{{c_4}}^-(v_2^1)|< \frac{\delta n}{9}$, then $|W_{c_2}^3\cap N_{{c_4}}^-(v_2^1)|\geq  |W_{c_2}^3|-(\frac{\delta}{9}+2\epsilon)n$. Thus, $|N_{{c_1}}^+(x)\cap N_{{c_4}}^-(v_2^1)|\geq |A_{c_1}\cap W_{c_2}^3|-(3\epsilon+\frac{\delta}{9}) n\geq \frac{\delta n}{9}$. Now, we connect $P^1$ and $P^2$ by using a rainbow directed  path $P^1v_4^1xyv_1^2P^2$ with colors ${\rm col}(P^1\cup P^2)\cup \{c_1,c_3,c_4\}$, as desired.

 This completes the proof of Claim \ref{numberz}.
 \end{claimproof}

In view of Claim \ref{numberz},  there are at least $\frac{\delta n}{10}$ choices for each of $c_1,c_2,c_3,c_4,x,y,(z)$ given any previous choices. Thus, we obtain the desired rainbow directed  path $P^1v_4^1xzyv_1^2P^2$ or $P^1v_4^1xyv_1^2P^2$. Applying the above process for each pair $(P^j,P^{j+1})$ with $j\in [s]$ where $P^{s+1}:=P^1$. Every time we use at most four unused colors in $\overline{\mathcal{U}}$ and at most three unused vertices.  {Therefore}, 
 for each choice, at most $4s\leq 4t\leq 8\kappa n\ll\frac{\delta n}{10}$ colors are forbidden and $3s\ll\frac{\delta n}{10}$ vertices are forbidden. Thus we obtain a rainbow directed  cycle $C$ of length at most $7s$ which contains every $P^j$ with $j\in [s]$ as a segment. 

By the construction of $C$, for every color $i\in \overline{\mathcal{U}}$, any $D_i$-good vertex $v$ and  for every vertex $u\in \overline{T_i}$, there are at least $2^{-16}\delta^2t\geq 2^{-17}\kappa \delta^3n$ disjoint directed $i$-absorbing paths of $(v,u)$. Moreover, the number of $D_i$-good vertices inside  $\overline{T_i}$ is at least $|\overline{T_i}|-2\epsilon n\geq (1-\frac{\sqrt{\epsilon}}{2})n-2\epsilon n\geq (1-{\sqrt{\epsilon}})n$. Thus $C$ is a directed  absorbing  cycle with parameters $(\frac{\delta}{3},\sqrt{\epsilon},\lambda,\lambda^2)$.
\end{proof}

\section{Stable case}

{{In this section, we employ the regularity-blow-up method for digraph collections $\mathcal{D}$ to prove the existence of a transversal directed Hamilton cycle when $\mathcal{D}$ is either strongly or weakly stable, combining the results from previous sections.}} Let $0<\gamma,\alpha,\epsilon,\delta<1$. We say that $\mathcal{D}$ is $(\gamma,\alpha,\epsilon,\delta)$-\textit{stable} if it is either $(\gamma,\alpha)$-strongly stable or $(\epsilon,\delta)$-weakly {stable}. We first show that the reduced digraph collection inherits stability.

\begin{lemma}\label{lem-reduced-stable}
   Let ${0<}\frac{1}{n}\ll \frac{1}{L_0}\ll \epsilon_0\ll d\ll \mu,\alpha\ll\gamma,\epsilon\leq \delta^2\ll1$. Assume that $\mathcal{D}=\{D_1,\ldots,D_n\}$ is a collection of digraphs on a common vertex set $V$ of size $n$ and $\delta^0(\mathcal{D})\geq \left(\frac{1}{2}-\mu\right)n$. Let $\mathcal{R}=\mathcal{R}(\epsilon_0,d,L_0)$ be the reduced digraph collection of {$\mathcal{D}$}. If {$\mathcal{D}$} is $(\gamma,\alpha,\epsilon,\delta)$-stable, then $\mathcal{R}$ is  $(\frac{\gamma}{2},\alpha^2,\epsilon,\frac{\delta}{2})$-stable.
\end{lemma}
\begin{proof}
We may assume that $\frac{1}{n}\le \frac{1}{n_0}$ where $n_0=n_0(\epsilon,1,L_0)$ is the constant from Lemma \ref{regularity-lemma}. Write $[L]$ for the common vertex set of digraphs in  $\mathcal{R}$ where $L_0\leq L\leq n_0$, and $[M]$ for the set of color clusters. Thus there is a partition $V_0,V_1,\ldots,V_L$ of $V$ and $\mathcal{C}_0,\mathcal{C}_1,\ldots,\mathcal{C}_M$ of $[n]$ and a digraph collection $\mathcal{D}'$ satisfying (i)-(v) of Lemma \ref{regularity-lemma}. 
Therefore, for each $\{(h,i),j\}\in {[L]\choose 2}\times [M]$, we have ${hi}\in R_j$ if and only if $\mathcal{D}_{hi,j}':=\{D_c'[V_h,V_i]:c\in \mathcal{C}_j\}$ is $(\epsilon_0,d)$-regular. Furthermore, $Lm\leq n\leq Mm+\epsilon_0n$ and $Mm\leq n\leq Lm+\epsilon_0n$, so $|L-M|\leq \frac{\epsilon_0n}{m}\leq \frac{\epsilon_0L}{1-\epsilon_0}$. Thus we may assume that $M=L$ at the expense of assuming the slightly worse bound $|V_0|+|\mathcal{C}_0|\leq 3\epsilon_0n$. Given $X\subset [L]$, we write $\hat{X}:=\bigcup_{j\in X}V_j$. Hence $|\hat{X}|=m|X|$.

\begin{claim}\label{claim-reduced}
Let $d\ll \alpha'$ and {$i\in [M]$}. Assume that there are sets $A,B\subset [L]$ of size at least
$(\frac{1}{2}-\alpha')L$ such that $e_{R_i}(A,B)\leq \alpha' L^2$. Then $e_{D_c}(\hat{A},\hat{B})\leq \sqrt{\alpha'}n^2$ for all but at most $2\sqrt{\alpha'}m$ colors $c\in \mathcal{C}_i$.
\end{claim}
\begin{claimproof}[Proof of Claim \ref{claim-reduced}]
  Let $t$ be the number of colors $c\in \mathcal{C}_i$ such that $e_{D_c}(\hat{A},\hat{B})> \sqrt{\alpha'}n^2$. Since $e(D_c)-e(D_c')\leq (24d+\epsilon_0)n^2$ for all $c\in [n]$, we have
  \begin{align*}
    \sqrt{\alpha'}n^2 t&\leq \sum_{c\in \mathcal{C}_i}e_{D_c}(\hat{A},\hat{B})\leq \sum_{c\in \mathcal{C}_i}\left(e_{D_c'}(\hat{A},\hat{B})+(24d+\epsilon_0)n^2\right)\\
    &\leq |\mathcal{C}_i|e_{R_i}(A,B)m^2+|\mathcal{C}_i|(24d+\epsilon_0)n^2\\
    &\leq \alpha'L^2m^3+25dmn^2\leq 2\alpha'mn^2.
\end{align*}
This implies that $t\leq 2\sqrt{\alpha'}m$.
\end{claimproof}

We consider the cases where $\mathcal{D}$ is strongly stable and weakly stable, respectively.
\medskip

{\bf Case 1.} $\mathcal{D}$ is $(\gamma,\alpha)$-strongly stable.
\medskip

We prove that $\mathcal{R}$ is $(\frac{\gamma}{2},\alpha^2)$-strongly stable. Suppose that there exists a subset $I\subseteq [L]$ with $|I|\geq (1-\frac{\gamma}{2})L$ such that $R_i$ is not $\alpha^2$-nice for any  $i\in I$. By Definition \ref{nice}, for each $i\in I$, there are two sets $A^i,B^i\subseteq [L]$ of size at least $(\frac{1}{2}-\alpha^2)L$ such that $e_{R_{i}}(A^i,B^i)<\alpha^2 L^2$. Applying Claim \ref{claim-reduced} with $\alpha':=\alpha^2$, we get that there are at least $(1-2\alpha)m$ colors {$c\in \mathcal{C}_{i}$} such that $e_{D_c}(\hat{A^i},\hat{B^i})\leq \alpha n^2$. Notice that  $|\hat{A^i}|,|\hat{B^i}|\geq (\frac{1}{2}-\alpha^2)Lm>(\frac{1}{2}-\alpha)n$. Then there are at least $(1-2\alpha)m$ colors $c\in \mathcal{C}_{i}$ such that $D_c$ is not $\alpha$-nice. Thus the number of colors $c\in [n]$ for which $D_c$ is not $\alpha$-nice is at least $(1-2\alpha)m|I|>(1-\gamma)n$ since $\alpha\ll \gamma$. This
contradicts the fact that $\mathcal{D}$ is $(\gamma,\alpha)$-{strongly stable}.

\medskip
{\bf Case 2.} $\mathcal{D}$ is $(\epsilon,\delta)$-weakly stable.
\medskip

Suppose that $\mathcal{R}$ is not $(\frac{\gamma}{2},\alpha^2)$-strongly stable. Then there is a subset $I\subseteq [L]$ with $|I|\geq (1-\frac{\gamma}{2})L$ such that $R_i$ is not $\alpha^2$-nice for any $i\in I$.  It suffices to
show that $\mathcal{R}$ is $(\epsilon,\frac{\delta}{2})$-weakly stable. Choose $i\in I$, then $R_i$ is $\alpha^2$-extremal. Based on Lemma \ref{degree-inheritance} (i), one has $d_{R_i}^+(j),d_{R_i}^-(j)\geq (\frac{1}{2}-2\mu)L$ for all but at most $d^{\frac{1}{4}}L$ vertices $j\in [L]$. Together with Lemma \ref{thm-single-graph}, we obtain that 
\begin{itemize}
  \item if $R_i$ is either $(\alpha^2,{\rm EC}1)$-extremal or $(\alpha^2,{\rm EC}2)$-extremal, then $R_i$ has a characteristic partition $(A_i',B_i',L_i')$ such that there are $Z,Y\in \{A,B\}$ satisfying  $e_{R_i}(Z_i',Y_i')\leq \alpha^2L^2$;
  \item if $R_i$ is $(\alpha^2,{\rm EC}3)$-extremal, then $R_i$ has a characteristic partition $({C_1^i}',{C_2^i}',{C_3^i}',{C_4^i}',L_i')$ such that $e_{R_i}(Z_i',Y_i')\leq \alpha^2L^2$ with $(Z_i',Y_i')\in\{({W_i^4}',{W_i^3}'),({W_i^2}',{W_i^1}')\}$.
\end{itemize}
Note that $|\hat{Z}_i'|,|\hat{Y}_i'|=(\frac{1}{2}-\alpha^2)Lm\geq (\frac{1}{2}-2\alpha)n$. In view of Claim \ref{claim-reduced}, we know that there is a subset $\mathcal{B}_i\subseteq \mathcal{C}_i$ with $|\mathcal{C}_i\setminus \mathcal{B}_i|\leq 2\alpha m$ such that for all $c\in \mathcal{B}_i$, we have $e_{D_c}(\hat{Z}_i',\hat{Y}_i')\leq \alpha n^2$.  It follows that $D_c$ is $2\alpha$-extremal for each $c\in \mathcal{B}_i$. 

If $R_i$ is either $(\alpha^2,{\rm EC}1)$-extremal or $(\alpha^2,{\rm EC}2)$-extremal, then each $D_c$ with $c\in \mathcal{B}_i$ admits a characteristic partition $(A_c,B_c,L_c)$, and there are $W\in \{A,B\}$ and $\{Z,Y\}=\{A,B\}$ such that $e_{D_c}(Z_c,W_c)\leq \sqrt{\alpha}n^2$ and $e_{D_c}(Y_c,W_c)\geq (\frac{1}{4}-\sqrt{\alpha})n^2$. Hence, we must have either $|\hat{Z}_i'\triangle Z_c|, |\hat{Y}_i'\triangle Y_c|\leq \epsilon n$ or $|\hat{Z}_i'\triangle Y_c|, |\hat{Y}_i'\triangle Z_c|\leq \epsilon n$. So, either $|\hat{A}_i'\triangle A_c|, |\hat{B}_i'\triangle B_c|\leq \epsilon n$ or $|\hat{A}_i'\triangle B_c|, |\hat{B}_i'\triangle A_c|\leq \epsilon n$. That is, for all $c\in \mathcal{B}_i$, the characteristic partition of $D_c$ is almost the same as the union of clusters corresponding to the characteristic partition of $R_i$.

If $R_i$ is $(\alpha^2,{\rm EC}3)$-extremal, then each $D_c$ admits a characteristic partition $(C_c^1, C_c^2, C_c^3, C_c^4, L_c)$ for $c\in \mathcal{B}_i$, and it satisfies $e_{D_c}(W_c^2,W_c^1),e_{D_c}(W_c^4,W_c^3)\leq \sqrt{\alpha}n^2$, $e_{D_c}(C_c^1),e_{D_c}(C_c^3)\geq |C_c^1|^2-\sqrt{\alpha}n^2$ and $e_{D_c}(C_c^2, C_c^4),\,e_{D_c}(C_c^4, C_c^2)\geq |C_c^2|^2-\sqrt{\alpha}n^2$. Therefore,  either $|\hat{Z}_i'\triangle W_c^2|, |\hat{Y}_i'\triangle W_c^1|\leq \epsilon n$ or $|\hat{Z}_i'\triangle W_c^4|, |\hat{Y}_i'\triangle W_c^3|\leq \epsilon n$. By relabeling, one may assume that $|{\hat{W_i^4}'}\triangle W_c^4|, |{\hat{W_i^3}'}\triangle W_c^3|\leq \epsilon n$ for each $c\in \mathcal{B}_i$. That is, for all $c\in \mathcal{B}_i$, the characteristic partition of $D_c$ is almost the same as the union of clusters corresponding to the characteristic partition of $R_i$.

Choose $c\in \mathcal{B}_i, c'\in \mathcal{B}_j$ for some $i,j\in [L]$. Suppose that $D_c$ and $D_{c'}$ are $\delta$-crossing.  We must have $i\neq j$. We consider the following possible cases. %Then, if $||$, we have

\medskip
$\bullet$ Suppose that $R_i$ and $R_j$ are  either $(\alpha^2,{\rm EC}1)$-extremal or $(\alpha^2,{\rm EC}2)$-extremal. %$D_c$ and $D_{c'}$ are either $(\epsilon,{\rm EC}1)$-extremal or $(\epsilon,{\rm EC}2)$-extremal. 
If $|\hat{A}_i'\triangle A_c|\leq \epsilon n$ and $|\hat{A}_j'\triangle A_{c'}|\leq \epsilon n$, then 
$$
|A_i'\triangle A_j'|m=|\hat{A}_i'\triangle \hat{A}_j'|\geq |A_c\triangle A_{c'}|-|\hat{A}_i'\triangle A_c|-|\hat{A}_j'\triangle A_{c'}|\geq \delta n-2\epsilon n\geq \frac{\delta n}{2}\geq \frac{\delta mL}{2}.
$$
The other cases are almost identical. Thus $R_i$ and $R_j$ are $\frac{\delta}{2}$-crossing.
\medskip

$\bullet$ %$D_c$ and $D_{c'}$ are $(\epsilon,{\rm EC}3)$-extremal. 
Suppose that $R_i$ and $R_j$ are $(\alpha^2,{\rm EC}3)$-extremal. Recall that $|{\hat{W_i^4}'}\triangle W_c^4|, |{\hat{W_i^3}'}\triangle W_c^3|\leq \epsilon n$. Then $|\hat{W_i^1}'\triangle W_c^1|\leq 10\epsilon n$ and $|\hat{W_i^2}'\triangle W_c^2|\leq 10\epsilon n$. Since $D_c$ and $D_{c'}$ are $\delta$-crossing, one may assume, without loss of generality, that $|W_c^1\triangle W_{c'}^1|\geq \delta n$ and $|W_c^1\triangle W_{c'}^3|\geq \delta n$. Hence 
$$
|{W_i^1}'\triangle {W_j^1}'|m=|{\hat{W_i^1}'}\triangle {\hat{W_j^1}'}|\geq |W_c^1\triangle W_{c'}^1|-|{\hat{W_i^1}'}\triangle W_c^1|-|{\hat{W_j^1}'}\triangle W_{c'}^1|\geq \delta n-21\epsilon n\geq \frac{\delta n}{2}\geq \frac{\delta mL}{2}.
$$
and similarly, $|{W_i^1}'\triangle {W_j^3}'|\geq \frac{\delta L}{2}$. Thus $R_i$ and $R_j$ are $\frac{\delta}{2}$-crossing.
%Similarly, $|{W_i^4}'\triangle {W_j^4}'|\geq \frac{\delta L}{2}$.

%If $|\hat{Z}_i'\triangle W_c^2|, |\hat{Y}_i'\triangle W_c^1|\leq \epsilon n$, then $|{W_i^1}'\triangle {W_j^1}'|\geq \frac{\delta L}{2}$ and $|{W_i^2}'\triangle {W_j^2}'|\geq \frac{\delta L}{2}$.

%The other cases are almost identical. 
\medskip

$\bullet$ Suppose that $R_i$ is either $(\alpha^2,{\rm EC}1)$-extremal or $(\alpha^2,{\rm EC}2)$-extremal, and $R_j$ is $(\alpha^2,{\rm EC}3)$-extremal. If $|\hat{A}_i'\triangle A_c|\leq \epsilon n$ and $|\hat{{W_j^3}'}\triangle W_c^3|\leq \epsilon n$, then 
$$
|A_i'\triangle {W_j^3}'|m= |\hat{A}_i'\triangle \hat{{W_j^3}'}|\geq |A_c\triangle W_{c'}^3|-|\hat{A}_i'\triangle A_c|-|\hat{{W_j^3}'}\triangle W_{c'}^3|\geq \delta n-2\epsilon n\geq \frac{\delta n}{2}\geq \frac{\delta mL}{2}.
$$ 
Similarly, $|B_i'\triangle {W_j^3}'|\geq \frac{\delta L}{2}$. 
The other cases are almost identical. 
Thus $R_i$ and $R_j$ are $\frac{\delta}{2}$-crossing. 
\medskip

Since $\mathcal{D}$ is $(\epsilon,\delta)$-weakly stable, we have $e(C_{\mathcal{D},k}^{\epsilon,\delta})\geq \delta n^2$ for some $k\in [3]$. The number of pairs $i,j\in [L]$ such that $R_i$ and $R_j$ are $\frac{\delta}{2}$-crossing is at least $\frac{\delta n^2}{m^2}\geq \frac{\delta L^2}{2}$. 
Thus $e(C_{\mathcal{R},k}^{\epsilon,\frac{\delta}{2}})\geq \frac{\delta L^2}{2}$, and hence $\mathcal{R}$ is $(\epsilon,\frac{\delta}{2})$-weakly stable.
\end{proof}

{{For a digraph collection $\mathcal{D}$  that is either strongly stable or weakly stable, Lemma~\ref{strongly} and Lemma~\ref{weakly} guarantee the existence of a directed absorbing cycle in 
$\mathcal{D}$. We shall then employ this absorbing property to prove that $\mathcal{D}$ contains a transversal directed Hamilton cycle.}}

\begin{lemma}\label{stable}
Let $0<\frac{1}{n}\ll \mu\ll \alpha\ll \gamma,8\epsilon^{1/2}\leq \delta\leq \epsilon^{1/3}\ll1$. Assume that $\mathcal{D}=\{D_1,\ldots,D_n\}$ is a collection of digraphs on a common vertex set $V$ of size $n$ with $\delta^0(\mathcal{D})\geq \left(\frac{1}{2}-\mu\right)n$. If $\mathcal{D}$ is $(\gamma,\alpha,\epsilon,\delta)$-stable, then $\mathcal{D}$ contains a transversal directed Hamilton cycle.
\end{lemma}
\begin{proof}
Choose additional parameters $n_0, L_0,\epsilon_0,d,\beta,\lambda$, where $n_0=n_0(\epsilon,1,L_0)$ is obtained from Lemma \ref{regularity-lemma}, % and $\zeta$ is obtained from Lemma \ref{thm-single-graph}, 
such that 
$$
  1<\frac{1}{n}\ll \frac{1}{n_0}\ll \frac{1}{L_0}\ll \epsilon_0\ll d\ll \mu\ll \beta\ll\lambda\ll \alpha\ll \gamma,8\epsilon^{1/2}\leq \delta\leq \epsilon^{1/3}\ll1. 
$$
%and $\delta\leq\min\{\zeta,\frac{1}{2}-\zeta-\epsilon\}$, where the previous lemmas hold with suitable parameters. 

Since $\mathcal{D}$ is $(\gamma,\alpha,\epsilon,\delta)$-stable, in view of Lemma \ref{strongly} and Lemma \ref{weakly}, we know that $\mathcal{D}$ has a directed  absorbing  cycle $C$ with parameters $(\frac{\delta}{3},\sqrt{\epsilon},\lambda,\lambda^2)$. Fix any two colors $c,c'\in [n]$, two vertices $x,y\in V$ and ${\rm{K\in \{I,II\}}}$, we say the triple $(c,x,y)$ is absorbable if there are at least $\lambda^2n$ disjoint Type-K directed $c$-absorbing paths of $(x,y)$ inside $C$; we say the pair $(c,x)$ is absorbable if there are at least $\lambda^2n$ disjoint Type-K directed $c$-absorbing paths of $(x,x)$ inside $C$; we say $(c,c',x,y)$ is totally absorbable if $(c,x),\ (c',y)$ and $(c,x,y)$ are all absorbable.  

By Definition \ref{def-abaosbing}, $C$ has length at most $\lambda n$, and there exists a color set $\mathcal{C}\subseteq [n]\setminus  {\rm col}(C)$ of
size at least $\frac{\delta n}{3}$ such that 
\begin{enumerate}
  \item[{\rm (a)}] given any color $c\in \mathcal{C}$ and any $D_c$-good vertex $v$, the triple $(c,v,u)$ is absorbable for all but at most $\sqrt{\epsilon}n$ vertices $u$,
  \item[{\rm (b)}] given any color $c\in \mathcal{C}$, for all but at most $\sqrt{\epsilon}n$ $D_c$-good vertices $v$, the pair $(c,v)$ is absorbable.
\end{enumerate}

\begin{claim}\label{claim:totally-absorbable}
  There is an integer $r\in[(2-2^{-10})\beta n,2\beta n]$ and for each $i\in [r]$, disjoint vertex pairs $(v_i,v_i')\in V^2$ and disjoint color pairs $(c_i,c_i')\in \mathcal{C}^2$ such that the family $\mathcal{Q}:=\{(c_i,c_i',v_i,v_i'):i\in [r]\}$ has the following properties:
  \begin{enumerate}
    \item[{\rm (i)}] $(c_i,c_i',v_i,v_i')$ is totally absorbable for all $i\in [r]$,
    \item[{\rm (ii)}] for every pair $(u_1,u_2)\in V^2$ and $c\in [n]$, there are at least $2^{-9}\beta n$ values $i\in [r]$ such that $c\in L({v_iu_1})$ and $c_i'\in L({u_2v_i'})$ (see Figure \ref{figabs}).
  \end{enumerate}
\end{claim}

%\begin{figure}[ht!]

%\centering
%  \psfrag{a}{$C$}\psfrag{b}{$v_1$}\psfrag{c}{$v_1'$}\psfrag{d}{$v_r$}
%  \psfrag{e}{$v_r'$}
%  \psfrag{f}{$v_i$}
%  \psfrag{g}{$v_i'$}
%  \psfrag{h}{$c_i$}
%  \psfrag{i}{$c_i'$}
%  \psfrag{j}{$u_1$}
%  \psfrag{k}{$u_2$}
%  \psfrag{m}{$c$}
%  \includegraphics[width=60mm]{ab-cycle.eps}\\
  %\caption{Extremal graphs.}
%\end{figure}

\begin{figure}[ht!]
\centering

\tikzset{every picture/.style={line width=0.75pt}} %set default line width to 0.75pt        

\begin{tikzpicture}[x=0.75pt,y=0.75pt,yscale=-1,xscale=1]
%uncomment if require: \path (0,235); %set diagram left start at 0, and has height of 235

%Straight Lines [id:da5858067870315564] 
\draw [color={rgb, 255:red, 248; green, 231; blue, 28 }  ,draw opacity=1 ]   (406,167.5) -- (494,171.5) ;
\draw [shift={(443.01,169.18)}, rotate = 2.6] [color={rgb, 255:red, 248; green, 231; blue, 28 }  ,draw opacity=1 ][line width=0.75]    (10.93,-3.29) .. controls (6.95,-1.4) and (3.31,-0.3) .. (0,0) .. controls (3.31,0.3) and (6.95,1.4) .. (10.93,3.29)   ;
%Straight Lines [id:da5869458365443816] 
\draw [color={rgb, 255:red, 155; green, 155; blue, 155 }  ,draw opacity=1 ]   (494,70.5) -- (406,73.5) ;
\draw [shift={(457,71.76)}, rotate = 178.05] [color={rgb, 255:red, 155; green, 155; blue, 155 }  ,draw opacity=1 ][line width=0.75]    (10.93,-3.29) .. controls (6.95,-1.4) and (3.31,-0.3) .. (0,0) .. controls (3.31,0.3) and (6.95,1.4) .. (10.93,3.29)   ;
%Curve Lines [id:da3535741183469501] 
\draw [color={rgb, 255:red, 208; green, 2; blue, 27 }  ,draw opacity=1 ]   (339.5,93) .. controls (332.11,77.75) and (314.61,65.75) .. (297.75,62) ;
\draw [shift={(327.29,77.3)}, rotate = 219.05] [color={rgb, 255:red, 208; green, 2; blue, 27 }  ,draw opacity=1 ][line width=0.75]    (10.93,-3.29) .. controls (6.95,-1.4) and (3.31,-0.3) .. (0,0) .. controls (3.31,0.3) and (6.95,1.4) .. (10.93,3.29)   ;
%Straight Lines [id:da4124360656382928] 
\draw [color={rgb, 255:red, 144; green, 19; blue, 254 }  ,draw opacity=1 ]   (342,147.25) -- (405,166.5) ;
\draw [shift={(366.81,154.83)}, rotate = 16.99] [color={rgb, 255:red, 144; green, 19; blue, 254 }  ,draw opacity=1 ][line width=0.75]    (10.93,-3.29) .. controls (6.95,-1.4) and (3.31,-0.3) .. (0,0) .. controls (3.31,0.3) and (6.95,1.4) .. (10.93,3.29)   ;
%Curve Lines [id:da22958146950292435] 
\draw [color={rgb, 255:red, 245; green, 166; blue, 35 }  ,draw opacity=1 ]   (305.5,181) .. controls (331.61,171.75) and (336.61,158.25) .. (342,147.25) ;
\draw [shift={(322.46,172.71)}, rotate = 322.59] [color={rgb, 255:red, 245; green, 166; blue, 35 }  ,draw opacity=1 ][line width=0.75]    (10.93,-3.29) .. controls (6.95,-1.4) and (3.31,-0.3) .. (0,0) .. controls (3.31,0.3) and (6.95,1.4) .. (10.93,3.29)   ;
%Curve Lines [id:da010676799119196057] 
\draw [color={rgb, 255:red, 144; green, 19; blue, 254 }  ,draw opacity=1 ]   (342,147.25) .. controls (350.11,130.75) and (349.61,105.75) .. (339.5,93) ;
\draw [shift={(347.42,126.5)}, rotate = 272.01] [color={rgb, 255:red, 144; green, 19; blue, 254 }  ,draw opacity=1 ][line width=0.75]    (10.93,-3.29) .. controls (6.95,-1.4) and (3.31,-0.3) .. (0,0) .. controls (3.31,0.3) and (6.95,1.4) .. (10.93,3.29)   ;
%Shape: Circle [id:dp7506141114473954] 
\draw  [fill={rgb, 255:red, 0; green, 0; blue, 0 }  ,fill opacity=1 ] (336.5,93) .. controls (336.5,91.34) and (337.84,90) .. (339.5,90) .. controls (341.16,90) and (342.5,91.34) .. (342.5,93) .. controls (342.5,94.66) and (341.16,96) .. (339.5,96) .. controls (337.84,96) and (336.5,94.66) .. (336.5,93) -- cycle ;
%Shape: Circle [id:dp3941887824989163] 
\draw  [fill={rgb, 255:red, 0; green, 0; blue, 0 }  ,fill opacity=1 ] (339,147.25) .. controls (339,145.59) and (340.34,144.25) .. (342,144.25) .. controls (343.66,144.25) and (345,145.59) .. (345,147.25) .. controls (345,148.91) and (343.66,150.25) .. (342,150.25) .. controls (340.34,150.25) and (339,148.91) .. (339,147.25) -- cycle ;
%Shape: Circle [id:dp506358263194884] 
\draw  [fill={rgb, 255:red, 0; green, 0; blue, 0 }  ,fill opacity=1 ] (294.75,62) .. controls (294.75,60.34) and (296.09,59) .. (297.75,59) .. controls (299.41,59) and (300.75,60.34) .. (300.75,62) .. controls (300.75,63.66) and (299.41,65) .. (297.75,65) .. controls (296.09,65) and (294.75,63.66) .. (294.75,62) -- cycle ;
%Shape: Circle [id:dp7124563540417581] 
\draw  [fill={rgb, 255:red, 0; green, 0; blue, 0 }  ,fill opacity=1 ] (302.5,181) .. controls (302.5,179.34) and (303.84,178) .. (305.5,178) .. controls (307.16,178) and (308.5,179.34) .. (308.5,181) .. controls (308.5,182.66) and (307.16,184) .. (305.5,184) .. controls (303.84,184) and (302.5,182.66) .. (302.5,181) -- cycle ;
%Curve Lines [id:da9204851897322937] 
\draw  [dash pattern={on 4.5pt off 4.5pt}]  (302.5,181) .. controls (195.7,198.41) and (196.37,50.41) .. (297.75,62) ;
\draw [shift={(222.25,117.67)}, rotate = 91.8] [color={rgb, 255:red, 0; green, 0; blue, 0 }  ][line width=0.75]    (10.93,-4.9) .. controls (6.95,-2.3) and (3.31,-0.67) .. (0,0) .. controls (3.31,0.67) and (6.95,2.3) .. (10.93,4.9)   ;
%Shape: Circle [id:dp8794843557269232] 
\draw  [fill={rgb, 255:red, 0; green, 0; blue, 0 }  ,fill opacity=1 ] (403,167.5) .. controls (403,165.84) and (404.34,164.5) .. (406,164.5) .. controls (407.66,164.5) and (409,165.84) .. (409,167.5) .. controls (409,169.16) and (407.66,170.5) .. (406,170.5) .. controls (404.34,170.5) and (403,169.16) .. (403,167.5) -- cycle ;
%Shape: Circle [id:dp5182818547443309] 
\draw  [fill={rgb, 255:red, 0; green, 0; blue, 0 }  ,fill opacity=1 ] (403,73.5) .. controls (403,71.84) and (404.34,70.5) .. (406,70.5) .. controls (407.66,70.5) and (409,71.84) .. (409,73.5) .. controls (409,75.16) and (407.66,76.5) .. (406,76.5) .. controls (404.34,76.5) and (403,75.16) .. (403,73.5) -- cycle ;
%Shape: Circle [id:dp6512992189478068] 
\draw  [fill={rgb, 255:red, 0; green, 0; blue, 0 }  ,fill opacity=1 ] (491,171.5) .. controls (491,169.84) and (492.34,168.5) .. (494,168.5) .. controls (495.66,168.5) and (497,169.84) .. (497,171.5) .. controls (497,173.16) and (495.66,174.5) .. (494,174.5) .. controls (492.34,174.5) and (491,173.16) .. (491,171.5) -- cycle ;
%Straight Lines [id:da6216632186503104] 
\draw [color={rgb, 255:red, 126; green, 211; blue, 33 }  ,draw opacity=1 ]   (403,73.5) -- (342.5,92) ;
\draw [shift={(379.44,80.7)}, rotate = 163] [color={rgb, 255:red, 126; green, 211; blue, 33 }  ,draw opacity=1 ][line width=0.75]    (10.93,-3.29) .. controls (6.95,-1.4) and (3.31,-0.3) .. (0,0) .. controls (3.31,0.3) and (6.95,1.4) .. (10.93,3.29)   ;
%Shape: Circle [id:dp3381946504965161] 
\draw  [fill={rgb, 255:red, 0; green, 0; blue, 0 }  ,fill opacity=1 ] (491,70.5) .. controls (491,68.84) and (492.34,67.5) .. (494,67.5) .. controls (495.66,67.5) and (497,68.84) .. (497,70.5) .. controls (497,72.16) and (495.66,73.5) .. (494,73.5) .. controls (492.34,73.5) and (491,72.16) .. (491,70.5) -- cycle ;
%Curve Lines [id:da5187142245445766] 
\draw  [dash pattern={on 4.5pt off 4.5pt}]  (494,70.5) .. controls (447.11,122.75) and (540.11,112.25) .. (494,171.5) ;
\draw [shift={(499.79,124.67)}, rotate = 226.55] [color={rgb, 255:red, 0; green, 0; blue, 0 }  ][line width=0.75]    (10.93,-4.9) .. controls (6.95,-2.3) and (3.31,-0.67) .. (0,0) .. controls (3.31,0.67) and (6.95,2.3) .. (10.93,4.9)   ;
%Shape: Circle [id:dp2569785834732201] 
\draw  [fill={rgb, 255:red, 0; green, 0; blue, 0 }  ,fill opacity=1 ] (157.75,58) .. controls (157.75,56.34) and (159.09,55) .. (160.75,55) .. controls (162.41,55) and (163.75,56.34) .. (163.75,58) .. controls (163.75,59.66) and (162.41,61) .. (160.75,61) .. controls (159.09,61) and (157.75,59.66) .. (157.75,58) -- cycle ;
%Shape: Circle [id:dp3584572972144391] 
\draw  [fill={rgb, 255:red, 0; green, 0; blue, 0 }  ,fill opacity=1 ] (157.75,87) .. controls (157.75,85.34) and (159.09,84) .. (160.75,84) .. controls (162.41,84) and (163.75,85.34) .. (163.75,87) .. controls (163.75,88.66) and (162.41,90) .. (160.75,90) .. controls (159.09,90) and (157.75,88.66) .. (157.75,87) -- cycle ;
%Shape: Circle [id:dp43559394129887075] 
\draw  [fill={rgb, 255:red, 0; green, 0; blue, 0 }  ,fill opacity=1 ] (157.75,153) .. controls (157.75,151.34) and (159.09,150) .. (160.75,150) .. controls (162.41,150) and (163.75,151.34) .. (163.75,153) .. controls (163.75,154.66) and (162.41,156) .. (160.75,156) .. controls (159.09,156) and (157.75,154.66) .. (157.75,153) -- cycle ;
%Shape: Circle [id:dp17445625998293357] 
\draw  [fill={rgb, 255:red, 0; green, 0; blue, 0 }  ,fill opacity=1 ] (157.75,183) .. controls (157.75,181.34) and (159.09,180) .. (160.75,180) .. controls (162.41,180) and (163.75,181.34) .. (163.75,183) .. controls (163.75,184.66) and (162.41,186) .. (160.75,186) .. controls (159.09,186) and (157.75,184.66) .. (157.75,183) -- cycle ;
%Shape: Circle [id:dp9512819459986228] 
\draw  [fill={rgb, 255:red, 0; green, 0; blue, 0 }  ,fill opacity=1 ] (159,106.5) .. controls (159,105.67) and (159.67,105) .. (160.5,105) .. controls (161.33,105) and (162,105.67) .. (162,106.5) .. controls (162,107.33) and (161.33,108) .. (160.5,108) .. controls (159.67,108) and (159,107.33) .. (159,106.5) -- cycle ;
%Shape: Circle [id:dp2648060843210751] 
\draw  [fill={rgb, 255:red, 0; green, 0; blue, 0 }  ,fill opacity=1 ] (159,121.5) .. controls (159,120.67) and (159.67,120) .. (160.5,120) .. controls (161.33,120) and (162,120.67) .. (162,121.5) .. controls (162,122.33) and (161.33,123) .. (160.5,123) .. controls (159.67,123) and (159,122.33) .. (159,121.5) -- cycle ;
%Shape: Circle [id:dp04533037529195871] 
\draw  [fill={rgb, 255:red, 0; green, 0; blue, 0 }  ,fill opacity=1 ] (159,134.5) .. controls (159,133.67) and (159.67,133) .. (160.5,133) .. controls (161.33,133) and (162,133.67) .. (162,134.5) .. controls (162,135.33) and (161.33,136) .. (160.5,136) .. controls (159.67,136) and (159,135.33) .. (159,134.5) -- cycle ;

% Text Node
\draw (352,63) node [anchor=north west][inner sep=0.75pt]   [align=left] {$c_i$};
% Text Node
\draw (395,49) node [anchor=north west][inner sep=0.75pt]   [align=left] {$v_i$};
% Text Node
\draw (445,52) node [anchor=north west][inner sep=0.75pt]   [align=left] {$c$};
% Text Node
\draw (491,48) node [anchor=north west][inner sep=0.75pt]   [align=left] {$u_1$};
% Text Node
\draw (491,182) node [anchor=north west][inner sep=0.75pt]   [align=left] {$u_2$};
% Text Node
\draw (395,183) node [anchor=north west][inner sep=0.75pt]   [align=left] {$v_i'$};
% Text Node
\draw (445,182) node [anchor=north west][inner sep=0.75pt]   [align=left] {$c_i'$};
% Text Node
\draw (135,50) node [anchor=north west][inner sep=0.75pt]   [align=left] {$v_1$};
% Text Node
\draw (135,79) node [anchor=north west][inner sep=0.75pt]   [align=left] {$v_1'$};
% Text Node
\draw (135,147) node [anchor=north west][inner sep=0.75pt]   [align=left] {$v_r$};
% Text Node
\draw (135,180) node [anchor=north west][inner sep=0.75pt]   [align=left] {$v_r'$};
% Text Node
\draw (274.78,116) node [anchor=north west][inner sep=0.75pt]   [align=left] {$C$};

\end{tikzpicture}
\caption{Absorbing.}\label{figabs}
\end{figure}
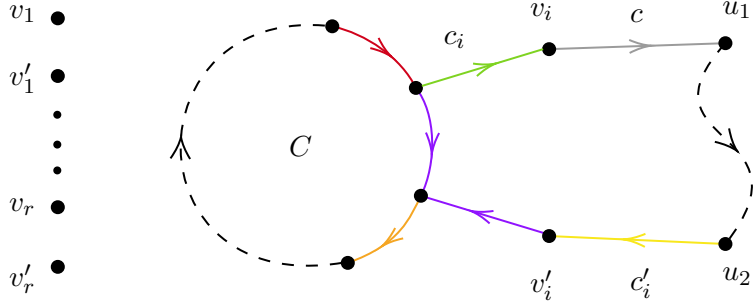

\begin{claimproof}[Proof of Claim \ref{claim:totally-absorbable}]
  For every color pair $(b_1,b_2)\in \mathcal{C}^2$, every vertex pair $(u_1,u_2)\in V^2$ and color $c\in \mathcal{C}$, let $S(b_1,b_2,u_1,u_2,c)$ be
the collection of pairs $(v_1,v_2)\in V^2$ such that $c\in L({v_1u_1})$, $b_2\in L({u_2v_2})$ and $(b_1,b_2,v_1,v_2)$ is totally
absorbable. We will prove that $|S(b_1,b_2,u_1,u_2,c)|\geq 2^{-4}n^2$.

For this, we first count the number of choices for $v_1$. Let $N_1:=\{x\in N_{{c}}^-(u_1):x\ \text{is}\ D_{b_1}\text{-good}\}$. Hence, $|N_1|\geq \left(\frac{1}{2}-\mu\right)n-2\epsilon n$. By (b), we know that all but at most $\sqrt{\epsilon}n$ vertices $v_1\in N_1$ satisfying that $(b_1,v_1)$ is absorbable. Hence there are at least $\frac{n}{4}$ choices of $v_1$ such that $v_1\in N_1$ and $(b_1,v_1)$ is absorbable. 

%Note that the number of vertices $v_1$ such that $v_1$ is not $D_{b_1}$-good or $(b_1,v_1)$ is not absorbable is at most $(2\epsilon +\sqrt{\epsilon})n$ (by (a)). Together with $d_{D_c}^-(u_1)\geq \left(\frac{1}{2}-\mu\right)n$, there are at least $\frac{n}{4}$ choices of $v_1$ such that $v_1\in N_{c}^-(u_1)$, $v_1$ is $D_{b_1}$-good and $(b_1,v_1)$ is absorbable. 

Now fix $v_1$. Let $N_2:=\{x\in N_{{b_2}}^+(u_2):x\ \text{is}\ D_{b_2}\text{-good}\}$. Hence, $|N_2|\geq \left(\frac{1}{2}-\mu\right)n-2\epsilon n$. 
%By (a), for all but at most $\sqrt{\epsilon}n$ vertices $v_2\in V$, the triple $(b_1,v_1,v_2)$ is absorbable; By (b), for all but at most $(2\epsilon +\sqrt{\epsilon})n$ vertices $v_2\in V$, the pair $(b_2,v_2)$ is absorbable. can delete at most $2(\epsilon+\sqrt{\epsilon})n+1$ vertices from $N_2$ and obtain a set $N_2'$  
By (a) and (b), we know that all but at most $2\sqrt{\epsilon}n+1$ vertices $v_2\in N_2$ %there is a subset $N_2'\subset N_2$ with $|N_2\setminus N_2'|\leq 2(\epsilon+\sqrt{\epsilon})n+1$ 
such that $v_2\neq v_1$ and $(b_1,b_2,v_1,v_2)$ is totally absorbable. Therefore, $|S(b_1,b_2,u_1,u_2,c)|\geq 2^{-4}n^2$, as desired. 

Let $\{c_i,c_i':i\in [2\beta n]\}\subseteq \mathcal{C}$ be a collection of distinct colors. For each $i\in [2\beta n]$, let $\mathcal{F}_i$ be the directed 2-graph on $V$ with %define $\mathcal{F}_i$ to be a directed $2$-graph with vertex set $V$ and 
edge set $\{(v_1,v_2)\in V^2:(c_i,c_i',v_1,v_2)\ \text{is totally absorbable}\}$, and define $\mathbf{H}:=\{\mathcal{F}_i:i\in [2\beta n]\}$.  Define the collection of directed multi-graphs:
$$
  \mathbf{Z}:=\left\{S(u_1,u_2,c):=\bigcup_{i\in [2\beta n]}S(c_i,c_i',u_1,u_2,c):(u_1,u_2)\in V^2,\ c\in [n]\right\}.
$$
For every $S:=S(u_1, u_2, c)\in \mathbf{Z}$ and every $i\in [2\beta n]$, we have $|E(S)\cap E(\mathcal{F}_i)|\geq 2^{-4}n^2$. Applying Lemma \ref{LEMMA:directed-k-graph-matching} with $t:=2\beta n$ and $\epsilon:=2^{-4}$, we obtain that there is a rainbow matching $M$ of $\mathbf{H}$ with size at least $(2-2^{-10})\beta n$ (and at most $2\beta n$) and $|E(S)\cap E(M)|\geq 2^{-9}\beta n$ for all $S\in \mathbf{Z}$, as desired.
\end{claimproof}

Let 
\begin{align*}
  &V_{\rm abs}:=\bigcup_{i\in [r]}\{v_i,v_i'\},\ \ U:=V\setminus (V(C)\cup V_{\rm abs}),\ \ 
  \mathcal{C}_{\rm abs}:=\bigcup_{i\in [r]}\{c_i,c_i'\},\\
 &\mathcal{C}_{\rm rem}=[n]\setminus ({\rm col}(C)\cup \mathcal{C}_{\rm abs}),\ \ \mathcal{J}:=\{J_i:i\in \mathcal{C}_{\rm rem}\}\ \ \text{where} \ \ J_i=D_i[U].
\end{align*}
Thus, $\big|[n]\setminus \mathcal{C}_{\rm rem}\big|=|V\setminus U|\leq \lambda n+4\beta n<2\lambda n$ since $\beta\ll \lambda$, and $\delta^0(\mathcal{J})\geq (\frac{1}{2}-\mu-2\lambda)n\geq (\frac{1}{2}-3\lambda)n$.

Since $\lambda\ll \gamma,\alpha,\epsilon,\delta$ and $\mathcal{D}$ is $(\gamma,\alpha,\epsilon,\delta)$-stable, it is easy to see that $\mathcal{J}$ is $(\frac{\gamma}{2},\frac{\alpha}{2},2\epsilon,\frac{\delta}{2})$-stable. Applying Lemma \ref{regularity-lemma} to $\mathcal{J}$ with parameters $(\epsilon_0, 1, d, L_0)$, we obtain a reduced digraph collection $\mathcal{R}$. Denote by $[L]$ the common vertex set of the digraphs in $\mathcal{R}$, where $L_0 \le L \le n_0$, and by $[M]$ the set of color clusters. Thus there is a partition $V_0,V_1,\ldots,V_L$ of $U$ and $\mathcal{C}_0,\mathcal{C}_1,\ldots,\mathcal{C}_M$ of $\mathcal{C}_{\rm rem}$ and a digraph collection $\mathcal{J}'$ satisfying  (i)-(v) of Lemma \ref{regularity-lemma}. Therefore, for $\{(h,i),j\}\in {[L]\choose 2}\times [M]$, we have that  ${hi}\in R_j$ if and only if $\mathcal{J}_{hi,j}':=\{J_c'[V_h,V_i]:c\in \mathcal{C}_j\}$ is $(\epsilon_0,d)$-regular. Furthermore, $Lm\leq |U|\leq Mm+\epsilon_0n$ and $Mm\leq |\mathcal{C}_{\rm rem}|\leq Lm+\epsilon_0n$, so $|L-M|\leq \frac{\epsilon_0n}{m}\leq \frac{\epsilon_0L}{1-\epsilon_0}$. Thus, we may assume that $M=L$ at the expense of assuming the slightly worse bound $|V_0|+|\mathcal{C}_0|\leq 3\epsilon_0n$. Based on Lemma \ref{degree-inheritance}, we obtain that, for each vertex $i\in [L]$, there are at least $(1-d^{1/4})M$ colors $j\in [M]$ such that $d_{R_j}^+(i),d_{R_j}^-(i)\geq (\frac{1}{2}-4\lambda)L$; for each color $j\in [M]$, there are at least $(1-d^{1/4})L$ vertices $i\in [L]$ such that $d_{R_j}^+(i),d_{R_j}^-(i)\geq (\frac{1}{2}-4\lambda)L$. 

It follows from Lemma \ref{lem-reduced-stable} that $\mathcal{R}$ is $(\frac{\gamma}{4},\frac{\alpha^2}{4},2\epsilon,\frac{\delta}{4})$-stable. It is routine to check that the bipartite graph collection corresponding to $\mathcal{R}$, say $\mathcal{B_R}$, is $(\frac{\gamma}{4},\frac{\alpha^2}{4},2\epsilon,\frac{\delta}{4})$-stable. 
Together with Theorem \ref{stable-matching}, there is a  transversal perfect matching inside $\mathcal{B_R}$. Therefore,  there exists a set of  disjoint rainbow directed cycles inside $\mathcal{R}$, say $C^1,C^2,\ldots,C^t$, such that $|V(C^1)|+\cdots+|V(C^t)|=L$. For each rainbow directed cycle $C^i$ in $\mathcal{R}$, by using colors in $\bigcup_{j\in {\rm col}(C^i)}\mathcal{C}_j$, we are to find a rainbow  directed path $P^i$ in $\mathcal{J}'$ with length at least $(1-\sqrt{\epsilon_0})|C^i|m$.

\begin{claim}\label{embedding}
  Let $C=v_1v_2\ldots v_sv_1$ be a rainbow directed cycle in $\mathcal{R}$ with $v_{s+1}:=v_1$ and  $v_iv_{i+1}\in E(R_i)$ for each $i\in [s]$. Then $\{J_i'[V_1,V_2,\cdots,V_s]:i\in \bigcup_{j\in [s]}\mathcal{C}_j\}$ contains a rainbow directed path with order at least $(1-\sqrt{\epsilon_0})sm$.% in $\mathcal{D}$ and colors from $\bigcup_{j\in {\rm col}(C^i)}\mathcal{C}_j$. .
\end{claim}
\begin{claimproof}[Proof of Claim \ref{embedding}]
  We construct a rainbow directed path greedily. Assume that we already have a rainbow directed path  $P=x_1\ldots x_i$ with $x_j\in V_{j_0}$ and $j_0\equiv j\pmod{s}$ for each $j\in [i]$. If $i\geq (1-\sqrt{\epsilon_0})sm$, then we are done. Otherwise, $|V_j\setminus V(P)|\geq \sqrt{\epsilon_0} m-1$ for each $j\in [s]$. We claim that $\max\{|N_{J'_{c}}^+(x_i)\cap (V_{i_0+1}\setminus V(P))|:c\in \mathcal{C}_{i_0}\setminus {\rm col}(P)\}\geq 2\epsilon_0 m$. If so, then choose $c\in \mathcal{C}_{i_0}\setminus {\rm col}(P)$ such that $|N_{J'_{c}}^+(x_i)\cap (V_{i_0+1}\setminus V(P))|$ is maximum and  choose $x_{i+1}\in N_{J_{c}'}^+(x_i)\cap (V_{i_0+1}\setminus V(P))$ such that $\max\{|N_{J_{c}'}^+(x_{i+1})\cap (V_{i_0+2}\setminus V(P))|:c\in \mathcal{C}_{i_0+1}\setminus {\rm col}(P)\}$ is maximum. % and one may extend $P$ to $x_1\ldots x_ix_{i+1}$ with colors ${\rm col}(P)\cup i_0$. 
  
  %there exists an unused color  $c\in \mathcal{C}_{i_0}$ such that $|N_{{c}}^+(x_i)\cap (V_{i_0+1}\setminus V(P))|\geq 2\epsilon_0 m$. 
  
  Suppose that $\max\{|N_{J_{c}'}^+(x_i)\cap (V_{i_0+1}\setminus V(P))|:c\in \mathcal{C}_{i_0}\setminus {\rm col}(P)\}<2\epsilon_0 m$. Clearly, $i\geq 2$. Let $W_{i_0}:=N_{J'_{c(x_{i-1}x_i)}}^+(x_{i-1})\cap (V_{i_0}\setminus V(P))$ and $W_{i_0+1}$ be a subset of $V_{i_0+1}\setminus V(P)$ with $|W_{i_0+1}|\geq \frac{\sqrt{\epsilon_0}}{2}m$. By the choice of each $x_{j}$ with $j\leq i$, we have $|W_{i_0}|\geq 2\epsilon_0 m$ and %$x_{i+1}$ cannot be chosen in such a way.  
 % We claim that there is an unused color $\hat{c}\in \mathcal{C}_k$ such that $|N_{{\hat{c}}}^+(x_i)\cap (V_{k+1}\setminus V(P))|\geq \epsilon_0 n$. Then 
 %$|N_{c}^+(x_i)\cap (V_{i_0+1}\setminus V(P))|< 2\epsilon_0 m$ holds for each color $c\in \mathcal{C}_{i_0}\setminus {\rm col}(P)$. Hence one may assume that every vertex 
 each $y\in W_{i_0}$ satisfies $|N_{J_c'}^+(y)\cap (V_{i_0+1}\setminus V(P))|< 2\epsilon_0 m$ for each color $c\in \mathcal{C}_{i_0}\setminus {\rm col}(P)$. Then 
  $$
  \frac{\sum_{c\in \mathcal{C}_{i_0}\setminus {\rm col}(P)}e_{J'_c}(W_{i_0},W_{i_0+1})}{|\mathcal{C}_{i_0}\setminus {\rm col}(P)||W_{i_0}||W_{i_0+1}|}\leq 2\sqrt{\epsilon_0}\ll d,
  $$
  which contradicts the fact that $\{J'_i[V_{i_0},V_{i_0+1}]:i\in \mathcal{C}_{i_0}\}$ is $(\epsilon_0,d)$-regular.
\end{claimproof}

Based on Claim \ref{embedding}, we know that there exists a set of disjoint rainbow directed paths $P^i$ in $\mathcal{J}$ with order at least $(1-\sqrt{\epsilon_0})|C^i|m$ and colors from $\bigcup_{j\in {\rm col}(C^i)}\mathcal{C}_j$ for all $i\in [t]$. %Notice that the collection of $\{P^i:i\in [t]\}$ are disjoint. 
The number of vertices or colors not in any $P^i$ is at most $\sqrt{\epsilon_0} n+3\epsilon_0n\leq 2\sqrt{\epsilon_0} n$. We consider each vertex in $U$ but not in any $P^i$ to be a single path (of length 0).

Without loss of generality, assume that $P^1,P^2,\ldots,P^s$ are all the disjoint rainbow directed paths obtained in the above. Then $s\leq 2\sqrt{\epsilon_0} n+L<2^{-10}\beta n$. Let $x_i, y_i$ be the start vertex and end vertex of $P^i$ for each $i\in [s]$ (where $x_i = y_i$ if $P^i$ has length 0). Clearly, $U=\bigcup_{i\in [s]}V(P^i)$. It follows that $\big|\mathcal{C}_{\rm rem}\setminus {\rm col}(\bigcup_{i\in [s]}P^i)\big|=s$.

Next, we do the following operation for each $i\in [s]$ in turn. Let $a_i\in \mathcal{C}_{\rm rem}$ be an arbitrary unused color. Choose an unused 4-tuple $Q_i:=(c_{j_i},c_{j_i}',v_{j_i},v_{j_i}')\in \mathcal{Q}$ where $j_i\in [r]$, $a_i\in L({v_{j_i}x_i})$ and $c_{j_i}'\in L({y_iv_{j_i}'})$. This is feasible since Claim \ref{claim:totally-absorbable} implies that there are at least $2^{-9}\beta n$ choices for $Q_i$, of which at most $s\leq 2^{-10}\beta n$ have been used. Now, $(c_{j_i},c_{j_i}',v_{j_i},v_{j_i}')$ is totally absorbable, so there are at least $\lambda^2n$ disjoint rainbow directed $c_{j_i}$-absorbing path of $(v_{j_i},v_{j_i}')$ inside ${C}$. Then choose one of them, say $S_i:=x_1^ix_2^ix_3^ix_4^i$, whose vertices have not been previously chosen since $s<2^{-10}\beta n<2^{-10}\lambda^2 n$, and whose colors are $b_1^i,b_2^i,b_3^i$ in order.

Notice that there remains $I\subseteq [r]$ such that the $(c_{j},c_{j}',v_{j},v_{j}')\in \mathcal{Q}$ with $j\in I$ are precisely the $4$-tuples which were not chosen to be some $Q_i$. For each one, there are at least $\lambda^2n$ disjoint rainbow directed $c_j$-absorbing paths of $(v_j,v_j)$ and disjoint rainbow directed $c_j'$-absorbing paths of $(v_j',v_j')$ inside $C$. Since $\beta< \lambda^2$, there are two such rainbow directed paths $T_i,T_i'$ for each $(c_j,v_j)$ and $(c_j',v_j')$, which are disjoint and whose vertices have not previously been chosen.

At the end of this process, we have a collection $\{S_i:i\in [s]\},\ \{T_i:i\in I\},\ \{T_i':i\in I\}$ of disjoint rainbow directed paths inside $C$ satisfying
\begin{itemize}
  \item for each $i\in [s]$, $S_i$ is a directed $c_{j_i}$-absorbing  path of $(v_{j_i},v_{j_i}')$, 
  \item for each $i\in I$, $T_i$ is a directed $c_i$-absorbing  path of $(v_i,v_i)$ and $T_i'$ is a $c_i'$-directed absorbing  path of $(v_i',v_i')$.
\end{itemize}
 For each $i\in [s]$, we replace $S_i$ by $x_1^ix_2^iv_{j_i}x_iP^iy_iv_{j_i}'x_3^ix_4^i$ with colors $b_1^i, c_{j_i}, a_i$ followed by the colors inherited from $P^i$, followed by $c_{j_i}',b_2^i,b_3^i$.  
That is, we have replaced $S_i$ by a rainbow directed  path with the same endpoints, vertices $V(S_i)\cup V(P^i)\cup \{v_{j_i},v_{j_i}'\}$, and colors ${\rm col}(S_i)\cup {\rm col}(P^i)\cup \{c_{j_i},c_{j_i}',a_i\}$. For each $i\in I$, we replace $T_i:=y_1^iy_2^iy_3^iy_4^i$ by $y_1^iy_2^iv_iy_3^iy_4^i$  where colors are inherited except ${\rm col}(y_2^iv_i)=c_i$ and ${\rm col}(v_iy_3^i)={\rm col}(y_2^iy_3^i)$. We do a similar replacement of $T_i'$, using new vertex $v_i'$ 
and new color $c_i'$. So we have replaced $T_i$ (resp. $T_i'$) by a rainbow directed path with the same endpoints, vertices $V(T_i)\cup \{v_i\}$ (resp. $V(T_i')\cup \{v_i'\}$) and colors ${\rm col}(T_i)\cup \{c_i\}$ (resp. ${\rm col}(T_i')\cup \{c_i'\}$). Thus we have obtained a  rainbow directed   cycle using vertices $V(C)\cup \bigcup_{i\in [s]}V(P^i)\cup \{v_i,v_i':i\in [r]\}=V$ and colors ${\rm col}(C)\cup \bigcup_{i\in [s]}{\rm col}(P^i)\cup \{c_i,c_i':i\in [r]\}=[n]$ where each color is used at most once (and hence exactly once). That is, $\mathcal{D}$ contains a transversal directed Hamilton cycle, as desired.
\end{proof}
\section{Extremal case}

In this section, we consider the extremal case, in which most digraphs in the collection are $(\epsilon, \mathrm{ECk})$-extremal for some $k \in [3]$. The following result follows directly from Lemma~\ref{stable}. 
\begin{lemma}\label{lemma-new1}
Assume that $0<\frac{1}{n}\ll 8\epsilon^{1/2}\leq  \delta\leq \epsilon^{1/3}\ll1$ and $m=(1-4\sqrt{\delta})n$.  
Let $\mathcal{D}=\{D_1,\ldots,D_n\}$ be a digraph collection on a common vertex set $V$ of size $n$ and $\delta^0(\mathcal{D})\geq \frac{n}{2}$. If $\mathcal{D}$ contains no transversal directed Hamilton cycles, then by swapping indices if necessary, one of the following holds: 
\begin{enumerate}
    \item[{\rm \bf (B1)}] either $D_i$ is $(\epsilon,{\rm EC1})$-extremal or $(\epsilon,{\rm EC2})$-extremal  and it admits an $\epsilon$-characteristic partition $(\Tilde{A}_i,\Tilde{B}_i,\Tilde{L}_i)$ for each $i\in [m]$,  
    \item[{\rm \bf (B2)}] or $D_i$ is $(\epsilon,{\rm EC3})$-extremal  and it admits an $\epsilon$-characteristic partition $(C_i^1,C_i^2,C_i^3,C_i^4,L_i)$ for each $i\in [m]$.
\end{enumerate}
In {\bf (B1)}, for every $i\in [m]$, $D_i$ also admits a partition $(A_i,B_i,L_i)$ satisfying %an $\epsilon$-characteristic partition $(A_i,B_i,C_i)$ satisfying $|{A_{i_0}}{\triangle} {A_i}|,|{B_{i_0}}{\triangle} {B_i}|<\delta n$ for each $i\in S$.
\allowdisplaybreaks
\begin{itemize}
\item[{\rm \bf (C1)}] if $i=1$ then   $(A_1,B_1,L_1):=(\Tilde{A}_1,\Tilde{B}_1,\Tilde{L}_1)$; for $i\geq 2$, 
    \item[{\rm \bf (C2)}] if $D_i$ is $(\epsilon,K_{\lceil\frac{n}{2}\rceil}\cup K_{\lfloor\frac{n}{2}\rfloor})$-extremal, then for 
$Y\in \{A,B\}$ we have  $\Tilde{Y_i}\subseteq Y_i$ and  $d_{D_i}(v,Y_i)\geq (\frac{1}{2}-3\sqrt{\delta})n$ for each vertex $v\in Y_i$,
\item[{\rm \bf (C3)}] if $D_i$ is $(\epsilon,K_{\lceil\frac{n}{2}\rceil,\lfloor\frac{n}{2}\rfloor})$-extremal, then for {$\{Y,Z\}= \{A,B\}$} we have  $\Tilde{Y_i}\subseteq Y_i$ and  $d_{D_i}(v,Z_i)\geq (\frac{1}{2}-3\sqrt{\delta})n$ for each vertex $v\in Y_i$,
\item[{\rm \bf (C4)}] subject to {\bf (C2)-(C3)}, {the partition $(A_i,B_i,L_i)$ is chosen such that $|A_i\cup B_i|$ is maximized for each  $i\in [2,m]$.} %each vertex in $V$ lies in as many as possible $A_i\cup B_i$ for $i\in [2,m]$.
\end{itemize}
Furthermore, for every  $i\in [m]$ and 
$Y\in \{A,B\}$, we have 
\begin{align*}%\label{difference}
    |L_i|\leq 2\epsilon n\ \ \text{and}\ \   |{Y}_1{\triangle} Y_{i}|<2\delta n.
\end{align*}
\end{lemma}

\begin{proof}
Choose constants $\alpha,\gamma,\epsilon,\delta$ such that  
\begin{align*}%\label{eq:parameter}
0<\frac{1}{n}\ll \alpha\ll \gamma,8\epsilon^{1/2}\leq  \delta\leq \epsilon^{1/3} \ll 1. %\ \text{and}\  \delta\leq\min\{\zeta,\frac{1}{2}-\zeta-\epsilon\}.
\end{align*}
%where $\zeta$ is defined in Lemma \ref{thm-single-graph}.
Let $\mathcal{D}=\{D_1,\ldots,D_n\}$ be a digraph collection on a common vertex set $V$ of size $n$ and  $\delta^0(\mathcal{D})\geq \frac{n}{2}$. Suppose that $\mathcal{D}$ does not contain transversal directed Hamilton cycles. By Lemma \ref{stable}, we know that  $\mathcal{D}$ is not $(\gamma,\alpha)$-strongly stable. Hence, there are at least $(1-\gamma)n$ digraphs in $\mathcal{D}$ such that each of them is $\alpha$-extremal. %{Therefore, there exists a constant $\gamma'\,(\leq \gamma)$ such that the number of $\alpha$-extremal digraphs in $\mathcal{D}$ is exactly $(1-\gamma')n$.} 
Without loss of generality, assume that $D_i$ is $\alpha$-extremal for each $i\in [(1-\gamma)n]$. Note that every $\alpha$-extremal digraph is also $\epsilon$-extremal, since $\alpha<\epsilon$. Denote 
\begin{align*}
  \tilde{\mathcal{C}}_1 &:=\{i\in [n]:D_i\ \text{is}\ (\epsilon,{\rm EC}1)\text{-extremal}\}, \\
  \tilde{\mathcal{C}}_2 &:=\{i\in [n]:D_i\ \text{is}\ (\epsilon,{\rm EC}2)\text{-extremal}\}, \\
  \tilde{\mathcal{C}}_3 &:=\{i\in [n]:D_i\ \text{is}\ (\epsilon,{\rm EC}3)\text{-extremal}\}.
\end{align*} 
Hence $|\tilde{\mathcal{C}}_1|+|\tilde{\mathcal{C}}_2|+|\tilde{\mathcal{C}}_3|\geq (1-\gamma)n$. By Lemma \ref{thm-single-graph}, we know that $D_i$ has a characteristic partition $(\Tilde{A_i},\Tilde{B_i},\Tilde{L_i})$ for each $i\in \tilde{\mathcal{C}}_1\cup \tilde{\mathcal{C}}_2$, or a characteristic partition  $(C_i^1,C_i^2,C_i^3,C_i^4,L_i)$ for each $i\in \tilde{\mathcal{C}}_3$. 

Under the above characteristic partition, applying  Lemma \ref{stable} again yields that $\mathcal{D}$ is not 
$(\epsilon,\delta)$-{weakly stable}. 
For each $i\in \tilde{\mathcal{C}}_1\cup \tilde{\mathcal{C}}_2$, define $I^1_i$ to be the set of $j\in \tilde{\mathcal{C}}_1\cup \tilde{\mathcal{C}}_2$ such that $D_i$ and $D_j$ are $\delta$-{crossing}. For each $i\in \tilde{\mathcal{C}}_3$, define $I^2_i$ to be the set of $j\in \tilde{\mathcal{C}}_3$ such that $D_i$ and $D_j$ are $\delta$-{crossing}. For each $i\in \tilde{\mathcal{C}}_1\cup \tilde{\mathcal{C}}_2$, define $I^3_i$ to be the set of $j\in \tilde{\mathcal{C}}_3$ such that $D_i$ and $D_j$ are $\delta$-{crossing}. Then we have 
$$
\sum_{i\in \tilde{\mathcal{C}}_1\cup \tilde{\mathcal{C}}_2}|I^1_i|< 2\delta n^2,\ \sum_{i\in \tilde{\mathcal{C}}_3}|I^2_i|< 2\delta n^2,\ \ \text{and}\ \  \sum_{i\in \tilde{\mathcal{C}}_1\cup \tilde{\mathcal{C}}_2}|I^3_i|< \delta n^2.
$$
%If $|\tilde{\mathcal{C}}_1|+|\tilde{\mathcal{C}}_2|\geq \sqrt{\delta}n$, then 
%Define $I_i$ to be the set of $j\in [(1-\gamma')n]$ such that $D_i$ and $D_j$ are $\delta$-{crossing}.  Therefore, $\sum_{i\in [(1-\gamma')n]}|I_i|< 2\delta (1-\gamma')^2n^2$. 
%Without loss of generality, assume that $|I^k_1|=\min\{|I^k_i|:i\in [(1-\gamma')n]\}$. %Then $|I^k_1|\leq 2\delta(1-\gamma')n$. 
%Based on the definition of $I_1$ and 
%Swapping the labels of colors, 
Hence there exists a color set, say $I^k$ (where $I^k\subseteq \tilde{\mathcal{C}}_1\cup \tilde{\mathcal{C}}_2$ if $k=1$ and $I^k\subseteq \tilde{\mathcal{C}}_3$ if $k\in \{2,3\}$), with size at most $\sqrt{\delta}n$ such that the following hold:  
%\begin{enumerate}
%  \item[{\rm \bf{(B1)}}] If $|\tilde{\mathcal{C}}_1|+|\tilde{\mathcal{C}}_2|\geq2\sqrt{\delta}n$, then $D_1$ is $\alpha$-close to EC1 or to EC2 and $|I_1^1|, |I_1^3|\leq \sqrt{\delta}n$, and $|\tilde{A}_1\triangle \tilde{A}_i|< \delta n$ and $|\tilde{B}_1\triangle \tilde{B}_i|<\delta n$ for all $i\in (\tilde{\mathcal{C}}_1\cup \tilde{\mathcal{C}}_2)\setminus I_1^1$,
%  $|\tilde{A}_1\triangle W_i^2|<\delta n$ and $|\tilde{B}_1\triangle W_i^4|<\delta n$, $|\tilde{A}_1\triangle W_i^1|<\delta n$ and $|\tilde{B}_1\triangle W_i^3|<\delta n$ for all $i\in \tilde{\mathcal{C}}_3\setminus I_1^3$. 
%  \item[{\rm \bf{(B2)}}] If $|\tilde{\mathcal{C}}_3|\geq2\sqrt{\delta}n$ and $|\tilde{\mathcal{C}}_1|+|\tilde{\mathcal{C}}_2|< 2\sqrt{\delta}n$, then $D_1$ is $\alpha$-close to EC3 and $|I_1^3|\sqrt{\delta}n$, and for all $k\in [4]$ and all $i\in  \tilde{\mathcal{C}}_3\setminus I_1^2$ we have $|W_1^k\triangle W_i^k|< \delta n$, in particular $|C_1^k\triangle C_i^k|<2\delta n$.
%\end{enumerate}
\begin{itemize}
  \item if $|\tilde{\mathcal{C}}_1|+|\tilde{\mathcal{C}}_2|\geq2\sqrt{\delta}n$, then  by swapping labels one has $|\tilde{A}_1\triangle \tilde{A}_i|< \delta n$ and $|\tilde{B}_1\triangle \tilde{B}_i|<\delta n$ for all $i\in (\tilde{\mathcal{C}}_1\cup \tilde{\mathcal{C}}_2)\setminus I^1$,
  %$D_1$ is $\alpha$-close to EC1 or to EC2 and $|I_1^1|, |I_1^3|\leq \sqrt{\delta}n$ satisfying the following:
  %\begin{itemize}
  %    \item $|\tilde{A}_1\triangle \tilde{A}_i|< \delta n$ and $|\tilde{B}_1\triangle \tilde{B}_i|<\delta n$ for all $i\in (\tilde{\mathcal{C}}_1\cup \tilde{\mathcal{C}}_2)\setminus I_1^1$,
  %    \item $|\tilde{A}_1\triangle W_i^2|<\delta n$ and $|\tilde{B}_1\triangle W_i^4|<\delta n$, $|\tilde{A}_1\triangle W_i^1|<\delta n$ and $|\tilde{B}_1\triangle W_i^3|<\delta n$ for all $i\in \tilde{\mathcal{C}}_3\setminus I_1^3$. 
  %\end{itemize} 
 \item if $|\tilde{\mathcal{C}}_3|\geq2\sqrt{\delta}n$, then by swapping labels, for all $k\in [4]$ and all $i\in  \tilde{\mathcal{C}}_3\setminus I^2$ we have $|W_1^k\triangle W_i^k|< \delta n$, in particular $|C_1^k\triangle C_i^k|< 2\delta n$,%, $|B_1\triangle B_i|\leq \delta n$, $|C_1\triangle C_i|\leq \delta n$ and $|D_1\triangle D_i|\leq \delta n$,
  \item if $|\tilde{\mathcal{C}}_1|+|\tilde{\mathcal{C}}_2|\geq2\sqrt{\delta}n$ and $|\tilde{\mathcal{C}}_3|\geq2\sqrt{\delta}n$, then by swapping labels one has  $|\tilde{A}_1\triangle W_i^1|<\delta n$ and either $|\tilde{A}_1\triangle W_i^2|<\delta n$ or $|\tilde{A}_1\triangle W_i^4|<\delta n$ for all $i\in \tilde{\mathcal{C}}_3\setminus I^3$. %$D_1$ is $\alpha$-close to EC1 and %subject to (i)-(ii), we have 
  %$|\tilde{A}_1\triangle W_i^2|<\delta n$ and $|\tilde{B}_1\triangle W_i^4|<\delta n$, $|\tilde{A}_1\triangle W_i^1|<\delta n$ and $|\tilde{B}_1\triangle W_i^3|<\delta n$ for all $i\in \tilde{\mathcal{C}}_3\setminus I_1^3$.
\end{itemize}
Note that if last itemize holds, then for each $i\in \tilde{\mathcal{C}}_3\setminus I^3$, there is an $i_0\in \{2,4\}$ such that %$|\tilde{\mathcal{C}}_3|\geq2\sqrt{\delta}n$ holds in $ {\bf (B1)}$, then there exist  $i\in \tilde{\mathcal{C}}_3\setminus I_1^3$, one has 
$$2\delta n>|\tilde{A}_1\triangle W_i^1|+|\tilde{A}_1\triangle W_i^{i_0}|\geq |W_i^1\triangle W_i^{i_0}|>2\delta n,$$ a contradiction. This implies that exactly one of  $|\tilde{\mathcal{C}}_1|+|\tilde{\mathcal{C}}_2|\geq 2\sqrt{\delta}n$ and $|\tilde{\mathcal{C}}_3|\geq 2\sqrt{\delta}n$ holds. 

Suppose that $|\tilde{\mathcal{C}}_1|+|\tilde{\mathcal{C}}_2|\geq 2\sqrt{\delta}n$.  Denote $[m]:=(\tilde{\mathcal{C}}_1\cup \tilde{\mathcal{C}}_2)\setminus I^1$ and $\mathcal{C}_{\rm bad}:=[m+1,n]$. By adding colors to $\mathcal{C}_{\rm bad}$ if necessary we may assume $m=(1-4\sqrt{\delta})n$. Recall that $\epsilon<\delta$ and $D_i$ has a characteristic partition $(\Tilde{A_i},\Tilde{B_i},\Tilde{L_i})$ for each $i\in [m]$. Hence for each $i\in [2,m]$, there exists a new partition $(A_i,B_i,L_i)$ of $D_i$ such that \allowdisplaybreaks
\begin{itemize}
    \item if $D_i$ is $(\epsilon,{\rm EC}1)$-extremal, then for 
$Y\in \{A,B\}$ we have  $\Tilde{Y_i}\subseteq Y_i$ and  $d_{i}^+(v,Y_i)\geq (\frac{1}{2}-3\sqrt{\delta})n$ for each vertex $v\in Y_i$,
\item if $D_i$ is $(\epsilon,{\rm EC}2)$-extremal, then for 
$Y\in \{A,B\}$ we have  $\Tilde{Y_i}\subseteq Y_i$ and  $d_{i}^+(v,Z_i)\geq (\frac{1}{2}-3\sqrt{\delta})n$ for each vertex $v\in Y_i$,
\item subject to the above two conditions, {the partition $(A_i,B_i,L_i)$ is chosen such that $|A_i\cup B_i|$ is maximized for each  $i\in [2,m]$.}%each vertex in $V$ lies in as many as possible $A_i\cup B_i$ for $i\in [2,m]$. 
\end{itemize}
  
Denote $(A_1,B_1,L_1):=(\Tilde{A}_1,\Tilde{B}_1,\Tilde{L}_1)$. 
Hence for every  $i\in [m]$ and 
$Y\in \{A,B\}$, we have $|Y_1\triangle Y_{i}|\leq |\Tilde{Y_1}\triangle \Tilde{Y_i}|+|\Tilde{L}_i|\leq  \delta n+2{\epsilon}n<2\delta n$, as desired. 
% Recall that $|W_i^1\triangle W_i^2|=2\zeta n>2\delta n$ for all $i\in \tilde{\mathcal{C}}_3$. 
%Hence {\bf (B1)}-{\bf (B3)} imply that if $|\tilde{\mathcal{C}}_3|\geq 2\sqrt{\delta}n$, then $|\tilde{\mathcal{C}}_1|+|\tilde{\mathcal{C}}_2|<2\sqrt{\delta}n$. 
\end{proof}
Based on Lemma \ref{lemma-new1}, we proceed {with the} proof by considering the following two theorems. 
\begin{theorem}\label{theorem-EC12}
Assume that $0<\frac{1}{n}\ll \epsilon\leq  \delta^2\ll1$ and $m=(1-4\sqrt{\delta})n$.  Let $\mathcal{D}=\{D_1,\ldots,D_n\}$ be a collection of digraphs on a common vertex set $V$ of size $n$ and $\delta^0(\mathcal{D})\geq \frac{n}{2}$. If $D_i$ is $(\epsilon,{\rm EC1})$-extremal or $(\epsilon,{\rm EC2})$-extremal for each $i\in [m]$, then $\mathcal{D}$ contains a transversal directed Hamilton cycle. 
\end{theorem}
\begin{theorem}\label{theorem-EC3}
Assume that $0<\frac{1}{n}\ll \epsilon\leq  \delta^2\ll1$ and $m=(1-4\sqrt{\delta})n$.  Let $\mathcal{D}=\{D_1,\ldots,D_n\}$ be a collection of digraphs on a common vertex set $V$ of size $n$ and $\delta^0(\mathcal{D})\geq \frac{n}{2}$. If $D_i$ is $(\epsilon,{\rm EC3})$-extremal for each $i\in [m]$, then $\mathcal{D}$ contains a transversal directed Hamilton cycle. 
\end{theorem}
In the following two sections, we prove the above two theorems respectively. 

\section{Proof of Theorem \ref{theorem-EC12}}

{{In this section we prove Theorem \ref{theorem-EC12}. Following the outline presented in Section 1.4,  we first identify several types of ``bad'' vertices and cover them by short rainbow directed paths whose endpoints possess additional properties. Next we handle the ``bad'' colors by choosing  matchings. We then connect  all these short rainbow directed paths via a connecting lemma. Finally, by the transversal blow-up lemma,  we construct a long transversal directed path, which can be closed to form the desired transversal directed cycle.}}

For the sake of clarity in presenting the following lemmas, we  summarize some  properties of the digraph collection $\mathcal{D}$. 

\begin{enumerate}[label=$(\dagger)$]
    \item\label{P1} Let $0<\frac{1}{n}\ll \epsilon\leq  \delta^2\ll \eta \ll1$ and $m=(1-4\sqrt{\delta})n$. Let $\mathcal{D}=\{D_1,\ldots,D_n\}$ be a collection of digraphs on a common vertex set $V$ of size $n$ and $\delta(\mathcal{D})\geq \frac{n}{2}$. We assume the following conditions hold. 
\begin{enumerate}
    \item For every $i\in [m]$, $D_i$ is either  $(\epsilon,{\rm EC1})$-extremal or $(\epsilon,{\rm EC2})$-extremal and it admits a partition $({A_i},{B_i},{L_i})$ satisfying {\bf (C1)-(C4)}. We extend $A_1\cup B_1$ to an equitable partition $A\cup B$ of $V$. 
    \item Define \begin{align*}
    &\mathcal{C}_1=\{i\in[m]:D_i ~\textrm{is}~(\epsilon,{\rm EC1})\textrm{-extremal}\},\\
    &\mathcal{C}_2=\{i\in[m]:D_i ~\textrm{is}~(\epsilon,{\rm EC2})\textrm{-extremal}\},\ \ \text{and}\ \ \mathcal{C}_{\rm bad}=\mathcal{C}\setminus (\mathcal{C}_1\cup \mathcal{C}_2).
    \end{align*}
    Let 
$
%\hat{\mathcal{C}}:=\cup_{i\in [2]}\psi_i\mathcal{G}_i\ \ \text{and}\ \ 
\hat{\mathcal{C}}:=\bigcup_{k\in [2]}\psi(\mathcal{C}_k),
$ 
where  $\psi(\mathcal{C}_k)=\mathcal{C}_k$ if $|\mathcal{C}_k|\geq \eta n$ and $\psi(\mathcal{C}_k)=\emptyset$ otherwise.
    \item After moving at most $4\delta n$ vertices from $A$ to $B$ (resp. from $B$ to $A$) and deleting at most $\frac{1}{3}\sqrt{\epsilon}n$ vertices in $A\cup B$, there exists a set $V_{\rm bad}\subseteq V$ with size at most $(\sqrt{\epsilon}+\frac{1}{2}\sqrt{\delta})n$ such that each vertex in $Y\setminus V_{\rm bad}$ lies in $Y_i$ for at least $(1-13\sqrt{\delta})|\hat{\mathcal{C}}|$ colors $i\in \hat{\mathcal{C}}$, where $Y\in \{A,B\}$.
    \item Define $$
    X:=\{x\in V: x\not\in A_i\cup B_i~\textrm{for at least}~ 6\sqrt{\delta}|\hat{\mathcal{C}}|~\textrm{colors}~i\in \hat{\mathcal{C}}\}.
    $$
    %Assume $X\setminus V(P)=\{x_1,\ldots,x_s\}$ with $s\leq \sqrt{\epsilon}n$. 
    Let $X'$ be a subset of $X$  consisting of vertices $x$ such that $d_{i}^{+}(x,Y_i)\geq (1-3\sqrt{\delta})n$ (resp. $d_{i}^{-}(x,Y_i)\geq (1-3\sqrt{\delta})n$) for at least $(1-3\sqrt{\delta})|\hat{\mathcal{C}}|$ colors $i\in \hat{\mathcal{C}}$ and $d_{i}^{-}(x,Z_i)\geq \frac{5}{2}\sqrt{\delta}n$ (resp. $d_{i}^{+}(x,Z_i)\geq \frac{5}{2}\sqrt{\delta}n$) for at least $3\sqrt{\delta}|\hat{\mathcal{C}}|$ colors $i\in \hat{\mathcal{C}}$.
\end{enumerate}
\end{enumerate}

The next lemma shows that any vertex contained in $L_i$ for many colors $i\in \mathcal{C}$ can be covered by a rainbow directed path $P_3$ with endpoints not in $V_{\rm bad}$. To enhance flexibility, we select a rainbow star $S_5$ (instead of a simple directed  $P_3$), which allows each endpoint of the resulting path to be replaced by an alternative available vertex when needed.

\begin{lemma}\label{claim5}
    Suppose that \ref{P1} holds and $\{Y,Z\}=\{A,B\}$.  Let $P$ be a rainbow directed path in $\mathcal{D}$ with length at most $\frac{7}{3}\sqrt{\delta}n$.   
    Assume $X\setminus  (X'\cup V(P))=\{x_1,\ldots,x_s\}$ with $s\leq \sqrt{\epsilon}n$. If 
$x_i\in X\cap Y$, then there exist  colors $c_i^1,c_i^2,c_i^3,c_i^4\in \mathcal{C}_j\setminus ({\rm col}(P)\cup \{c_{\ell}^1,c_{\ell}^2,c_{\ell}^3,c_{\ell}^4:{\ell}\in  [i-1]\})$ and vertices $x_i^1,\,x_i^2,\,x_i^3,x_i^4\in W\setminus (V(P)\cup  \{x_{\ell}^1,\,x_{\ell}^2,\,x_{\ell}^3,\,x_{\ell}^4:{\ell}\in  [i-1]\})$ such that ${x_i^kx_i}\in E(D_{c_i^k})$ for $k\in[2]$ and ${x_ix_i^k}\in E(D_{c_i^k})$ for $k\in\{3,4\}$, where 
    \begin{itemize}
    \item $j=2$ and $W=Z$ if $|\mathcal{C}_1|<\eta n$,
        \item $j=1$ and $W=Y$ if $|\mathcal{C}_2|<\eta n$,
        \item either $j=1$ and $W=Y$, or $j=2$ and $W=Z$ otherwise.
        
    \end{itemize}
\end{lemma}
\begin{proof}%[Proof of Lemma~\ref{claim5}]
    We only prove the case that $|\mathcal{C}_1|< \eta n$, and the other cases can be proved by similar  arguments. Suppose there exists an $i_0\in [s]$ such that  Lemma~\ref{claim5} holds for all $i\in [i_0-1]$ but does not hold for $i_0$. Assume for convenience that $x_{i_0}\in X\cap A$.  Denote $\mathcal{C}_2':=\mathcal{C}_2\setminus ({\rm col}(P)\cup \{c_{\ell}^1,c_{\ell}^2,c_{\ell}^3,c_{\ell}^4:{\ell}\in  [i_0-1]\})$. Let
    \begin{align*}
        &\mathcal{C}_2^1:=\big\{j\in \mathcal{C}_2':d_{j}^+(x_{i_0},B)\geq 4(i_0-1)+4+|V(P)|\big\}\ \text{and}\\
        &\mathcal{C}_2^2:=\big\{j\in \mathcal{C}_2':d_{j}^-(x_{i_0},B)\geq 4(i_0-1)+4+|V(P)|\big\}.      
    \end{align*}
    Hence either $|\mathcal{C}_2^1|\leq 3$ or $|\mathcal{C}_2^2|\leq 3$. Without loss of generality, assume that $|\mathcal{C}_2^1|\leq 3$. Thus, for each $j\in \mathcal{C}_2'\setminus \mathcal{C}_2^1$, we have 
    \allowdisplaybreaks
    \begin{align*}
        d_{j}^+(x_{i_0},A_j)&\geq d_{j}^+(x_{i_0},A)-|A\setminus A_1|-|A_1\triangle A_j|-|X_A\setminus X_A^2|-|X'|\\
        &\geq \delta^0(D_j)-d_{j}^+(x_{i_0},B)-|A\setminus A_1|-|A_1\triangle A_j|-|X_A\setminus X_A^2|-|X'|
        \\
        &\geq \frac{n}{2}-\left(4(i_0-1)+4+|V(P)|\right)-2{\epsilon}n-2\delta n-4\delta n-\sqrt{\epsilon}n\\
        &\geq \left(\frac{1}{2}-3\sqrt{\delta}\right)n.
    \end{align*}   
Furthermore, $|\mathcal{C}_2'\setminus \mathcal{C}_2^1|\geq (1-3\sqrt{\delta})|\mathcal{C}_2|$.
     
Notice that $x_{i_0}\notin X'$. Then, for at least $(1-3\sqrt{\delta})|\mathcal{C}_2|$ colors $j\in \mathcal{C}_2$, $d_j^-(x_{i_0},B_j)< \frac{5}{2}\sqrt{\delta}n$, and consequently, $d_j^-(x_{i_0},A_j)\geq (\frac{1}{2}-3\sqrt{\delta})n$ for the same set of colors. By {\bf (C4)}, we know that  $x_{i_0}\in A_j\cup B_j$ for all but at most $6\sqrt{\delta}|\mathcal{C}_2|$ colors $j\in \mathcal{C}_2$,  
    which leads to $x_{i_0}\not\in X$, a contradiction. 
     \end{proof}
     
Next, we give an application of the transversal blow-up lemma (see \cite{cheng2023transversals}) for embedding transversal undirected Hamilton paths inside very dense bipartite graph collections, which can be proved by minor modifications to the proof of \cite[Lemma 6.1]{cheng2024stability}, and we omit the proof here. We say that $P$ is an undirected subgraph of a digraph collection $\mathcal{D}=\{{D_1},\ldots,D_n\}$ if, for every edge $uv\in E(P)\cap E(D_i)$, the reverse arc $vu\in E(D_i)$. 
\begin{lemma}\label{lemma4.1}
   Suppose \ref{P1} holds and $\{W,Z\}=\{A,B\}$. Let $W^*\subseteq W\setminus V_{\rm bad}$ and $Z^*\subseteq Z\setminus V_{\rm bad}$, where $|W^*|,|Z^*|\geq \eta n$, $W^*\cap Z^*=\emptyset$ and $|W^*|-|Z^*|=t\in \{0,1\}$. Let $T^*=Z^*$ if $t=0$ and $T^*=W^*$ if $t=1$. Let $\mathcal{C^*}\subseteq \mathcal{C}$ satisfy $|\mathcal{C^*}|=|W^*|+|Z^*|-1$, where $\mathcal{C^*}\subseteq \mathcal{C}_1$ if $W=Z$ and $\mathcal{C^*}\subseteq \mathcal{C}_2$ if $W\neq Z$. Let $W^-\subseteq W^*$ and $T^+\subseteq T^*$ with $|W^-|,|T^+|\geq \frac{\eta n}{8}$. Then there is a transversal undirected Hamilton path in $\{D_i^{\pm}[W^*,Z^*]:i\in \mathcal{C}^*\}$ starting at $W^-$ and ending   at $T^+$.
    \end{lemma}
    
The subsequent result can be used to connect two disjoint  short rainbow paths into a single short rainbow path. 
\begin{lemma}
  [Connecting tool]\label{conn}
    Suppose that \ref{P1} holds and $\{Y,Z\}=\{A,B\}$. Assume $P=u_1u_2\ldots u_s$ and $Q=v_1v_2\ldots v_t$ are two disjoint rainbow directed paths inside $\mathcal{D}$ with $u_s,v_1\notin V_{\rm bad}\cup X'$ and   $s+t\leq 5\eta n$.  
    \begin{enumerate}
        \item[{\rm (i)}] If $u_s\in Y$, $v_1\in Z$ and $|\mathcal{C}_2\setminus {\rm col}(P\cup Q)|\geq 14\sqrt{\delta} n$, then there are three colors  $c_1,c_2,c_3\in\mathcal{C}_2\setminus {\rm col}(P\cup Q)$ and two vertices $w_1\in Z\setminus (V(P\cup Q)\cup V_{\rm bad})$,  $w_1'\in Y\setminus (V(P\cup Q)\cup V_{\rm bad})$ such that $u_1Pu_sw_1w_1'v_1Qv_t$ is a rainbow directed path with colors ${\rm col}(P\cup Q)\cup \{c_1,c_2,c_3\}$.
        \item[{\rm (ii)}]  If $u_s,v_1\in Y$ and $|\mathcal{C}_2\setminus {\rm col}(P\cup Q)|\geq 14\sqrt{\delta}n$, then there are two colors  $c_1,c_2\in\mathcal{C}_2\setminus {\rm col}(P\cup Q)$ and a vertex $w_1\in Z\setminus (V(P\cup Q)\cup V_{\rm bad})$ such that $u_1Pu_sw_1v_1Qv_t$ is a rainbow directed path with colors ${\rm col}(P\cup Q)\cup \{c_1,c_2\}$.
        \item[{\rm (iii)}] If $u_s,v_1\in Y$ and $|\mathcal{C}_1\setminus {\rm col}(P\cup Q)|\geq 14\sqrt{\delta}n$, then there are two colors  $c_1,c_2\in \mathcal{C}_1\setminus {\rm col}(P\cup Q)$ and a vertex $w_1\in Y\setminus (V(P\cup Q)\cup V_{\rm bad})$ such that $u_1Pu_sw_1v_1Qv_t$ is a rainbow directed path with  colors ${\rm col}(P\cup Q)\cup \{c_1,c_2\}$. 
 \end{enumerate}
\end{lemma}
\begin{proof}%[Proof of Lemma~\ref{conn}]
We only give the proof of (i), the other two statements can be proved by similar discussions, whose procedures are omitted.

Note that $u_s,v_1\not\in V_{\rm bad}\cup X'$. Then $u_s$ (resp. $v_1$) lies in $Y_i$ (resp. $Z_i$) for at least $(1-13\sqrt{\delta})|\hat{\mathcal{C}}|$ colors $i\in \hat{\mathcal{C}}$.  Since $|\mathcal{C}_2\setminus {\rm col}(P\cup Q)|\geq 14\sqrt{\delta}n$, we obtain that $u_s$ (resp. $v_1$) lies in $Y_i$ (resp. $Z_i$) for at least $\sqrt{\delta}n$ colors $i\in \mathcal{C}_2\setminus {\rm col}(P\cup Q)$. Hence there are two distinct colors $c_1,c_2\in \mathcal{C}_2\setminus {\rm col}(P\cup Q)$ such that $u_s\in Y_{c_1}$ and $v_1\in Z_{c_2}$. It is routine to check that 
\begin{align*}
    |N_{{c_1}}^+(u_s)\cap (Z\setminus (V_{\rm bad}\cup X'))|\geq &|N_{{c_1}}^+(u_s)\cap Z_{c_1}|-|Z_1\triangle Z_{c_1}|-|X\cup X_A\cup X_B|\\
\geq& \left(\frac{1}{2}-3\sqrt{\delta}\right)n-\left(2\delta+\sqrt{\epsilon}+\frac{1}{2}\sqrt{\delta}\right)n\geq \left(\frac{1}{2}-4\sqrt{\delta}\right)n.
\end{align*}
Hence there exists a vertex $w_1\in N_{{c_1}}^+(u_s)\cap Z$ 
that avoids $V(P\cup Q)\cup V_{\rm bad}\cup X'$. Similarly, since $w_1\notin V_{\rm bad}\cup X'$, there is a color $c_3\in \mathcal{C}_2\setminus ({\rm col}(P\cup Q)\cup \{c_1,c_2\})$ such that $w_1\in B_{c_3}$ and $$|N_{{c_2}}^-(v_1)\cap N_{{c_3}}^+(w_1)\cap  (Y\setminus (V_{\rm bad}\cup X'))|\geq \left(\frac{1}{2}-10\sqrt{\delta}\right)n.$$ %. On the other hand,
%\begin{align*}
%&|N_{{c_2}}^-(v_1)\cap N_{{c_3}}^+(w_1)\cap  (Y\setminus (V_{\rm bad}\cup X'))|\\
%\geq &|N_{{c_2}}^-(v_1)\cap (A\setminus (V_{\rm bad}\cup X'))|+|N_{{c_3}}^+(w_1)\cap  (Y\setminus (V_{\rm bad}\cup X'))|-\frac{n}{2}\\
%\geq &\left(\frac{1}{2}-8\sqrt{\delta}\right)n.
%\end{align*}
Therefore, there is a vertex $w_1'\in N_{{c_2}}^-(v_1)\cap N_{{c_3}}^+(w_1)\cap  A$ that avoids $V(P\cup Q)\cup V_{\rm bad}\cup X'$. It follows that  $u_1Pu_sw_1w_1'v_1Qv_t$ is a rainbow directed path inside $\mathcal{D}$ with colors ${\rm col}(P\cup Q)\cup \{c_1,c_2,c_3\}$.
\end{proof}

To prove Theorem \ref{theorem-EC12}, we need the following lemma. This establishes that every digraph collection $\mathcal{D}$ contains a transversal directed {Hamilton} cycle if it admits an almost balanced partition, with the small parts inducing an empty graph for almost all colors. % The following lemma shows that In order to complete the proof of Theorem \ref{theorem-EC12}, we need to show the existence of transversal directed Hamilton cycles in Step 1 of Case 1. 
We first state a preliminary result from \cite{chengsun}.
\begin{lemma}[\cite{chengsun}]\label{claim4.2}
     Let $\mathcal{C}$ be a set of colors, and $\mathcal{G}=\{G_i[Y,B]:i\in \mathcal{C}\}$ be a collection of bipartite graphs with the common bipartition $Y\cup B$  such that $7|Y|<|B|\leq\frac{3}{5}|\mathcal{C}|$. If $\sum_{i\in \mathcal{C}}|E(D_i[Y,B])|\geq t|B||\mathcal{C}|$ for some integer $t$ with $1\leq t\leq |Y|$, then $\mathcal{G}$ contains $t$ disjoint rainbow stars on $5$ vertices, and
each of them has its center in $Y$ and other vertices in $B$.
\end{lemma}

\begin{lemma}\label{Y-large}
Assume $0<\frac{1}{n}\ll\delta\ll1$ and $0\leq \gamma\leq3{\delta}$. Let $\mathcal{C}$ be a set of $n$ colors, and let  $\mathcal{D}=\{D_i:i\in \mathcal{C}\}$ be a collection of digraphs with the common vertex set $V$ of size $n$ such that  $\delta^0(\mathcal{D})\geq \left\lceil\frac{n}{2}\right\rceil$. Let $A\cup B$ be a partition of $V$ with  $|A|=\left\lceil\frac{n}{2}\right\rceil+\gamma n$, and let $\mathcal{C}'\cup \mathcal{C}''$ be a partition of $\mathcal{C}$ with  $|\mathcal{C}''|\leq \delta n$. Assume that $D_i[B]=\emptyset$ for all $i\in \mathcal{C}'$. Then $\mathcal{D}$ contains a transversal directed Hamilton cycle. 
\end{lemma}
\begin{proof}
Define $Y_1:=\{v\in A: d_{i}^+(v,B)\leq (1-\delta^{\frac{1}{4}})|B|\ \textrm{for at least}\ \delta^{\frac{1}{4}}|\mathcal{C}'|\ \textrm{colors}\  i\in \mathcal{C}'\}$, $Y_2:=\{v\in A: d_{i}^-(v,B)\leq (1-\delta^{\frac{1}{4}})|B|\ \textrm{for at least}\ \delta^{\frac{1}{4}}|\mathcal{C}'|\ \textrm{colors}\  i\in \mathcal{C}'\}$.  
    Since $D_i[B]=\emptyset$ for all $i\in \mathcal{C}'$, it is routine to check that 
\begin{align}\notag
    \left\lceil\frac{n}{2}\right\rceil|B||\mathcal{C}'|&\leq \sum_{i\in \mathcal{C}'}|E({D_i}[A,B])|\\\notag
    &\leq |Y_1|(1-\delta^{\frac{1}{4}})|B|\delta^{\frac{1}{4}}|\mathcal{C}'|+|Y_1||B|(1-\delta^{\frac{1}{4}})|\mathcal{C}'|+(|A|-|Y_1|)|B||\mathcal{C}'|\\\notag
    &=(|A|-\delta^{\frac{1}{2}}|Y_1|)|B||\mathcal{C}'|.
\end{align}
It implies that $|Y_1|\leq \frac{\gamma n}{\sqrt{\delta}}\leq 3\sqrt{\delta}n$. Similarly, $|Y_2|\leq \frac{\gamma n}{\sqrt{\delta}}\leq 3\sqrt{\delta}n$. For a vertex $v\in Y_1\setminus Y_2$ (resp. $v\in Y_2\setminus Y_1$), we have $d_{i}^-(v,B)>(1-\delta^{\frac{1}{4}})|B|$ (resp. $d_{i}^+(v,B)> (1-\delta^{\frac{1}{4}})|B|$) for at least $(1-\delta^{\frac{1}{4}})|\mathcal{C}'|$ colors $i\in \mathcal{C}'$. 
{Notice that Lemma~\ref{lemma4.1} and Lemma~\ref{conn} (i)-(ii) hold by setting  $V_{\rm bad}:=Y_1\cup Y_2$ and $\mathcal{C}_2=\mathcal{C}'$.}

Since $\delta^0(\mathcal{D})\geq \left\lceil\frac{n}{2}\right\rceil$, we have $|E(D_i[A,B])|,|E(D_i[B,A])| \geq |B|\left\lceil\frac{n}{2}\right\rceil$ for each $i\in \mathcal{C}'$, i.e., there are at most $|B|(\left\lceil\frac{n}{2}\right\rceil+\gamma n)-|B|\left\lceil\frac{n}{2}\right\rceil=|B|\gamma n$ non-edges from $A$ to $B$ (resp., from $B$ to $A$) in $D_i$. Hence 
$$
\sum_{i\in \mathcal{C}'}|E(D_i[Y_1,B])|\geq (|Y_1|-\gamma n)|B||\mathcal{C}'|\ \  \text{and}\ \ \sum_{i\in \mathcal{C}'}|E(D_i[B,Y_2])|\geq (|Y_2|-\gamma n)|B||\mathcal{C}'|.
$$ 
According to the sizes of $|Y_1|$ and $|Y_2|$, we proceed by considering the following three cases.

\medskip
{\bf Case 1. $|Y_1|, |Y_2|\geq \gamma n$.}
\medskip

In view of Lemma \ref{claim4.2}, there exists a subset $Y_1'\subseteq Y_1$ with size $\gamma n$ such that vertices in $Y_1\setminus Y_1'$ can be covered by  disjoint rainbow $P_3$ copies inside $\{D_i[Y_1,B]:i\in \mathcal{C}'\}$ with centers in $Y_1\setminus Y_1'$; and a subset $Y_2'\subseteq Y_2$ with size $\gamma n$ such that vertices in $Y_2\setminus Y_2'$ can be covered by  disjoint rainbow $P_3$ copies inside $\{D_i[B,Y_2]:i\in \mathcal{C}'\}$ with centers in $Y_2\setminus Y_2'$. Recall that each vertex in $Y_1$ (resp. $Y_2$) has at least $\left\lceil\frac{n}{2}\right\rceil-(1-\delta^{\frac{1}{4}})|B|>4|Y|$ out-neighbors (resp. in-neighbors) in $A$ for at least $\delta^{\frac{1}{4}}|\mathcal{C}'|>\delta^{\frac{1}{4}}(1-\delta)n>4|Y|$ digraphs $D_i$ with $i\in\mathcal{C}'$. Therefore, by using colors in $\mathcal{C}'$,  
\begin{itemize}
  \item vertices in $(Y_1\cup Y_2)\setminus (Y_1'\cup Y_2')$ can be covered by disjoint rainbow directed $P_3$ copies with centers in $(Y_1\cup Y_2)\setminus (Y_1'\cup Y_2')$ and endpoints in $B$,
  \item vertices in $Y_1'\setminus Y_2'$ can be covered by disjoint rainbow directed $P_3$ copies starting at $B$ and ending   at $A$ with centers in $Y_1'\setminus Y_2'$,
  %\item vertices in $(Y_1\setminus Y_1')\cap (Y_2\setminus Y_2')$ can be covered with disjoint rainbow directed $P_3$ with centers in $(Y_1\setminus Y_1')\cap (Y_2\setminus Y_2')$ and endpoints in $B$. 
  \item vertices in $Y_2'\setminus Y_1'$ can be covered by disjoint rainbow directed $P_3$ copies starting at $A$ and ending   at $B$ with centers in $Y_2'\setminus Y_1'$,
  \item vertices in $Y_1'\cap Y_2'$ can be covered by disjoint rainbow directed $P_3$ copies with centers in $Y_1'\cap Y_2'$  and endpoints in $A$.
  %\item vertices in $Y_2'\setminus Y_1$ can be covered with disjoint rainbow directed $P_3$ with centers in $Y_2'\setminus Y_1$, head in $A$ and tail in $B$.
\end{itemize}
For each $y\in Y_1\cup Y_2$, let $P_y:=y^1yy^2$  denote the rainbow directed $P_3$ with center $y$, where $P_y$ has colors $c_y^1$ and $c_y^2$. 
Let $\mathbf{P}=\{P_y:y\in Y_1\cup Y_2\}$. By using Lemma~\ref{conn}, one may connect all of the above rainbow directed $P_3$ copies into a single rainbow directed path $P^1$  that starts at $A$ and ends  at $B$. Then 
$$
|V(P^1)\cap A|-|V(P^1)\cap B|=|Y_1'\setminus Y_2'|+|Y_2'\setminus Y_1'|+2|Y_1'\cap Y_2'|=|Y_1'|+|Y_2'|=2\gamma n.
$$
Hence, $|A\setminus V(P^1)|=|B\setminus V(P^1)|+\sigma$, where $\sigma=0$ if $n$ is odd and $\sigma=1$ otherwise. 
%  For such vertices in $Y_1\setminus Y_2$ which cannot be covered by rainbow directed paths with both end vertices in $B$, there is an out-neighbor in $A\setminus (Y_1\cup Y_2)$ and an in-neighbor in $B$. 

%Similarly, there exist  $|Y_2|-\gamma n$ disjoint rainbow directed $P_3$ inside $\{D_i[Y_2,B]:i\in \mathcal{C}'\}$ with centers in $Y_1\setminus Y_2$. For such vertices in $Y_2\setminus Y_1$ which cannot be covered by rainbow directed paths with both end vertices in $B$, there is an in-neighbor in $A\setminus (Y_1\cup Y_2)$ and an out-neighbor in $B$.  

%Let $\mathbf{P}_1$ be a set consisting of those rainbow $P_3$ and $Y_1$ be a subset of $Y$ consisting of the centers of them. Denote $Y_2:=Y\setminus Y_1$. Then $|Y_2|= \gamma n+1$. 

%Recall that each vertex in $Y_1$ is adjacent to at least $\left\lceil\frac{n}{2}\right\rceil-(1-\delta^{\frac{1}{4}})|B|>4|Y_1|$ vertices in $A$ for at least $\delta^{\frac{1}{4}}|\mathcal{C}'|>\delta^{\frac{1}{4}}(1-\delta)n>4|Y|$ colors $i\in\mathcal{C}'$. Then there exist  $\gamma n$ disjoint rainbow $P_3$ inside $\{D_i[Y_1',A\setminus Y]:i\in \mathcal{C}'\setminus {\rm col}(\mathbf{P}_1)\}$ with centers in $Y_2$. 

It is routine to check that there exists a rainbow matching inside $\{D_i^{\pm}[A\setminus V(P^1),B\setminus V(P^1)]: i\in \mathcal{C}''\}$, say $M$, such that $D_j[A\setminus V(P^1\cup M),B\setminus V(P^1\cup M)]\neq \emptyset$ for each $j\in {\rm col}(M)$. 
Furthermore, $D_{j}$ is $(19\sqrt{\delta},{\rm EC}1)$-extremal for all $j\in \mathcal{C}''\setminus {\rm col}(M)$. 
% We proceed by considering the following two cases. 

\medskip
{\bf Subcase 1.1.} $n$ and $|\mathcal{C}''\setminus {\rm col}(M)|$ have the same parity. 
\medskip
% and is odd and $|\mathcal{C}''\setminus {\rm col}(M)|$ is odd, or $n$ is even  and $|\mathcal{C}''\setminus {\rm col}(M)|$ is even.

We greedily choose two disjoint  rainbow directed paths $P_A$ and $P_B$ inside $\{D_i[A\setminus V(P^1\cup M)]\cup D_i[B\setminus V(P^1\cup M)]:i\in \mathcal{C}''\setminus {\rm col}(M)\}$ with lengths $\lceil\frac{|\mathcal{C}''\setminus {\rm col}(M)|}{2}\rceil$ and
$\lfloor\frac{|\mathcal{C}''\setminus {\rm col}(M)|}{2}\rfloor$ respectively, such that $V(P_A)\subseteq A$ and $V(P_B)\subseteq B$. 
%In view of Lemma~\ref{conn} (i)-(ii), by using colors in $\mathcal{C}'$, one may connect all rainbow paths  in $\mathbf{P}$ in turn to get a single rainbow path $P^1$ with $2|Y_1|+3|Y_2|$ vertices in $A$ and $2|Y_1|+|Y_2|$ vertices in $B$, whose endpoints are in different parts. Therefore, $|B\setminus V(P^1)|-|A\setminus V(P^1)|=\sigma$. 
Applying Lemma~\ref{conn}~(i), we successively  connect $P^1$, all rainbow edges in $M$, $P_A$ and $P_B$ into a single rainbow path $P^2$ whose endpoints lie in different parts. Clearly, $|A\setminus V(P^2)|=|B\setminus V(P^2)|$. 
Together with Lemma~\ref{lemma4.1}, we know that  $\mathcal{D}$ contains a transversal directed Hamilton cycle, as desired.   

\medskip
{\bf Subcase 1.2.} $n$ and $|\mathcal{C}''\setminus {\rm col}(M)|$ have { different} parity.
\medskip%$n$ is odd and  $|\mathcal{C}''\setminus {\rm col}(M)|$ is even, or $n$ is even and $|\mathcal{C}''\setminus {\rm col}(M)|$ is odd. 

Notice that $|B|=\lfloor\frac{n}{2}\rfloor-\gamma n$. Hence for each $D_i$ with $i\in \mathcal{C}$, each vertex in $B$ has at least $\gamma n+1$ in-neighbors in $A$. Therefore, each vertex in $B$ has at least one in-neighbor outside $Y_1'$. 

We first assume  $|\mathcal{C}''\setminus {\rm col}(M)|\geq 1$. Choose an unused vertex $y\in B$ and a color $c\in \mathcal{C}''\setminus {\rm col}(M)$. Assume that ${xy}\in E(D_c[A,B])$ with $x\notin Y_2'$. 
%$\bullet$ 
If $x\in A\setminus (Y_1\cup Y_2)$, then $x$ can be chosen outside $V(P^1\cup M)$. Using Lemma~\ref{conn}~(i)-(ii), we connect $xy$ and $P^1$ into a single rainbow directed path $P$. If $x\in (Y_1\cup Y_2)\setminus Y_1'$, then we replace the end vertex of $P_x$ by $y$. In either situation the problem reduces to Subcase 1.1,  yielding a transversal directed Hamilton cycle in  $\mathcal{D}$, as desired.

%Applying a similar discussion as Subcase 1.1, we obtain a transversal directed Hamilton cycle inside $\mathcal{D}$, as desired. 

%$\bullet$ Suppose that $x\in (Y_1\cup Y_2)\setminus Y_1'$. Then we replace the end  vertex of $P_x$ by $y$. Applying a similar discussion as Subcase 1.1, we obtain a transversal directed Hamilton cycle inside $\mathcal{D}$, as desired. 

Now, we consider that $n$ is odd and  $|\mathcal{C}''\setminus {\rm col}(M)|=0$. For each $D_i$ with $i\in \mathcal{C}$, each vertex in $A$ has at least one in-neighbor outside $Y_1'$. Choose an unused vertex $y\in A\setminus (Y_1\cup Y_2)$ and a color $c\in \mathcal{C}$. Then there exists a vertex $x\in A\setminus Y_1'$ such that ${xy}\in E(D_c[A])$.  By a similar discussion as above, we obtain a transversal directed Hamilton cycle inside $\mathcal{D}$, as desired.

\medskip
{\bf Case 2.} At least one of $Y_1$ and $Y_2$ has size less than $\gamma n$.
\medskip

In this case, we consider only $|Y_1|\geq \gamma n$ and $|Y_2|<\gamma n$; the remaining two cases are analogous.  Following the argument of Case 1, we know that all vertices in $Y_1$ can be covered by a set of disjoint rainbow directed copies of $P_3$ with centers in $Y_1$ and endpoints in $A\setminus (Y_1\cup Y_2)$. Furthermore, each vertex in $Y_2\setminus Y_1$ can be covered by a rainbow directed $P_3$ with its center in  $Y_2\setminus Y_1$,  starting from   $A\setminus (Y_1\cup Y_2)$ and using colors in $\mathcal{C}'$. (In this process, for the choice of the other endpoint of each rainbow directed $P_3$, we prefer using vertices in $A$ over vertices in $B$.) %or  is adjacent to at least $\left\lceil\frac{n}{2}\right\rceil-(1-\delta^{\frac{1}{4}})|B|>5|Y|$ vertices in $A$ for at least $\delta^{\frac{1}{4}}|\mathcal{C}'|>\delta^{\frac{1}{4}}(1-\delta)n>5|Y|$ colors $i\in\mathcal{C}'$. Then there exist  $|Y|$ disjoint rainbow paths $P_3$ with centers in $Y$ and endpoints in $A\setminus Y\cup B$ using colors in $\mathcal{C}'$. 
By using unused colors in $\mathcal{C}$ and unused vertices in $A\setminus Y$, one may extend those rainbow directed $P_3$ as long as possible or choose all other disjoint rainbow directed paths. Let $\mathbf{P}=\{Q_1,Q_2,\ldots,Q_t\}$ be a set consisting of all disjoint rainbow directed paths in the above, where each $Q_i$ has at least one endpoint in $A\setminus Y$ and length $s_i\ (1\leq i\leq t)$ inside $A$. We next consider the following two possible cases. 

\medskip
{\bf Subcase 2.1.}  $|B|\geq |A|-(s_1+\cdots+s_t)-\sigma$. 
\medskip

In this subcase, there exists a set $\mathbf{P}'=\{Q_1',Q_2',\ldots,Q_{\ell}'\}$ such that $|B|=|A|-(s_1'+\cdots+s_{\ell}')-\sigma$, the endpoints of each  $Q_i'$ lie outside  $Y$ and $Y_2\cup Y_1'\subseteq V(Q_1'\cup \ldots\cup Q_{\ell}')$, where $s_i'=|E(Q_i'[A])|$ for each $i\in [\ell]$. 
By Lemma~\ref{conn} (i)-(ii), we connect all rainbow directed paths in $\mathbf{P}'$ into a single rainbow directed path $P^1$, whose endpoints are in different parts. 
Since $|A|-|B|=2\gamma n+\sigma$, we have $\ell\leq s_1'+\cdots+s_{\ell}'= 2\gamma n$. Therefore, 
$|E(P^1)|\leq 8\gamma n$ and $|A\setminus V(P^1)|-\sigma=|B\setminus V(P^1)|$.

Recall that all vertices in $Y_1\setminus V(\mathbf{P}')$ can also be covered by a set of disjoint rainbow directed copies of $P_3$ with centers in $Y_1\setminus V(\mathbf{P}')$ and both endpoints in $B$. Applying Lemma~\ref{conn} (i) again to connect those rainbow directed paths with $P^1$, we get a rainbow directed path $P^2$ with length at most $8\gamma n+12\sqrt{\delta}n$, and with endpoints in different parts.

Choose a maximum rainbow matching inside $\{D_i^{\pm}[A\setminus V(P^2),B\setminus V(P^2)]: i\in \mathcal{C}''\setminus {\rm col}(P^2)\}$, say $M$, such that $D_j^{\pm}[A\setminus V(P^2\cup M),B\setminus V(P^2\cup M)]\neq\emptyset$  for each $j\in {\rm col}(M)$. %each edge in $M$ can be replaced by an edge between $A$ and $B$ using the same color and  vertices not in $V(P^1\cup M)$. 
Clearly, $D_{j}$ is $(13\sqrt{\delta},{\rm EC}1)$-extremal for all $j\in \mathcal{C}''\setminus {\rm col}(P^1\cup M)$.  In fact, when constructing  $P^1-B-Y$,  we may prioritize the use of colors from  $\mathcal{C}''\setminus {\rm col}(P^1\cup M)$ over other colors. Hence either  ${\rm col}(P^1-B-Y)\subseteq \mathcal{C}''\setminus {\rm col}(M)$ or $\mathcal{C}''\setminus {\rm col}(M) \subseteq {\rm col}(P^1-B-Y)$. 

Based on Lemma~\ref{conn} (i), one may  connect $P^1$ and all edges of $M$ into a single rainbow directed path $P^2$, whose endpoints are in different parts and length is at most $8\gamma n+4{\delta} n$. Next, we are to choose two disjoint  rainbow paths $P_A$  and $P_B$ inside $\{D_i[A\setminus V(P^1\cup M)]\cup D_i[B\setminus V(P^1\cup M)]:i\in \mathcal{C}''\setminus {\rm col}(P^1\cup M)\}$ such that $V(P_A)\subseteq A$ and $V(P_B)\subseteq B$ respectively, whose lengths are determined by the parity of $|\mathcal{C}''\setminus {\rm col}(P^1\cup M)|$. Clearly, if $n$ and $|\mathcal{C}''\setminus {\rm col}(P^1\cup M)|$ have the same parity, then by Lemma~\ref{lemma4.1}, we know that  $\mathcal{D}$  contains a transversal directed Hamilton cycle, as desired. Hence $n$ and $|\mathcal{C}''\setminus {\rm col}(P^1\cup M)|$ have {different} parity. By an argument analogous to that in the proof of Step 4 in Case 1, we can obtain the desired transversal Hamilton cycle in $\mathcal{D}$. 

\medskip
{\bf Subcase 2.2.} $|B|<|A|-(s_1+\cdots+s_t)-\sigma$. 
\medskip

In this case, $\sum_{i=1}^t s_i<2\gamma n$. % and so $\sum_{i=1}^t \frac{s_i}{2}\leq \gamma n-\frac{1}{2}$. %Let $t'$ be the number of rainbow directed paths each of which has head and tail in different parts. Using  Lemma~\ref{conn} (ii) in $B$, one may connect all rainbow paths of  $\mathbf{P}$ into a single rainbow paths $P^3$, whose endpoint are in different parts. Let $\Tilde{A}=A\setminus V(P^3)$, $\Tilde{B}=B\setminus V(P^3)$ and $\Tilde{\mathcal{C}}=\mathcal{C}\setminus {\rm col}(P^3)$. Then $|\Tilde{A}|=|A|-(s_1+\cdots+s_t+t+(t'-1))$, $|\Tilde{B}|=|B|-t-(t'-1)$ and $|\Tilde{\mathcal{C}}|=|\mathcal{C}|-(s_1+\cdots+s_t+2(t'-1)+2t-1)$. 
Let $w$ be an arbitrary vertex in $A\setminus V(\mathbf{P})$ and $c_1,c_2$ be two colors  in $\mathcal{C}\setminus {\rm col}(\mathbf{P})$. %Without loss of generality, assume that $d_{D_{c_1}}(w,V(\mathbf{P}))\leq d_{D_{c_2}}(w,V(\mathbf{P}))$.  
Clearly, in $D_{c_1}$ and $D_{c_2}$, $w$ has at least $\left\lceil\frac{n}{2}\right\rceil-(\lfloor\frac{n}{2}\rfloor-\gamma n)=\gamma n+\sigma$ out-neighbors (resp. in-neighbors) in $V(\mathbf{P})$ and it cannot be adjacent to the start vertex of $Q_i[A]$ (resp. from the end   vertex of $Q_i[A]$) for all $i\in [t]$ (by the maximality of $\mathbf{P}$). For convenience, assume that  $s_0=0$ and  $V(Q_i)\cap A:=\{v_{s_1+\cdots+s_{i-1}+i},\ldots,v_{s_1+\cdots+s_{i}+i}\}$ for each $i\in [t]$. 
Define
$$
  I_1:=\left\{i:v_i\in N_{{c_1}}^-(w)\right\}\ \ \text{and}\ \ I_2:=\left\{i:v_{i+1}\in N_{{c_2}}^+(w)\right\}.
$$
Hence $
  I_1,I_2\subseteq \bigcup_{i\in [t]}\left[\sum_{j\in [i-1]}s_{j}+i,\sum_{j\in [i]}s_j+i-1\right]$. %[1,s_1]\cup [s_1+2,s_1+s_2+1]\cup \cdots \cup [s_1+\cdots+s_{t-1}+t,s_1+\cdots+s_t+t-1].$ 
Therefore, $I_1\cap I_2\neq \emptyset$, which contradicts the maximality of $\mathbf{P}$. 
\end{proof}

Now, we are ready to give the proof of Theorem \ref{theorem-EC12}. 
\begin{proof}[Proof of Theorem \ref{theorem-EC12}]
Choose constant $\epsilon,\delta,\eta$ such that
$$\frac{1}{n}\ll \epsilon\leq \delta^2\ll \eta\ll 1,$$
let $\mathcal{D}=\{D_1,\ldots,D_n\}$ be a collection of digraphs on a common vertex set $V$ of size $n$ and $\delta^0(\mathcal{D})\geq \frac{n}{2}$. 
Suppose that $\mathcal{D}$ does not contain transversal directed Hamilton cycles. By Lemma \ref{lemma-new1}, for each $i\in [m]$, %Denote $[m]:=(\tilde{\mathcal{C}}_1\cup \tilde{\mathcal{C}}_2)\setminus I^1$ and $\mathcal{C}_{\rm bad}:=[m+1,n]$. By adding colors to $\mathcal{C}_{\rm bad}$ if necessary we may assume $m=(1-4\sqrt{\delta})n$. Recall that $\epsilon<\delta$ and 
  $D_i$ has a characteristic partition $({A_i},{B_i},{L_i})$ satisfying {\bf (C1)}-{\bf (C4)}. Moreover, for every  $i\in [m]$ and 
$Y\in \{A,B\}$, we have 
\begin{align}\label{difference}
    |L_i|\leq 2\epsilon n\ \text{and}\  |{Y}_1{\triangle} Y_{i}|<2\delta n.
\end{align}%Hence for each $i\in [2,m]$, there exists a new partition $(A_i,B_i,L_i)$ of $D_i$ such that \allowdisplaybreaks
%\begin{itemize}
%    \item[{\rm \bf (C1)}] if $D_i$ is $(\epsilon,{\rm EC}1)$-extremal, then for $Y\in \{A,B\}$ we have  $\Tilde{Y_i}\subseteq Y_i$ and  $d_{i}^+(v,Y_i)\geq (\frac{1}{2}-3\sqrt{\delta})n$ for each vertex $v\in Y_i$,
%\item[{\rm \bf (C2)}] if $D_i$ is $(\epsilon,{\rm EC}2)$-extremal, then for  $Y\in \{A,B\}$ we have  $\Tilde{Y_i}\subseteq Y_i$ and  $d_{i}^+(v,Z_i)\geq (\frac{1}{2}-3\sqrt{\delta})n$ for each vertex $v\in Y_i$,
%\item[{\rm \bf (C3)}] subject to {\bf (C1)-(C2)}, {the partition $(A_i,B_i,L_i)$ is chosen such that $|A_i\cup B_i|$ is maximized for each  $i\in [2,m]$.}%each vertex in $V$ lies in as many as possible $A_i\cup B_i$ for $i\in [2,m]$. 
%\end{itemize}
  
%Denote $(A_1,B_1,L_1):=(\Tilde{A}_1,\Tilde{B}_1,\Tilde{L}_1)$.  
Expand $A_1\cup B_1$ to an equitable partition $A\cup B$ of $V$. For simplicity, we assume $Y\in\{A,B\}$ and $Z\in \{A,B\}\setminus \{Y\}$ in the following.  Let $\mathcal{C}_1\cup \mathcal{C}_2$ be a partition  of $[m]$ with 
\begin{align*}
    &\mathcal{C}_1=\{i\in[m]:  D_i ~\textrm{is}~(\epsilon,{\rm EC}1){\text{-extremal}}\},\\
    &\mathcal{C}_2=\{i\in[m]: D_i ~\textrm{is}~(\epsilon,{\rm EC}2){\text{-extremal}}\}.
\end{align*}
%Hence for $i\in[2,m]$, we have $|Y_1\triangle Y_{i}|\leq |\Tilde{Y_1}\triangle \Tilde{Y_i}|+|\Tilde{L}_i|\leq  \delta n+2{\epsilon}n<2\delta n$. 
 Let 
$
\hat{\mathcal{C}}:=\bigcup_{k\in [2]}\psi(\mathcal{C}_k),
$ 
where  $\psi(\mathcal{C}_k)=\mathcal{C}_k$ if $|\mathcal{C}_k|\geq \eta n$ and $\psi(\mathcal{C}_k)=\emptyset$ otherwise. Clearly, $|\hat{\mathcal{C}}|\geq (1-4\sqrt{\delta}-\eta)n$. 
Define 
\begin{align*}
  X:=&\{x\in V: x\not\in A_i\cup B_i~\textrm{for at least}~ 6\sqrt{\delta}|\hat{\mathcal{C}}|~\textrm{colors}~i\in \hat{\mathcal{C}}\}, \\
  %\cup \\
  %&\{x\in X: d_{D_i}^-(x,Z_i)\geq (1-3\sqrt{\delta})n~\textrm{for at least}~(1-3\sqrt{\delta})|\hat{\mathcal{C}}|~\textrm{colors}~i\in \hat{\mathcal{C}}\\ &\text{and}\ d_{D_i}^+(x,Y_i)\geq 3\sqrt{\delta}n~\textrm{for at least}~3\sqrt{\delta}n~\textrm{colors}~i\in \hat{\mathcal{C}}\ \},\\
  X_Y:=&\{x\in Y\setminus X:x\not\in Y_{i}~\textrm{for at least}~ 10\sqrt{\delta}|\hat{\mathcal{C}}|~\textrm{colors}~i\in \hat{\mathcal{C}}\}.
\end{align*}
%Recall that $|L_{i}|\leq 2{\epsilon}n$ and $|Y_1\triangle Y_{i}|<2\delta n$ for each $i\in \hat{\mathcal{C}}$. 
By \eqref{difference}, one has  $6\sqrt{\delta}|\hat{\mathcal{C}}||X|\leq 2{\epsilon}n|\hat{\mathcal{C}}|$ and
$$ 10\sqrt{\delta}|\hat{\mathcal{C}}||X_Y|\leq \sum_{i\in \hat{\mathcal{C}}}|Y\setminus Y_i|\leq \sum_{i\in \hat{\mathcal{C}}}(|Y\setminus Y_1|+|Y_1\triangle Y_i|)\leq (2\delta+2{\epsilon}) n|\hat{\mathcal{C}}|.
$$
It follows that $|X|< \frac{1}{3}\sqrt{\epsilon}n$ and $|X_Y|< \frac{1}{4}\sqrt{\delta} n$. For $k\in [2]$, define 
\begin{equation*}
    X_Y^k:=
    \left\{
    \begin{array}{ll}
        \{x\in X_Y:x\in Y_{i}~\textrm{for at least}~ 3\sqrt{\delta}n~\textrm{colors}~i\in {\mathcal{C}_k}\}, &\textrm{if $|\mathcal{C}_k|\geq \eta n$,}\\[5pt]
        \emptyset, &\textrm{otherwise.}
    \end{array}
    \right.
\end{equation*}
For a vertex $x\in X_Y\setminus (\bigcup_{k\in [2]}X_Y^k)$, we know $x\in Y_{i}$  for at most $6\sqrt{\delta}n$ colors $i\in \hat{\mathcal{C}}$ and $x\in A_i\cup B_i$ for at least $(1-6\sqrt{\delta})|\hat{\mathcal{C}}|$ colors $i\in \hat{\mathcal{C}}$. Thus, $x\in Y\setminus Y_{i}$ for at least $|\hat{\mathcal{C}}|-6\sqrt{\delta}n$ colors $i\in \hat{\mathcal{C}}$ and $x\in Z_{i}$ for at least $(1-6\sqrt{\delta})|\hat{\mathcal{C}}|-6\sqrt{\delta}n\geq (1-13\sqrt{\delta})|\hat{\mathcal{C}}|$ colors $i\in \hat{\mathcal{C}}$. %Notice that  $x\in Y$. Therefore, if $x\in Z_i$ for some $i\in \hat{\mathcal{C}}$, then $x\in Y\setminus Y_i$. 
Hence  
$$
\Big|X_Y\setminus \Big(\bigcup_{k\in [2]}X_Y^k\Big)\Big|\big(|\hat{\mathcal{C}}|-6\sqrt{\delta}n\big)\leq \sum_{i\in \hat{\mathcal{C}}}|Y\setminus Y_i|\leq (2\delta +2\epsilon)n|\hat{\mathcal{C}}|.
$$
It follows that $|X_Y\setminus (\bigcup_{k\in [2]}X_Y^k)|<4\delta n$. Then  we move vertices in $X_A\setminus (\bigcup_{k\in [2]}X_A^k)$ to $B$ and vertices in $X_B\setminus (\bigcup_{k\in [2]}X_B^k)$ to $A$.

Let $X'$ be a subset of $X$  consisting of vertices $x$ such that $d_{i}^{+}(x,Y_i)\geq (1-3\sqrt{\delta})n$ (resp. $d_{i}^{-}(x,Y_i)\geq (1-3\sqrt{\delta})n$) for at least $(1-3\sqrt{\delta})|\hat{\mathcal{C}}|$ colors $i\in \hat{\mathcal{C}}$ and $d_{i}^{-}(x,Z_i)\geq \frac{5}{2}\sqrt{\delta}n$ (resp. $d_{i}^{+}(x,Z_i)\geq \frac{5}{2}\sqrt{\delta}n$) for at least $3\sqrt{\delta}|\hat{\mathcal{C}}|$ colors $i\in \hat{\mathcal{C}}$. {Delete vertices in $X'$ from $A\cup B$.}  Without loss of generality, assume that $|B|-|A|=r$ with  $0\leq r\leq 8{\delta}n+\sqrt{\epsilon}n$. %, where $\sigma=0$ if $n$ is even and $\sigma=1$ if $n$ is odd. 
Define {$V_{\rm bad}:=\bigcup_{k\in [2]}(X^k_A\cup X^k_B)\cup (X\setminus X')$}. Obviously, $|V_{\rm bad}|< (\sqrt{\epsilon}+\frac{1}{2}\sqrt{\delta})n$. Moreover, each vertex in $Y\setminus V_{\rm bad}$ lies in $Y_i$ for at least $(1-13\sqrt{\delta})|\hat{\mathcal{C}}|$ colors $i\in \hat{\mathcal{C}}$.

%The subsequent result can be used to connect two disjoint  short rainbow directed paths into a single short rainbow directed path.

Next, based on the sizes of 
$|\mathcal{C}_1|$ and $|\mathcal{C}_2|$, we distinguish our proof into the following three cases. 

\medskip
{\bf Case 1}. $|\mathcal{C}_1|< \eta n$. 
\medskip

In this case, each vertex in $Y\setminus (X\cup X_Y^2)$ belongs to $Y_i$ for at least $(1-13\sqrt{\delta})|\mathcal{C}_2|$ colors $i\in \mathcal{C}_2$. Recall that $|B|-|A|=r$ with  $0\leq r\leq (8{\delta}+\sqrt{\epsilon})n$. The proof is divided into four steps.%+\sigma$, where $\sigma=0$ if $n$ is even and $\sigma=1$ if $n$ is odd. 

\medskip
{\bf Step 1. Balance the number of vertices in $A$ and $B$.}
\medskip

In this step, we claim that by moving vertices in $B\cap V_{\rm bad}$ or deleting rainbow directed paths inside $\mathcal{D}[B]$, we can make the number of remaining vertices in $B$ and $A$ differ by  $\sigma$, where $\sigma=1$ if $|A|+|B|$ is odd and $\sigma=0$ otherwise. Denote $s:=|V_{\rm bad}\cap B|$. If $|B|-|A|-\sigma\leq 2s$, then move $\frac{|B|-|A|-\sigma}{2}$ vertices in $V_{\rm bad}\cap B$ to $A$, and our desired result holds. Hence, it suffices to consider $|B|-|A|-\sigma>2s$. Now, we move all vertices in $V_{\rm bad}\cap B$ to $A$. Hence, $B\cap V_{\rm bad}=\emptyset$. Assume that $\{Q_1,\ldots,Q_t\}$ is a set of disjoint maximal rainbow directed paths in $\mathcal{D}[B]$. 

On the one hand, suppose that  $|E(Q_1)|+\cdots+|E(Q_{t})|\geq |B|-|A|-\sigma$. Then there must exist a set of disjoint rainbow directed paths, say $\{Q_1',\ldots,Q_{t'}'\}$, such that  $|E(Q_1')|+\cdots+|E(Q_{t'}')|= |B|-|A|-\sigma$. That is,  $|A|-t'=|B|-(|V(Q_1')|+\cdots+|V(Q_{t'}')|)-\sigma$. Using  Lemma~\ref{conn} (ii), we connect $Q_1',\ldots,Q_{t'}'$ into a single rainbow directed path $P^1$ with length at most $27\delta n$, starting at $A$ and ending at $B$. (In fact, by using Lemma~\ref{conn} (ii), we only get a rainbow path with two endpoints in $B$. But we usually want to find a rainbow directed path with endpoints in different parts. Hence one may extend it by using an unused color of $\mathcal{C}_2$ and an unused vertex in $A$). Therefore, $|A\setminus V(P^1)|=|B\setminus V(P^1)|-\sigma$, as desired. 

On the other hand, suppose that $|E(Q_1)|+\cdots+|E(Q_{t})|< |B|-|A|-\sigma= r-2s-\sigma$. In this case, we move all vertices of $Q_1,\ldots,Q_t$ to $A$.  
Let $\mathcal{C}':=\mathcal{C}\setminus {\rm col}(\bigcup_{i\in [t]}Q_i)$. 
Hence \allowdisplaybreaks
\begin{align*}
    |A|
    &\leq \frac{n-|X'|-r}{2}+s+2(r-2s-\sigma)\leq \frac{n}{2}+18{\delta}n,  \\
|B|&\geq \frac{n-|X'|+r}{2}-s-2(r-2s-\sigma)\geq \frac{n}{2}-18{\delta}n,\\
|\mathcal{C}'|&=|\mathcal{C}|-(|E(Q_1)|+\cdots+|E(Q_{t})|)\geq (1-9\delta)n.
\end{align*}
%\bigcup_

By the maximality of $\{Q_1,\ldots,Q_t\}$, one has  $D_i[B]=\emptyset$ for all  $i\in \mathcal{C}'$. It follows from  $\delta^0(\mathcal{D})\geq \left\lceil\frac{n}{2}\right\rceil$ that $\left\lceil\frac{n}{2}\right\rceil\leq|A|+|X'|\leq \frac{n}{2}+19{\delta}n$. Now, by Lemma \ref{Y-large}, % (which is presented in the final of this subsection),  
we  obtain that $\mathcal{D}$ contains a transversal directed Hamilton cycle, a contradiction.
 
In this step, if we only move vertices in $V_{\rm bad}$ from $B$ to $A$, then let $P^1$ be a null digraph (i.e., no vertices); otherwise, assume that $P^1$ is a rainbow directed path obtained in the above, whose length is at most $27\delta n$ and endpoints are  $u_1\in A\setminus V_{\rm bad},\,v_1\in B\setminus V_{\rm bad}$. Hence $|A\setminus V(P^1)|=|B\setminus V(P^1)|-\sigma$. 

\medskip
{\bf Step 2. Construct a sequence of disjoint rainbow  directed $P_3$ copies such that all centers of them are exactly all vertices in $V_{\rm bad}\cup X'$.}
\medskip

Recall that $|X_A^2\cup X_B^2|\leq |X_A\cup X_B|<\frac{1}{2}\sqrt{\delta}n$.  Let $v\in X_Y^2$. If $v\in Z$, then $v$ is moved from $B$ to $A$ in Step 1, so $v\in X_B$. Therefore, $v\in A_i\cup B_i$ for at least $(1-6\sqrt{\delta})|\mathcal{C}_2|$ colors $i\in \mathcal{C}_2$ and $v\not\in B_i$ for at least $10\sqrt{\delta}|\mathcal{C}_2|$ colors $i\in \mathcal{C}_2$. Hence, $v\in A_i$ for at least $4\sqrt{\delta}|\mathcal{C}_2|-27\delta n>4|V_{\rm bad}\cup X'|$ colors $i\in \mathcal{C}_2\setminus {\rm col}(P^1)$. If $v\in Y$, then $v$ belongs to $Y_i$ for at least $(3\sqrt{\delta}-27\delta)n>4|V_{\rm bad}\cup X'|$ colors $i\in \mathcal{C}_2\setminus {\rm col}(P^1)$. Hence, using colors in $\mathcal{C}_2\setminus {\rm col}(P^1)$, we can greedily choose 
$|X_A^2\cup X_B^2|$  disjoint rainbow directed $P_3$ copies, such that their centers correspond to all vertices in $X_A^2\cup X_B^2$. Moreover, if $v\in (X_A^2\cup X_B^2)\cap Y$, then the rainbow directed  $P_3$ with center $v$ has endpoints in $Z\setminus (V_{\rm bad}\cup X'\cup V(P^1))$. 

For vertices in $X'$, we know from its definition that by avoiding vertices and colors in $P^1$ and all rainbow directed paths chosen for vertices in  $X_A^2\cup X_B^2$, there exists a set of disjoint  rainbow directed $P_3$ copies with centers in $X'$ and endpoints in different parts. 

By using Lemma~\ref{conn} (i)-(ii), we connect $P^1$, all rainbow directed $P_3$ copies with centers in $X_A^2\cup X_B^2\cup X'$ (as obtained above) into a single rainbow directed path.  %with two endpoints in different parts by Lemma~\ref{conn} (ii). Then connect the two resulting rainbow directed paths, all rainbow directed paths with centers in $X$ and $P^1$ 
This results in a rainbow directed path $P^2=u_1\ldots v_2$ with $v_2\in B\setminus V_{\rm bad}$ and length at most 
$$
|E(P^1)|+4|X_A^2\cup X_B^2|+5|X'|+6\leq 27\delta n+2\sqrt{\delta}n+5\sqrt{\epsilon}n+6.
$$

For vertices in $X\setminus X'$, we consider the following claim. Notice that the position (in $A$ or $B$) of vertices in $X\setminus X'$ does not affect our proof, so if $v\in X$ is a vertex moved from $B$ to $A$, then we only need consider $v\in A$. By Lemma~\ref{claim5}, for each $x_i\in (X\setminus X')\cap Y$, there exists a  rainbow directed path $Q_{x_i}=x_i^1x_ix_i^2$ with ${\rm col}(Q_{x_i})=\{c_i^1,c_i^2\}\subseteq \mathcal{C}_2$ and $x_i^1,x_i^2\in Z\setminus V_{\rm bad}$. 
Moreover, each color or endpoint in $Q_{x_i}$ can be replaced by an unused color or vertex (if needed). Using Lemma~\ref{conn} (i)-(ii), one may connect $P^2$ and all rainbow directed paths in $\{Q_{x_i}:x_i\in X\setminus X'\}$ into a single rainbow directed path $P^3$ with endpoints $u_1\in A\setminus V_{\rm bad}$ and $v_3\in B\setminus V_{\rm bad}$, whose length is at most $27\delta n+2\sqrt{\delta}n+5\sqrt{\epsilon}n+10.$
%and all rainbow directed $P_3$ with centers in $X\setminus X'$ by Lemma~\ref{conn}~(i)-(ii) into a single rainbow path with endpoints in different parts. Then using Lemma~\ref{conn} (i)-(ii) to connect $P^2$ and the resulting two rainbow directed paths into a single rainbow directed path $P^3$ with endpoints $u_1\in A\setminus V_{\rm bad}$ and $v_3\in B\setminus V_{\rm bad}$, whose length is at most $|E(P^2)|+5|X\setminus X'|+4\leq 27\delta n+2\sqrt{\delta}n+5\sqrt{\epsilon}n+8$.

\medskip
{\bf Step 3. Select a rainbow matching  with colors in $\mathcal{C}_{\rm bad}\setminus {\rm col}(P^3)$.} 
\medskip% inside $\{D_i[A\setminus V(P^3),B\setminus V(P^3)]:i\in \mathcal{C}_{\rm bad}\setminus {\rm col}(P^3)\}$.}

Choose a maximum rainbow matching  in $\{D_i^{\pm}[A\setminus V(P^3),B\setminus V(P^3)]:i\in \mathcal{C}_{\rm bad}\setminus {\rm col}(P^3)\}$, say $\Tilde{M}$. If there exists an edge  with color $j$ in $\Tilde{M}$  such that  $D_j^{\pm}[A\setminus V(P^3\cup \Tilde{M}),B\setminus V(P^3\cup \Tilde{M})]$ contains no $2$-matching, then delete all such edges from $\Tilde{M}$ and denote the resulting rainbow directed  matching by $M$.  {Connecting} $P^3$ and all rainbow directed edges in $M$ by Lemma~\ref{conn} (i), we obtain a rainbow directed path $P^4$ with endpoints $u_1\in A\setminus V_{\rm bad}$ and $v_4\in B\setminus V_{\rm bad}$, whose length is at most $|E(P^3)|+16\sqrt{\delta} n\leq 18\sqrt{\delta}n+5\sqrt{\epsilon}n+27\delta n+10$. 

Denote $\mathcal{C}_{\rm bad}':=\mathcal{C}_{\rm bad}\setminus {\rm col}(P^4)$ and let  $j\in \mathcal{C}_{\rm bad}'$. Hence in $D_j$, all but at most one vertex in  $Y\setminus V(P^3\cup \Tilde{M})$ has at least $\frac{n}{2}-|V(P^3\cup \Tilde{M})|-1\geq (\frac{1}{2}-19\sqrt{\delta})n$ out-neighbors in $Y$. Therefore, $D_j$ is $(19\sqrt{\delta},{\rm EC}1)$-extremal for all $j\in \mathcal{C}_{\rm bad}'$.

\medskip
{\bf Step 4. Construct two rainbow directed paths inside $\mathcal{D}[A\setminus V(P^4)]$ and $\mathcal{D}[B\setminus V(P^4)]$ by using unused colors in $\mathcal{C}_1\cup \mathcal{C}_{\rm bad}'$}.
\medskip

It is straightforward to check that 
$$
 |B\setminus V(P^4)|-|A\setminus V(P^4)|=|B\setminus V(P^1)|-|A\setminus V(P^1)|=\sigma.
$$
By the construction of $P^4$, we know that if $w\in V(P^4)$ is a vertex not in $(V(P^1)\cap B)\cup V_{\rm bad}\cup X$, then it is possible to avoid $w$ when constructing $P^4$; if $i\in {\rm col}(P^4)$ is a color not in ${\rm col}(P^1-A)\cup {\rm col}(M)$, then it is possible to avoid $i$ when constructing $P^4$. In other words, we may  assume that $w\not\in V(P^4)$ unless $w\in (V(P^1)\cap B)\cup V_{\rm bad}\cup X'$, and $i\not\in {\rm col}(P^4)$ unless $i\in {\rm col}(P^1-A)\cup {\rm col}(M)$.

Furthermore, while selecting rainbow directed paths inside $\mathcal{D}[B]$ (in Step 1), one may give priority to  {using} colors from $\mathcal{C}_1\cup \mathcal{C}_{\rm bad}'$ and avoid vertices in $X'\cup V_{\rm bad}\cup V(M)$. This implies that either $\mathcal{C}_1\cup \mathcal{C}_{\rm bad}'\subseteq {\rm col}(P^1)$ or ${\rm col}(P^1-A)\subseteq \mathcal{C}_1\cup \mathcal{C}_{\rm bad}'$. Denote $\Tilde{\mathcal{C}}:=(\mathcal{C}_1\cup \mathcal{C}_{\rm bad}')\setminus {\rm col}(P^1)$. 

Notice that for each color $i\in \mathcal{C}_2\setminus {\rm col}(P^4)$, every vertex $v\in A\setminus V(P^4)$ has at least $(\frac{1}{2}-3\sqrt{\delta})n-|V(P^4)|-|B\triangle B_i|$ out-neighbors (resp. in-neighbors) in $B\setminus V(P^4)$ in $D_i$. % with $i\in \mathcal{C}_2\setminus {\rm col}(P^4)$. 
Hence,  
$$
\left|N_{i}^+(v)\cap N_{i}^-(v)\cap (B\setminus V(P^4))\right|\geq \left(\frac{1}{2}-45\sqrt{\delta}\right)n.
$$
Therefore,  %for each $i\in \mathcal{C}_2\setminus {\rm col}(P^4)$, 
by replacing every pair of arcs with the same endpoints and  opposite directions by an undirected edge, the digraph $D_i^{\pm}[A\setminus V(P^4),B\setminus V(P^4)]$ can be viewed as an undirected graph $G_i$ with vertex set $(A\cup B)\setminus V(P^4)$ and  minimum degree at least $(\frac{1}{2}-45\sqrt{\delta})n$. %Similarly, for each $i\in \Tilde{\mathcal{C}}$, each directed graph $D_i[A\setminus V(P^4)]\cup D_j[B\setminus V(P^4)]$ can be viewed as an undirected graph $G_i$, which is the union of two undirected graphs with  minimum degree at least $(\frac{1}{2}-42\sqrt{\delta})n$.

    The next claim employs Lemma~\ref{lemma4.1} to construct a transversal directed Hamilton cycle. 
\begin{claim}\label{parity}
Let $Q=x_1\ldots x_s$ be a rainbow directed path inside $\mathcal{D}$ with $s\leq 19\sqrt{\delta}n$, $V_{\rm bad}\subseteq V(Q)$, $(\mathcal{C}_1\cup \mathcal{C}_{\rm bad})\setminus \Tilde{\mathcal{C}}\subseteq {\rm col}(Q)$, $x_1\in A\setminus V_{\rm bad}$ and $x_s\in B\setminus V_{\rm bad}$. 
If $|\Tilde{\mathcal{C}}|-\big||B\setminus V(Q)|-|A\setminus V(Q)|\big|$ is a nonnegative even integer, then $\mathcal{D}$ has a transversal directed Hamilton cycle.
\end{claim}
\begin{claimproof}[Proof of Claim \ref{parity}]
    Using colors in $\Tilde{\mathcal{C}}$, one may greedily choose two disjoint  rainbow directed paths $P^5$ and $P^6$ in $\mathcal{D}[A\setminus V(Q)]$ and $\mathcal{D}[B\setminus V(Q)]$ with lengths  $$\frac{|\Tilde{\mathcal{C}}|-(|B\setminus V(Q)|-|A\setminus V(Q)|)}{2}\ \ \text{and}\ \ \frac{|\Tilde{\mathcal{C}}|+(|B\setminus V(Q)|-|A\setminus V(Q)|)}{2}$$ respectively, whose endpoints are $u_5,v_5\in A\setminus V_{\rm bad}$ and  $u_6,v_6\in B\setminus V_{\rm bad}$. (Note that if a rainbow path has length $0$, then it contains one vertex.) Applying Lemma~\ref{conn} (i)-(ii) to connect $Q,P^5$ and $P^6$ in turn, we get a rainbow directed  path $P$ with endpoints $x_1\in A\setminus V_{\rm bad}$ and $v_6\in B\setminus V_{\rm bad}$. Clearly, $|V(P)|\leq 19\sqrt{\delta}n+4\sqrt{\delta} n+\eta n+4$ and $|A\setminus V(P)|=|B\setminus V(P)|$. %Delete all used vertices and all color from $A\cup B$ and $\mathcal{C}$. Then $|B|-|A|=0$.

Notice that $x_1,v_6\not\in V_{\rm bad}$. Hence there exist $i_1,i_2\in \mathcal{C}_2\setminus {\rm col}(P)$ such that $x_1\in A_{i_1}$ and $v_6\in B_{i_2}$. Let  $W^*:=A\setminus V(P),\,T^*:=B\setminus V(P)$, $\mathcal{C}^*:=\mathcal{C}\setminus ({\rm col}(P)\cup \{i_1,i_2\})$, $W^-:=N_{G_{i_2}}(v_6)\cap W^*$ and $T^+:=N_{G_{i_1}}(x_1)\cap T^*$. It is routine to check that $|W^*|=|T^*|$, $|\mathcal{C}^*|=|W^*|+|T^*|-1$, 
\begin{align*}
    |W^-|&\geq d_{G_{i_2}}(v_6,A)-|V(P)|\geq d_{G_{i_2}}(v_6,A_{i_2})-|A_1\triangle A_{i_2}|-|X_A\setminus X_A^2|-|X'|-|V(P)|\\
    &\geq \left(\frac{1}{2}-45\sqrt{\delta}\right)n-2\delta n-4{\delta}n-\sqrt{\epsilon}n-|V(P)|> \left(\frac{1}{2}-2\eta\right)n,
\end{align*}
and similarly $|T^+|>\left(\frac{1}{2}-2\eta\right)n$.  Applying Lemma~\ref{lemma4.1} yields that $\{G_i[W^*,T^*]:i\in \mathcal{C}^*\}$ contains a transversal undirected path $P'$ starting at $v'\in W^-$ and ending   at $u'\in T^+$. Thus, $x_1Pv_6v'P'u'x_1$ is a transversal directed Hamilton cycle inside $\mathcal{D}$.  
\end{claimproof}

If $n$ and $|\Tilde{\mathcal{C}}|$ have the same parity, %  is even and $|\Tilde{\mathcal{C}}|$ is even, or $n$ is odd and $|\Tilde{\mathcal{C}}|$ is odd, 
then by Claim \ref{parity}, there exists a transversal directed Hamilton cycle inside $\mathcal{D}$, a contradiction. Hence, it suffices to consider that $n$ and $|\Tilde{\mathcal{C}}|$ have { different} parity. %$n$ is even and $|\Tilde{\mathcal{C}}|$ is odd, or $n$ is odd and $|\Tilde{\mathcal{C}}|$ is even. 
Therefore, either $|\Tilde{\mathcal{C}}|-(|B\setminus V(P^4)|-|A\setminus V(P^4)|)$ is a positive odd integer or $|\Tilde{\mathcal{C}}|=0$ and $n$ is odd.  %Delete vertices in $X'$ from $A\cup B$. 

Based on the definition of $X'$, we move vertices in $X'$  to $Y$ if it is covered by a rainbow directed path from $Z$ to $Y$ (in Step 2). The following claim  considers the local structure of   $D_i$. 
\begin{claim}\label{claim4.7}
\begin{enumerate}
   \item[{\rm (i)}] %If $|\Tilde{\mathcal{C}}|\geq 1$, then $E(D_{i}[A])=E(D_{i}^+[A\cap X',A])$ for all $i\in \mathcal{C}\setminus \Tilde{\mathcal{C}}$,
       If $|\Tilde{\mathcal{C}}|=0$, then $D_{i}[B\setminus (X'\cup V(P^1))]=\emptyset$ for all $i\in \mathcal{C}\setminus {\rm col}(P^1-A)$.
    \item[{\rm (ii)}] For all $i\in \Tilde{\mathcal{C}}$, $D_{i}[A\setminus  (V_{\rm bad}\cup X'),B]=\emptyset$ and $D_{i}[(B\setminus (V_{\rm bad}\cup X'\cup V(P^1)),A]=\emptyset$.
    %For all $i\in \Tilde{\mathcal{C}}$, $E(D_{i}[A,B])=E(D_{i}[A\cap (V_{\rm bad}\cup X'),B])$ and $E(D_{i}[B,A])=E(D_{i}[(B\cap (V_{\rm bad}\cup X'\cup V(P^1)),A])$ .
\end{enumerate}

\end{claim}
\begin{claimproof}[Proof of Claim \ref{claim4.7}]
(i)\ %Since $|\Tilde{C}|\geq 1$, one has ${\rm col}(P^1-A)\subseteq \mathcal{C}_1\cup \mathcal{C}_{\rm bad}'$. 
Suppose that there exists an $i_0\in \mathcal{C}\setminus {\rm col}(P^1-A)$ such that $D_{i_0}[B\setminus (X'\cup V(P^1))]\ne \emptyset$. Choose ${w_1w_2}\in E(D_{i_0}[B\setminus (X'\cup V(P^1))])$. %We prove that $w_1\in B\cap (X'\cup V(P^1))$. 
Recall that each of $w_1$ and $w_2$ can be chosen outside $V(P^4)$ unless it is in  $(V(P^1)\cap B)\cup V_{\rm bad}\cup X'$, and $i_0$ can be chosen outside ${\rm col}(P^4)$ unless $i_0\in {\rm col}(M)$. % If $i_0\in {\rm col}(P^1-A)\cup {\rm col}(M)$, then delete the edge with color $i_0$ from $M$. 
If $i_0\in {\rm col}(M)$, then let $\Tilde{P}^4$ be a rainbow directed path obtained by deleting the edge with color $i_0$ from $P^4$ and  connecting its two endpoints by Lemma~\ref{conn}~(i). % (if $i_0\in {\rm col}(M)$) or Lemma~\ref{conn} (ii) (if $i_0\in {\rm col}(P^1-A)$). 
It is routine to check that $|B\setminus V(\Tilde{P}^4)|-|A\setminus V(\Tilde{P}^4)|=|B\setminus V({P}^4)|-|A\setminus V({P}^4)|$. Let $P^4:=\Tilde{P}^4$ if $i_0\in {\rm col}(M)$, which
reduces our proof to the case where $i_0\not\in {\rm col}(P^4)$. We consider the following several cases.%In the following, we will construct a new rainbow directed path containing the edge $w_1w_2$ with color $i_0$ and satisfying all conditions in Claim \ref{parity}. 

\medskip
$\bullet$ Suppose $w_1,w_2\not\in V_{\rm bad}$. Then $w_1,w_2\not\in V(P^4)$. 
Connect $w_1$ with the endpoint $v_4$ of $P^4$ by Lemma~\ref{conn} (ii). Denote the resulting rainbow directed path by $Q_1$.
\medskip

$\bullet$ Suppose $w_1\not\in V_{\rm bad}$ and $w_2\in V_{\rm bad}$. Hence $w_1\not\in V(P^4)$ and there exists a rainbow subpath $u_1P^4w_2^1w_2^2w_2$ of $P^4$ with $w_2^1\in B\setminus (V_{\rm bad}\cup X')$, $w_2^2\in A\setminus (V_{\rm bad}\cup X')$. Replace $w_2^1w_2^2w_2$ with $w_1w_2$ in $P^4$, and connect $w_1$ with the vertex in $N_{P^4}^-(w_2^1)$ by Lemma~\ref{conn} (i). Let $Q_2$ be the resulting rainbow directed path. The case for $w_1\in V_{\rm bad}$ and $w_2\not\in V_{\rm bad}$ is similar. 
\medskip

$\bullet$ Suppose $w_1,w_2\in V_{\rm bad}$. % and $\in V_{\rm bad}\cup X'$. 
By rearranging, one may assume that  $w_1$ and $w_2$ are connected by a rainbow directed path $w_1w_1^1w_2^1w_2^2w_2$ in  $P^4$, where $w_1^1,w_2^2\in A\setminus (V_{\rm bad}\cup X')$ and $w_2^1\in B\setminus (V_{\rm bad}\cup X')$. Replace $w_1w_1^1w_2^1w_2^2w_2$ with $w_1w_2$ in $P^4$; and denote the resulting rainbow directed path by $Q_3$.
\medskip

%$\bullet$ $w_1\notin V_{\rm bad}\cup X',w_2\in A\cap X'$. Hence $w_1\not\in V(P^4)$. Then there exists a rainbow subpath $u_1P^4w_2^1w_2^2w_2$ of $P^4$ with $w_2^2\in A\setminus V_{\rm bad}$, $w_2^1\in B\setminus V_{\rm bad}$. Replace $w_2^1w_2^2w_2$ with $w_1w_2$ in $P^4$, and connect $w_1$ with the vertex in $N_{P^4}^-(w_2^1)$ (if exists) by Lemma~\ref{conn} (i). Let $Q_3$ be the resulting rainbow directed path.

%$\bullet$ $w_1\in V_{\rm bad}, w_2\in A\cap X'$. By rearranging, one may assume $w_1$ and $w_2$ is connected by a rainbow directed path $w_1w_1^1w_2^1w_2^2w_2$ in  $P^4$, where $w_1^1,w_2^2\in B\setminus V_{\rm bad}$ and $w_2^1\in A\setminus V_{\rm bad}$. Replace $w_1w_1^1w_2^1w_2^2w_2$ with $w_1w_2$ in $P^4$. Denote by $Q_4$ the resulting rainbow directed path.

In each of the above cases, we get a rainbow directed path $Q_i$ ($i\in [3]$) such that $|\Tilde{\mathcal{C}}|-\big||B\setminus V(Q_i)|-|A\setminus V(Q_i)|\big|$ is a nonnegative even integer.  It follows from Claim \ref{parity} that $\mathcal{D}$ contains  a transversal directed Hamilton cycle, a contradiction.

\medskip
(ii)\ We will only prove that $D_{i}[A\setminus  (V_{\rm bad}\cup X'),B]=\emptyset$ for all $i\in \Tilde{\mathcal{C}}$, the second statement can be proved by a similar discussion. Suppose that $D_{i_1}[A\setminus  (V_{\rm bad}\cup X'),B]\neq\emptyset$ for some $i_1\in \Tilde{\mathcal{C}}$. Choose ${z_1z_2}\in E(D_{i_1}[A\setminus  (V_{\rm bad}\cup X'),B])$. Hence  $z_1\notin V(P^4)$. Next, we consider several cases.%Since $|\Tilde{\mathcal{C}}|\geq 1$, one has  . 

%Suppose that $E(D_{i_1}[A,B])\neq \emptyset$ for some $i_1\in \Tilde{\mathcal{C}}$. Choose ${z_1z_2}\in E(D_{i_1}[A,B])$. Notice that ${\rm col}(P^1-A)\subseteq \mathcal{C}_1\cup \mathcal{C}_{\rm bad}'$. If $z_i\in V_{\rm bad}\cap (Y\setminus X')$ for $i\in [2]$, then $P^4$ can be written as $u_1P^4z_i^0z_i^1z_iz_i^2P^4v_4$, where $z_i^0\in Y\setminus V_{\rm bad}$ and $z_i^1,z_i^2\in Z\setminus V_{\rm bad}$; if $z_i\in Y\cap X'$ for $i\in [2]$, then $P^4$ can be written as $u_1P^4z_i^0z_i^1z_iz_i^2P^4v_4$, where $z_i^1\in Z\setminus V_{\rm bad}$ and $z_i^0,z_i^2\in Y\setminus V_{\rm bad}$; if $z_2\in V(P^1)\cap B$, then $P^4$ can be written as $u_1P^4z_2'z_2P^4v_4$, where $z_2'z_2\in E(P^1)\cap E(D_{i_2})$ for some $i_2\in \mathcal{C}$. Clearly,  $i_2\in \mathcal{C}_2$ if $z_2'\in A$ and $i_2\in (\mathcal{C}_1\cup \mathcal{C}_{\rm bad}')\setminus \Tilde{\mathcal{C}}$ if $z_2'\in B$. Notice that $V(P^1)\cap B\cap V_{\rm bad}=\emptyset$. Hence $z_2'\not\in V_{\rm bad}$. 

\medskip
$\bullet$ Suppose $z_2\not\in (V(P^1)\cap B)\cup V_{\rm bad}\cup X'$. Then $z_2\not\in V(P^4)$. Using Lemma~\ref{conn}~(i), connect  $z_2$ to the endpoint $u_1$ of $P^4$. Denote the resulting rainbow directed path by  $Q_1$, and set $\Tilde{\mathcal{C}}_1:=\Tilde{\mathcal{C}}\setminus \{i_1\}$. 

%$\bullet$ $z_1,z_2\not\in (V(P^1)\cap B)\cup V_{\rm bad}\cup X'$. Then $z_1,z_2\not\in V(P^4)$. Connect $z_2$ with the endpoint $u_1$ of $P^4$ by Lemma~\ref{conn} (i). Let  $Q_0$ be the resulting rainbow directed path and $\Tilde{\mathcal{C}}_0=\Tilde{\mathcal{C}}\setminus \{i_1\}$. 
\medskip

$\bullet$ Suppose $z_2\in  V(P^1)\cap B$. 
Then $P^4$ can be written as $u_1P^4z_2'z_2P^4v_4$, where $z_2'z_2\in E(P^1)\cap E(D_{i_2})$ for some $i_2\in \mathcal{C}$. Clearly,  $i_2\in \mathcal{C}_2$ if $z_2'\in A$, and $i_2\in (\mathcal{C}_1\cup \mathcal{C}_{\rm bad}')\setminus \Tilde{\mathcal{C}}$ if $z_2'\in B$ (since ${\rm col}(P^1-A)\subseteq \mathcal{C}_1\cup \mathcal{C}_{\rm bad}'$). Notice that $V(P^1)\cap B\cap V_{\rm bad}=\emptyset$. Hence $z_2'\not\in V_{\rm bad}$. 
Apply Lemma~\ref{conn}~(i)-(ii) to connect $u_1P^1z_2'$ and  $z_1z_2P^4v_4$ into a single rainbow directed path, which we denote by  $Q_2$. Let $\Tilde{\mathcal{C}}_2:=\Tilde{\mathcal{C}}\setminus \{i_1\}$ if $z_2'\in A$, and $\Tilde{\mathcal{C}}_2:=(\Tilde{\mathcal{C}}\setminus \{i_1\})\cup \{i_2\}$ if $z_2'\in B$.

\medskip
$\bullet$ Suppose $z_2\in V_{\rm bad}\cup X'$. Then $P^4$ can be written as $u_1P^4z_2^1z_2P^4v_4$, where $z_2^1\in A\setminus V_{\rm bad}$. By Lemma~\ref{conn} (ii), we connect $u_1P^4z_2^1$ with $z_1z_2z_2^2P^4v_4$ into a single rainbow directed path, say   $Q_3$. Set $\Tilde{\mathcal{C}}_3:=\Tilde{\mathcal{C}}\setminus \{i_1\}$. %Similar result holds for $z_1\in V_{\rm bad}$ and $z_2\not\in  (V(P^1)\cap B)\cup V_{\rm bad}\cup X'$. 

\medskip
In each of the above cases, we get a rainbow directed path $Q_i$ ($i\in [3]$) such that $|\Tilde{\mathcal{C}}|-\big||B\setminus V(Q_i)|-|A\setminus V(Q_i)|\big|$ is a nonnegative even integer.  It follows from Claim \ref{parity} that $\mathcal{D}$ contains  a transversal directed Hamilton cycle, a contradiction. 
\end{claimproof}

If $|\Tilde{\mathcal{C}}|=0$, then $n$ is odd. In view of Claim \ref{claim4.7} (i), we know $D_i[B\setminus (X'\cup V(P^1))]=\emptyset$ for all $i\in \mathcal{C}\setminus {\rm col}(P^1-A)$. Notice that $|B\setminus (X'\cup V(P^1))|\geq \frac{n}{2}-28\delta n$ and $|{\rm col}(P^1-A)|\leq 27\delta n$. Together with Lemma \ref{Y-large} (see the end of this subsection), we know that $\mathcal{D}$ contains a transversal directed Hamilton cycle, a contradiction. %  By a similar discussion as the proof of Claim \ref{claim4.7}, we can show that $E(D_i[B])=E(D_i[(V(P^1)\cup X')\cap B])$ for all $i\in \mathcal{C}$. Together with the minimum semi-degree condition, one has $|B|\geq |A|\geq \frac{n+1}{2}$, a contradiction. 

If $|\Tilde{\mathcal{C}}|\geq 1$, then based on Claim \ref{claim4.7} (ii) and the fact that $\delta^0(\mathcal{D})\geq \frac{n}{2}$, we obtain $|A|,|B|\geq \frac{n+1}{2}$, a contradiction.% Thus $V(P^1)=\emptyset$. We are to prove that $N_{i}(x,B)=\emptyset$ for any $i\in \mathcal{C}\setminus (\tilde{\mathcal{C}}\cup {\rm col}(P^4))$ and any vertex $x\in A\cap X'$. Choose $x_1\in X'\cap A$, suppose that there exists a vertex $x_2\in B$ and a color $c\in \mathcal{C}\setminus (\tilde{\mathcal{C}}\cup {\rm col}(P^4))$ such $x_2\in N_{i}^(x_1,B)$. If $x_i\in Y\cap X'$, then $P^4$ can be written as $u_1P^4x_i^0x_i^1x_ix_i^2P^4v_4$, where $x_i^1\in Z\setminus V_{\rm bad}$ and $x_i^0,x_i^2\in Y\setminus V_{\rm bad}$. 

\medskip
{\bf Case 2}. $|\mathcal{C}_2|< \eta n$.
\medskip

In this case, each vertex in $Y\setminus (X\cup X_Y^1)$ belongs to $Y_i$ for at least  $(1-13\sqrt{\delta})|\mathcal{C}_1|$ colors $i\in \mathcal{C}_1$. Following the argument from Step 2 of Case 1, we know that there are $|V_{\rm bad}|$ disjoint rainbow directed $P_3$ copies, each having its center in $V_{\rm bad}$
  and both endpoints in the same part. We now  consider vertices in $X'$. 
 Based on the definition of $X'$, we move a vertex in $X'$  to $Y$ whenever it is covered by a rainbow directed  $P_3$ from $Z$ to $Y$ (in Step 2 of Case 1). %Then each vertex in $Y\cap X'$ can be covered by a rainbow  directed $P_3$ from $Z$ to $Y$. 
 In order to connect those rainbow directed $P_3$ with centers in $X'$ into a single rainbow directed path, we balance the number of rainbow directed paths from $A$ to $B$ and those from $B$ to $A$. 

 \begin{claim}\label{rainbowX'}
    There exists a rainbow directed path $P^0$ with length at most $8\sqrt{\epsilon}n$ such that $X'\subseteq V(P^0)$, ${\rm col}(P^0)\subseteq \mathcal{C}_1$, and its endpoints lie outside 
$V_{\rm bad}\cup X'$. Furthermore, $P^0$ can be chosen to either start at $A$ and end at $B$,   or have both endpoints in the same part.
 \end{claim}
\begin{claimproof}[Proof of Claim \ref{rainbowX'}]
If $|X'|=0$, then by the minimum semi-degree condition, our result holds obviously.  
Assume therefore that  $|X'|\geq 1$ and, without loss of generality, that $|A\cap X'|-|B\cap X'|\geq 0$. We first select all disjoint non-trivial maximum rainbow directed paths, say $Q_1,\ldots,Q_t$, inside $\{D_i[A\cap X']:i\in \mathcal{C}_1\}$. Let $t':=|V(Q_1)|+\cdots+|V(Q_t)|$ and $Q_i:=x_i\ldots y_i$ for each $i\in [t]$. Next, we choose a maximum rainbow directed  matching, denoted by $\tilde{M}$, inside $\{D_i[A,B]:i\in \mathcal{C}_1\setminus {\rm col}(Q_1\cup Q_2\cup \ldots \cup Q_t)\}$ avoiding vertices in $\bigcup_{i\in [t]}V(Q_i-y_i)$. Partition $\tilde{M}$ {into}  $M_1\cup M_2$ such that $V(M_2)\cap A=V(\tilde{M})\cap A\cap X'$. Hence, each vertex in $V(M_2)\cap A$ or each $Q_i$ with $y_i\in V(M_2)\cap A$ can be covered by a rainbow directed path with length at most $|X'|+2$ and both endpoints in $B\setminus (V_{\rm bad}\cup X')$. Denote $\hat{X}:=(A\cap X')\setminus V\big(\bigcup_{i\in [t]}(Q_i-y_i)\cup \tilde{M}\big)$.

%$|E(M_1)|=0$ and $|E(\tilde{M})|\leq |X'|\leq \sqrt{\epsilon}n$. Therefore, each vertex in $A\setminus X'$ cannot be adjacent to any vertices in $B$. Otherwise, by the definition of $X'$, there will be a large rainbow matching , a contradiction.   $A\cap X'\subseteq V(\tilde{M})$. Therefore, the number of rainbow directed paths from $A$ to $B$ and those from $B$ to $A$. 

Assume that  $|\hat{X}|\geq 1$. We are to prove that $|E(M_1)|\geq |\hat{X}|+1$. 
Suppose for contradiction that  $|E(M_1)|\leq |\hat{X}|\leq \sqrt{\epsilon}n$. Together with the definition of $X'$, one has $B\cap X'\subseteq V(\tilde{M})\cap B$ and $\tilde{M}$ can be chosen such that each matching edge incident to a vertex in $B\cap X'$ is not incident to  vertices in $A\cap X'$. 

%Hence $|\hat{X}|+t=0$ if and only if $|B\cap X'|=0$, $|E(\tilde{M})|= |A\cap X'|$ and $A\cap X'=V(\tilde{M})\cap A$. 
%We first consider that $|\hat{X}|\geq 1$. 
Choose $u_1\in \hat{X}$, $u_2\in B\setminus V(\tilde{M})$ and two colors $c_1,c_2\in \mathcal{C}_1\setminus {\rm col}(\tilde{M}\cup Q_1\cup \ldots \cup Q_t)$. We consider $N_{c_1}^+(u_1)$ and $N_{c_2}^-(u_2)$ respectively. 
It is easy to see that $u_1$ has at most $t'-t$ out-neighbors inside $A\cap X'$ and $N_{c_1}^+(u_1,B)=N_{c_1}^+(u_1,B\cap V(\tilde{M}))$. Hence 
\begin{align}\label{eq:X'1}
  \left\lceil\frac{n}{2}\right\rceil\leq |N_{c_1}^+(u_1)|\leq |A\setminus X'|+(t'-t)+|N_{c_1}^+(u_1,B\cap V(\tilde{M}))|.
\end{align}
Similarly, we have 
\begin{align}\label{eq:X'2}
  \left\lceil\frac{n}{2}\right\rceil\leq |N_{c_2}^-(u_2)|\leq |N_{c_2}^-(u_2,A\cap V(\tilde{M}))|+|B|-1.
\end{align}
Therefore, 
$$
  |N_{c_1}^+(u_1,B\cap V(\tilde{M}))|+|N_{c_2}^-(u_2,A\cap V(\tilde{M}))|\geq |\hat{X}|+|E(M_2)|+1.
$$
Notice that, if there exists an edge $w_1w_2\in E(\tilde{M})$ such that  $u_1w_1\in E(D_{c_1})$ and $w_2u_2\in E(D_{c_2})$, then we will find a large rainbow matching, a contradiction. Hence $|N_{c_1}^+(u_1,B\cap V(\tilde{M}))|+|N_{c_2}^-(u_2,A\cap V(\tilde{M}))|\leq |E(\tilde{M})|$. It follows that $|E(\tilde{M})|\geq |\hat{X}|+|E(M_2)|+1$, which implies $|E(M_1)|\geq |\hat{X}|+1$. Thus, together with Lemma~\ref{conn} (iii), the desired rainbow directed path $P^0$ can be obtained. 

%from $A$ to $B$ and those from $B$ to $A$ can be balanced.%, a contradiction.

Now, we consider $|\hat{X}|=0$ and $|B\cap X'|=0$. If $|E(M_1)|\geq 1$, then we are done. Otherwise, by the maximality of $\tilde{M}$ and the definition of $X'$, we deduce that $D_i[A\setminus X',B]=\emptyset$. Using a similar argument as \eqref{eq:X'1} and \eqref{eq:X'2}, there is a vertex in $A\setminus (V_{\rm bad}\cup X')$ that is adjacent to some vertex in $\{x_1,\ldots,x_t\}\cup ((A\cap X')\setminus V(Q_1\cup \ldots\cup Q_t))$. %, $|A|$ it is routine to check that there exists a rainbow directed path starting at $A\setminus (V_{\rm bad}\cup X')$ and ending at $B\setminus (V_{\rm bad}\cup X')$. 
Together with Lemma~\ref{conn} (iii), the desired rainbow directed path can be obtained immediately.

Finally, we consider  $|\hat{X}|=0$ and $|B\cap X'|\geq 1$. In this case, the desired rainbow directed path also follows immediately from part (iii) of Lemma~\ref{conn}.  (Note that if there is a rainbow directed path with some endpoint in  $V_{\rm bad}$, then it can be extended by at most two vertices so that both of its endpoints remain in the corresponding part and are no longer in  $V_{\rm bad}$.)  
\end{claimproof}

By Lemma~\ref{claim5} and a similar discussion as Steps 1-3 in Case 1, the following hold. 
\begin{enumerate}
    \item[{\rm \bf (D1)}] In $\{D_i[A\setminus V(P^0)]\cup D_i[B\setminus V(P^0)]:i\in \mathcal{C}_1\setminus {\rm col}(P^0)\}$, there are  $|V_{\rm bad}\setminus V(P^0)|$ disjoint rainbow directed $P_3$ copies, each with its center in $V_{\rm bad}\setminus V(P^0)$ and both endpoints in the same part. Let $\mathbf{P}$ be the set of these rainbow directed $P_3$ copies.
    \item[{\rm \bf (D2)}] For colors in  $\mathcal{C}_{\rm bad}$, there exists a maximal rainbow matching, say $M$,  inside $\{D_i[A\setminus V(\mathbf{P}\cup P^0)]\cup D_i[B\setminus V(\mathbf{P}\cup P^0)]:i\in\mathcal{C}_{\rm bad}\}$ such that for each edge in $M[Y]$ with color $j$ we have $D_j[Y\setminus V(\mathbf{P}\cup P^0\cup {M})]$ contains a $2$-matching. 
    Denote $\mathcal{C}_{\rm bad}':=\mathcal{C}_{\rm bad}\setminus {\rm col}(M)$. Then $D_{j}$ is $(19\sqrt{\delta},{\rm EC}2)$-extremal for all $j\in \mathcal{C}_{\rm bad}'$.
    \item[{\rm \bf (D3)}] Let $q:=|\mathcal{C}_2\cup \mathcal{C}_{\rm bad}'|$. In the digraph collection $\{D_i^{\pm}[A,B]:i\in \mathcal{C}_2\cup \mathcal{C}_{\rm bad}')\}$, there exists a rainbow directed path $Q$ of length $q-1$ or $q$ that avoids vertices in $\mathbf{P}\cup P^0\cup M$.  % matching $M'=\{{u_1v_1},{u_2v_2},\ldots,{u_qv_q}\}$. 
   \item[{\rm \bf (D4)}]  We can assume that a vertex $v\notin V(\mathbf{P}\cup M\cup Q)$ unless $v\in V_{\rm bad}$, and a color $c\notin {\rm col}(\mathbf{P})$ if  $c\in \mathcal{C}_1$. 
\end{enumerate}

Next, we give a claim that will be used later to obtain transversal directed Hamilton cycles. 
\begin{claim}\label{path}
    Assume that $Q_1=z_1\ldots z_s$  and $Q_2=z_t'\ldots z_1'$ are  two disjoint rainbow directed paths inside $\mathcal{D}$ such that $z_1,z_1'\in A\setminus V_{\rm bad}$, $z_s,z_t'\in B\setminus V_{\rm bad}$ and $s+t<4\eta n$. If $V_{\rm bad}\subseteq V(Q_1\cup Q_2)$ and $\mathcal{C}_2\cup \mathcal{C}_{\rm bad}\subseteq {\rm col}(Q_1\cup Q_2)$, then $\mathcal{D}$ contains a transversal directed Hamilton cycle. 
\end{claim}
\begin{claimproof}[Proof of Claim \ref{path}]
Note that 
$z_1,z_1'\in A\setminus V_{\rm bad}$ and $z_s,z_t'\in B\setminus V_{\rm bad}$. There exist $i_1,i_2,i_3,i_4\in \mathcal{C}_1\setminus {\rm col}(Q_1\cup Q_2)$ such that $z_1\in A_{i_1}$, $z_1'\in A_{i_2}$, $z_s\in B_{i_3}$ and $z_t'\in B_{i_4}$. Split $\mathcal{C}_1\setminus ({\rm col}(Q_1\cup Q_2)\cup \{i_1,i_2,i_3,i_4\})$ into two parts $\mathcal{C}_a\cup \mathcal{C}_b$, where $|\mathcal{C}_a|=|A\setminus V(Q_1\cup Q_2)|-1$ and $|\mathcal{C}_b|=|B\setminus V(Q_1\cup Q_2)|-1$. Let $I_1\cup I_2$ and $J_1\cup J_2$ be equitable partitions of $A\setminus V(Q_1\cup Q_2)$ and $B\setminus V(Q_1\cup Q_2)$ with $|I_1|\geq |I_2|$ and $|J_1|\geq |J_2|$, respectively. Define
$$
\mathcal{D}_a:=\{D_i^{\pm}[I_1,I_2]:i\in \mathcal{C}_a\}\ \ \text{and}\ \ \mathcal{D}_b:=\{D_i^{\pm}[J_1,J_2]:i\in \mathcal{C}_b\}.
$$
Notice that 
$$
|Y\setminus V(Q_1\cup Q_2)|\geq \frac{n-|X'|-r}{2}-|V(Q_1\cup Q_2)|>\frac{n}{2}-5\eta n,
$$
and for $i\in \mathcal{C}_a\cup \{i_1,i_2\}$, $j\in [2]$ and $z\in I_{3-j}\cup \{z_1,z_1'\}$, we have 
\begin{align*}
|N_{{i}}^{+}(z,I_j)|&\geq d_{i}^{+}(z,A\setminus V(Q_1\cup Q_2))-|I_{3-j}|\\
&\geq d_{i}^+(z,A_{i})-|A_1\triangle A_{i}|-|X_A\setminus X_A^1|-|X'|-|V(Q_1\cup Q_2)|-|I_{3-j}|\\
&\geq \left(\frac{1}{2}-3\sqrt{\delta}\right)n-\left(2\delta+4\delta+4\eta+\sqrt{\epsilon}\right)n-\left\lceil\frac{n+r}{4}\right\rceil >\frac{n}{4}-5\eta n.
\end{align*}
Similarly, we have $|N_{{i}}^-(z,I_j)|\geq\frac{n}{4}-5\eta n$. This  implies that
$$
  |N_{{i}}^{+}(z,I_j)\cap N_{{i}}^-(z,I_j)|\geq 2\left(\frac{n}{4}-5\eta n\right)-|I_j|\geq \frac{n}{4}-10\eta n. 
$$
Hence, for each color  $i\in \mathcal{\mathcal{C}}_a\cup \{i_1,i_2\}$, one may view $D_i^{\pm}[I_1,I_2]$ as an undirected graph $G_i[I_1,I_2]$ with minimum degree at least $\frac{n}{4}-10\eta n$, where two vertices are adjacent in $G_i[I_1,I_2]$ if and only if there are two opposite arcs between them. By Lemma~\ref{lemma4.1}, there exists a transversal undirected path $Q_a$ inside $\mathcal{D}_a$ with two endpoints $x_a\in N_{G_{i_1}}(z_1,I_1)$ and $y_a\in N_{G_{i_2}}(z_1',I)$ (here $y_a\in I_2$ if and only if $|I_1|=|I_2|$). 
%Similarly, we have  $|N_{G_{i_2}}(z_1',I_j)|,|N_{G_{i_3}}(z_s,J_j)|,|N_{G_{i_4}}(z_t',J_j)|>\frac{n}{4}-10\eta n$ for each $j\in  [2]$. 
%The similar form of second inequality also holds for $z_s,z_1',z_t'$. 

%By Lemma~\ref{lemma4.1}, there exists a transversal undirected path $Q_a$ inside $\mathcal{D}_a$ with two endpoints $x_a\in N_{G_{i_1}}(z_1,I_1)$ and $y_a\in N_{G_{i_2}}(z_1',I)$; and
Similarly, there exists a transversal undirected path $Q_b$ inside $\mathcal{D}_b$ with two endpoints $x_b\in N_{G_{i_3}}(z_s,J_1)$ and $y_b\in N_{G_{i_4}}(z_t',J)$ (here $y_b\in J_2$ if and only if $|J_1|=|J_2|$). Thus, $\mathcal{D}$ contains a transversal directed Hamilton cycle $z_1Q_1z_sx_bQ_by_bz_t'Q_2z_1'y_aQ_ax_az_1$.
\end{claimproof}

Now, we proceed by considering the value of $q$.

\medskip
{\bf Subcase 2.1.} $q$ is even  and $q\geq 2$. 
\medskip

Using Lemma~\ref{conn} (iii), we connect $P^0$ (with endpoints in the same part), all rainbow directed paths in $\mathbf{P},$ all edges in $M$ and $Q$ (from $A$ to $B$ with length $q-1$) into a single rainbow directed path $P$, whose 
length is at most $(8\sqrt{\epsilon}+3\sqrt{\delta}+16\sqrt{\delta} +\eta )n$ and  endpoints are $x\in A\setminus (V_{\rm bad}\cup X')$, $y\in B\setminus (V_{\rm bad}\cup X')$. Choose $c\in (\mathcal{C}_2\cup \mathcal{C}_{\rm bad}')\setminus {\rm col}(P)$ and let $u_qv_q$ be an edge in  $D_c[B\setminus V(P),A\setminus V(P)]$. %, where $c\in (\mathcal{C}_2\cup \mathcal{C}_{\rm bad}')\setminus {\rm col}(P)$. 
Consequently, $P$ and $u_qv_q$ are two disjoint rainbow directed paths satisfying all conditions in Claim~\ref{path}. Thus,  $\mathcal{D}$ contains a transversal directed Hamilton cycle, a contradiction. 
\medskip

{\bf Subcase 2.2.}  $q$ is odd.  
\medskip

Using Lemma~\ref{conn} (iii), we  connect all rainbow paths in  $\mathbf{P},$ all edges in $M$ and $Q$ (from 
$B$ to $A$ with length $q$) into a single rainbow directed path $P$, whose 
length is at most $(3\sqrt{\delta}+16\sqrt{\delta} +\eta )n$ and endpoints are $x\in B\setminus V_{\rm bad}$, $y\in A\setminus V_{\rm bad}$. Therefore, $P$ and $P^0$ (from $A$ to $B$) are two disjoint rainbow paths satisfying all conditions in Claim  \ref{path}. It follows that  $\mathcal{D}$ contains a transversal directed Hamilton cycle, a contradiction. 

\medskip

{\bf Subcase 2.3.}  $q=0$. 
\medskip

If $|X'|\geq 2$, then choose $u\in A\cap X'$ if $|A\cap X'|>|B\cap X'|$, and choose $u\in B\cap X'$ if $|A\cap X'|=|B\cap X'|$. Delete $u$ from $X'$ when constructing $P^0$. %Using Lemma~\ref{conn}, one may connect the resulting two rainbow directed paths into a single rainbow directed path $P^2$, which is starting at $A$ and ending   at $B$. 
%Choose $e\in E(\tilde{M})$ and delete the vertices of $e$ from $P_0$. Then connect the resulting rainbow directed paths into a single rainbow directed path $P^2$ by Lemma~\ref{conn}. Notice that $P^2$ can be chosen such that its two endpoints are in the  part. 
Based on the definition of $X'$, there exists a rainbow directed copy of $P_3$ from $B$ to $A$  with center $u$, whose vertices are disjoint from $P^0$.  %Furthermore, there exists a rainbow directed $P^e$ from $A$ to $B$ with $e\subseteq E(P^e)$ and its two endpoints are not in $V_{\rm bad}\cup X'$. 
Applying Lemma~\ref{conn} (iii) again to connect $P^0$ (from $A$ to $B$), all rainbow directed paths in $\mathbf{P}, M$ into a single rainbow path $P$, whose length is at most $(8\sqrt{\epsilon}+3\sqrt{\delta}+16\sqrt{\delta} +\eta )n$, with   endpoints in different parts. % are $x\in A\setminus V_{\rm bad}$, $y\in B\setminus V_{\rm bad}$. 
Thus, $P$ and $P_x$ are two disjoint rainbow paths satisfying all conditions in Claim~\ref{path}. Therefore,  $\mathcal{D}$ contains a transversal directed Hamilton cycle, a contradiction. 

If $|X'|\leq 1$, then together with the condition that $\delta^0(\mathcal{D})\geq \frac{n}{2}$, it is easy to find two disjoint rainbow edges with opposite directions between $A$ and $B$. By Claim  \ref{path} and a similar discussion as above, we know that   $\mathcal{D}$ contains a transversal directed Hamilton cycle, a contradiction. 
% a matching edge $e$ in $\tilde{M}$ and delete . Delete the rainbow path containing $x$, say $P_x$, and the rainbow matching edge from $P^0$ (endpoints in the same part), and denoted the resulting path by $P^2$. Then  $P^2$ has endpoints in the same part. Using Lemma~\ref{conn} (iii) to connect all rainbow paths in $e$, $P^2$, $\mathbf{P},M$ and $P_y$ into a single rainbow path $P$, whose length is at most $(4\sqrt{\epsilon}+2\sqrt{\delta}+9\delta +3\eta )n$ and  endpoints are $x\in A\setminus V_{\rm bad}$, $y\in B\setminus V_{\rm bad}$. Thus, $P$ and $P_x$ are two disjoint rainbow paths satisfying all conditions in Claim  \ref{path}. Thus,  $\mathcal{D}$ contains a transversal directed Hamilton cycle, a contradiction. %Hence $B\cap X'=\emptyset$. Hence each vertex in $A\cap X'$ is covered by a rainbow directed path with endpoints in $A\setminus V_{\rm bad}$. 

%Recall that each vertex in $B$ is adjacent to at least one vertex in $A$ for every $D_i$. By a similar discussion as Subcase 2.2, we know that $yx$ can be extended to a rainbow directed path with head in $A\setminus V_{\rm bad}$ and tail $y$. Together with Claim  \ref{path},  $\mathcal{D}$ contains a transversal directed Hamilton cycle, a contradiction. 

\medskip
{\bf Case 3.} $|\mathcal{C}_1|\geq \eta n$ and $|\mathcal{C}_2|\geq \eta n$.
\medskip

In this case, each vertex in $Y\setminus V_{\rm bad}$ lies in $Y_{i}$ for at least $(1-13\sqrt{\delta})|{\mathcal{C}_1}\cup \mathcal{C}_2|$ colors $i\in \mathcal{C}_1\cup {\mathcal{C}_2}$. 
Denote $\bigcup_{k\in[2]}(X_A^k\cup X_B^k)=:\{v_1,v_2,\ldots,v_{t}\}$. Notice that $t<\frac{1}{2}\sqrt{\delta}n$. 
By the definition of $X_Y^k$, we obtain that for each $i\in [t]$, there exist $c_{4i-3},c_{4i-2},c_{4i-1},c_{4i}\in (\mathcal{C}_1\cup \mathcal{C}_2)\setminus \{c_1,c_2,\ldots,c_{4i-5},c_{4i-4}\}$ and $v_i^1,v_i^2,v_i^3,v_i^4\in (A\cup B)\setminus (V_{\rm bad}\cup \{v_{\ell}^1,v_{\ell}^2,v_{\ell}^3,v_{\ell}^4:\ell\in [i-1]\})$ such that 
\begin{itemize}
    \item if $v_i\in X_Y^1$, then $c_{4i-4+j}\in \mathcal{C}_1$, $v_i^j\in Y\cap Y_{c_{4i-4+j}}$,  ${v_i^jv_i}\in E(D_{c_{4i-4+j}})$ for all $j\in [2]$ and ${v_iv_i^j}\in E(D_{c_{4i-4+j}})$ for all $j\in \{3,4\}$,
    \item if $v_i\in X_Y^2$, then $c_{4i-4+j}\in \mathcal{C}_2$, $v_i^j\in Z\cap Z_{c_{4i-4+j}}$,  ${v_i^jv_i}\in E(D_{c_{4i-4+j}})$ for all $j\in [2]$ and ${v_iv_i^j}\in E(D_{c_{4i-4+j}})$ for all $j\in \{3,4\}$.
\end{itemize}
Therefore, there exists a set of disjoint rainbow directed paths   $\mathbf{P}^0=\{v_i^1v_iv_i^2:i\in [t]\}$ with ${\rm col}(v_i^1v_iv_i^2)=\{c_{4i-3},c_{4i-2}\}$ for each $i\in [t]$. 

Based on the definition of $X'$, we move a vertex from  $X'$  to $Y$ whenever it is covered by a rainbow directed $P_3$ from $Z$ to $Y$ (in Step 2 of Case 1).  Avoiding all vertices and colors appearing in $\mathbf{P}^0$, we then obtain a family of disjoint rainbow directed paths $\mathbf{P^1}=\{v^1vv^2:v\in X'\}$, where for each vertex $v\in Y\cap X'$ we have  $v^1\in Z\setminus (V_{\rm bad}\cup X')$ and  $v^2\in Y\setminus (V_{\rm bad}\cup X')$. Furthermore, for every such path at least one of the colors of ${v^1v}$ and ${vv^2}$ can be chosen from both $\mathcal{C}_1$ and $\mathcal{C}_2$. Without loss of generality, assume that $|A|\leq |B|$. 

Using colors in $\mathcal{C}_1\cup \mathcal{C}_2$ and applying  Lemma~\ref{conn} (i)-(iii), we connect all rainbow directed paths in $\mathbf{P}^0\cup \mathbf{P}^1$ into a single rainbow directed path $P^1$ with endpoints  $u_1\in A\setminus V_{\rm bad}$ and $w_1\in B\setminus V_{\rm bad}$, whose length is at most $4t+4|X'|\leq 2\sqrt{\delta}n+4|X'|$.

%By a similar discussion as Lemma~\ref{claim5}, we obtain the following claim. 
%\begin{claim}\label{claim5}
%    Assume $X\setminus X'=\{x_1,\ldots,x_s\}$ with $s\leq \sqrt{\epsilon}n$. Then for each $i\in [s]$, there exist $c_i^1,c_i^2,c_i^3,c_i^4\in (\mathcal{C}_1\cup \mathcal{C}_2)\setminus ({\rm col}(P^1)\cup \{c_{\ell}^1,c_{\ell}^2,c_{\ell}^3,c_{\ell}^4:{\ell}\in [i-1]\})$ and $x_i^1,\,x_i^2,\,x_i^3,\,x_i^4\in V\setminus (V(P^1)\cup  \{x_{\ell}^1,\,x_{\ell}^2,\,x_{\ell}^3,\,x_{\ell}^4:{\ell}\in [i-1]\})$  such that: if $x_i\in Y$, then 
%\begin{itemize}
%    \item ${x_i^jx_i}\in E(D_{c_{i}^j})$ for $j\in [2]$, and ${x_ix_i^j}\in E(D_{c_{i}^j})$ for  $j\in \{3,4\}$;
%    \item either $x_i^j\in Y$ and $c_i^j\in \mathcal{C}_1$ for all $j\in [4]$, or $x_i^j\in Z$ and $c_i^j\in \mathcal{C}_2$ for all $j\in [4]$.
%\end{itemize}
%\end{claim}
By Lemma~\ref{claim5}, there exists a set of disjoint  rainbow directed paths  $\mathbf{P}^2=\{x_i^1x_ix_i^2:i\in [s]\}$,  each of which has colors $c_i^1,c_i^2$ for $i\in [s]$. Applying Lemma~\ref{conn}~(i)-(ii), we connect all these rainbow directed paths together with $P^1$ into a single rainbow directed path $P^2$, which has endpoints   $u_1\in A\setminus V_{\rm bad}$ and $w_2\in B\setminus V_{\rm bad}$ and length  at most $|E(P^1)|+4|X\setminus X'|+2\leq (2\sqrt{\delta}+4\sqrt{\epsilon})n+2$.

Choose a maximal rainbow {matching,} say $M$,  inside $\{D_i[A\setminus V(P^2)]\cup D_i[B\setminus V(P^2)]:i\in\mathcal{C}_{\rm bad}\}$ such that for each edge in $M[Y]$ with color $j$, we have $D_j[Y\setminus V(P^2\cup M)]$ contains a $2$-matching. 
Denote $\mathcal{C}_{\rm bad}':=\mathcal{C}_{\rm bad}\setminus {\rm col}(M)$. Note that $D_{j}$ is $(19\sqrt{\delta},{\rm EC}2)$-extremal for all $j\in \mathcal{C}_{\rm bad}'$.  Hence, $\{D_i[A\setminus V(P^2\cup M),B\setminus V(P^2\cup M)]:i\in \mathcal{C}_{\rm bad}'\}$ contains a transversal matching, say $M'$. By using Lemma~\ref{conn} (i)-(iii) to connect $P^2$ and all rainbow edges in $M\cup M'$, we obtain a  rainbow directed path $P^3$  with endpoints  $u_1\in A\setminus V_{\rm bad}$ and $w_3\in B\setminus V_{\rm bad}$, whose length is at most $|E(P^2)|+16\sqrt{\delta} n\leq (18\sqrt{\delta}+4\sqrt{\epsilon})n+2$. 

Since $u_1,w_3\notin V_{\rm bad}$, there exist $i_0,i_1\in \mathcal{C}_1\setminus {\rm col}(P^3)$ such that $u_1\in A_{i_0}$ and $w_3\in B_{i_1}$. By the construction of $P^3$, one may assume that $v\not\in V(P^3)$ unless $v\in V_{\rm bad}\cup X'$, and $c\not\in {\rm col}(P^3)$ unless $c\in \mathcal{C}_{\rm bad}$. Similar to the above two cases, we know that each digraph $D_i^{\pm}[A\setminus V(P^3), B\setminus V(P^3)]$ with $i\in \mathcal{C}_2$ can be viewed as an undirected graph with minimum degree at least $(\frac{1}{2}-45\sqrt{\delta})n$; each digraph $D_i[A\setminus V(P^3)]\cup D_i[B\setminus V(P^3)]$ with $i\in \mathcal{C}_1$ can be viewed as the union of two disjoint undirected graphs with minimum degree at least $(\frac{1}{2}-45\sqrt{\delta})n$. 
We first prove the following claim. %In what follows, we proceed by considering the parity of $|\mathcal{C}_2\setminus {\rm col}(P^3)|$. 

\begin{claim}\label{claim4.9}
    Assume that $Q_1=z_1\ldots z_k$  and $Q_2=z_t'\ldots z_1'$ are two disjoint rainbow directed paths inside $\mathcal{D}$ such that $z_1,z_1'\in A\setminus V_{\rm bad}$, $z_k,z_t'\in B\setminus V_{\rm bad}$ and each of them has length at most $19\sqrt{\delta} n$. If $V_{\rm bad}\subseteq V(Q_1\cup Q_2)$, $\mathcal{C}_{\rm bad}\subseteq {\rm col}(Q_1\cup Q_2)$ and $|\mathcal{C}_2\setminus {\rm col}(Q_1\cup Q_2)|$ is even, then $\mathcal{D}$ contains a transversal directed Hamilton cycle. 
\end{claim}
\begin{claimproof}[Proof of Claim \ref{claim4.9}]
Note that 
$z_1,z_1'\in A\setminus V_{\rm bad}$ and $z_k,z_t'\in B\setminus V_{\rm bad}$. There exist  $\ell_1,\ell_2,\ell_3\in \mathcal{C}_1\setminus {\rm col}(Q_1\cup Q_2)$ and $\ell_4\in \mathcal{C}_2\setminus {\rm col}(Q_1\cup Q_2)$ such that $z_1\in A_{\ell_1}$, $z_k\in B_{\ell_2}$, $z_1'\in A_{\ell_4}$ and $z_t'\in B_{\ell_3}$. 
%Then $x,z\not\in V(P^3)$ and there exists $i_3\in \mathcal{C}_2\setminus {\rm col}(P^3)$ and $i_4\in \mathcal{C}_1\setminus {\rm col}(P^3)$ such that $x\in A_{i_3}$ and $z\in B_{i_4}$. %Delete $x,z$ from $A\cup B$ and $c_3',c_4'$ from $\mathcal{C}$ respectively. 

Split $A\setminus V(Q_1\cup Q_2)$ into $I_1$ and $I_2$ such that $|I_1|=\frac{|\mathcal{C}_2\setminus {\rm col}(Q_1\cup Q_2)|}{2}$, and split $B\setminus V(Q_1\cup Q_2)$ into $J_1$ and $J_2$ such that $|J_1|=\frac{|\mathcal{C}_2\setminus {\rm col}(Q_1\cup Q_2)|}{2}$. Let $\mathcal{D}_1:=\{D_i
^{\pm}[I_1,J_1]:i\in \mathcal{C}_2\setminus ({\rm col}(Q_1\cup Q_2)\cup \{\ell_4\})\}$. By Lemma~\ref{lemma4.1}, $\mathcal{D}_1$ contains a transversal undirected path $Q$ with endpoints $z_1''\in N_{{\ell_4}}^+(z_1')\cap J_1$ and $y\in I_1$. Since $y\not\in V_{\rm bad}$, there exists $\ell_5\in \mathcal{C}_1\setminus ({\rm col}(Q_1\cup Q_2)\cup \{\ell_1,\ell_2,\ell_3\})$ such that $y\in A_{\ell_5}$. 

Assume that $|I_2|=k_1$ and $|J_2|=k_2$. 
We further split $I_2$ into $I_2'$ and $I_2''$ 
such that $|I_2'|=\lceil\frac{k_1}{2}\rceil$, and split $J_2$ into $J_2'$ and $J_2''$ 
such that $|J_2'|=\lceil\frac{k_2}{2}\rceil$. Split the color
set $\mathcal{C}_1\setminus ({\rm col}(Q_1\cup Q_2)\cup \{\ell_1,\ell_2,\ell_3,\ell_5\})$ into two subsets $\mathcal{C}_a$ and $\mathcal{C}_b$ such that $|\mathcal{C}_a|=k_1-1$ and $|\mathcal{C}_b|=k_2-1$. Let $$\mathcal{D}_a:=\{D_i^{\pm}[I_2',I_2'']:i\in \mathcal{C}_a\}\ \ \text{and}\ \ \mathcal{D}_b:=\{D_i^{\pm}[J_2',J_2'']:i\in \mathcal{C}_b\}.$$ %It is routine to check that $\mathbf{G}_1,\mathbf{G}_a$ and $\mathbf{G}_b$ are $(\delta',1-\delta')$-superregular.
Applying Lemma~\ref{lemma4.1} again yields that \begin{itemize}
    \item $\mathcal{D}_a$ contains  a transversal undirected path $Q_a$  with endpoints $x_a\in N_{{\ell_1}}^-(z_1)\cap I_2',\,y_a\in N_{{\ell_5}}^+(y)\cap I_2$; \item  $\mathcal{D}_b$ contains a transversal undirected path $Q_b$  with endpoints $x_b\in N_{{\ell_2}}^+(z_k)\cap J_2',\,y_b\in N_{{\ell_3}}^-(z_t')\cap J_2$.
\end{itemize} 
 Hence, $z_1Q_1z_kx_bQ_by_bz_t'Q_2z_1'z_1''Qyy_aQ_ax_az_1$ is a transversal directed Hamilton cycle in $\mathcal{D}$. 
\end{claimproof}

Assume that $|\mathcal{C}_2\setminus {\rm col}(P^3)|$ is odd. Notice that  there  exists an edge $z_1'z_2'\in E(D_{i_2}[A\setminus V(P^3),B\setminus V(P^3)])$ for some $i_2\in \mathcal{C}_2\setminus {\rm col}(P^3)$. This edge together with the path $P^3$ form two rainbow directed paths that satisfy all conditions in Claim \ref{claim4.9}. It follows that $\mathcal{D}$ contains a transversal directed  Hamilton cycle, a contradiction. Hence $|\mathcal{C}_2\setminus {\rm col}(P^3)|$ is even.

Since $\delta^0(\mathcal{D})\geq \frac{n}{2}$ and $|A|\leq |B|$, each vertex in $A$ has an in-neighbor in $B$ in each $D_i$ with $i\in \mathcal{C}_1\cup {\rm col}(M)$. Suppose that there exists an edge $y_2y_1\in E(D_{i_3}[B,A])$ for some $i_3\in \mathcal{C}_1\cup {\rm col}(M)$, where $y_1\in A\setminus (V(P^3)\cup V_{\rm bad})$ and $y_2\in B$. If $i_3\in {\rm col}(M)$, then assume ${uv}\in E(M[A])\cap E(D_{i_3})$. Applying Lemma~\ref{conn}~(i)-(ii), we successively  connect $P^2$, all rainbow edges in $M-uv$ and all rainbow edges in $M'$. This results {in} a rainbow directed path $\Tilde{P}^3$ with endpoints $\Tilde{u}_1\in A\setminus V_{\rm bad}$ and $\Tilde{w}_3\in B\setminus V_{\rm bad}$. Therefore, $P^3$ can be written as 
$uvv'\Tilde{u}_1\Tilde{P}^3\Tilde{w}_3$, where the colors of  ${vv'}$ and ${v'\Tilde{u}_1}$ are in $\mathcal{C}_2$. Since $|\mathcal{C}_2\setminus {\rm col}(P^3)|$ is even, one has $|\mathcal{C}_2\setminus {\rm col}(\Tilde{P}^3)|$ is even. Hence one may let $P^3:=\Tilde{P}^3$ if $i_3\in {\rm col}(M)$ and further assume  $i_3\not\in {\rm col}(P^3)$.

Notice that $y_2\not\in V(P^3)$ unless it is in $V_{\rm bad}\cup X'$. If $y_2\not\in V_{\rm bad}\cup X'$, then $P^3$ and $y_2y_1$ are two disjoint rainbow directed paths satisfying all conditions in Claim \ref{claim4.9}. Hence,  $\mathcal{D}$ contains a transversal Hamilton cycle, a contradiction. Therefore, $y_2\in V_{\rm bad}\cup X'$. %Without loss of generality, assume that $y_2\in V_{\rm bad}\cup X'$. 

Recall that $y_2^1y_2y_2^2$ is the rainbow directed path that covers $y_2$.  Assume ${y_2^1y_2}$ and ${y_2y_2^2}$ have colors $j_1$ and $j_2$ in $P^3$, respectively. Hence,  there exists a rainbow directed $\hat{P^3}$ starting at $A$ and ending   at $B$ such that $P^3$ is obtained by connecting $y_2^1y_2y_2^2$ and $\hat{P^3}$ if $y_2^1\in A$, and $P^3$ is obtained by connecting $\hat{P^3}$ and $y_2^1y_2y_2^2$ otherwise. %and  if $y_2^1\in A$ then $P^3$ can be obtained by connecting $y_2^1y_2y_2^2$ and a rainbow directed $\hat{P^3}$ with endpoints in different parts; if $y_2^1\in B$, then $P^3$ can be obtained by connecting $y_2^1y_2y_2^2$ and a rainbow directed $\hat{P^3}$ with endpoints in different parts. Therefore, $|\mathcal{C}_2\setminus {\rm col}(\tilde{P^3})|$ and $|\mathcal{C}_2\cap \{j_1,j_2\}|$ have different parity. 

If $y_2\in X'$, then by adjusting the selections of $y_2^1,y_2^2$ and $j_1,j_2$, one may get a rainbow directed path $P^3$ such that 
$|\mathcal{C}_2\setminus {\rm col}({P^3})|$ is odd. This situation therefore reduces to the case where $|\mathcal{C}_2\setminus {\rm col}({P^3})|$ is odd. 

If $y_2\in V_{\rm bad}$, then $|\mathcal{C}_2\setminus {\rm col}(\hat{P^3})|$ is even. Let 
$\hat{P}:=y_2^1y_2y_1$ if $y_2^1\in B$, and let $\hat{P}:=y'y_2^1y_2y_1$ with $y'y_2^1\in E(D_{c'})$ for unused color $c'\in \mathcal{C}_2$ and unused vertex $y'\in B$ if $y_2^1\in A$. Consequently, $\hat{P^3}$ and $\hat{P}$ are two rainbow directed paths  satisfying all conditions in Claim \ref{claim4.9}. Hence,  $\mathcal{D}$ contains a transversal directed Hamilton cycle, a contradiction.
\end{proof}

%\subsection{Almost all digraphs are extremal to EC1 or EC2.}

%\begin{theorem}\label{theorem-EC12}
% Let $\mathcal{D}=\{D_1,\ldots,D_n\}$ be a digraph collection on a common vertex set $V$ of size $n$ and  $\delta^0(\mathcal{D})\geq \frac{n}{2}$. Assume that $|\tilde{\mathcal{C}}_3|<\sqrt{\delta}n$. Then $\mathcal{D}$ contains a transversal directed Hamilton cycle. 
%\end{theorem}
\section{Proof of Theorem \ref{theorem-EC3}}

{{In this section we present the proof of Theorem~\ref{theorem-EC3}. Since the overall strategy parallels that of Theorem~\ref{theorem-EC12}, the argument itself is relatively short; therefore we do not separate out preliminary claims but instead give a complete, self-contained proof of Theorem \ref{theorem-EC3}.}}  %For the sake of clarity in presenting the following lemmas, we summarize some  properties of the digraph collection $\mathcal{D}$.
%\begin{theorem}\label{theorem-EC3}
% Let $\mathcal{D}=\{D_1,\ldots,D_n\}$ be a digraph collection on a common vertex set $V$ of size $n$ and  $\delta^0(\mathcal{D})\geq \frac{n}{2}$. Assume that $|\mathcal{C}_3|\geq \sqrt{\delta}n$. Then $\mathcal{D}$ contains a transversal directed Hamilton cycle. 
%\end{theorem}
\begin{proof}[Proof of Theorem \ref{theorem-EC3}]
Choose constant $\epsilon,\delta$ such that
$\frac{1}{n}\ll \epsilon\leq \delta^2\ll 1,$ 
let $\mathcal{D}=\{D_1,\ldots,D_n\}$ be a collection of digraphs on a common vertex set $V$ of size $n$ and $\delta^0(\mathcal{D})\geq \frac{n}{2}$. Suppose that $\mathcal{D}$ does not contain transversal directed Hamilton cycles. For each $i\in [m]$,  since $D_i$ is $(\epsilon,{\rm EC3})$-extremal, one may assume that $(C_i^1,C_i^2,C_i^3,C_i^4,L_i)$ is a  $\epsilon$-characteristic partition of $D_i$. %Recall that if $|\mathcal{C}_3|\geq 2\sqrt{\delta}n$, then $|\mathcal{C}_1|+|\mathcal{C}_2|<2\sqrt{\delta}n$. Denote $[m]:= \mathcal{C}_3\setminus I^2$ and $\mathcal{C}_{\rm bad}:=[m+1,n]$. By adding colors to $\mathcal{C}_{\rm bad}$ if necessary, we may assume $m=(1-4\sqrt{\delta})n$. 
By Lemma \ref{lemma-new1},  one has $|C_1^k\triangle C_i^k|\leq 2\delta n$ for all $k\in [4]$ and all $i\in [m]$. Expand $C_1^1\cup C_1^2\cup C_1^3\cup C_1^4$ into a partition $C^1\cup C^2\cup C^3\cup C^4$ of $V$ such that $|C^1|-|C^3|\leq 1$ and $|C^2|=|C^4|$. Hence, for $i\in[2,m]$ and $k\in[4]$, we have $|C^k\triangle C_{i}^k| \leq  2\delta n+2{\epsilon}n<3\delta n$.
  
  Now, we define the set of bad vertices. Let $(Y,Z)\in \{(C^1,C^3),(C^2,C^4)\}$. Define %\not\in \bigcup_{k\in [4]}C_i^k
\begin{align*}
  &X:=\left\{x\in V: x\in L_i~\textrm{for at least}~ 3\sqrt{\delta}|\mathcal{C}_3|~\textrm{colors}~i\in {\mathcal{C}}_3\right\}, \\
  %X':=&\{x\in X: d_{D_j}^+(x,Y)\geq (1-3\sqrt{\delta})n~\textrm{for at least}~(1-3\sqrt{\delta})|\mathcal{C}_3|~\textrm{colors}~i\in \mathcal{C}_3\\ &\text{and}\ d_{D_j}^-(x,Z)\geq 3\sqrt{\delta}n~\textrm{for at least}~3\sqrt{\delta}n~\textrm{colors}~i\in \mathcal{C}_3\ \},\\
  &X_{YZ}:=\left\{x\in (Y\cup Z)\setminus X:x\not\in Y_{i}\cup Z_i~\textrm{for at least}~ 10\sqrt{\delta}|\mathcal{C}_3|~\textrm{colors}~i\in {\mathcal{C}}_3\right\},\\
  &X_{YZ}':=\left\{x\in X_{YZ}:x\in W_{i}~\textrm{for some $W\in \{Y,Z\}$ with at least}~ \frac{3}{2}\sqrt{\delta}n~\textrm{colors}~i\in {\mathcal{C}_3}\right\}.
\end{align*}

Recall that for all 
$i\in {\mathcal{C}}_3$, we have  $|L_{i}|\leq 2{\epsilon}n$ and $|(Y_1\cup Z_1)\triangle (Y_{i}\cup Z_i)|<4\delta n$. Hence 
$$
3\sqrt{\delta}|{\mathcal{C}}_3||X|\leq 2{\epsilon}n|{\mathcal{C}}_3|\ \ \text{and}\ \ 10\sqrt{\delta}|{\mathcal{C}}_3||X_{YZ}|\leq \sum_{i\in \mathcal{C}_3}|(Y\cup Z)\setminus (Y_i\cup Z_i)|\leq (4\delta+2{\epsilon}) n|{\mathcal{C}}_3|.
$$
It follows that $|X|\leq \sqrt{\epsilon}n$ and $|X_{YZ}|\leq \frac{9}{20}\sqrt{\delta} n<\frac{1}{2}\sqrt{\delta} n$. 
 For a vertex $x\in X_{YZ}\setminus X_{YZ}'$, we know $x\in Y_{i}\cup Z_i$  for at most $3\sqrt{\delta}n$ colors $i\in {\mathcal{C}}_3$ and $x\in \bigcup_{k\in [4]}C_i^k$ for at least $(1-3\sqrt{\delta})|{\mathcal{C}}_3|$ colors $i\in {\mathcal{C}}_3$. Thus, $x\in (Y\cup Z)\setminus (Y_{i}\cup Z_i)$ for at least $(1-4\sqrt{\delta})|{\mathcal{C}}_3|$ colors $i\in \mathcal{C}_3$ and $x\in \bigcup_{k\in [4]}C_i^k\setminus (Y_i\cup Z_i)$ for at least $(1-3\sqrt{\delta})|{\mathcal{C}}_3|-3\sqrt{\delta}n\geq (1-10\sqrt{\delta})|{\mathcal{C}}_3|$ colors $i\in \mathcal{C}_3$. Hence  
\begin{align*}
 |X_{YZ}\setminus X_{YZ}'|(1-4\sqrt{\delta})|{\mathcal{C}}_3|\leq &\sum_{i\in {\mathcal{C}}_3}|(Y\cup Z)\setminus (Y_i\cup Z_i)|\\
\leq &\sum_{i\in \mathcal{C}_3}\left(|(Y\cup Z)\setminus (Y_1\cup Z_1)|+|(Y_1\cup Z_1)\triangle (Y_i\cup Z_i)|\right)\\
\leq &(4\delta +2\epsilon)n|{\mathcal{C}}_3|.
\end{align*}
It follows that $|X_{YZ}\setminus X_{YZ}'|\leq 8\delta n$. Then move vertices in $X_{YZ}\setminus X_{YZ}'$ to $V\setminus (Y\cup Z)$. Therefore, each vertex in $(Y\cup Z)\setminus (X\cup X_{YZ}')$ lies in $Y_i\cup Z_i$ for at least $(1-10\sqrt{\delta})|\mathcal{C}_3|$ colors $i\in \mathcal{C}_3$. %For a vertex in $X_{YZ}'$, if it is in $Y_i$ for at least $\frac{3}{2}\sqrt{\delta}n$ colors $i\in \mathcal{C}_3$, then move it to $Y$.

Let $\{Y,Z\}\in \left\{\{C^1,C^3\},\{C^2,C^4\}\right\}$. Define 
\begin{align*}
  X_{Y}:=&\left\{x\in Y\setminus (X\cup X_{YZ}'):x\not\in Y_{i}~\textrm{for at least}~ 20\sqrt{\delta}|\mathcal{C}_3|~\textrm{colors}~i\in {\mathcal{C}}_3\right\},\\
  X_{Y}':=&\left\{x\in X_{Y}:x\in Y_{i}~\textrm{for at least}~ 6\sqrt{\delta}n~\textrm{colors}~i\in {\mathcal{C}_3}\right\}.
\end{align*}
By a similar discussion as above, we know that $|X_{Y}|<\frac{1}{2}\sqrt{\delta} n$ and each vertex of  $X_Y\setminus X_Y'$ lies in $Z_i$ for at least $(1-20\sqrt{\delta})|{\mathcal{C}}_3|$ colors $i\in \mathcal{C}_3$.  Furthermore,  $|X_{Y}\setminus X_{Y}'|\leq 4\delta n$. We then move vertices in $X_{Y}\setminus X_{Y}'$ to $Z$. Hence $\left||C^2|-|C^4|\right|\leq 24\delta n$.

 Define $V_{\rm bad}:=X\cup \big(\bigcup_{(Y,Z)\in \{(C^1,C^3),(C^2,C^4)\}} X_{YZ}'\big)\cup  \big(\bigcup_{k\in [4]}X_{C^k}'\big)$. Obviously, $|V_{\rm bad}|< (\sqrt{\epsilon}+3\sqrt{\delta})n$. Moreover, each vertex in $C^k\setminus V_{\rm bad}$ lies in $C_i^k$ for at least $(1-20\sqrt{\delta})|{\mathcal{C}}_3|$ colors $i\in {\mathcal{C}}_3$.

The subsequent result allows us to connect two disjoint  short rainbow directed paths into a single short rainbow directed path.  Its proof follows the same lines as that of Lemma~\ref{conn}, and is therefore omitted. 
\begin{claim}[Connecting tool]\label{conn2}
    Let $P=u_1u_2\ldots u_s$ and $Q=v_1v_2\ldots v_t$ be two disjoint rainbow directed paths inside $\mathcal{D}$ with $u_s\in Y\setminus V_{\rm bad}$, $v_1\in Z\setminus V_{\rm bad}$, where $Y,Z\in \{C^1,C^2,C^3,C^4\}$. Assume that $|\mathcal{C}_3\setminus {\rm col}(P\cup Q)|>21\sqrt{\delta} n$, and that for all $i\in [4]$, if $C^i$
 is a candidate set for selecting $w_1$ or $w_1'$, then $|C^i\setminus (V(P\cup Q)\cup V_{\rm bad})|>25\delta n$. % Assume that $u_s\in Y\setminus V_{\rm bad}$, $v_1\in Z\setminus V_{\rm bad}$, where $Y,Z\in \{C^1,C^2,C^3,C^4\}$. 
    \begin{enumerate}
        \item[{\rm (i)}] If $(Y,Z)\in \{(C^1,C^2), (C^3,C^4)\}$, then there are two colors  $c_1,c_2\in\mathcal{C}_3\setminus {\rm col}(P\cup Q)$ and a vertex $w_1\in Y\setminus (V(P\cup Q)\cup V_{\rm bad})$ such that $u_1Pu_sw_1v_1Qv_t$ is a rainbow directed path with colors ${\rm col}(P\cup Q)\cup \{c_1,c_2\}$. A similar result holds for  $(Y,Z)\in \{(C^2,C^3), (C^4,C^1)\}$. 
        \item[{\rm (ii)}] If $(Y,Z)\in \{(C^1,C^1), (C^3,C^3)\}$, then there are two colors  $c_1,c_2\in \mathcal{C}_3\setminus {\rm col}(P\cup Q)$ and one vertex $w_1\in Y\setminus (V(P\cup Q)\cup V_{\rm bad})$ such that $u_1Pu_sw_1v_1Qv_t$ is a rainbow directed path with  colors ${\rm col}(P\cup Q)\cup \{c_1,c_2\}$. 
         \item[{\rm (iii)}]  If $\{Y,Z\}=\{C^2,C^4\}$, then there are three colors $c_1,c_2,c_3\in\mathcal{C}_3\setminus {\rm col}(P\cup Q)$ and two vertices $w_1\in Z\setminus (V(P\cup Q)\cup V_{\rm bad})$,  $w_1'\in Y\setminus (V(P\cup Q)\cup V_{\rm bad})$ such that $u_1Pu_sw_1w_1'v_1Qv_t$ is a rainbow directed path with colors ${\rm col}(P\cup Q)\cup \{c_1,c_2,c_3\}$.
         \item[{\rm (iv)}] If $(Y,Z)=(C^i,C^j)$ with $(i,j)\in\{(1,3),(1,4),(4,3)\}$, then there are two colors  $c_1,c_2\in\mathcal{C}_3\setminus {\rm col}(P\cup Q)$ and a vertex $w_1\in C^2\setminus (V(P\cup Q)\cup V_{\rm bad})$ such that $u_1Pu_sw_1v_1Qv_t$ is a rainbow directed path with colors ${\rm col}(P\cup Q)\cup \{c_1,c_2\}$. A similar result holds for $(Y,Z)=(C^i,C^j)$ with $(i,j)\in\{(3,1),(3,2),(2,1)\}$.
         \item[{\rm (v)}]  If $(Y,Z)=(C^i,C^i)$ with $i\in \{2,4\}$, then there are two colors $c_1,c_2\in\mathcal{C}_3\setminus {\rm col}(P\cup Q)$ and a vertex  $w_1\in C^{i+2}\setminus (V(P\cup Q)\cup V_{\rm bad})$ such that $u_1Pu_sw_1v_1Qv_t$ is a rainbow directed path with colors ${\rm col}(P\cup Q)\cup \{c_1,c_2\}$. %Similar result holds for  $(Y,Z)=(C^4,C^4)$.
 \end{enumerate}
\end{claim}
Delete vertices in $X$ from $\bigcup_{k\in [4]}C^k$ and choose a vertex $x\in X$. Then for some $j\in [4]$, we have $|N_{i}^-(x,C^j\setminus V_{\rm bad})|\geq 200\delta n$ for at least $200\delta n$ colors $i\in \mathcal{C}_3$. Recall that $|X|\leq \sqrt{\epsilon}n$.  We first consider $j=1$ and do the following operations.
\begin{enumerate}
  \item[{\rm (a)}] If $|N_{i}^+(x,C^1\setminus V_{\rm bad})|\geq 200\delta n$ for at least $200\delta n$ colors $i\in \mathcal{C}_3$, then there is a rainbow directed copy of $P_3$ centered at $x$ with both endpoints in $C^1$. We then move such $x$ to $C^1$.
  \item[{\rm (b)}] Notice that each vertex in $C^1\setminus V_{\rm bad}$ has at least $|C^4\setminus V_{\rm bad}|-16\delta n$  in-neighbors inside $C^4\setminus V_{\rm bad}$ for all but at most $20\sqrt{\delta}|\mathcal{C}_3|$ digraphs $D_{\ell}$ with ${\ell}\in \mathcal{C}_3$. If $|N_{i}^+(x,C^2\setminus V_{\rm bad})|\geq 200\delta n$ for at least $200\delta n$ colors $i\in \mathcal{C}_3$, then there is a rainbow directed copy of $P_4$ centered at $x$ with endpoints in $C^4,C^2$ respectively. We then move such $x$ to $C^1$.
  
  Suppose (a)-(b) do not hold. Then for at least $(1-400\delta)|\mathcal{C}_3|$ digraphs $D_{\ell}$ with $\ell\in \mathcal{C}_3$, $x$ has at least $|C^3\cup C^4|-400\delta n-|V_{\rm bad}|$ out-neighbors in $(C^3\cup C^4)\setminus V_{\rm bad}$. Observe that every vertex in $C^3\setminus V_{\rm bad}$ has at least $|C^4\setminus V_{\rm bad}|-16\delta n$  out-neighbors inside  $C^4\setminus V_{\rm bad}$ for all but at most $20\sqrt{\delta}|\mathcal{C}_3|$ digraphs $D_{\ell}$ with ${\ell}\in \mathcal{C}_3$.
  
  \item[{\rm (c)}] If $|N_{i}^-(x,C^2\setminus V_{\rm bad})|\geq 200\delta n$ for at least $200\delta n$ colors $i\in \mathcal{C}_3$, then we can find a  rainbow directed copy of $P_3$ or $P_4$ centered at  $x$ with endpoints in $C^2,C^4$ respectively; we then move such $x$ to $C^3$. The same conclusion holds when $|N_{i}^-(x,C^3\setminus V_{\rm bad})|\geq 200\delta n$ for at least $200\delta n$ colors $i\in \mathcal{C}_3$; in that case we also move $x$ to $C^3$.
  
  \item[{\rm (d)}] Suppose (a)-(c) do not hold. Then in each $D_{\ell}$ with $\ell\in \mathcal{C}_3$, $x$ has at least $|(C^1\cup C^4)\setminus V_{\rm bad}|-400\delta n-|V_{\rm bad}|$ in-neighbors inside $(C^1\cup C^4)\setminus V_{\rm bad}$ for all but at most $(1-400\delta )|\mathcal{C}_3|$ colors $i\in \mathcal{C}_3$. We then move $x$ to $C^2$. 
\end{enumerate}

We perform the above operation successively for  $j=2,3,4$. Let $X'$ denote the subset of $X$ formed by the vertices described in item (d). Hence, for the vertices in $X\setminus X'$, there exists a set of disjoint  rainbow directed copies of  $P_3$ or $P_4$ centered at those vertices, where each copy either has both endpoints in the same part $C^k$ with $k\in \{1,3\}$, or has one endpoint {in} $C^2$ and the other in $C^4$. Moreover, $\left||C^2|-|C^4|\right|\leq 24\delta n+2\sqrt{\epsilon}n$. 

By the definition of $V_{\rm bad}\setminus X$, for vertices in  $(V_{\rm bad}\setminus X)\cap Y$, we can find a set of disjoint rainbow directed copies of $P_3$ with centers in $(V_{\rm bad}\setminus X)\cap Y$ and endpoints in $Y\setminus V_{\rm bad}$ (if $Y\in \{C^1,C^3\}$) or in $(C^2\cup C^4)\setminus (Y\cup V_{\rm bad})$ (if $Y\in \{C^2,C^4\}$). % ; if $Y\in \{C^2,C^4\}$, we can find a set of disjoint  rainbow directed $P_3$ with centers in  $(V_{\rm bad}\setminus X)\cap Y$ and endpoints in $(C^2\cup C^4)\setminus (Y\cup V_{\rm bad})$. 
Denote the rainbow directed path with center $x$ obtained above by $P_x$ for each $x\in V_{\rm bad}\setminus X'$. 

In order to use the transversal blow-up lemma inside $\{D_i[C^2,C^4]:i\in \mathcal{C}_3\}$, we need to balance the sizes of $|C^2|$ and $|C^4|$. Without loss of generality, assume that $|C^2|-|C^4|=r$ with $0\leq r\leq 25{\delta} n$. We now claim that by moving vertices in $C^2\cap V_{\rm bad}$ or deleting rainbow directed paths inside $\mathcal{D}[C^2]$ and $\mathcal{D}[C^2,C^1\cup C^3]$, we can make the number of remaining vertices in $C^2$ and $C^4$ differ by $\sigma$, where $\sigma=1$ if $|C^2|+|C^4|$ is odd and $\sigma=0$ otherwise. 

Denote $s_1:=|X_{C^2C^4}'\cap C^2|$ and $s_2:=|X_{C^2}'|$. If $|C^2|-|C^4|-\sigma\leq s_1+2s_2$, then by moving   $s_1'$ vertices in 
$X_{C^2C^4}'\cap C^2$ to $C^1\cup C^3$ and $s_2'$ vertices in 
$X_{C^2}'$ to $C^4$ 
with $s_1'+2s_2'=|C^2|-|C^4|-\sigma$, we can get our desired result. Hence, it suffices to consider $|C^2|-|C^4|-\sigma>s_1+2s_2$. We first move all vertices in 
$X_{C^2C^4}'\cap C^2$ to $C^1\cup C^3$ and all  vertices in 
$X_{C^2}'$ to $C^4$. Then $(V_{\rm bad}\setminus X)\cap C^2=\emptyset$. Choose a set of disjoint maximal rainbow directed paths in $\mathcal{D}[C^2]\cup\mathcal{D}[C^2,C^1]$, denote such a set by $\{Q_1,\ldots,Q_t\}$. Assume that $\{Q_1',\ldots,Q_k'\}$ is a set of disjoint rainbow directed copies of $P_3$ in $\mathcal{D}[C^2,C^3]$ with centers in $C^2$ and endpoints in $C^3$. %We consider the following two possible cases. 

On the one hand, we consider that  $|E(Q_1)|+\cdots+|E(Q_{t})|+k\geq |C^2|-|C^4|-\sigma$. Then there exist an integer $k'\in [k]$ and a set of disjoint rainbow directed subpaths of $\{Q_1,\ldots,Q_t\}$, say $\{\tilde{Q}_1,\ldots,\tilde{Q}_{t'}\}$, such that  $|E(\tilde{Q}_1)|+\cdots+|E(\tilde{Q}_{t'})|+k'= |C^2|-|C^4|-\sigma$. That is,  $|C^4|-t'=|C^2|-(|V(Q_1')|+\cdots+|V(Q_{t'}')|)-k'-\sigma$. If the end   vertex of $\tilde{Q_i}$ with $i\in [t']$ (resp. the start vertex of $Q_i'$ with $i\in [k']$) lies in $V_{\rm bad}\setminus X'$, then we extend it by a vertex not in $V_{\rm bad}$. Applying Claim~\ref{conn2}, we then connect the modified paths $\tilde{Q}_1,\ldots,\tilde{Q}_{t'},Q_1',\ldots,Q_{k'}'$ into a single rainbow path $P^0$ with endpoints in $C^2$ and $C^4$ respectively, whose length is at most $190\delta n$. Therefore, $|C^4\setminus V(P^0)|=|C^2\setminus V(P^0)|-\sigma$, as desired. 

In each of the above cases, that is, either $|C^2|-|C^4|-\sigma\leq s_1+2s_2$, or $|C^2|-|C^4|-\sigma> s_1+2s_2$ and $|E(Q_1)|+\cdots+|E(Q_{t})|+k\geq |C^2|-|C^4|-\sigma$, the number of the remaining vertices in $C^2$ and $C^4$ are balanced. Since $|E(P^0)|\leq 190\delta n$, by avoiding vertices in $V(P^0)$,  one may choose a set of disjoint rainbow directed paths $\mathbf{P}^0=\{P_x:x\in V_{\rm bad}\setminus (X'\cup P^0)\}$. For  colors in $\mathcal{C}_{\rm bad}\setminus {\rm col}(P^0)$, we select a maximal rainbow matching $M$ inside $\mathcal{D}^{\pm}[C^2,C^4]\cup \mathcal{D}[C^1]\cup \mathcal{D}[C^3]\cup (\bigcup_{i\in [4]}\mathcal{D}[C^i,C^{i+1}])$ ensuring it avoids vertices in $V(P^0\cup \mathbf{P}^0)$. Notice that if there exists a color $c\in {\mathcal{C}_{\rm bad}}\setminus ({\rm col}(P^0\cup M))$, then $D_c^{\pm}[C^1,C^3]$ is almost an undirected complete bipartite graph and $D_c[C^2],D_c[C^4]$ are almost undirected complete graphs. Consequently, for colors in ${\mathcal{C}_{\rm bad}}\setminus ({\rm col}(P^0\cup M))$, we can  choose a set of disjoint rainbow directed paths, say $\mathbf{P}^1$, with centers in $C^1$ and endpoints inside $C^3$, ensuring that at most one color in $\mathcal{C}_{\rm bad}$ {remains}  unused.
%Furthermore, each unused vertex $v\in C^2$ satisfying $N_{i}^+(v,C^4)\geq |C^4|-200\delta n$ and $N_{i}^-(v,C^4)\geq |C^4|-200\delta n$ for all but at most $(1-190\delta n)|\mathcal{C}_3|$ colors $i\in \mathcal{C}_3$. 

%each vertex in $A$ is adjacent to almost all vertices in $C$, and each vertex is adjacent to almost all vertices in $A$. By choosing a set of disjoint rainbow directed path with center in $A$ and endpoints inside $C$, at most one color in $\mathcal{C}_{\rm bad}$ is unused.  If such a color exists, then choose a directed edge from $A$ to $C$ in this graph. 

Let $\mathbf{P}:=\mathbf{P}^0\cup \mathbf{P}^1\cup M\cup \{P^0\}$. By Claim \ref{conn2} (ii) and (v), we  connect all rainbow directed paths inside $\mathbf{P}$ with one endpoints in $C^2$ and the other in $C^4$ into $Q_0=u_0\ldots v_0$, all rainbow directed paths inside $\mathbf{P}$ with both endpoints inside $C^1$ and $C^3$ into $Q_1=u_1\ldots v_1$ and $Q_2=u_2\ldots v_2$ respectively. Since $u_0,v_0,u_1,v_2\notin V_{\rm bad}$, there are four distinct colors $i_1,i_2,i_3,i_4\in \mathcal{C}_3\setminus {\rm col}(Q_0\cup Q_1\cup Q_3)$ such that $u_0\in C_{i_1}^2$, $v_0\in C_{i_2}^4$, $u_1\in C_{i_3}^1$ and $v_2\in C_{i_4}^3$. We proceed by considering the following two cases. 

\medskip
$\bullet$ Suppose that $\sigma=0$, i.e., $|C^4\setminus V(P^0)|=|C^2\setminus V(P^0)|$. We first assume that $\mathcal{C}_{\rm bad}\setminus {\rm col}(M\cup \mathbf{P}^0)=\emptyset$. %Delete all used vertices from $C^i$ and all used colors from $\mathcal{C}_3$. 
By applying Lemma~\ref{lemma4.1} to the digraph collection  $\{D_i[C^1]\cup D_i[C^3]:i\in \mathcal{C}_3\setminus \{i_1,i_2,i_3,i_4\}\}$, while avoiding vertices and colors in $Q_0\cup Q_1\cup Q_2$, we obtain two disjoint rainbow directed paths $P^1=x_1\ldots y_1$ and $P^2=x_2\ldots y_2$. Here $y_1\in N_{i_3}^-(u_1,C^1)$, $x_2\in N_{i_4}^+(v_2,C^3)$, $x_1\in C_{c_1}^1\setminus V_{\rm bad}$ and $y_2\in C_{c_2}^3\setminus V_{\rm bad}$ for two unused colors $c_1,c_2\in \mathcal{C}_3$. We then connect  $Q_1,Q_2$ via Claim \ref{conn2} (iv) to form a rainbow directed path $P=x_1\ldots y_2$, which satisfies $V_{\rm bad}\setminus X'\subseteq V(P)$, $\mathcal{C}_{\rm bad}\subseteq {\rm col}(P)$ and $|C^4\setminus V(P)|=|C^2\setminus V(P)|+1$.

Applying Lemma~\ref{lemma4.1} to the digraph collection $\{D_i^{\pm}[C^2\setminus V(P\cup Q_0),C^4\setminus V(P\cup Q_0)]:i\in \mathcal{C}_3\setminus ({\rm col}(P\cup Q_0)\cup \{i_1,i_2,c_1,c_2\})\}$, we obtain two disjoint rainbow directed paths $P^3=x_3\ldots y_3$ and $P^4=x_4\ldots y_4$, where $x_3\in N_{{c_1}}^+(y_2,C^4\setminus V_{\rm bad})$, $y_3\in N_{{i_1}}^-(u_0,C^4\setminus V_{\rm bad})$, $x_4\in N_{{i_2}}^+(v_0,C^2\setminus V_{\rm bad})$ and $y_4\in N_{{c_2}}^-(x_1,C^4\setminus V_{\rm bad})$, respectively. Hence $\mathcal{D}$ contains a transversal directed Hamilton cycle, a contradiction. 

Now, assume that there exists a color $c\in \mathcal{C}_{\rm bad}\setminus {\rm col}(M\cup \mathbf{P}^0)$.  Since $\delta^0(\mathcal{D})\geq \frac{n}{2}$, there exists an edge $uv\in E(D_c[C^2\setminus V(Q_0\cup Q_1\cup Q_2),C^3\cup C^4])$ if $|C^1|+|C^2|\leq |C^3|+|C^4|$, and $vu\in E(D_c[C^4\setminus V(Q_0\cup Q_1\cup Q_2),C^1\cup C^2])$ otherwise. By an argument analogous to that above, and considering whether  $v\in V_{\rm bad}$, we obtain that $\mathcal{D}$ contains a transversal directed Hamilton cycle, a contradiction.

%$D_c[C^1]\cup D_c[C^3]\neq \emptyset$ or $D_c[C^2,C^4]\neq \emptyset$. Choose such a directed edge inside $D_c$. By a similar discussion as above, $\mathcal{D}$ contains a transversal directed Hamilton cycle.% Applying the transversal blow-up lemma, we get two rainbow directed paths $P_a=u_a\ldots v_a$ and $P_c=u_c\ldots v_c$ with $u_a\in A_{c_3}\setminus V_{\rm bad}$ and $v_c\in N_{{c_1}}^-(x,C)\setminus V_{\rm bad}$. Connect $P_a,Q_a,Q_b,P_b$ in turn by using vertices in $A,B,C$ respectively and colors in $\mathcal{C}_3$. Choose an  By using Lemma \ref{conn2}, connect $P_a$ and $P_c$ by a vertex in $B$, we get a rainbow path with endpoints in $A$ and $C$ respectively. Then extend it to a rainbow paths $u_a\ldots v_d$ with endpoints in $A$ and $C$ respectively.
\medskip

$\bullet$ Suppose that $\sigma=1$, i.e., $|C^4\setminus V(P^0)|=|C^2\setminus V(P^0)|-1$. %Connecting all rainbow paths with both endpoints inside $A$ or $C$, denoted by $Q_a$ and $Q_c$. 
We first consider that there exists a color $c\in \mathcal{C}_{\rm bad}\setminus {\rm col}(M\cup \mathbf{P}^0)$. Choose a directed edge, say ${xy}$, in $E(D_c[C^3\setminus V(Q_0\cup Q_1\cup Q_3),C^1\setminus V(Q_0\cup Q_1\cup Q_3)])$. %Delete all used vertices from $C^i$ and all used colors from $\mathcal{C}_3$. 
By applying Lemma~\ref{lemma4.1} to the digraph collection  $\{D_i[C^1\setminus (V(Q_0\cup Q_1\cup Q_2)\cup \{x\})]\cup D_i[C^3\setminus (V(Q_0\cup Q_1\cup Q_2)\cup \{y\})]:i\in \mathcal{C}_3\setminus ({\rm col}(Q_0\cup Q_1\cup Q_2)\cup \{i_1,i_2,i_3,i_4,c\})\}$, we obtain two disjoint rainbow directed paths $P^1=x_1\ldots y_1$ and $P^2=x_2\ldots y_2$, where $y_1\in N_{i_3}^-(u_1,C^1)$, $x_2\in N_{i_4}^+(v_2,C^3)$, $x_1\in C_{c_1}^1\setminus V_{\rm bad}$ and $y_2\in C_{c_2}^3\setminus V_{\rm bad}$ for two unused colors $c_1,c_2\in \mathcal{C}_3$.
%we get two disjoint rainbow directed paths $P^1=u_1\ldots v_1$ and $P^2=u_2\ldots v_2$ with $u_1\in C_{c_1}^1\setminus V_{\rm bad}$ and $v_2\in C_{c_2}^3\setminus V_{\rm bad}$ for two unused colors $c_1,c_2\in \mathcal{C}_3$. 
By sequentially connecting $Q_1,{xy},Q_3$ via Claim \ref{conn2} (iv), we obtain a rainbow directed path $P=x_1\ldots y_2$ such that $V_{\rm bad}\setminus X'\subseteq V(P)$,  $\mathcal{C}_{\rm bad}\subseteq {\rm col}(P)$ and $|C^4\setminus V(P)|=|C^2\setminus V(P)|+1$. Similar to the previous discussion, $\mathcal{D}$ contains a transversal directed Hamilton cycle, a contradiction. %By transversal blow-up lemma, there exists a rainbow directed path $P^0=u_b\ldots w_d$ with endpoints in $u_b\in N_{{c_1}}^+(v_d,D\setminus V_{\rm bad})$ and $w_d\in N_{{c_2}}^-(u_a,D\setminus V_{\rm bad})$ respectively by using colors inside $\mathcal{C}_3$. Hence 

Now, assume that $\mathcal{C}_{\rm bad}\setminus {\rm col}(M\cup \mathbf{P}^0)=\emptyset$. Since $|C^2|\geq |C^4|$, either $|C^1|+|C^2|\geq |C^3|+|C^4|$ or $|C^3|+|C^2|\geq |C^1|+|C^4|$. Hence, for each $c\in \mathcal{C}_3$, there exists an edge $uv\in E(D_c[C^2\setminus V(Q_0\cup Q_1\cup Q_2),C^1\cup C^2])$ or $uv\in E(D_c[C^2\cup C^3,C^2\setminus V(Q_0\cup Q_1\cup Q_2)])$. By a similar argument as above, and considering whether the vertex  $v$ belongs to $V_{\rm bad}$ or not, we obtain that $\mathcal{D}$ contains a transversal directed Hamilton cycle, a contradiction.% Let $P$ be the rainbow directed path obtained in the case for $\sigma=0$ with $x_1\in N_{i_2}^+(v_0,C^1\setminus V_{\rm bad})$. Then $|C^4\setminus V(P)|=|C^2\setminus V(P)|$.  Applying Lemma~\ref{lemma4.1}  inside the digraph collection $\{D_i^{\pm}[C^2\setminus V(P\cup Q_0),C^4\setminus V(P\cup Q_0)]:i\in \mathcal{C}_3\setminus ({\rm col}(P\cup Q_0)\cup \{i_1,c_2\})\}$, there exist a rainbow directed path $P^3=x_3\ldots y_3$  with endpoints $x_3\in N_{{c_2}}^+(y_2,C^2\setminus V_{\rm bad})$ and $y_3\in N_{{i_1}}^-(u_0,C^4\setminus V_{\rm bad})$ respectively. Hence $\mathcal{D}$ contains a transversal directed Hamilton cycle, a contradiction.

\medskip

On the other hand, we consider that 
$|E(Q_1)|+\cdots+|E(Q_{t})|+k< |C^2|-|C^4|-\sigma= r-s_1-2s_2-\sigma$. In this case, we move all vertices inside $Q_1,\ldots,Q_t$, all vertices in $(V(Q_1')\cup \ldots\cup V(Q_k'))\cap C^2$ along with their in-neighbors in $C^3$ to $C^4$. Denote the resulting partition by $\tilde{C}^1\cup \tilde{C}^2\cup \tilde{C}^3\cup \tilde{C}^4$. %If there are some rainbow directed $P_3$ with center in $B$ and both endpoints in $A$, or in $C$ by using unused colors in $\mathcal{C}_3$, then move such vertices in $B$ to $D$. Then each graph $D_c$ with an unused color $c\in \mathcal{C}_3$, each vertex in $B\setminus V(P)$ has at most $4\delta n$ out-neighbors in $A$ and at most $4\delta n$ in-neighbors in $C$. If $|B\setminus V(P)|\leq  |D\setminus V(P)|$, then we are done. If $|B\setminus V(P)|> |D\setminus V(P)|$ still holds, then move all vertices of the above rainbow $P_2$ and $P_3$ to $D$.  
Let $\tilde{\mathcal{C}}:=\mathcal{C}\setminus ({\rm col}(\bigcup_{i\in [t]}Q_i)\cup {\rm col}(\bigcup_{i\in [k]}Q_i'[C^3,C^2]))$. 
Therefore,  \allowdisplaybreaks
\begin{align*}
    &|\tilde{C}^4|
    \leq \frac{|C^2|+|C^4|-r}{2}+s_2+2(r-s_1-2s_2-\sigma)\leq \frac{|C^2|+|C^4|}{2}+50{\delta}n,  \\
&|\tilde{C}^2|\geq \frac{|C^2|+|C^4|+r}{2}-s_1-s_2-2(r-s_1-2s_2-\sigma)\geq \frac{|C^2|+|C^4|}{2}-50{\delta}n,\\
&|\tilde{\mathcal{C}}|=|\mathcal{C}|-(|E(Q_1)|+\cdots+|E(Q_{t})|)-k\geq (1-25\delta)n.
\end{align*}

By the maximality of $\{Q_1,\ldots,Q_t\}$, we know that $D_i[\tilde{C}^2]=D_i[\tilde{C}^2,\tilde{C}^1]=D_i[\tilde{C}^3,\tilde{C}^2]=\emptyset$  for all color $i\in \tilde{\mathcal{C}}$. % and each vertex in $\tilde{C}^2$ has at most $8\delta n$ out-neighbors in $A$ (resp. in-neighbors in $C$) for all unused color $i\in \mathcal{C}'$. 
It follows from  $\delta^0(\mathcal{D})\geq \left\lceil\frac{n}{2}\right\rceil$ that $\left\lceil\frac{n}{2}\right\rceil\leq|\tilde{C}^3|+|\tilde{C}^4|$ and $\left\lceil\frac{n}{2}\right\rceil\leq |\tilde{C}^1|+|\tilde{C}^4|$. Therefore, 
$$
|\tilde{C}^1|+|\tilde{C}^2|+|\tilde{C}^3|+|\tilde{C}^4|=n\leq \left\lceil\frac{n}{2}\right\rceil+\left\lceil\frac{n}{2}\right\rceil\leq |\tilde{C}^1|+|\tilde{C}^3|+2|\tilde{C}^4|,
$$
which implies that $|\tilde{C}^2|\leq |\tilde{C}^4|$. We complete our proof by using the subsequent lemma (i.e., Lemma~\ref{Y-large-2}), which leads to a contradiction by showing  that $\mathcal{D}$ contains a transversal directed Hamilton cycle.
\end{proof}
\begin{lemma}\label{Y-large-2}
Suppose $0<\frac{1}{n}\ll\delta\ll1$ and $0\leq \gamma<\delta$. Let $\mathcal{C}$ be a set of $n$ colors, and $\mathcal{D}=\{D_i:i\in \mathcal{C}\}$ be a collection of digraphs on a common vertex set $V$ of size $n$ such that  $\delta^0(\mathcal{D})\geq \left\lceil\frac{n}{2}\right\rceil$. Assume that
\begin{itemize}
  \item $C^1\cup C^2\cup C^3\cup C^4$ is a partition of $V$ with $|C^2|\geq 50\delta^{\frac{1}{4}}$, $\min\{|C^3|+|C^4|,|C^1|+|C^4|\}\geq \left\lceil\frac{n}{2}\right\rceil$ and  $|C^4|=\left\lfloor\frac{|C^2|+|C^4|+1}{2}\right\rfloor+\gamma n$,
  \item $\mathcal{C}'\cup \mathcal{C}''$ is a partition of $\mathcal{C}$ with  $|\mathcal{C}''|\leq \delta n$,
  \item $D_i[C^2]=D_i[C^2,C^1]=D_i[C^3,C^2]=\emptyset$ for all $i\in \mathcal{C}'$.
\end{itemize}
Then $\mathcal{D}$ contains a transversal directed Hamilton cycle.
\end{lemma}
\begin{proof}
It is easy to see that $|C^3|+|C^4|+|C^1|+|C^4|=n+2\gamma n+\sigma$ and $\max\{|C^3|+|C^4|,|C^1|+|C^4|\}\leq \left\lceil\frac{n}{2}\right\rceil+2\gamma n+\sigma$, where $\sigma=1$ if $|C^2|+|C^4|$ is odd and $\sigma=0$ otherwise. Without loss of generality, assume that $|C^3|+|C^4|=\left\lceil\frac{n}{2}\right\rceil+\gamma_1 n$ and $|C^1|+|C^4|=\left\lceil\frac{n}{2}\right\rceil+\gamma_2 n$ with $2\left\lceil\frac{n}{2}\right\rceil+\gamma_1 n+\gamma_2 n=n+2\gamma n+\sigma$.
 Define $Y_1:=\{v\in C^4: d_{i}^+(v,C^2)\leq (1-\delta^{\frac{1}{4}})|C^2|\ \textrm{for at least}\ \delta^{\frac{1}{4}}|\mathcal{C}'|\ \textrm{colors}\  i\in \mathcal{C}'\}$, $Y_2:=\{v\in C^4: d_{i}^-(v,C^2)\leq (1-\delta^{\frac{1}{4}})|C^2|\ \textrm{for at least}\ \delta^{\frac{1}{4}}|\mathcal{C}'|\ \textrm{colors}\  i\in \mathcal{C}'\}$, and $Y:=Y_1\cup Y_2$.  
    It is routine to check that 
\begin{align}\notag
    \left\lceil\frac{n}{2}\right\rceil|C^2||\mathcal{C}'|&\leq \sum_{i\in \mathcal{C}'}|E({D_i}[C^2,C^3\cup C^4])|=\sum_{i\in \mathcal{C}'}|E({D_i}[C^2,C^3])|+\sum_{i\in \mathcal{C}'}|E({D_i}[C^2,C^4])|\\\notag
    &\leq |C^2||C^3||\mathcal{C}'|+ |Y_2|(1-\delta^{\frac{1}{4}})|C^2|\delta^{\frac{1}{4}}|\mathcal{C}'|+|Y_2||C^2|(1-\delta^{\frac{1}{4}})|\mathcal{C}'|
    +(|C^4|-|Y_2|)|C^2||\mathcal{C}'|\\\notag
    &=(|C^4|-\delta^{\frac{1}{2}}|Y_2|+|C^3|)|C^2||\mathcal{C}'|.
\end{align}
Thus, $|Y_2|\leq 3\sqrt{\delta}n$. Similarly, $|Y_1|\leq 3\sqrt{\delta}n$. {{Notice that Lemma~\ref{lemma4.1} and Claim \ref{conn2} hold by setting $V_{\rm bad}:=Y$ and $\mathcal{C}_3:=\mathcal{C}'$.}}

For a vertex $v\in Y_1\setminus Y_2$ (resp. $v\in Y_2\setminus Y_1$), we have $d_{i}^-(v,C^2)\geq (1-\delta^{\frac{1}{4}})|C^2|$ (resp. $d_{i}^+(v,C^2)\geq (1-\delta^{\frac{1}{4}})|C^2|$) for at least $(1-\delta^{\frac{1}{4}})|\mathcal{C}'|$ colors $i\in \mathcal{C}'$. For a vertex $y\in Y_2$,  one has $d_i^-(y,C^1\cup C^4)\geq 40\sqrt{\delta}n$ %it has at least $40\sqrt{\delta}n$ in-neighbors inside $C^1\cup C^4$ 
for at least ${\delta}^{\frac{1}{4}}|\mathcal{C}'|-1$ colors  $i\in \mathcal{C}'$. Otherwise, by the definition of $Y_2$, there exists a color $c\in \mathcal{C}'$ such that $d_{c}^-(y,C^2)\leq (1-\delta^{\frac{1}{4}})|C^2|$ and $d_{c}^-(y,C^1\cup C^4)\leq 40\sqrt{\delta}n$. Together with $|C^2|\geq 50{\delta}^{\frac{1}{4}}n$, one has 
$$
d_{c}^-(y)\leq 40\sqrt{\delta}n+(1-\delta^{\frac{1}{4}})|C^2|+|C^3|\leq \left\lceil\frac{n}{2}\right\rceil-\delta^{\frac{1}{4}}|C^2|+40\sqrt{\delta}n< \left\lceil\frac{n}{2}\right\rceil,
$$
a contradiction. Similarly, each vertex in $Y_1$ has at least $40\sqrt{\delta}n$ out-neighbors inside $C^3\cup C^4$ for at least ${\delta}^{\frac{1}{4}}|\mathcal{C}'|-1$ digraphs $D_i$ with $i\in \mathcal{C}'$.

Since $\delta^0(\mathcal{D})\geq \frac{n}{2}$, we have $|E(D_i[C^2,C^3\cup C^4])|,|E(D_i[C^1\cup C^4,C^2])| \geq |C^2|\left\lceil\frac{n}{2}\right\rceil$ for each $i\in \mathcal{C}'$. This implies that for each $i\in \mathcal{C}'$, there are at most $|C^2|(\left\lceil\frac{n}{2}\right\rceil+\gamma_1 n)-|C^2|\left\lceil\frac{n}{2}\right\rceil=|C^2|\gamma_1 n$ non-edges in $D_i[C^2,C^3\cup C^4]$ and at most $|C^2|\gamma_2 n$ non-edges in $D_i[C^1\cup C^4,C^2]$. Therefore, %for each $i\in \mathcal{C}'$, there are at most $|C^2|\gamma_1 n$ non-edges in $D_i[C^2, C^4]$ and at most $|C^2|\gamma_2 n$ non-edges in $D_i[C^4,C^2]$. 
%there are at most $|C^2|\gamma_1 n$ (resp., $|C^2|\gamma_2 n$) non-edges from $C^2$ to $C^4$ (resp., from $C^4$ to $C^2$) in $D_i$. 
$$
\sum_{i\in \mathcal{C}'}|E(D_i[Y_1,C^2])|\geq (|Y_1|-\gamma_1 n)|C^2||\mathcal{C}'|\ \ \text{and}\ \ \sum_{i\in \mathcal{C}'}|E(D_i[C^2,Y_2])|\geq (|Y_2|-\gamma_2 n)|C^2||\mathcal{C}'|.
$$
We proceed {with the} proof by considering the values of $|Y_1|$ and $|Y_2|$.

\medskip
{\bf Case 1.} $|Y_1|\geq \gamma_1 n,\ |Y_2|\geq \gamma_2 n$.
\medskip

In view of Lemma \ref{claim4.2}, there exists a subset $Y_1'\subseteq Y_1$ with size $\gamma_1 n$ such that vertices in  $Y_1\setminus Y_1'$ can be covered by  disjoint rainbow directed $P_3$ copies inside $\{D_i[Y_1,C^2]:i\in \mathcal{C}'\}$ with centers in $Y_1\setminus Y_1'$; and a subset $Y_2'\subseteq Y_2$ with size $\gamma_2 n$ such that vertices in $Y_2\setminus Y_2'$ can be covered by  disjoint rainbow directed $P_3$ copies inside $\{D_i[C^2,Y_2]:i\in \mathcal{C}'\}$ with centers in $Y_2\setminus Y_2'$. %Recall that each vertex in $Y_1$ (resp. $Y_2$) has at least $\left\lceil\frac{n}{2}\right\rceil-(1-\delta^{\frac{1}{4}})|B|>4|Y|$ out-neighbors (resp. in-neighbors) in $A$ for at least $\delta^{\frac{1}{4}}|\mathcal{C}'|>\delta^{\frac{1}{4}}(1-\delta)n>4|Y|$ colors $i\in\mathcal{C}'$.
Therefore, by using colors in $\mathcal{C}'$,  
\begin{itemize}
  \item vertices in $(Y_1\cup Y_2)\setminus (Y_1'\cup Y_2')$ can be covered by a set of disjoint rainbow directed $P_3$ copies with centers in $(Y_1\cup Y_2)\setminus (Y_1'\cup Y_2')$ and endpoints in $C^2$,
  %\item vertices in $(Y_1\setminus Y_1')\cap (Y_2\setminus Y_2')$ can be covered with disjoint rainbow directed $P_3$ with centers in $(Y_1\setminus Y_1')\cap (Y_2\setminus Y_2')$ and endpoints in $C^2$. 
  \item  vertices in $Y_2'\setminus Y_1'$ can be covered by a set of disjoint rainbow directed $P_3$ or $P_4$ copies with centers in  $Y_2'\setminus Y_1'$, starting at $C^4$ and ending  at $C^2$,
  \item vertices in $Y_1'\setminus Y_2'$ can be covered by a set of disjoint rainbow directed $P_3$ or $P_4$ copies with centers in $Y_1'\setminus Y_2'$, starting at $C^2$ and ending at 
 $C^4$,
  \item vertices in $Y_1'\cap Y_2'$ can be covered by a set of disjoint rainbow directed $P_3$, $P_4$ or $P_5$ copies with centers in $Y_1'\cap  Y_2'$ and endpoints in $C^4$.
  %\item vertices in $Y_2'\setminus Y_1$ can be covered with disjoint rainbow directed $P_3$ or $P_4$ with centers in $Y_2'\setminus Y_1$, head in $C^4$ and tail in $C^2$.
\end{itemize}
For each $y\in Y$, denote the rainbow directed path with center $y$ by $P_y=y^1*y*y^2$ with colors in $\mathcal{C}_3$. Let $\mathbf{P}:=\{P_y:y\in Y\}$. 

It is routine to verify that there exists a rainbow matching inside $D_i^{\pm}[C^2\setminus V(\mathbf{P}),C^4\setminus V(\mathbf{P})]\cup D_i[C^1]\cup D_i[C^3]\cup (\bigcup_{k\in [4]}D_i[C^k,C^{k+1}]): i\in \mathcal{C}''\}$, say $M$. For all but at most one unused {color}  in $\mathcal{C}''$, we can find a rainbow directed copy of $P_3$ with center in $Y$ and endpoints in $Z$, where $\{Y,Z\}=\{C^1,C^3\}$. If some color $c\in \mathcal{C}''$ remains unused, then $D_c[C^1,C^3]$ is nearly a complete bipartite graph and $D_c[C^2],D_c[C^4]$ are almost complete graphs.  Then, by a similar argument to that in Theorem \ref{Y-large}, $\mathcal{D}$ contains a transversal directed Hamilton cycle.

\medskip
{\bf Case 2.} At most one of $|Y_1|\geq \gamma_1n$ and $|Y_2|\geq \gamma_2n$ holds.
\medskip

We only consider the case that $|Y_1|\geq \gamma_1 n$ and $|Y_2|<\gamma_2 n$; other cases can be discussed similarly. Notice that vertices in $Y_1$ can be covered by a set of disjoint rainbow directed paths with centers in $Y_1$ and ending in   $C^4\setminus (Y_1\cup Y_2)$. In particular, those rainbow directed paths can be chosen such that none of them  contains edges inside $\mathcal{D}[C^4,C^3]$ or $\mathcal{D}[C^1,C^4]$. Furthermore, vertices in $Y_2\setminus Y_1$ can be covered by a set of disjoint rainbow directed $P_3$ copies with centers in $Y_2\setminus Y_1$,  starting at  $C^4\setminus Y$, and using colors in $\mathcal{C}'$.

%Let $Y=Y_1\cup Y_2$. By a similar discussion as above, we know there are $|Y_1|-\gamma n$ disjoint rainbow directed $P_3$ with centers in $Y_1$ and endpoints in $B$. Furthermore,  For a vertex $v\in Y_1$, one of the following holds:
%\begin{itemize}
%  \item it has at least $4|Y_1|$ in-neighbors and out-neighbors in $D$ for at least $4|Y_1|$ colors $i\in\mathcal{C}'$;
%  \item it has at least $4|Y_1|$ in-neighbors in $A$ and out-neighbors in $C$ for at least $4|Y_1|$ colors $i\in\mathcal{C}'$;
%  \item almost all its in-neighbors are in $A\cup B$. In this case, it has at most $|A|+(1-\delta^{\frac{1}{4}})|B|+8|Y_1|\leq \frac{n}{2}+8\delta n-\delta^{\frac{1}{4}}|B|+40\sqrt{\delta}n\leq \frac{n}{2}$ since $|B|\geq 50{\delta}^{\frac{1}{4}}n$, a contradiction. 
%\end{itemize}
%Then there exist  $\gamma n+|Y_2\setminus Y_1|$ disjoint rainbow directed  paths $P_i(i\in \{3,4,5\})$ with centers in unused vertices of  $Y_1\cup Y_2$ and endpoints either in $B$ or one endpoint in $D\setminus (Y_1\cup Y_2)$ and the other in $B$ using colors in $\mathcal{C}'$. 
In the digraph collection $\mathcal{D}[C^4\setminus Y]$, extend those rainbow directed  paths or choose other disjoint rainbow directed paths into a set of disjoint maximal rainbow directed  paths. Let $\mathbf{P}:=\{Q_1,Q_2,\ldots,Q_t\}$ be a set consisting of all disjoint rainbow paths in the above, each of which has at least one endpoint in $C^4\setminus Y$ and length $s_i\ (1\leq i\leq t)$ inside $C^4$. Let $\{e_1,e_2,\ldots,e_r\}$ be a set consisting of all disjoint rainbow edges in $\mathcal{D}[C^4\setminus Y,C^3]$, and let $\{\tilde{Q}_1,\ldots,\tilde{Q}_k\}$ be a set of disjoint maximal rainbow directed  $P_3$ in $\mathcal{D}^{\pm}[C^4\setminus Y,C^1]$ with centers in $C^4$ and endpoints in $C^1$, each of which is chosen while avoiding previously used vertices and colors.

\medskip
{\bf Subcase 2.1.} $|C^2|\geq |C^4|-k-r-(s_1+\cdots+s_t)-\sigma$. 
\medskip

In this subcase, there exist two integers $r'\in [r]$, $k'\in [k]$ and a set of disjoint rainbow directed paths $\mathbf{P}'=\{{Q}_1',{Q}_2',\ldots,{Q}_{\ell}'\}$ such that $|C^4\setminus V(\mathbf{P}')|-r'-k'-\sigma=|C^2|-\ell$, where $Y_1'\cup Y_2\subseteq V(\mathbf{P}')$ and the endpoints of $\tilde{Q}_i$ are not in $Y$ for all $i\in [\ell]$. 
%Assume $|E(Q_i')|=s_i'$ for each $i\in [\ell]$.   
By Claim \ref{conn2}, we connect all the rainbow directed paths in $\mathbf{P}'$ into a single rainbow path $P^1$, whose endpoints lie in different parts. 
Since $|C^4|-|C^2|=2\gamma n+\sigma$, we have $\ell\leq |E(Q_1')|+\cdots+|E(Q_{\ell}')|= 2\gamma n+\sigma$. Therefore, 
$|E(P^1)|\leq 6\gamma n+3\sigma$ and $|C^4\setminus V(P^1)|-\sigma=|C^2\setminus V(P^1)|$. 
By a similar discussion as Subcase 2.1 of Theorem \ref{Y-large}, $\mathcal{D}$ contains a transversal directed Hamilton cycle.

\medskip
{\bf Subcase 2.2.} $|C^2|<|C^4|-k-r-(s_1+\cdots+s_t)-\sigma$. 
\medskip

%In this case, in each graph $D_c$ with unused color $c\in \mathcal{C}$, each unused vertex in $D$ has at most $\gamma n$ in-neighbors (out-neighbors) in $A$ (resp. $C$). Without loss of generality, assume that $|A|\geq |C|$. Then $|B|+|C|=n-|A|-|D|\leq n-|C|-\frac{|B|+|D|}{2}-\gamma n=n-\frac{|C|+|B|}{2}-\frac{|C|+|D|}{2}-\gamma n$. This implies that $|B|+|C|$

In this subcase, $\sum_{i=1}^t s_i+k+r<2\gamma n$. %Using  Claim \ref{conn2}, one may connect all rainbow paths of  $\mathbf{P}$ (resp. $\{e_1,\ldots,e_k\}$, $\{Q_1',\ldots,Q_k'\}$) into a single rainbow paths $P^3$ (resp. $P^4$, $P^5$), whose endpoints are in different parts of $C^2$ and $C^4$ (resp. $C^2$ and $C^4$, $C^3$). Let $\Tilde{C}^4=C^4\setminus V(P^3\cup P^4\cup P_5)$ and $\Tilde{C}^2=C^2\setminus V(P^3\cup P^4\cup P^5)$. Then $|\Tilde{C}^4|=|C^4|-(s_1+\cdots+s_t+t+(t_1+t_2-1))-k$, $|\Tilde{C}^2|=|C^2|-t-((t_1+t_2-1))$ and $|\Tilde{\mathcal{C}}|=|\mathcal{C}|-(s_1+\cdots+s_t+2(t_1+t_2-1)+2t-1)-2k$. 
Assume that $V(Q_i[C^4]):=\{v_{s_1+\cdots+s_{i-1}+i},\ldots,v_{s_1+\cdots+s_{i}+i}\}$ for each $i\in [t]$. Let $w$ be an unused vertex in $C^4$ and $c_1,c_2$ be two unused colors in $\mathcal{C}$. %Without loss of generality, assume that $d_{D_{c_1}}^-(w,D\setminus \Tilde{D})\leq d_{D_{c_2}}^+(w,D\setminus \Tilde{D})$.  $w$ has at least $\left\lceil\frac{n}{2}\right\rceil-(\lceil\frac{n-1}{2}\rceil-\gamma n)=\gamma n+\sigma$ out-neighbors (resp. in-neighbors) in $D\setminus \Tilde{D}$ and it
By the maximality of $\mathbf{P}$, in $D_{c_1}$ (resp. $D_{c_2}$), $w$ cannot be adjacent to the starting  vertices of $Q_i[C^4]$ (resp. cannot be adjacent from the ending vertices of  $Q_i[C^4]$) for all $i\in [t]$. 

Define
$$
  I_1:=\{i:v_i\in N_{{c_1}}^-(w)\}\ \ \text{and}\  \ 
  I_2:=\{i:v_{i+1}\in N_{{c_2}}^+(w)\}.
$$  
Hence $I_1,I_2\subseteq \bigcup_{i\in [t]}\left[\sum_{j\in [i-1]}s_{j}+i,\sum_{j\in [i]}s_j+i-1\right]$ with $s_0=0$. 
By the maximality of $\mathbf{P}$, we know that  
\begin{itemize}
  \item $I_1\cap I_2= \emptyset$, which implies $|I_1|+|I_2|\leq s_1+\cdots+s_t$,
  \item each vertex in $N_{c_1}^+(w,C^3)$ must be an end vertex of some $e_i$ for $i\in [r]$ (hence   $|N_{c_1}^+(w,C^3)|\leq r$), %cannot be adjacent to the head of $\{e_1,\ldots,e_r\}$ and it can be adjacent to at most $r$ vertices in $C^3$,
  \item each vertex in $N_{c_2}^-(w,C^1)$ must be a start vertex of some $\tilde{Q}_i$ for $i\in [k]$ (hence $|N_{c_2}^-(w,C^1)|\leq k$). % is a subset of  $w$ can only be adjacent  has at most $k$ in-neighbors in $C^1$.
\end{itemize}
Therefore, 
\begin{align*}
  n\leq & d_{D_{c_1}}^+(w)+d_{D_{c_2}}^-(w)\leq s_1+\cdots+s_t+r+|C^3|+|C^2|+k+|C^1|+|C^2|\\
  \leq & 2\gamma n-1+n-|C^4|+|C^2|=n+2\gamma n-1-(2\gamma n+\sigma)=n-1-\sigma,
\end{align*} 
a contradiction. 
\end{proof}

\section{Concluding remarks}
In this paper, we establish a transversal analogue of Ghouila-Houri's theorem \cite{Houri} (Theorem~\ref{Ghouila-Houri}), thereby resolving a problem proposed by Chakraborti, Kim, Lee, and Seo \cite{2023Tournament}  for all sufficiently large $n$. As a consequence, we recover the transversal version of Dirac's theorem, previously obtained by Joos and Kim~\cite{2021jooskim}.

An \textit{anti-directed cycle}  is a digraph in which the underlying graph forms a cycle, and no pair of consecutive edges forms a directed path. DeBiasio and Molla \cite{ADHC} proved {that the} anti-directed Hamilton cycle is guaranteed to appear in a digraph $D$ if $\delta^0(D) \geq \frac{n}{2}+1$. 
DeBiasio, K\"uhn, Molla, Osthus, and Taylor \cite{AOHC} showed that for sufficiently large $n$, every $n$-vertex digraph $D$ with $\delta^0 (D)\geq \frac{n}{2}$ contains every orientation of a Hamilton cycle except, possibly, the anti-directed one. It would be interesting to generalize the above two results to {a} transversal version.

An \textit{oriented graph} 
is a digraph with no cycle of length two. Keevash, K\"uhn, and Osthus \cite{2009KeevashOriented} proved the oriented version of Dirac's theorem, which states that every $n$-vertex digraph with minimum semi-degree at least $\frac{3n-4}{8}$ contains a directed Hamilton cycle. It would be interesting to consider a transversal version of this theorem. 

\begin{question}\label{dirac-oriented}
    Let $\mathcal{D} = \left \{ D_{1}, \ldots , D_{n}\right \}$ be a collection of oriented graphs with common vertex set $V$ of size $n.$ If $\delta^0 (\mathcal{D})\geq \lceil\frac{3n-4}{8}\rceil$, does $\mathcal{D}$ contain a transversal directed  Hamilton cycle?
\end{question}
We remark that our main technique used in this paper cannot be directly generalized to the above oriented case. This is because the absorbing structure we used does not exist in the above oriented graph collections since their semi-degree is much lower than $\frac{n}{2}$. To overcome this difficulty, one has to introduce new stable conditions and build an absorption structure using Lemma \ref{LEMMA:directed-k-graph-matching}. %When the stable conditions does not hold, the oriented graph collections should be close to the extremal graph families and thus one can use structural analysis to prove the tight bound for it. 
It seems that  new ideas are  needed to prove the tight bound in Question \ref{dirac-oriented}.     

\section{Acknowledgment}
We would like to express our gratitude to the anonymous reviewers for their 
valuable comments that greatly improved the presentation of this paper.

\bibliographystyle{abbrv}
\bibliography{Digraph}

\begin{appendices}
\section{Regularity for digraph collections}
In this section, we prove the regularity lemma for digraph collections (i.e., Lemma \ref{regularity-lemma}). Before proceeding with the proof, we first make some preparations, including some necessary definitions and key theorems.

A \textit{$k$-uniform hypergraph} (or \textit{$k$-graph}) $\mathscr{F}=(V,E)$ consists of a vertex set $V$ and an edge set $E$ which is a family of $k$-element subsets of $V$, i.e., $E\subseteq\binom Vk.$ For a vertex $v\in V$, the degree of $v$ in $\mathscr{F}$, denoted by $\deg_{\mathscr{F}}(v)$, is the number of edges containing $v$. Given a $k$-graph $\mathscr{F}$, and $k$ disjoint sets $V_1, V_2,\ldots, V_k \subseteq V(\mathscr{F})$, define 
$\mathscr{F}[V_1, V_2,\ldots, V_k]$ to be the subhypergraph of $\mathscr{F}$ with edge set $\left\{\{v_1,v_2,\ldots,v_k\}\in E(\mathscr{F}):v_i\in V_i\ {\text{for all}}\ i\in [k]\right\}$. We call
\[
d_{\mathscr{F}}(V_1, V_2, \ldots, V_k) := \frac{e(\mathscr{F}[V_1, V_2,\ldots, V_k])}{|V_1||V_2|\ldots|V_k|}
\]
the {\em density} of the tuple $(V_1, V_2, \ldots, V_k)$ in $\mathscr{F}$. If $d_{\mathscr{F}}(V_1, V_2, \ldots, V_k)=0$, then we say that $(V_1, V_2, \ldots, V_k)$ is empty in $\mathscr{F}$. Given $\epsilon > 0$ and $d\in [0,1)$, we say that $(V_1, V_2, \ldots, V_k)$ is $\epsilon$-{\em regular} in $\mathscr{F}$
if for all $i\in [k]$, whenever $V_i'\subseteq V_i$ with $|V_i'|\geq \epsilon|V_i|$ one has  
\begin{center}
$|d_{\mathscr{F}}(V_1', V_2',\ldots, V_k')-d_{\mathscr{F}}(V_1, V_2,\ldots , V_k)|\leq \epsilon$, 
\end{center}
and the tuple is $(\epsilon, d)$-{\em regular} if it is $\epsilon$-{\em regular} and $d_{\mathscr{F}}(V_1, V_2, \ldots, V_k)\geq d$. The following result was proved by Chung \cite{chung1991regularity}, and its proof follows the approach of the original Regularity Lemma for graphs \cite{szemeredi1975regular}.

\begin{theorem}[Weak hypergraph regularity lemma]\label{Weak hypergraph regularity lemma-chung}
For all integers $k\geq 2$, $L_0\geq 
1$, and every $\epsilon > 0$ there exists $N = N(\epsilon, L_0, k)$ such that if $\mathscr{F}$ is a $k$-graph on $n \geq N$ vertices, then $V(\mathscr{F})$ has a partition $V_0,V_1,\ldots, V_L$ such that the following hold:\
\begin{itemize}
    \item $L_0\leq L\leq N$ and $|V_0|\leq \epsilon n$,
    \item $|V_1|=\cdots=|V_L|$,
    \item for all but at most $\epsilon \binom{L}{k}$ $k$-tuples $\{i_1, \ldots, i_k\}\in \binom{[L]}{k}$, we have that $(V_{i_1}, \ldots, V_{i_k})$ is $\epsilon$-regular in $\mathscr{F}$.
\end{itemize}
\end{theorem}

Given two partitions $V_0,V_1,\ldots,V_k$  and $U_1,\ldots,U_{\ell}$ of a vertex set, we say that $V_0,V_1,\ldots,V_k$ \textit{refines} $U_1,\ldots,U_{\ell}$ if for all $V_i$ with $i\in [k]$ there is $U_j$ for some $j\in [\ell]$ that contains $V_i$. 
%We say that a partition $P_0$ refines a partition $P$ if every (non-exceptional) set in $P_0$ is a subset of some set in $P$. 
For some integer $s>0$, let $\mathscr{F}_1, \mathscr{F}_2,\ldots, \mathscr{F}_s$ be $s$ $K$-partite $k$-graphs sharing a common vertex set $V:=V_1\cup \ldots \cup V_K$. As a consequence of the proof of Theorem \ref{Weak hypergraph regularity lemma-chung}, we can deduce the following corollary. 

\begin{corollary}\label{Weak hypergraph regularity lemma-chung-s-graphs}
For all integers $K\geq k\geq 2$, $s\geq 1$, $L_0\geq 
1$, and every $\epsilon > 0$, there exists $N = N(\epsilon, s, L_0, K, k)$ such that if $\mathscr{F}_1, \mathscr{F}_2,\ldots, \mathscr{F}_s$ are $s$ $K$-partite $k$-graphs sharing a common vertex set $V:=V_1\cup \ldots \cup V_K$ of size $n \geq N$, then $V$ can be refined into a partition $U_0,U_1,\ldots, U_L$ satisfying the following properties:
\begin{itemize}
    \item $L_0\leq L\leq N$ and $|U_0|\leq \epsilon n$,
    \item $|U_1|=\cdots=|U_L|$,
    \item for each $i\in [s]$ and for all but at most $\epsilon \binom{L}{k}$ $k$-tuples $\{i_1, \ldots, i_k\}\in \binom{[L]}{k}$, we have that $(U_{i_1}, \ldots, U_{i_k})$ is $\epsilon$-regular in $\mathscr{F}_i$.
\end{itemize}
\end{corollary}

The degree form of the weak hypergraph regularity lemma \cite{townsend2016extremal} is proved in the same way as the original
degree form of the regularity lemma, which in fact can be derived from Theorem \ref{Weak hypergraph regularity lemma-chung} via some cleaning.

\begin{theorem}[Degree form of the weak hypergraph regularity lemma, \cite{townsend2016extremal}]\label{THEOREM:Degree form of the weak hypergraph regularity lemma} 
For all integers $k\geq 2$, $L_0 \geq 1$ and every $\epsilon> 0$, there exists $N=N(\epsilon, L_0, k)$ such that for every $d \in [0, 1)$ and for every $k$-graph $\mathscr{F}$ on $n \geq N$ vertices, there exists a partition of $V(\mathscr{F})$ into $V_0, V_1, \ldots , V_L$ and a spanning subhypergraph $\mathscr{F}'$ of $\mathscr{F}$ such that the following properties hold:
\begin{itemize}
    \item  $L_0 \leq L \leq N$ and $|V_0| \leq \epsilon n$,
    \item $|V_1| = \cdots = |V_L|$,
    \item ${\rm deg}_{\mathscr{F}'}(v) \ge {\rm deg}_{\mathscr{F}}(v) - (d + \epsilon)n^{k-1}$
for all  $v\in V(\mathscr{F})$,
    \item every edge of $\mathscr{F}'$ with more than one vertex in a single cluster $V_i$ for some $i \in [L]$ has at least one vertex in $V_0$,
    \item for all $k$-tuples $\{i_1,i_2,\ldots,i_k\}\in \binom{[L]}{k}$, we have that $(V_{i_1}, V_{i_2},\ldots, V_{i_k})$ is either empty or $(\epsilon, d)$-regular in $\mathscr{F}'$.  
\end{itemize}
\end{theorem}

In fact, from the proof of Theorem \ref{THEOREM:Degree form of the weak hypergraph regularity lemma} in \cite{townsend2016extremal} and Corollary \ref{Weak hypergraph regularity lemma-chung-s-graphs} we can derive the following stronger version of the conclusion.

\begin{theorem}\label{THEOREM:strong-Degree form of the weak hypergraph regularity lemma} 
For all integers $K\geq k\geq 2$, $s\geq1$, $L_0\geq 
1$, and every $\epsilon > 0$, there exists $N = N(\epsilon, s, L_0, K, k)$ such that, for any $s$ $K$-partite $k$-graphs $\mathscr{F}_1, \mathscr{F}_2,\ldots, \mathscr{F}_s$ sharing a common vertex set $V:=V_1\cup \ldots \cup V_K$ of size $n \geq N$, we can find a refinement partition $U_0,U_1,\ldots, U_L$ of $V$, and a spanning subhypergraph $\mathscr{F}_i'\subseteq \mathscr{F}_i$ for each $i\in [s]$ such that the following properties hold:
\begin{itemize}
    \item  $L_0 \leq L \leq N$ and $|U_0| \leq \epsilon n$,
    \item $|U_1| = \cdots = |U_L|$,
    \item for each $i\in [s]$ and each $v\in V$, ${\rm deg}_{\mathscr{F}'_i}(v) \ge {\rm deg}_{\mathscr{F}_i}(v) - (d + \epsilon)n^{k-1}$,
    \item for each $i\in [s]$, every edge of $\mathscr{F}_i'$ with more than one vertex in a single cluster $V_j$ for some $j \in [L]$ has at least one vertex in $V_0$,
    \item for each $i\in [s]$ and all $k$-tuples $\{i_1,i_2,\ldots,i_k\}\in \binom{[L]}{k}$, we have that $(U_{i_1}, U_{i_2},\ldots, U_{i_k})$ is either empty or $(\epsilon, d)$-regular in $\mathscr{F}'_i$.  
\end{itemize}
\end{theorem}
Since the proof of Theorem \ref{THEOREM:strong-Degree form of the weak hypergraph regularity lemma} closely follows that of Theorem \ref{THEOREM:Degree form of the weak hypergraph regularity lemma}, requiring only a repeated application of the cleaning operation on $s$ hypergraphs, we omit it here for brevity.

Let $\mathcal{D}=\{D_i:i\in \mathcal{C}\}$ be a collection of digraphs on a common vertex set $[n]$.  Based on $\mathcal{D}$, we construct the auxiliary $4$-graph $\mathscr{H}$ as follows. {The vertex set is  $V(\mathscr{H})=[n]\cup \mathcal{C} \cup S_1\cup S_2$, where $S_1$ and $S_2$ are new sets with $|S_1|=|S_2|=n$, and the four sets $[n]$, $\mathcal{C}$, $S_1$, $S_2$ are pairwise disjoint.} The edge set of $\mathscr{H}$ is defined by  %. The vertex set of $\mathscr{H}$ is $[n]\cup \mathcal{C} \cup S_1\cup S_2$, where $|S_1|=|S_2|=n$. Suppose that $i,j$ is an edge in $D_c$ for some $i, j\in [n]$ and $c\in \mathcal{C}$. If $i< j$, then $\{i, j, c, x\}\in \mathscr{H}$ for every $x\in S_1$; if $i> j$, then $\{i, j, c, y\}\in \mathscr{H}$ for every $y\in S_2$. 
\begin{align*}
&{E(\mathscr{H})=\left\{\{i,j,c,x\}:i,j\in [n],\ c\in \mathcal{C},\ ij\in E(D_c);\ x\in S_1\ \text{if}\ i<j\ \text{and}\ x\in S_2\ \text{otherwise}\right\}}.%&E(\mathscr{H})=\left\{\{i,j,c,x\}:i,j\in [n]\ \text{with}\ ij\in E(D_c)\ {\text{for some}}\ c\in \mathcal{C}\ \text{and}\ x\in S_1\ \text{if}\ i<j,\ x\in S_2\ \text{if}\ i>j\right\}.
\end{align*}
%We say that a partition $P_0$ refines a partition $P$ if any (non-exceptional) set in $P_0$ is a subset of some set in $P$. 

%Let us provide a concise overview of the proof strategy for Lemma \ref{regularity-lemma}. Fix a digraph collection $\mathcal{D}$.Initially, we utilize Theorem \ref{THEOREM:Degree form of the weak hypergraph regularity lemma} (degree form of the weak hypergraph regularity lemma) to obtain a vertex partition for the auxiliary hypergraph $\mathscr{H}$ constructed based on $\mathcal{D}$. Subsequently, we derive a corresponding partition for the vertex and color sets associated with $\mathcal{D}$. We then verify whether this partition aligns with the conditions stipulated in Lemma \ref{regularity-lemma}. If not,  by Theorem \ref{THEOREM:strong-Degree form of the weak hypergraph regularity lemma}, we iteratively refine the partition until the desired partition is achieved.  
Now, we prove Lemma \ref{regularity-lemma}.
\begin{proof}[Proof of Lemma \ref{regularity-lemma}]
By increasing $L_0$ and decreasing $\epsilon$, $\delta$ as necessary, we may assume that $0< \frac{1}{L_0} \ll \epsilon \ll \delta \ll 1$. Let $L_1:=\frac{24L_0}{\delta}$. 
Choose further positive constants $\epsilon', \alpha$ satisfying $\frac{1}{L_1}\ll \epsilon' \ll  \alpha\ll \epsilon \ll \delta \ll 1$.
Let $n_1$ be derived from Theorem \ref{THEOREM:Degree form of the weak hypergraph regularity lemma}  by applying it with the parameters $(\epsilon', L_1, 4)$. Next, let $n_2$ be obtained from Theorem \ref{THEOREM:strong-Degree form of the weak hypergraph regularity lemma} using the parameters $(\epsilon', 2, L_0, n_1, 3)$. Define $n_0:=\frac{4n_2}{\alpha}$. By increasing $n_0$ if necessary, we may assume that $ L_1\ll n_1\ll n_2\ll n_0$.  Altogether,
\[
0<\frac{1}{n_0}\ll\frac{1}{n_2}\ll\frac{1}{n_1}\ll \frac{1}{L_1}\ll \epsilon' \ll  \alpha\ll \epsilon \ll \delta \ll 1.
\]

Let $n \geq n_0$ be an integer and $\mathcal{D}$ be a digraph collection on vertex set $V:=[n]$ and with color set $\mathcal{C}$, where $\delta n\leq |\mathcal{C}|\leq \frac{n}{\delta}$. Let $d \in [0, 1)$ and let $\mathscr{H}$ be the auxiliary $4$-graph obtained from $\mathcal{D}$ with vertex set $U:=V\cup \mathcal{C}\cup S_1\cup S_2$. Theorem \ref{THEOREM:Degree form of the weak hypergraph regularity lemma} implies that there is a partition $U_0, U_1, \ldots, U_K$ of $U$ and a spanning subhypergraph $\mathscr{H}_0$ of $\mathscr{H}$ such that
\begin{itemize}
    \item [{\rm (a)}] $L_1 \leq K \leq n_1$ and $|U_0| \leq \epsilon' |U|$,
    \item [{\rm(b)}]$|U_1| = \cdots = |U_K| =: m'$,
    \item [{\rm(c)}]${\rm deg}_{\mathscr{H}_0}(v) \ge {\rm deg}_{\mathscr{H}}(v) - (d + \epsilon')|U|^3$
for all  $v\in U$,
    \item [{\rm(d)}] every edge of $\mathscr{H}_0$ with more than one vertex in a single cluster $U_i$ for some $i \in [K]$ has at least one vertex in $U_0$,
    \item [{\rm(e)}] for all $4$-tuples $\{i_1, i_2, i_3, i_4\}\in \binom{[K]}{4}$, we have that $(U_{i_1}, U_{i_2}, U_{i_3}, U_{i_4})$ is either empty or $(\epsilon', d)$-regular in $\mathscr{H}_0$.  
\end{itemize}

%Next, we will make some adjustments to the partition $U_0\cup U_1\cup\ldots \cup U_K$ of $V(\mathscr{H})$. 

Delete all edges $e\in E(\mathscr{H}_0)$ satisfying $|e\cap U_i|\geq 2$ for some $i\in [K]$ and $e\cap(\mathcal{C}\cap U_0)=\emptyset$. Denote the resulting spanning subhypergraph by $\mathscr{H}'$. Note that for every $v \in U$, the number of edges incident to $v$ that are deleted is fewer than
\[
\epsilon'|U|^3+Km'^2|U|\leq\epsilon'|U|^3+m'|U|^2\leq\epsilon'|U|^3+\frac{|U|^3}{L_1}\leq 2\epsilon' |U|^3. 
\]
Combining $(c)$, for each $v\in U$ we have
\begin{align}\notag%\label{ALIGN:D_H_'}
{\rm deg}_{\mathscr{H}'}(v) \ge {\rm deg}_{\mathscr{H}}(v) - (d +3 \epsilon')|U|^3.
\end{align}
Partition each cluster $U_i$, $i\in [K]$, into $\frac{1}{\alpha}$ subclusters of size at most $m_0$ so that all but at most $4$ subclusters of $U_i$ have size exactly $m_0$ and the property that they lie entirely within one of $V$, $\mathcal{C}$, $S_1$ and $S_2$. 
If a subcluster does not have this property, then add it to $U_0$. The new exceptional set has size at most
$|U_0| + 4\alpha m'K$, and let $V_0$ be its intersection with $V$, $\mathcal{C}_0$ be its intersection with $\mathcal{C}$. Relabel the subclusters so that those which are subsets of $V$ are $V_1,\ldots , V_L$ and those which are subsets of $\mathcal{C}$ are $\mathcal{C}_1,\ldots ,\mathcal{C}_M$. For each $c\in \mathcal{C}$, let $D_c'$ be the spanning sub-digraph of $D_c$ with vertex set $V=:[n]$ and edge set 
\[
\left\{uv :u<v,\ x\in S_1\ \text{and}\ \{u,v,c,x\} \in \mathscr{H}' \right\}\cup\left\{vu :u<v,\ y\in S_2\ \text{and}\ \{u,v,c,y\} \in \mathscr{H}' \right\}.
\]
We claim that the properties of Lemma \ref{regularity-lemma} (i)-(iv) are satisfied. 

\medskip
\textbf{Property (i)}. Note that $m' K\leq |U|$. Since $\epsilon'\ll \epsilon$, one has 
\begin{align}\label{ALIGN:CON1}
    |V_0|+|\mathcal{C}_0|\leq |U_0|+4\alpha m' K\leq (\epsilon'+4\alpha)|U|= (\epsilon'+4\alpha)\left(3n+\frac{n}{\delta}\right)\leq \frac{\epsilon n}{2}.
\end{align}
Also, 
\begin{align*}%\label{ALIGN:CON2}
 L+M\leq \frac{2K}{\alpha}\leq \frac{2n_1}{\alpha}\leq \frac{n_0}{2},   
\end{align*}
proving the required upper bound for both $L$ and $M$. Furthermore, $m_0(L+M)\leq n+|\mathcal{C}|$ and 
\[
\frac{L_1}{6}\leq \frac{3n+|\mathcal{C}|}{6m_0}\leq \frac{n+|\mathcal{C}|}{2m_0}\leq \frac{n+|\mathcal{C}|-\epsilon n}{m_0}\leq L+M,
\]
so
\begin{align}\label{ALIGN:CON21}
L= \frac{n-|V_0|}{m_0}\geq \frac{(n-|V_0|)(L+M)}{n+|\mathcal{C}|}\geq \frac{(n-|V_0|)L_1}{6(n+\frac{n}{\delta})}\geq \frac{\delta L_1}{24}=L_0, 
\end{align}
and similarly for $M$.  

\medskip
\textbf{Property (ii)}. It follows by construction.
\medskip

\textbf{Property (iii)}. By the construction of $\mathscr{H}'$, for all $v\in V$ we have 
$$\sum_{c\in \mathcal{C}}\left(d^+_{D_c'}(v)+ d^-_{D_c'}(v)  \right)\geq \frac{{\rm deg}_{\mathscr{H'}}(v)}{n},$$  and $e(D_c')\geq \frac{{\rm deg}_{\mathscr{H}'}(c)}{n}$ for all $c\in \mathcal{C}$. Furthermore,
\begin{align}\label{ALIGN-d+Epn^3}
    (d+3\epsilon')(3n+|\mathcal{C}|)^3\leq \left(\frac{8d}{\delta^3}+\frac{24\epsilon'}{\delta^3}\right)n^3\leq \left(\frac{8d}{\delta^3} +\frac{\epsilon}{4}\right)n^3.
\end{align}
Notice that $\sum_{c\in \mathcal{C}}\left(d^+_{D_c}(v)+ d^-_{D_c}(v)  \right)= \frac{{\rm deg}_{\mathscr{H}}(v)}{n}$
for all $v\in V$. Thus,
\begin{align}\label{ALIGN:CON3}
    \sum_{c\in \mathcal{C}}d^+_{D_c'}(v)\geq&  \frac{{\rm deg}_{\mathscr{H'}}(v)}{n}-\sum_{c\in \mathcal{C}}d^-_{D_c'}(v)  \notag\\
    \geq&  \frac{{\rm deg}_{\mathscr{H}}(v)}{n}-\sum_{c\in \mathcal{C}}d^-_{D_c}(v)-\left(\frac{8d}{\delta^3} +\frac{\epsilon}{4}\right)n^2  \notag\\
    =& \sum_{c\in \mathcal{C}}d^+_{D_c}(v)-\left(\frac{8d}{\delta^3} +\frac{\epsilon}{4}\right)n^2
\end{align}
for all $v\in V$, and similarly for $\sum_{c\in \mathcal{C}}d^-_{D_c'}(v)$.
Additionally, observe that $e(D_c)= \frac{{\rm deg}_{\mathscr{H}}(c)}{n}$ for each  $c\in \mathcal{C}$. It follows that for each $c\in \mathcal{C}$,  
\begin{align}\label{ALIGN:CON31}
    e(D_c')\geq \frac{{\rm deg}_{\mathscr{H}'}(c)}{n}\geq \frac{{\rm deg}_{\mathscr{H}}(c)}{n}-\left(\frac{8d}{\delta^3} +\frac{\epsilon}{4}\right)n^2= e(D_c)-\left(\frac{8d}{\delta^3} +\frac{\epsilon}{4}\right)n^2.
\end{align}

\medskip
\textbf{Property (iv).} 
 It follows by construction.
 \medskip
 
%We will show that if some digraph $D_c'$ has an edge with both vertices in a single cluster $V_i$ for some $i\in [L]$, then $c\in \mathcal{C}_0$. Suppose that $D_c'$ has an edge $uv$ with $u, v \in V_i$ for some $i \in [L]$. Since $V_i\subseteq U_{i'}$ for some $i'\in [K]$, it follows that $u, v \in U_{i'}$. By the definition of $D_c'$ there exists $z\in S_1\cup S_2$ such that $\{u, v, c, z\}\in \mathscr{H}'$.  According to property (d), we have $c\in U_0$ or $z \in U_0$. If $c\in U_0$, then $c\in \mathcal{C}_0$ and we are done. Otherwise, $z\in U_0$ and $c\notin U_0$. Without loss of generality, suppose that $u<v$, implying $z\in S_1$. Since $|U_0|\ll n=|S_1|$, there exists $z'\in S_1\setminus U_0$ such that $\{u, v, c, z'\}\in \mathscr{H}'$. Notice that $\{u, v, c, z\}\cap U_0 =\emptyset$. This is a contradiction to property (d). Thus, we conclude that $c \in \mathcal{C}_0$, as required.

\textbf{Property (v).} %If for all triples $\{(h,i),j\}\in \left({[L]\choose 2}\times [M]\right)$, we have either $D_c'[V_h,V_i]=\emptyset$ for all $c\in \mathcal{C}_j$, or $\mathcal{D}_{hi,j}':=\{D_c'[V_h,V_i]:c\in \mathcal{C}_j\}$ is $(\epsilon,d)$-regular, then we are done. Otherwise, 
Suppose that $(V_0\cup \mathcal{C}_0, V_1, \ldots, V_L, \mathcal{C}_1,\ldots, \mathcal{C}_M)$ does not satisfy property (v), we are to refine it and get a desired partition. Now, define the following auxiliary $3$-graphs $\mathscr{H}_1$ and $\mathscr{H}_2$ on the common vertex set $U':=([n]\cup \mathcal{C})\setminus (V_0\cup \mathcal{C}_0)$ and 
\begin{center}
    $E(\mathscr{H}_1)= \left\{\{u, v, c\} : uv\in D'_c[V_i, V_j], u, v \in V\setminus V_0, c\in \mathcal{C}\setminus \mathcal{C}_0, 1\leq i<j\leq L  \right\}$,
    $E(\mathscr{H}_2)=  \left\{\{u, v, c\} : uv\in D'_c[V_j, V_i], u, v \in V\setminus V_0, c\in \mathcal{C}\setminus \mathcal{C}_0, 1\leq i<j\leq L   \right\}$.
\end{center} 
Notice that the partition $P=(V_1, \ldots, V_L, \mathcal{C}_1, \ldots, \mathcal{C}_M)$ {can} be viewed as two distinct partitions  $P_1$, $P_2$ for the  hypergraphs$\mathscr{H}_1$, $\mathscr{H}_2$ (both partitions divide $V$ and $\mathcal{C}$ in the same way,  but the sets of edges among the partition sets are different). By increasing $n_2$, we may assume that $n_2\gg N(\epsilon', 2, L_0, L+M, 3)$ (which is defined in Theorem \ref{THEOREM:strong-Degree form of the weak hypergraph regularity lemma}). Consequently, $n_0\gg N(\epsilon', 2, L_0, L+M, 3)$ as well. Applying Theorem~\ref{THEOREM:strong-Degree form of the weak hypergraph regularity lemma}, we obtain a refinement \(U_0',U_1',\dots,U_{K'}'\) of the partition \(P\) together with spanning subhypergraphs \(\mathscr{H}_1'\subseteq\mathscr{H}_1\) and \(\mathscr{H}_2'\subseteq\mathscr{H}_2\) that satisfy the following: 
\begin{itemize}
    \item [(a')]  $L_0 \leq K' \leq n_2$ and $|U_0'| \leq \epsilon' |U'|$,
    \item [(b')] $|U'_1| = \cdots = |U'_{K'}|$,
    \item [(c')]for each $i\in [2]$ and each $v\in V\setminus V_0$, ${\rm deg}_{\mathscr{H}'_i}(v) \ge {\rm deg}_{\mathscr{H}_i}(v) - (d + \epsilon')|U'|^{2}$,
    \item [(d')] for each $i\in [2]$, every edge of $\mathscr{H}_i'$ with more than one vertex in a single cluster $U'_j$ for some $j \in [K']$ has at least one vertex in $U'_0$,
    \item [(e')] for each $i\in [2]$ and all $3$-tuples $\{i_1,i_2, i_3\}\in \binom{[K']}{3}$, we have that $(U_{i_1}', U_{i_2}',U_{i_3}')$ is either empty or $(\epsilon', d)$-regular in $\mathscr{H}'_i$.  
\end{itemize}
Following the same operation as earlier, let $V'_0:=V\cap U_0'$ and $\mathcal{C}'_0:=\mathcal{C}\cap U_0'$. Relabel the subclusters such that those which are subsets of $V$ are $V'_1,\ldots , V'_{L'}$ and those which are subsets of $\mathcal{C}$ are $\mathcal{C}_1',\ldots ,\mathcal{C}_{M'}'$. For each $c\in \mathcal{C}\setminus (\mathcal{C}_0\cup \mathcal{C}_0')$, let $D_c''$ be the spanning sub-digraph of $D_c$ with vertex set $V$ and edge set 
\[
\left\{uv : \{u,v,c\} \in \mathscr{H}_1'\,\,\text{or}\,\,\{u,v,c\} \in \mathscr{H}_2'  \right\}\cup E(D_c'^{\pm}[V_0,V\setminus (V_0\cup V_0')]).
\]
%where
%\[
%Z=: \{uv\in D'_c: u\in V_0\,\,\text{and}\,\, v\in V\setminus (V_0\cup V_0')\,\,\text{or}\,\,v\in V_0\,\,\text{and}\,\, u\in V\setminus (V_0\cup V_0')\}.
%\]
Let $D_i^*:=D_i'$ for $i\in \mathcal{C}_0\cup \mathcal{C}_0'$ and $D_i^*:=D_i''$ for $i\in \mathcal{C}\setminus (\mathcal{C}_0\cup \mathcal{C}_0')$. Let
$$P^*:=(V_0\cup \mathcal{C}_0\cup V_0'\cup \mathcal{C}_0', V_1',\ldots, V_{L'}', \mathcal{C}_1'\ldots, \mathcal{C}_{M'}')$$
be a partition of $V\cup \mathcal{C}$. 
Next, we will re-verify that $P^*$ and digraph collection $\mathcal{D}^*:=\{D_i^* : i\in \mathcal{C}\}$ satisfies (i)-(v) of Lemma \ref{regularity-lemma}. By the construction of $\mathscr{H}_1$ and $\mathscr{H}_2$, it is easy to see that $P^*$ and $\mathcal{D^*}$ satisfy the property (v). Thus, we only need to re-verify properties (i)-(iv).

\medskip
\textbf{Property (i).} By \eqref{ALIGN:CON1}, \eqref{ALIGN:CON21} and (a'), we have
\[|V_0|+|\mathcal{C}_0|+|V_0'|+|\mathcal{C}_0'|\leq \frac{\epsilon n}{2}+\epsilon' |U'|\leq \epsilon n,\]
and
\[
L_0\leq L \leq L' \leq n_2 \leq  n_0\ \ \text{and}\ \ L_0\leq M \leq M' \leq n_2 \leq  n_0.
\]
\medskip

\textbf{Property (ii).} It follows from (b') that $|V_1'|=\cdots=|V_{L'}'|=|\mathcal{C}_1'|=\cdots=|\mathcal{C}_{M'}'|$, as required.
\medskip

\textbf{Property (iii).} For each vertex $v\in V$ and each color $c\in \mathcal{C}$, by the construction of $D_{c}^*$, (c'), \eqref{ALIGN-d+Epn^3}, \eqref{ALIGN:CON3} and \eqref{ALIGN:CON31}, we have 
\[
 e(D_c^*)\geq e(D_c)-\left(\frac{8d}{\delta^3} +\frac{\epsilon}{4}\right)n^2-2\left(\frac{8d}{\delta^3} +\frac{\epsilon}{4}\right)n^2-\frac{4\epsilon' n^2}{\delta}\geq e(D_c)-\left(\frac{24d}{\delta^3} +\epsilon\right)n^2
\]
and
\[
\sum_{c\in \mathcal{C}}d^+_{D_c^*}(v)\geq \sum_{c\in \mathcal{C}}d^+_{D_c}(v)-\left(\frac{8d}{\delta^3} +\frac{\epsilon}{4}\right)n^2-2\left(\frac{8d}{\delta^3} +\frac{\epsilon}{4}\right)n^2-\frac{2\epsilon' n^2}{\delta^2}\geq \sum_{c\in \mathcal{C}}d^+_{D_c}(v)-\left(\frac{24d}{\delta^3} +\epsilon\right)n^2.
\]
Analogously, the same inequality holds for $\sum_{c\in \mathcal{C}}d^-_{D_c^*}(v)$. 
\medskip

\textbf{Property (iv).} Notice that $\mathscr{H}_1$ and $\mathscr{H}_2$ are $(L+M)$-partite hypergraphs. Thus, there is no edge of $\mathscr{H}_i'$ with more than one vertex in a single cluster $U'_j$ for some $j\in [K']$. This implies that $P^*$ and $\mathcal{D}^*$ satisfy (iv).

This completes the proof of Lemma \ref{regularity-lemma}.
\end{proof}

\section{Transversal perfect matching in bipartite graph  collections}
\subsection{Proof of Lemma \ref{LEMMA:MATCHING-CONS}}
In this subsection we prove Lemma \ref{LEMMA:MATCHING-CONS}, which provides a characteristic partition of each balanced bipartite graph whose minimum degree is just below $\frac{n}{2}$, allowing a tiny fraction of vertices that may have even smaller degrees.
%\begin{lemma}\label{LEMMA:MATCHING-CONS}
%Suppose that $0<1/n\ll d\ll\epsilon \leq 1$. Let $G$ be a balanced bipartite graph on vertex set $V=V_1\cup V_2$ of size $2n$ with $d_{G}(x)\geq (1/2-\epsilon^3)n$ for all but at most $dn$ vertices $x\in V$ which is $\epsilon$-extremal. Then there is a characteristic partition $(A_1, B_1, A_2, B_2, C_1, C_2)$ of $G$ such that the following holds:
%\begin{itemize}
 %   \item[(1)] $A_i, B_i, C_i\subseteq V_i$, $|A_i|=|B_i|=(\frac{1}{2}-\epsilon)n$ and $|C_i|=2\epsilon n$ for $i=1, 2$;
 %   \item[(2)] $d_G(v, X_i)\geq (\frac{1}{2}-2\epsilon)n$ for $v\in X_{3-i}$ and $X\in \{A, B\}$, $i\in [2]$;
 %   \item[(3)] either $e(G[A_1, B_2])\leq \epsilon n^2$ or $e(G[A_2, B_1])\leq \epsilon n^2$.
%\end{itemize}
%\end{lemma}
\begin{proof}[Proof of Lemma \ref{LEMMA:MATCHING-CONS}]
We first add vertices with  degree at most $\left(\frac{1}{2}-\epsilon^5\right)n$ to $C_1\cup C_2$. Let $\mu =\epsilon^5$. Since $G$ is $\epsilon$-extremal, there are $X\subseteq V_1,\,Y\subseteq V_2$ each of size at least $\left(\frac{1}{2}-\mu\right)n$ such that $e_G(X, Y) \leq  \mu n^2$. Thus, it is easy to see that for all but at most $\sqrt{\mu}n$ vertices $v$ in $X$ we have $d_G(v, Y) \leq \sqrt{\mu}n$. Similarly, for all but at most $\sqrt{\mu}n$ vertices $v$ in $Y$ we have $d_G(v, X) \leq \sqrt{\mu}n$. We add these exceptional vertices from $X$ and from $Y$ to $C_1$ and $C_2$, respectively. Now, 
 the new sets $X$ and $Y$ each has size at least 
 $\left(\frac{1}{2}-\mu \right)n-\sqrt{\mu}n\geq \left(\frac{1}{2}-2\sqrt{\mu}\right)n.$ 
 Choose $X_1\subseteq X, Y_1\subseteq Y$ such that $|X_1|=|Y_1|=\left(\frac{1}{2}-2\sqrt{\mu}\right)n$. 
 
 Notice that $|C_i|\leq (d+\sqrt{\mu})n\leq 2\sqrt{\mu}n$ by $d\ll \epsilon$. Thus, we can choose $Z_1\subseteq V_1\setminus (X_1\cup C_1)$ and $Z_2\subseteq V_2\setminus (Y_1\cup C_2)$ such that $|Z_1|=|Z_2|=(\frac{1}{2}-2\sqrt{\mu})n$. Recall that $d_G(v, Y_1)\leq \sqrt{\mu}n$ for each vertex $v\in X_1$. This implies that for each $v\in X_1$, one has 
 $$d_G(v, Z_2)\geq \left(\frac{1}{2}-\mu-5\sqrt{\mu}\right)n-|C_2| \geq \left(\frac{1}{2}-8\sqrt{\mu}\right)n.$$
Similarly, $d_G(v,  Z_1) \geq \left(\frac{1}{2}-8\sqrt{\mu}\right)n$ for each $v \in Y_1$. 

Since the number of edges between $X_1$ and $Z_2$ is at least $(\frac{1}{2}-2\sqrt{\mu})n\times \left(\frac{1}{2}-8\sqrt{\mu}\right)n$ and $|Z_2|=(\frac{1}{2}-2\sqrt{\mu})n$, there are at most $50\mu^{\frac{1}{4}}n$ vertices $v$ in $Z_2$ such that $d_G(v, X_1)\leq (\frac{1}{2}-10\mu^{\frac{1}{4}})n$. 
Similarly, for all but at most $50\mu^{\frac{1}{4}}n$ vertices $v$ in $Z_1$ we have $d_G(v, Y_1) \geq (\frac{1}{2}-10\mu^{\frac{1}{4}})n$. Thus, we can choose $A_1\subseteq X_1$, $B_1\subseteq Z_1$, $A_2\subseteq Y_1$, $B_2\subseteq Z_2$ such that $|A_1|=|B_1|=|A_2|=|B_2|=(\frac{1}{2}-\epsilon)n$ and $d_G(v, X_i)\geq (\frac{1}{2}-2\epsilon)n$ for $v\in X_{3-i}$ and $X\in \{A, B\}$, $i\in [2]$. Let $C_1=V_1\setminus (A_1\cup B_1)$ and $C_2=V_2\setminus (A_2\cup B_2)$. Hence, $(A_1, B_1, C_1, A_2, B_2, C_2)$ is the desired partition of $G$. 
\end{proof}

\subsection{Proof of Theorem \ref{stable-matching}}
In this subsection we prove Theorem \ref{stable-matching}, giving a sufficient condition for the existence of a transversal perfect matching in bipartite graph collections. Before proceeding with the proof, we first present some key lemmas that will be used in the proof. 

Assume that $\mathcal{G}=\{G_1,\ldots,G_n\}$ is a collection of bipartite graphs on a common vertex partition $V_1\cup V_2$ with $|V_1|=|V_2|=n$.  For $\mu >0$, we say that $\mathcal{G}$ is  $\mu$-{\em nice} if for every $A\subseteq V_1$ and $B\subseteq V_2$ of size $\lfloor \frac{n}{2} \rfloor$, we have $e_{\mathcal{G}}(A, B):=\sum_{i\in [n]}e_{G_i}(A,B) \geq \mu n^3$.% Recall the definition of $(\gamma, \alpha, \epsilon, \delta)$-stable. The following lemma shows that if $\mathcal{G}$ is $(\gamma, \alpha, \epsilon, \delta)$-stable, then $\mathcal{G}$ is $\mu$-nice for some sufficiently small $\mu$.

\begin{lemma}\label{LEMMA:stable-nice-matching}
    Suppose that $0< \frac{1}{n}\ll\mu'\ll \alpha \ll  \gamma, {\epsilon}\ll \delta\ll 1$. Let $\mathcal{G}=\{G_1,\ldots,G_n\}$ be a collection of bipartite graphs on a common vertex partition $V=V_1\cup V_2$ with $|V_1|=|V_2|=n$. If $\mathcal{G}$ is $(\gamma, \alpha, \epsilon, \delta)$-stable, then $\mathcal{G}$ is $\mu'$-nice. 
\end{lemma}
\begin{proof}
    For the sake of contradiction, suppose  that $\mathcal{G}$ is not $\mu'$-nice. Then there are sets $A\subseteq V_1$ and $B\subseteq V_2$, each of size $\lfloor \frac{n}{2} \rfloor$, such that $e_{\mathcal{G}}(A, B)\leq \mu' n^3$. Therefore, for all but at most $\sqrt{\mu'}n$ colors $i\in [n]$, we have $e_{G_i}(A, B)< \sqrt{\mu'} n^2$, implying that such  $G_i$ is not $\sqrt{\mu'}$-nice. This further implies that $\mathcal{G}$ is not $(1-\sqrt{\mu'}, \sqrt{\mu'})$-{strongly} stable and hence not $(\gamma, \alpha)$-{strongly} stable, since $\mu'\ll \gamma, \alpha$. 
    
    Without loss of generality, assume that for each $i\in [(1-\sqrt{\mu'})n]$, $G_i$ is not $\sqrt{\mu'}$-nice. Since $\epsilon > \mu'^{1/10}$, $G_i$ is $\epsilon$-extremal for all $i\in [(1-\sqrt{\mu'})n]$. By Lemma \ref{LEMMA:MATCHING-CONS}, the graph $G_i$ has a characteristic partition, denoted as $(A_1^i, B_1^i, C_1^i, A_2^i, B_2^i, C_2^i)$. It follows that, after possible relabeling $A_1^i, B_1^i, A_2^i, B_2^i$, for all but at most $\delta^2 n$ colors $i\in [(1-\sqrt{\mu'})n]$ we have 
    $$|A\triangle A_1^i|\leq \frac{\delta n}{2}\ \ \text{and}\ \  |B\triangle B_2^i|\leq \frac{\delta n}{2}.$$
     Otherwise, $e_{\mathcal{G}}(A,B)\geq \frac{\delta^3 n^3}{8}> \mu' n^3$ by $\mu'\ll\epsilon \ll \delta$, 
     %, we have at least     $\frac{\delta^3 n^3}{4}> \mu' n^3$    edges of $\mathcal{G}$ inside $\mathcal{G}[A, B]$,
     a contradiction. Delete such colors from $[(1-\sqrt{\mu'})n]$ and let $\mathcal{C}$ be the remaining color set. Thus, for every $i, j\in \mathcal{C}$ the graphs $G_i$ and $G_j$ are not $\delta$-crossing. Therefore, 
$$
e(C_{\mathcal{G}}^{\epsilon, \delta})\leq \sqrt{\mu'}n^2+\delta^2 n^2<\delta n^2\  \text{since}\ \mu'\ll \delta.
$$ 
Thus, $\mathcal{G}$ is not $(\epsilon, \delta)$-weakly stable, which contradicts the assumption that $\mathcal{G}$ is $(\gamma, \alpha, \epsilon, \delta)$-stable.
\end{proof}

The following result shows that if a bipartite graph collection $\mathcal{G}$ is $\mu$-nice under the minimum degree constraint, then we can find a large rainbow matching in $\mathcal{G}$.
 
\begin{lemma}\label{LEMMA:N-2-MATCHING}
Let $0 < \frac{1}{n} \ll d \ll \mu \ll 1$. Suppose that $\mathcal{G}=\{G_1,\ldots,G_n\}$ is a collection of bipartite graphs on a common vertex partition $V=V_1\cup V_2$ with $|V_1|=|V_2|=n$, and for each vertex $v \in V$ we have $d_{G_i}(v)\geq \left(\frac{1}{2}-\mu\right)n$ for all but at most $dn$ colors $i \in[n]$. If $\mathcal{G}$ is $6\mu$-nice, then $\mathcal{G}$ contains a rainbow matching $\mathcal{M}$ with $e(\mathcal{M}) \geq n-dn-2$. 

%Moreover, if for every $i\in [n]$, $d_{G_i}(x)\geq \left(\frac{1}{2}-\mu\right)n$ for all but at most $dn$ vertices $x\in V$, then we can find a rainbow matching $\mathcal{M}$ of size at least $n-dn-2$ such that the following holds:
%Suppose $V_1\setminus V(\mathcal{M})=\{u_1, u_2, \ldots, u_{t}\}$ and $[n]\setminus {\rm col}(\mathcal{M})=\{c_1, c_2, \ldots, c_{t}\}$. %Then, $d_{G_{c_i}}(u_i)\geq \left(\frac{1}{2}-\mu\right)n$ for each $i\in [t]$.
\end{lemma}

\begin{proof}
Let $\mathcal{M}$ be a maximum rainbow matching in $\mathcal{G}$, and suppose for a contradiction that $e(\mathcal{M})\leq n-dn-3$. Since for each vertex $v \in V$ we have $d_{G_i}(v)\geq \left(\frac{1}{2}-\mu\right)n$ for all but at most $dn$ colors $i \in[n]$, there are vertices $v_1\in V_1\setminus V(\mathcal{M}), v_2\in V_2\setminus V(\mathcal{M})$ and distinct colors $c_1, c_2\notin {\rm col}(\mathcal{M})$ 
such that $d_{G_{c_{\ell}}}(v_{\ell})\geq \left(\frac{1}{2}-\mu\right)n$ for $\ell\in \{1, 2\}$. Put $N_{\ell}:=N_{G_{c_{\ell}}}(v_{\ell})$. Thus, $|N_{\ell}|\geq \left(\frac{1}{2}-\mu\right)n$. Let $\mathcal{C} := [n] \setminus ({\rm col}(\mathcal{M}) \cup \{c_1, c_2\})$. 
For $x \in V(\mathcal{M})$, let $x^+$ denote the unique neighbour of $x$ in $\mathcal{M}$, and for $A \subseteq V(\mathcal{M})$, let $A^+ := \{x^+ : x \in A\}$. 

By the maximality of $\mathcal{M}$, we have 
$N_1\cap N_2^+=\emptyset,\, N_2\cap N_1^+=\emptyset$ and $N_1,N_2\subseteq V(\mathcal{M})$. % Notice that there is no matching edge $e_c \in E(G_c[N_1, N_2])$ for some $c\in {\rm col}(\mathcal{M})$ and $N_{\ell}\subseteq V(\mathcal{M})$ for $l\in \{1,2\}$ by the maximality of $\mathcal{M}$.  
This implies that
$$e(\mathcal{M})\geq |N_1|+|N_2|\geq 2\left(\frac{1}{2}-\mu\right)n=n-2\mu n.$$
Moreover, for each $c\in \mathcal{C}$, we have $e_{G_c}(N_1^+, N_2^+)=0$ by the maximality of $\mathcal{M}$. Thus, it suffices to show that for all $c\in {\rm col}(\mathcal{M}[N_1, N_1^+]\cup \mathcal{M}[N_2, N_2^+])$, we have $e_{G_c}(N_1^+, N_2^+)\leq 4\mu n^2$. If this holds, then  
$$e_{\mathcal{G}}(N_1^+, N_2^+)\leq 4\mu n^2(n-2\mu n)+2\mu n^3 < 6\mu n^3,$$
contradicting the assumption that $\mathcal{G}$ is  $6\mu$-nice. Next, we consider the case $c\in {\rm col}(\mathcal{M}[N_1, N_1^+])$;  the proof for the case $c\in {\rm col}(\mathcal{M}[N_2, N_2^+])$ is similar and is omitted.

Let $c\in {\rm col}(\mathcal{M}[N_1, N_1^+])$ and $w\in N_1$ such that ${\rm col}(ww^+)=c$. Now let $\mathcal{M}' = \mathcal{M} - w w^+ + v_1w$ be a rainbow matching with ${\rm col}(\mathcal{M}') = ({\rm col}(\mathcal{M}) \setminus \{c\})\cup \{c_1\}$. Clearly, $\mathcal{M}'$ is a maximum rainbow matching. %, otherwise the maximality of $\mathcal{M}$ would be contradicted. 
Notice that there exists a color $c_3\in \mathcal{C}\setminus \{c_1,c_2\}$ such that $d_{G_{c_3}}(w^+)\geq \left(\frac{1}{2}-\mu\right)n$. Let $N_3:=N_{G_{c_3}}(w^+)$. Consider the new vertex pair $(w^+, v_2)$ and color pair $(c_3, c_2)$. By the maximality of $\mathcal{M}'$, we have $N_3^+\cap N_2=\emptyset$. Recall that $N_1^+\cap N_2=\emptyset$ and $|N_i|\geq \left(\frac{1}{2}-\mu\right)n$ for each $i\in [3]$. Thus, 
\begin{equation}\label{N3deltaN1}
    |N_3^+\triangle N_1^+|= |N_3^+\cup N_1^+|-|N_3^+\cap N_1^+|\leq (n-|N_2|)-(|N_1|+|N_3|-(n-|N_2|))\leq 4\mu n.
\end{equation}
On the other hand, by the maximality of $\mathcal{M}'$, we have $e_{G_c}(N_3^+, N_2^+)=0$. Thus, together with \eqref{N3deltaN1}, we have $e_{G_{c}}(N_1^+, N_2^+)\leq 4\mu n^2$, as desired.
\end{proof}

Let $n$ be a positive integer and let $d, \epsilon>0$ such that $0<\frac{1}{n}\ll d\ll \epsilon\leq 1$. Consider a bipartite graph collection  $\mathcal{G}=\{G_1,\ldots,G_n\}$ on a common vertex partition $V=V_1\cup V_2$ with $|V_1|=|V_2|=n$ such that $d_{G}(x)\geq (\frac{1}{2}-\epsilon^5)n$ for all but at most $dn$ vertices $x\in V$. Lemma \ref{LEMMA:MATCHING-CONS} implies that for each $i\in [n]$ if $G_i$ is $\epsilon$-extremal, then we can fix a characteristic partition
\begin{center}
 $(A_1^i,B_1^i,C_1^i,A_2^i,B_2^i,C_2^i)$
of $G_i$.   
\end{center}
A vertex $v\in V$ is called $i$-{\em good} if either $G_i$ is $\epsilon$-extremal and $v\in V\setminus (C^1_i\cup C^2_i)$ or $G_i$ is not $\epsilon$-extremal and $d_{G_i}(v)\geq (\frac{1}{2}-\epsilon^3)n$. Note that for each $i\in [n]$ and $j\in\{1,2\}$, there are at most $2\epsilon n$ vertices in $V_j$ that are not $i$-good.
 
\begin{definition}[absorbing edge, absorbing matching]  Let $\mathcal{G}=\{G_1,\ldots,G_n\}$ be a collection of bipartite graphs on a common vertex partition $V=V_1\cup V_2$ with $|V_1|=|V_2|=n$. 
Given two vertices $u\in V_1$ and $v\in V_2$, an edge $ e=w_1w_2$ with $u, v \not\in \{w_1,w_2\}$ is called a $c$-\textit{absorbing edge} of $(u, v)$ if $c \in L(w_1v)$ and ${\rm col}(w_1w_2) \in L(uw_2)$ (see Figure \ref{absorbing-edge}).

Given $\delta, \delta', \gamma, \gamma'\ge 0$, a rainbow matching $\mathcal{M}$ is an \textit{absorbing matching} with parameters $(\delta, \delta', \gamma, \gamma')$ if $|E(\mathcal{M})| \leq \gamma n$ and there exists a color set $\mathcal{C}$ of size at least $\delta n$ such that given any color $c \in \mathcal{C}$ and any $c$-good vertex $v\in V_2$, for all but at most $\delta' n$ vertices $u \in V_1$, there are at least $\gamma' n$ $c$-absorbing edges of $(u, v)$ within $\mathcal{M}$.
\end{definition}

\begin{figure}[htbp]
    \centering
\tikzset{every picture/.style={line width=0.75pt}} %set default line width to 0.75pt        

%\tikzset{every picture/.style={line width=0.75pt}} %set default line width to 0.75pt        

\begin{tikzpicture}[x=0.75pt,y=0.75pt,yscale=-1.1,xscale=1.1]
%uncomment if require: \path (0,235); %set diagram left start at 0, and has height of 235

%Straight Lines [id:da45445843389178053] 
\draw [color={rgb, 255:red, 144; green, 19; blue, 254 }  ,draw opacity=1 ]   (207,152.9) -- (293.78,153.06) ;
%Straight Lines [id:da23501259873745894] 
\draw [color={rgb, 255:red, 144; green, 19; blue, 254 }  ,draw opacity=1 ]   (293.78,153.06) -- (293.78,80.06) ;
%Straight Lines [id:da41790946133706] 
\draw [color={rgb, 255:red, 201; green, 208; blue, 2 }  ,draw opacity=1 ][fill={rgb, 255:red, 248; green, 231; blue, 28 }  ,fill opacity=1 ]   (206.78,80.06) -- (207,152.9) ;
%Shape: Circle [id:dp28412354607667023] 
\draw  [fill={rgb, 255:red, 0; green, 0; blue, 0 }  ,fill opacity=1 ] (204,152.9) .. controls (204,151.25) and (205.34,149.9) .. (207,149.9) .. controls (208.66,149.9) and (210,151.25) .. (210,152.9) .. controls (210,154.56) and (208.66,155.9) .. (207,155.9) .. controls (205.34,155.9) and (204,154.56) .. (204,152.9) -- cycle ;
%Shape: Circle [id:dp8775000545882401] 
\draw  [fill={rgb, 255:red, 0; green, 0; blue, 0 }  ,fill opacity=1 ] (290.78,153.06) .. controls (290.78,151.4) and (292.12,150.06) .. (293.78,150.06) .. controls (295.43,150.06) and (296.78,151.4) .. (296.78,153.06) .. controls (296.78,154.71) and (295.43,156.06) .. (293.78,156.06) .. controls (292.12,156.06) and (290.78,154.71) .. (290.78,153.06) -- cycle ;
%Shape: Circle [id:dp782095066838151] 
\draw  [fill={rgb, 255:red, 0; green, 0; blue, 0 }  ,fill opacity=1 ] (203.78,80.06) .. controls (203.78,78.4) and (205.12,77.06) .. (206.78,77.06) .. controls (208.43,77.06) and (209.78,78.4) .. (209.78,80.06) .. controls (209.78,81.71) and (208.43,83.06) .. (206.78,83.06) .. controls (205.12,83.06) and (203.78,81.71) .. (203.78,80.06) -- cycle ;
%Shape: Circle [id:dp39058111533784423] 
\draw  [fill={rgb, 255:red, 0; green, 0; blue, 0 }  ,fill opacity=1 ] (290.78,80.06) .. controls (290.78,78.4) and (292.12,77.06) .. (293.78,77.06) .. controls (295.43,77.06) and (296.78,78.4) .. (296.78,80.06) .. controls (296.78,81.71) and (295.43,83.06) .. (293.78,83.06) .. controls (292.12,83.06) and (290.78,81.71) .. (290.78,80.06) -- cycle ;

% Text Node
\draw (288,62) node [anchor=north west][inner sep=0.75pt]   [align=left] {$u$};
% Text Node
\draw (201,62) node [anchor=north west][inner sep=0.75pt]   [align=left] {$v$};
% Text Node
\draw (199,163) node [anchor=north west][inner sep=0.75pt]   [align=left] {$w_1$};
% Text Node
\draw (288,165) node [anchor=north west][inner sep=0.75pt]   [align=left] {$w_2$};
% Text Node
\draw (195,110) node [anchor=north west][inner sep=0.75pt]   [align=left] {$c$};

\end{tikzpicture}
\caption{$w_1w_2$ is a $c$-absorbing edge of $(u, v)$.}
\label{absorbing-edge}
\end{figure}
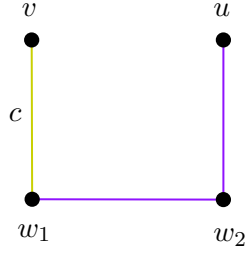

 The following two lemmas establish the existence of absorbing matchings when $\mathcal{G}$ is stable.  %show that if $\mathcal{G}$ is $(\gamma, \alpha, \epsilon, \delta)$-stable, then $\mathcal{G}$ contains an  absorbing matching. 
 We first consider the case that $\mathcal{G}$ is strongly stable.

\begin{lemma}\label{strongly-matching}
    Let $0<\frac{1}{n}\ll d\ll \lambda ,\mu\ll \gamma,\alpha\ll 1$. Suppose that $\mathcal{G}=\{G_1,\ldots,G_n\}$ is a collection of bipartite graphs on a common vertex partition $V=V_1\cup V_2$ with $|V_1|=|V_2|=n$, satisfying the following conditions:
   \begin{itemize}
    \item for every $i\in [n]$, $d_{G_i}(x)\geq \left(\frac{1}{2}-\mu\right)n$ for all but at most $dn$ vertices $x\in V_1\cup V_2$, 
    \item for every $x\in V_1\cup V_2$, $d_{G_i}(x)\geq \left(\frac{1}{2}-\mu\right)n$ for all but at most $dn$ colors $i\in [n]$.
\end{itemize}
If $\mathcal{G}$ is $(\gamma,\alpha)$-strongly stable, then there exists an absorbing matching with parameters $(1,0,\lambda,\lambda^2)$.
\end{lemma}
\begin{proof}
Without loss of generality, assume that $G_1,\ldots,G_{\gamma n}$ are $\alpha$-nice. For each $i\in[\lambda n]$, define 
$$
\mathcal{F}_i:=\left\{(u,v)\in V_1\times V_2 : i\in L(uv)\right\}.
$$
The degree condition of $G_i$ implies that 
$|\mathcal{F}_i|=e(G_i)\geq (n-dn) \left(\frac{1}{2}-\mu\right)n> \frac{1}{4}n^2.$ Fix a color $c\in [n]$ 
and two vertices $u\in V_1, v\in V_2$ such that $v$ is $c$-good, let
$$Z_i(c, uv):=\big\{ (w,w')\in \mathcal{F}_i : ww' \,\, \text{is a}\,\,c\text{-absorbing edge of}\,\, (u, v)\big\}\ \ \text{and}\ \ Z(c,uv):=\bigcup_{i\in [\lambda n]}Z_i(c,uv).$$
Since $G_{i}$ is $\alpha$-nice, $e(G_i[N_{G_i}(u), N_{G_c}(v)])\geq \frac{\alpha n^2}{2}$ for all but at most $dn$ indices $i
\in [\gamma n]$. This means that for all but at most $dn$ indices $i\in [\gamma n]$, we have $|Z_i(c,uv)|\geq \frac{\alpha n^2}{2}$. %Let $$Z(c,uv):=\bigcup_{i\in [\lambda n]}Z_i(c,uv).$$ 
Thus, we have $|Z(c,uv)\cap \mathcal{F}_i|\geq \frac{\alpha n^2}{2}$ for at least $(\lambda -2d)n$ indices $i\in [\lambda n]$. Applying Lemma \ref{LEMMA:directed-k-graph-matching} with 
$$\mathbf{H}:=\{\mathcal{F}_i:i\in [\lambda n]\},\,\,\,
\mathbf{Z}:=\{Z(c,uv) : c\in [n],(u, v)\in V_1\times V_2\,\,\text{and}\,\, v\,\,\text{is}\,\,c\textrm{-good}\},$$ 
and parameters $t:=\lambda n$, $\epsilon:=\frac{\alpha}{2}$, to obtain a rainbow matching $\mathcal{M}$ inside $\mathbf{H}$ of size at least $(1-\frac{\alpha^2}{16})\lambda n$ such that $|E(Z(c,uv))\cap E(\mathcal{M})|\geq \frac{\alpha^2\lambda n}{16}$ for all $Z(c,uv)\in \mathbf{Z}$. Therefore, there exists a subset $I\subseteq [\lambda n]$ of size at least $(1-\frac{\alpha^2}{16})\lambda n$ such that for each $i\in I$, there is an edge $w_iw'_i$ in $\mathcal{M}$ with $i\in L(w_iw'_i)$. For every $c\in [n]$ and $(u, v)\in V_1\times V_2$ satisfying $v$ is $c$-good, there are at least $\frac{\alpha^2\lambda n}{16}$ matching edges in $\mathcal{M}$ which are $c$-absorbing edges of $(u,v)$.
By construction, $\mathcal{M}$ is an absorbing matching with parameters $(1,0,\lambda,\frac{\alpha^2\lambda}{16})$ and hence with parameters $(1,0,\lambda,\lambda^2)$ since $\lambda\ll \alpha$.
\end{proof}

Next, we demonstrate that an absorbing matching can be found when $\mathcal{G}$ is weakly stable.

\begin{lemma}\label{weakly-matching}
    Let $0<\frac{1}{n}\ll d\ll \lambda,\mu\ll \epsilon\ll \delta <1$. Suppose that $\mathcal{G}=\{G_1,\ldots,G_n\}$ is a collection of bipartite graphs on a common vertex partition $V=V_1\cup V_2$ with $|V_1|=|V_2|=n$ such that the following conditions hold:
   \begin{itemize}
    \item for every $i\in [n]$, $d_{G_i}(x)\geq \left(\frac{1}{2}-\mu\right)n$ for all but at most $dn$ vertices $x\in V_1\cup V_2$, 
    \item for every $x\in V_1\cup V_2$, $d_{G_i}(x)\geq \left(\frac{1}{2}-\mu\right)n$ for all but at most $dn$ colors $i\in [n]$.
\end{itemize}If $\mathcal{G}$ is $(\epsilon,\delta)$-weakly stable, then there exists an absorbing matching with parameters $(\frac{\delta}{2},\sqrt{\epsilon},\lambda,\lambda^2)$.
\end{lemma}
\begin{proof}
For
each $i\in [n]$, we define 
$$
\mathcal{F}_i:=\{(u,v)\in V_1\times V_2 : i\in L(uv)\}.
$$
Then $|\mathcal{F}_i|> \frac{n^2}{4}$. We first prove the following claim, which only requires the degree condition on $\mathcal{G}$.

\begin{claim}\label{weakly-absorbing-matching}
Assume $ij\in E(C_{\mathcal{G}}^{\epsilon, \delta})$. Let $u\in V_1$ be a $j$-good vertex and $v\in V_2$ be an $i$-good vertex.
Define 
$$Z_j(i,uv):=\{(w,w')\in \mathcal{F}_{j} :  ww'\,\, \text{is an}\,\, i\text{-absorbing edge
of}\,\, (u,v)\}.$$
Then $|Z_j(i,uv)|\geq 2^{-4}\delta n^2$.
\end{claim}

\begin{proof}[Proof of Claim \ref{weakly-absorbing-matching}]
For each $\ell\in [n]$, if $G_{\ell}$ is $\epsilon$-extremal, then by Lemma \ref{LEMMA:MATCHING-CONS}, there is a characteristic partition $(A_1^{\ell},B_1^{\ell},C_1^{\ell},A_2^{\ell},B_2^{\ell},C_2^{\ell})$ corresponding to $G_{\ell}$. Fix any such indices $i,j$ and vertices $u,v$.  Since $v$ is $i$-good, we have 
$$d_{G_i}(v,A^i_1)\geq \left(\frac{1}{2}-2\epsilon\right)n\ \ \text{or}\ \  d_{G_i}(v,B^i_1)\geq \left(\frac{1}{2}-2\epsilon\right)n.$$ Assume, without loss of generality, that the former case holds (the latter case is analogous). Since  $G_i$ and $G_j$ are $\delta$-crossing and $\epsilon\ll \delta$, one has 
 $$|A^i_1\cap  A^j_1|\geq \frac{\delta n}{2},\ |A^i_1\cap B^j_1|\geq \frac{\delta n}{2},\  |A_1^j\cap B_1^i|\geq \frac{\delta n}{2},\ \ \text{and}\ \ |B_1^i\cap B_1^j|\geq \frac{\delta n}{2}.$$
Combining $|A^i_1\cap  A^j_1|\geq \frac{\delta n}{2}$ with  $d_{G_i}(v,A^i_1)\geq (\frac{1}{2}-2\epsilon)n$, we obtain
$$d_{G_i}(v,A^j_1)\geq |A^i_1\cap  A^j_1|-\left(|A^i_1|-d_{G_i}(v,A^i_1)\right)\geq \frac{\delta n}{2}-\epsilon n\geq \frac{\delta n}{3}.$$
Similarly, we have $d_{G_i}(v,B^j_1)\geq\frac{\delta n}{3}$. Recall that $u\in V_1$ and $u$ is $j$-good, which implies that $u\in X^j_1$ for some $X\in \{A,B\}$. Since $d_{G_j}(u, X_2^j)\geq (\frac{1}{2}-2\epsilon)n$, there are at least $(\frac{1}{2}-2\epsilon)n$ choices for $w'\in X_2^j$.
Notice that in $G_j$, every vertex in $X^j_2$ is adjacent to all but at most $\epsilon n$ vertices in $X^j_1$. Fixing $w'\in X_2^j$, and using $d_{G_i}(v,X^j_1)\geq\frac{\delta n}{3}$, we conclude that there are at least $\frac{\delta n}{4}$ choices for $w\in X^j_1$. Thus, 
$$|Z_j(i, uv)|\geq \frac{\delta n}{4}\left(\frac{1}{2}-2\epsilon\right)n\geq 2^{-4}\delta n^2,$$
as desired.
\end{proof}

%$u\in A^j_1\cup B^j_1$. We consider the following two cases.

%{\bf Case 1.} $u\in A^j_1$.

%Since $d_{G_j}(u, A_2^j)\geq (\frac{1}{2}-2\epsilon)n$, there are at least $(\frac{1}{2}-2\epsilon)n$ choices for $w'\in A_2^j$. Notice that in $G_j$, every vertex in $A^j_2$ is adjacent to all but at most $\epsilon n$ vertices in $A^j_1$. Fixing $w'\in A_2^j$, and using $d_{G_i}(v,A^j_1)\geq\frac{\delta n}{3}$, we find that there are at least $\frac{\delta n}{4}$ choices for $w\in A^j_1$. Thus, 
%$$|Z_j(i, uv)|\geq \frac{\delta n}{4}\left(\frac{1}{2}-2\epsilon\right)n\geq 2^{-4}\delta n^2.$$

%{\bf Case 2.} $u\in B^j_1$. 

%Since $d_{G_j}(u, B_2^j)\geq (\frac{1}{2}-2\epsilon)n$, there are at least $(\frac{1}{2}-2\epsilon)n$ choices for $w'\in B_2^j$. Notice that in $G_j$, every vertex in $B^j_2$ is adjacent to all but at most $\epsilon n$ vertices in $B^j_1$. Fixing $w'\in B^j_2$, and using $d_{G_i}(v,B^j_1)\geq\frac{\delta n}{3}$, we find that there are at least $\frac{\delta n}{4}$ choices for $w\in B^j_1$. Thus,
%$$|Z_j(i, uv)|\geq \frac{\delta n}{4}\left(\frac{1}{2}-2\epsilon\right)n\geq 2^{-4}\delta n^2.$$

For any vertex $v\in V$, let
$$
\mathcal{C}_v:=\{i\in [n]:v\ \text{is}\ i\text{-good}\}.
$$
Since $e(C_{\mathcal{G}}^{\epsilon,\delta})\geq \delta n^2$, there exists a subgraph $H$ of $C_{\mathcal{G}}^{\epsilon,\delta}$ such that $|V(H)|\geq \delta n$ and $\delta(H)\geq \delta n$. For each $i\in V(H)$, define the set
$$
T_i:=\left\{v\in V: |N_H(i)\setminus \mathcal{C}_v|\geq \frac{\delta n}{2}\right\}.
$$

Observe that there are at most $4\epsilon n^2$ pairs $(u, i)\in V\times [n]$ such that $u$ is not $i$-good. This implies that
$\frac{\delta n}{2}|T_i|\leq 4\epsilon n^2$. Hence, by $\epsilon\ll \delta$, we have
$|T_i|\leq \frac{\sqrt{\epsilon}n}{2}$. 

Let $\overline{T_i}:=V\setminus T_i$. Clearly, $|\overline{T_i}|\geq 2n-\frac{\sqrt{\epsilon}n}{2}$.  For each $i\in V(H)$ and $u\in \overline{T_i}$, we have $|\mathcal{C}_u\cap N_H(i)|\geq \frac{\delta n}{2}$. 
Now we independently and randomly select vertices from $V(H)$ with probability $\kappa:=\frac{\lambda}{14}$ to
obtain a set $\mathcal{U}$ of colors. A Chernoff bound implies that, with high probability,
\begin{enumerate}
    \item[{\rm (i)}] $\frac{\kappa |V(H)|}{2}\leq t:=|\mathcal{U}|\leq 2\kappa |V(H)|$,
    \item[{\rm (ii)}] for every $i\in V(H)$ and $u\in \overline{T_i}$, we have $|\mathcal{C}_u\cap N_H(i,\mathcal{U})|\geq \frac{\delta \kappa n}{4}$.
\end{enumerate}
Fix such a set $\mathcal{U}$ and let $\overline{\mathcal{U}}:=V(H)\setminus \mathcal{U}$. By relabeling colors, one may assume $\mathcal{U}=:[t]$. Clearly,  $$|\overline{\mathcal{U}}|=|V(H)|-|\mathcal{U}|\geq \frac{\delta n}{2},\ \  \text{and}\ \ \delta(H[\overline{\mathcal{U}}])\geq \delta(H)-|\mathcal{U}|\geq\frac{\delta n}{2}.$$ 
Define the multi-graph collection 
$$
  \mathbf{Z}:=\left\{Z(i,uv):=\bigcup_{j\in \mathcal{C}_u\cap N_H(i,\mathcal{U})}Z_j(i,uv):i\in \overline{\mathcal{U}}, u\in \overline{T_i}\ \text{and}\ v\ \text{is}\ i\text{-good}\right\}.
$$
Given $i$ and $u$, the number of choices for $j$ is at least $\frac{\delta\kappa n}{4}$ by (ii). Since $ij\in E(H)$, $u$ is $j$-good and $v$ is $i$-good, Claim \ref{weakly-absorbing-matching} implies that there are at least $2^{-4}\delta n^2$ $i$-absorbing edges of $(u,v)$ whose ordered vertex set is in $\mathcal{F}_{j}$. Thus, $|Z(i, uv)\cap \mathcal{F}_j|\geq 2^{-4}\delta n^2$ for at least $\frac{\delta\kappa n}{4}\geq 2^{-3}t$ indices $j\in [t]$. Applying Lemma \ref{LEMMA:directed-k-graph-matching} with 
\begin{center}
$\mathbf{H}:=\{\mathcal{F}_j : j\in [t]\}$, $\mathbf{Z}$ and $\epsilon:=2^{-4}\delta$,
\end{center}
we can find a rainbow matching $\mathcal{M}$ in $\mathbf{H}$ of size at least $(1-2^{-10}\delta^2)t$ such that $|E(Z(i,uv))\cap E(\mathcal{M})|\geq 2^{-10}\delta^2t$ for all $Z(i,uv)\in \mathbf{Z}$. That is, there is $I\subseteq [t]$ with $|I|\geq (1-2^{-10}\delta^2)t$ such that for each $j\in I$ there is an edge $e_j\in E(\mathcal{M})$ with $j\in L(e_j)$. For every $i\in \overline{\mathcal{U}}$, $u\in \overline{T_i}$ and $v\in V_2$ is $i$-good, there are at least $2^{-10}\delta^2 t$ indices $j\in I$ such that $e_j$ is an $i$-absorbing edge of $(u,v)$.
Thus, $\mathcal{M}$ is an absorbing matching with parameters $(\frac{\delta}{2}, \sqrt{\epsilon}, (1-2^{-10}\delta^2)t, 2^{-10}\delta^2 t)$ and hence with $(\frac{\delta}{2},\sqrt{\epsilon},\lambda,\lambda^2)$.
\end{proof}

Now, we are ready to prove Theorem \ref{stable-matching}.
%\begin{lemma}\label{stable-matching}
%Let $0<\frac{1}{n}\ll \mu\ll \alpha\ll \gamma,\epsilon\ll \delta\ll1$. Suppose that $\mathcal{G}=(G_1,\ldots,G_n)$ is a digraph collection on a common vertex set $V=V_1\cup V_2$ of size $2n$ with $\delta(\mathcal{G})\geq \left(\frac{1}{2}-\mu\right)n$. If $\mathcal{G}$ is $(\gamma,\alpha,\epsilon,\delta)$-stable, then $\mathcal{G}$ contains a transversal perfect matching.
%\end{lemma}
\begin{proof}[Proof of Theorem \ref{stable-matching}]
Choose additional parameters $\beta,\lambda$ and $\mu'$ such that 
$$
  1<\frac{1}{n}\ll d\ll \mu\ll \beta\ll\lambda\ll \mu'\ll \alpha\ll \gamma,\epsilon\ll \delta\ll1,
$$
where the previous lemmas in this section hold with these suitable parameters.

By Lemma \ref{strongly-matching} and Lemma \ref{weakly-matching}, $\mathcal{G}$ contains an absorbing matching $\mathcal{M}$ with parameters $(\frac{\delta}{2},\sqrt{\epsilon},$ $\lambda,\lambda^2)$. For any color $c\in [n]$ and vertices $u \in V_1$, $v\in V_2$, we define the triple $(c,u,v)$ as absorbable if there are at least $\lambda^2n$ disjoint $c$-absorbing edges of $(u,v)$ inside $\mathcal{M}$. 
By the definition of absorbing matching, one has $|\mathcal{M}|\leq \lambda n$. Furthermore, there exists a color set $\mathcal{C}\subseteq [n]\setminus {\rm col}(\mathcal{M})$ of size at least $\frac{\delta n}{2}$ such that, for any given color $c\in \mathcal{C}$ and $c$-good vertex $v\in V_2$, the triple $(c,u,v)$ is absorbable for all but at most $\sqrt{\epsilon}n$ vertices $u\in V_1$.

\begin{claim}\label{claim:totally-absorbable-matching}
  There exists an integer $r$ with $(2-2^{-10})\beta n \leq r \leq 2\beta n$. For each $i \in [r]$, there is a tuple $q_i:=(c_i^1, c_i^2, w_i^1, w_i^{1'}, w_i^2, w_i^{2'})$ where $c_i^1, c_i^2 \in \mathcal{C}$ are distinct colors, $w_i^1, w_i^{2} \in V_1$ and $w_i^{1'}, w_i^{2'} \in V_2$ are distinct vertices, satisfying the  following properties:
  \begin{enumerate}
    \item[{\rm (i)}] $(c_i^1, w_i^1, w_i^{1'})$ and $(c_i^2, w_i^2, w_i^{2'})$ are absorbable for all $i \in [r]$,
    \item[{\rm (ii)}] for every pair $(u, v) \in V_1 \times V_2$ and every color $c \in [n]$, there are at least $2^{-20}\beta n$ indices $i \in [r]$ such that $c_i^1 \in L(uw_i^{1'})$, $c_i^2 \in L(vw_i^{2})$ and $c \in L(w_i^1w_i^{2'})$.
  \end{enumerate}
\end{claim}
\begin{proof}[Proof of Claim \ref{claim:totally-absorbable-matching}]
For every two distinct  colors $c_i^1, c_i^2\in \mathcal{C}$, and vertex pair $(u, v)\in V_1\times V_2$ such that $d_{G_{c_i^1}}(u)\geq \left(\frac{1}{2}-\mu\right) n$ and $d_{G_{c_i^2}}(v)\geq \left(\frac{1}{2}-\mu\right) n$, and for every color $c\in [n]$, let $S(c_i^1, c_i^2, u, v, c)$ be the collection of tuples 
$$(w_i^1, w_i^{1'}, w_i^2, w_i^{2'})\in V_1\times V_2 \times V_1\times V_2$$
that satisfy the following properties: $(c_i^1, w_i^1, w_i^{1'})$ and $(c_i^2, w_i^2, w_i^{2'})$ are absorbable, $c_i^1 \in L(uw_i^{1'})$, $c_i^2 \in L(vw_i^{2})$, and $c \in L(w_i^1w_i^{2'})$. We will prove that $|S(c_i^1,c_i^2,u,v,c)|\geq 2^{-7}n^4$. For this, we count the number of possible choices for $w_i^{1'},\,w_i^{2'},\,w_i^1$ and $w_i^2$ in turn. 

$\bullet$ Recall that all but at most $2\epsilon n$ vertices in $V_2$ are $c_i^1$-good and $d_{G_{c_i^1}}(u)\geq \left(\frac{1}{2}-\mu\right)n$. Therefore, there are at least $
   \left(\frac{1}{2}-\mu\right) n-2\epsilon n-1\geq \frac{n}{3} 
$ 
choices for $w_i^{1'}$ such that $w_i^{1'}$ is $c_i^1$-good and $c_i^1\in L(uw_i^{1'})$.

$\bullet$ For vertices in $V_2$, there are at most $2\epsilon n+dn$ vertices $w$ that are either not $c_i^2$-good or $d_{G_c}(w)< \left(\frac{1}{2}-\mu\right)n$. Thus, we have at least 
$
n-2\epsilon n-dn-2\geq \frac{n}{2}
$ 
choices for $w_i^{2'}$ such that $w_i^{2'}$ is $c_i^2$-good and $d_{G_c}(w_i^{2'})\geq \left(\frac{1}{2}-\mu\right)n$.

$\bullet$ Fix $w_i^{1'}$ and $w_i^{2'}$. Since $w_i^{1'}$ is $c_i^1$-good and $c_i^1\in \mathcal{C}$, we obtain that $(c_i^1, w, w_i^{1'})$ is absorbable for all but at most $\sqrt{\epsilon}n$ vertices $w\in V_1$. Recall that $d_{G_c}(w_i^{2'})\geq\left(\frac{1}{2}-\mu\right)n$. Therefore, there are at least 
$
\left(\frac{1}{2}-\mu\right) n- \sqrt{\epsilon}n-1\geq \frac{n}{3}
$ 
choices for $w_i^{1}$ such that $(c_i^1, w_i^1, w_i^{1'})$ is absorbable and $c\in L(w_i^1w_i^{2'})$.

$\bullet$  Similarly, since $d_{G_{c_i^2}}(v)\geq \left(\frac{1}{2}-\mu\right)n$, $c_i^2\in \mathcal{C}$ and $w_i^{2'}$ is $c_i^2$-good, there are at least 
$
\left(\frac{1}{2}-\mu\right) n-\sqrt{\epsilon}n-2\geq \frac{n}{3}
$ 
choices for $w_i^{2}$ such that $c_i^2\in L(vw_i^2)$ and $(c_i^2, w_i^2, w_i^{2'})$ is absorbable.

Therefore, $|S(c_i^1, c_i^2, u, v, c)|\geq 2^{-7}n^4$, as desired.

Let $\{c_i^1,c_i^2:i\in [2\beta n]\}\subseteq \mathcal{C}$ be a set of distinct colors. %For each $i\in [2\beta n]$, set $\{c_i^1, c_i^2\}:= \{c_{2i-1}, c_{2i}\}$ and d
Define
$$\mathcal{F}_i:=\{(w_i^1, w_i^{1'}, w_i^2, w_i^{2'})\in V_1\times V_2 \times V_1\times V_2 : (c_i^1, w_i^1, w_i^{1'})~ \text{and}~(c_i^2, w_i^2, w_i^{2'})~\text{are absorbable}\}.$$
Since for each color $c\in \mathcal{C}$, there are at most $2\epsilon n$ vertices in each part that are not $c$-good. Hence, there exist at least 
$(n-2\epsilon n)(n-\sqrt{\epsilon}n)\geq \frac{n^2}{2}$ 
vertex pairs $(x, y)\in V_1\times V_2$ such that $(c, x, y)$ is absorbable. Thus, $|\mathcal{F}_i|\geq 2^{-4}n^4$. 
Let $\mathbf{H}:=\{\mathcal{F}_i : i\in [2\beta n]\}$ be the collection of directed 4-graphs. 

For a vertex pair $(u, v)\in V_1\times V_2$, let 
$$I_{u,v}:=\left\{i\in [2\beta n] : d_{G_{c_i^1}}(u)\geq  \left(\frac{1}{2}-\mu\right) n\ \  \text{and}\ \  d_{G_{c_i^2}}(v)\geq  \left(\frac{1}{2}-\mu\right) n\right\}.$$
Define the collection of directed multi-4-graphs as follows:
$$
  \mathbf{Z}:=\left\{S(u,v,c):=\bigcup_{i\in I_{u, v}}S(c_i^1, c_i^2, u, v, c) : (u,v)\in V_1\times V_2,\, c\in [n]\right\}.
$$
Since for each vertex $w\in V$,  there are at most $dn$ indices $i\in [n]$ such that $d_{G_i}(w)<\left(\frac{1}{2}-\mu\right) n$. Hence  for every $S:=S(u, v, c)\in \mathbf{Z}$, we have $|E(S)\cap E(\mathcal{F}_i)|\geq 2^{-7}n^4$ for all but at most $2dn$ indices $i\in [2\beta n]$. By applying Lemma \ref{LEMMA:directed-k-graph-matching} with $t:=2\beta n$ and $\epsilon:=2^{-7}$,  we deduce that there is a rainbow matching $M$ in $\mathbf{H}$ of size at least $(2-2^{-10})\beta n$ (and at most $2\beta n$) and $|E(S)\cap E(M)|\geq 2^{-20}\beta n$ for all $S\in \mathbf{Z}$.
\end{proof}

Let 
\begin{align*}
  &V_{{\rm abs}}:=\bigcup_{i\in [r]}\{w_i^1,w_i^{1'}, w_i^2, w_i^{2'}\},\ \ \mathcal{C}_{{\rm abs}}:=\bigcup_{i\in [r]}\{c_i^1, c_i^2\},\ \
  \mathcal{C}_{{\rm rem}}:=[n]\setminus ({\rm col}(\mathcal{M})\cup \mathcal{C}_{{\rm abs}}), \\
 &U:=V\setminus (V(\mathcal{M})\cup V_{{\rm abs}}),\ \ \mathcal{J}:=\{J_i:i\in \mathcal{C}_{{\rm rem}}\}\ \text{where} \ J_i=G_i[U].
\end{align*}
Note that $|[n]\setminus \mathcal{C}_{{\rm rem}}|=\frac{|V\setminus U|}{2}\leq \lambda n+4\beta n$. 
%For each color $c\in \mathcal{C}_{\rm rem}$ for all but at most $dn$ vertices in $w\in U$ such that 
%$$d_{J_c}(v)\geq \left(\frac{1}{2}-\mu\right) n-\lambda n-2\beta n\geq \frac{|U|}{4}-\mu' \frac{|U|}{12}.$$
For each vertex $v\in U$, for all but at most $dn$ colors $c\in \mathcal{C}_{\rm rem}$ we have 
$$d_{J_c}(v)\geq \left(\frac{1}{2}-\mu\right) n-\lambda n-4\beta n\geq \frac{|U|}{4}-\mu' \frac{|U|}{12}.$$
%, and $\delta(\mathcal{J})\geq (\frac{1}{2}-\mu-2\lambda)n\geq (\frac{1}{2}-3\lambda)n$. 
Since $\beta ,\lambda, \mu'\ll \gamma, \alpha, \epsilon, \delta$, it is easy to see that $\mathcal{J}$ is $(\frac{\gamma}{2}, \frac{\alpha}{2}, 2\epsilon, \frac{\delta}{2})$-stable. By Lemmas  \ref{LEMMA:stable-nice-matching} and \ref{LEMMA:N-2-MATCHING}, we deduce that $\mathcal{J}$ is $\mu'$-nice and there exists a rainbow matching $\mathcal{M}_1$ in $\mathcal{J}$ with size at least $|\mathcal{C}_{{\rm rem}}|-dn-2$. 
Assume that 
    $$(U\setminus V(\mathcal{M}_1))\cap V_1=\{u_1,\ldots,u_t\},\,\, (U\setminus V(\mathcal{M}_1))\cap V_2=\{v_1,\ldots,v_t\}, \,\,\mathcal{C}_{\rm rem}\setminus {\rm col}(\mathcal{M}_1)=\{c_1,\ldots, c_t\}.$$

%Now, we consider vertices in $U\setminus V(\mathcal{M}_1)$.
Do the following for each 
$(u_i,v_i,c_i)$ with  $i=1,2,\ldots,t$ in turn. Choose an unused 6-tuple $q_{j_i}=(c_{j_i}^1, c_{j_i}^2, w_{j_i}^1, w_{j_i}^{1'}, w_{j_i}^2, w_{j_i}^{2'})$ for some $j_i \in [r]$, such that $c_{j_i}^1 \in L(u_i w_{j_i}^{1'})$, $c_{j_i}^2 \in L(v_i w_{j_i}^{2})$, and $c_i \in L(w_{j_i}^1 w_{j_i}^{2'})$. This is possible since Claim \ref{claim:totally-absorbable-matching} implies that there are at least $2^{-20}\beta n$ choices for $q_{j_i}$, of which at most $dn+2$ have been used. Thus, all vertices in $U\setminus V(\mathcal{M}_1)$ and all colors in $\mathcal{C}_{\rm rem}\setminus {\rm col}(\mathcal{M}_1)$ are contained in a rainbow matching of $\mathcal{G}$. % we can choose one such that its vertices have not been previously chosen.

At the end of this process, there remains a set $I\subseteq [r]$ such that the $q_i$ with {$i\in I$} are precisely the 6-tuples that were not chosen. Recall that $(c_i^1, w_i^1, w_i^{1'})$ and $(c_i^2, w_i^2, w_i^{2'})$ are absorbable for each $i\in I$. It follows that there are at least $\lambda^2n$ disjoint $c_i^j$-absorbing edges of $(w_i^j, w_i^{j'})$ inside $\mathcal{M}$ for $i\in I$ and $j\in [2]$. Together with $\beta\ll \lambda$, we obtain a transversal perfect matching inside $\mathcal{G}$.
%For each of these,  there are at least $\frac{\lambda^2 n}{2}$ disjoint color pair $\{c_i^1, c_i^2\}\in {\rm col}(\mathcal{M})\times {\rm col}(\mathcal{M})$ such that both triples $(c_i^1, w_i^1, w_i^{1'})$ and $(c_i^2, w_i^2, w_i^{2'})$ are absorbable. 
\end{proof}

\end{appendices}
\end{document}